\documentclass[twoside,platex,a4paper]{article}
\usepackage{amsthm,amsmath,amssymb,amscd,graphicx,amsfonts,stmaryrd}
\usepackage{mathrsfs}
\usepackage[all]{xy}
\usepackage[margin=25mm]{geometry}
\usepackage{bigints}
\usepackage{mathtools}
\usepackage{enumitem}
\usepackage{comment}
\usepackage{fancyhdr}
\usepackage{latexsym}
\usepackage{amsfonts}
\usepackage{bm}
\usepackage[d]{esvect}
\usepackage{array}
\usepackage{tikz}
\usetikzlibrary{positioning, intersections, calc, arrows.meta,math}

\usepackage{hyperref}

\DeclareMathOperator{\parent}{par}
\DeclareMathOperator{\gparent}{gpar}
\DeclareMathOperator{\DFprec}{\prec_{\tiny{DF}}}
\DeclareMathOperator{\hght}{ht}
\DeclareMathOperator{\Leb}{Leb}

\DeclareMathOperator{\dom}{dom}
\DeclareMathOperator{\idty}{id}
\DeclareMathOperator{\cl}{cl}

\allowdisplaybreaks[1]
\thepage

\mathtoolsset{showonlyrefs=true}

\numberwithin{equation}{section}

\theoremstyle{definition}
\newtheorem{exm} {Example}[section]
\newtheorem{dfn}[exm] {Definition}
\newtheorem{rem}[exm] {Remark}
\newtheorem{const}[exm]{Construction}

\theoremstyle{plane}
\newtheorem{lem}[exm]{Lemma}
\newtheorem{prop}[exm]{Proposition}
\newtheorem{thm}[exm]{Theorem}
\newtheorem{cor}[exm]{Corollary}
\newtheorem{assum}[exm]{Assumption}

\newcommand{\erdosrenyi}{Erd\H{o}s-R\'{e}nyi}
\newcommand{\mbR}{\mathbb{R}}
\newcommand{\mbRp}{\mathbb{R}_{+}}
\newcommand{\mbM}{\mathbb{M}}
\newcommand{\cX}{\mathcal{X}}
\newcommand{\cY}{\mathcal{Y}}

\title{Convergence of local times of stochastic processes associated with resistance forms} 
\date{}
\author{Ryoichiro Noda\thanks{Research Institute for Mathematical Sciences, Kyoto University, Kyoto, 606-8502,
JAPAN. E-mail:sgrndr@kurims.kyoto-u.ac.jp}}

\begin{document}

\maketitle
\begin{abstract}
  In this paper,
  it is shown that 
  if a sequence of resistance metric spaces equipped with measures converges with respect to the local Gromov-Hausdorff-vague topology,
  and certain non-explosion and metric-entropy conditions are satisfied,
  then the associated stochastic processes and their local times also converge.
  The metric-entropy condition can be checked by applying volume estimates of balls.
  Whilst similar results have been proved previously, the approach of this article is more widely applicable.
  Indeed, 
  we recover various known conclusions for scaling limits of some deterministic self-similar fractal graphs,
  critical Galton-Watson trees,
  the critical Erd\H{o}s-R\'{e}nyi random graph and the configuration model 
  (in the latter two cases, we prove for the first time the convergence of the models with respect to the resistance metric
  and also, for the configuration model, we overcome an error in the existing proof of local time convergence).
  Moreover,
  we derive new ones for scaling limits of uniform spanning trees and random recursive fractals.
  The metric-entropy condition also implies convergence of associated Gaussian processes.
\end{abstract}

\tableofcontents


\section{Introduction} \label{sec: introduction}
It is natural to attempt to construct a stochastic process on a continuous medium as a scaling limit of random walks on corresponding discrete media,
such as Brownian motion in Euclidean space as a scaling limit of random walks on lattices.
In this paper, we deal with scaling limits of stochastic processes and associated local times,
assuming that the underlying spaces are equipped with resistance metrics (see Section \ref{sec: introduction to resistance metrics}).
This can be interpreted as the setting of ``low-dimensional'' disordered media,
including fractals such as the Sierpi\'{n}ski gasket and tree-like metric spaces.
A resistance metric characterizes the electrical energy,
a corresponding bilinear form and,
combined with a measure on the space (and under certain technical conditions),
determines uniquely a Dirichlet form and a stochastic process on the space.
In \cite{Croydon_Hambly_Kumagai_17_Time-changes},
it was shown that
if a sequence of spaces equipped with resistance metrics and measures
converge with respect to the local Gromov-Hausdorff-vague topology (see Section \ref{subsec: the GHV}),
and a uniform volume doubling (UVD) condition
(see \cite[Definition 1.1]{Croydon_Hambly_Kumagai_17_Time-changes}) is satisfied,
then the associated stochastic processes and local times also converge. One should note, however,
that the UVD condition is too strong for many sequences of random graphs.
In \cite{Croydon_18_Scaling},
the UVD condition was relaxed and the convergence of stochastic processes was established
under a weaker non-explosion condition,
which enabled a wider range of examples to be handled.
In particular, when one considers a sequence of compact spaces,
no volume condition is needed for the convergence of stochastic processes.
However in the more general setting of \cite{Croydon_18_Scaling},
convergence of local times was left open. In the subsequent work \cite{Andriopoulos_23_Convergence},
convergence of local times of stochastic processes on some random graphs was obtained,
but the arguments of that paper strongly relied on specific properties of the specific random graphs considered.
This article aims to establish convergence of local times in much greater generality.

For the setting and notation of the present work,
we closely follow \cite{Croydon_18_Scaling}.
For a metric space $(S,d)$, we set $B_{d}(x,r)\coloneqq \{y \in S : d(x,y) < r\}$ 
and $D_{d}(x,r)\coloneqq \{y \in S : d(x,y) \leq r\}$.
We say that $(S,d)$ is \textit{boundedly compact} 
if every closed bounded set in $(S,d)$ is compact (note this implies $(S,d)$ is complete, separable and locally compact).
A tuple $G=(S, d, \rho, \mu)$ is called a \textit{rooted-and-measured boundedly-compact metric space} 
if and only if $(S, d)$ is a boundedly-compact metric space,
$\rho$ is an element of $S$ called the \textit{root} 
and $\mu$ is a Radon measure on $(S,d)$,
that is, $\mu$ is a Borel measure with $\mu(K) < \infty$ for every compact subset $K \subseteq S$.
For each $r>0$,
we define $G^{(r)} = (S^{(r)}, d^{(r)}, \rho^{(r)}, \mu^{(r)})$
by setting 
\begin{equation}  \label{1. eq: dfn of restriction operator}
  S^{(r)} \coloneqq \cl (B_{d}(S, r)),
  \quad 
  d^{(r)} \coloneqq d|_{S^{(r)} \times S^{(r)}},
  \quad 
  \rho^{(r)} \coloneqq \rho, 
  \quad 
  \mu^{(r)}(\cdot) \coloneqq \mu(\cdot \cap S^{(r)}),
\end{equation}
where $\cl(\cdot)$ denotes the closure of a set.
We write $\mathbb{G}$ for the collection of 
rooted-and-measured isometric equivalence classes of rooted-and-measured boundedly-compact metric spaces,
and equip $\mathbb{G}$ with the local Gromov-Hausdorff-vague topology.
We define $\mathbb{G}_{c}$ to be the collection of $(S, d, \rho, \mu) \in \mathbb{G}$
such that $(S, d)$ is compact.
We then equip $\mathbb{G}_{c}$ and $\mathbb{G}$
with the Gromov-Hausdorff-Prohorov topology
and the local Gromov-Hausdorff-vague topology, respectively.
(See Section \ref{subsec: the GHV} for details).

\begin{dfn} [{The space $\mathbb{F}$ and $\mathbb{F}_{c}$}]
  We define the subspace $\mathbb{F}$ of $\mathbb{G}$
  to be the collection of $(F,R,\rho,\mu) \in \mathbb{G}$
  such that $\mu$ is of full support and $R$ is a resistance metric 
  which is associated with a regular resistance form and satisfies
  \begin{equation} \label{1. dfn eq: recurrence of R in the space F}
    \lim_{r \to \infty} R(\rho, B_{R}(\rho, r)^{c}) = \infty. 
  \end{equation}
  We write $\mathbb{F}_{c}$ for the subspace of $\mathbb{G}_{c}$ 
  consisting of $(F, R, \rho, \mu) \in \mathbb{F}$ 
  such that $(F, R)$ is compact.
  (For the definitions of resistance metric and regular resistance form, 
  see Definition \ref{4. dfn: resistance forms} and \ref{dfn: regular resistance forms}.)
  We equip $\mathbb{F}$ and $\mathbb{F}_{c}$ with the relative topologies 
  induced from $\mathbb{G}$ and $\mathbb{G}_{c}$, respectively.
\end{dfn}

\begin{rem}
  There is a follow-up paper \cite{Noda_pre_Scaling}.
  In \cite[Section 3]{Noda_pre_Scaling},
  resistance forms and associated Dirichlet spaces are studied
  and
  it is shown that if a resistance metric satisfies \eqref{1. dfn eq: recurrence of R in the space F},
  then the associated resistance form is regular
  (see Remark \ref{4. rem: recurrence implies regularity} below).
\end{rem}

Let $G=(F, R, \rho, \mu)$ be an element of $\mathbb{F}$.
Note that $G^{(r)}$ defined in \eqref{1. eq: dfn of restriction operator} 
belongs to $\mathbb{F}_{c}$ 
(see \cite[Lemma 2.6]{Croydon_18_Scaling}).
Let $(\mathcal{E}, \mathcal{F})$ be a resistance form corresponding to $R$.
The regularity of $(\mathcal{E}, \mathcal{F})$ ensures 
the existence of a related regular Dirichlet form $(\mathcal{E},\mathcal{D})$ on $L^{2}(F,\mu)$ 
and also an associated Hunt process $((X_{G}(t))_{t \geq 0}, (P_{x}^{G})_{x \in F})$,
which is recurrent by the condition \eqref{1. dfn eq: recurrence of R in the space F}.
In order to state one of our main assumptions,
we require the following notion from metric geometry.

\begin{dfn} [{Metric entropy}]\label{dfn: epsilon-net}
  Let $(S,d)$ be a compact metric space.
  For $\varepsilon > 0$, a subset $A$ is called an \textit{$\varepsilon$-covering} of $(S,d)$,
  if for each $x \in S$ there exists $a \in A$ such that $d(x,a) \leq \varepsilon$.
  We define $N_{d}(S,\varepsilon)$ by setting
  \begin{equation}
    N_{d}(S,\varepsilon)
    =
    \min \{ |A|: A\ \text{is an}\  \varepsilon \text{-covering of}\ (S,d) \},
  \end{equation}
  where $|A|$ denotes the cardinality of $A$.
  An $\varepsilon$-covering $A$ with $|A|=N_{d}(S,\varepsilon)$ is called a \textit{minimal $\varepsilon$-covering} of $(S,d)$.
  We call the family $(N_{d}(S,\varepsilon) : \varepsilon >0 )$ the \textit{metric entropy} of $(S,d)$.
\end{dfn}

\begin{rem}
  In Definition \ref{dfn: epsilon-net},
  we borrow the definition of metric entropy given in \cite{Marcus_Rosen_06_Markov},
  but one should note that the metric entropy is defined to be the logarithm of $N_{d}(S, \varepsilon)$ elsewhere in the literature.
\end{rem}

In Section \ref{sec: resistance},
we study the joint continuity of local times of a Hunt process associated with a resistance form.
Within this, an important subset of $\mathbb{F}$ is $\check{\mathbb{F}}$ defined below.
In particular, in Corollary \ref{cor: joint continuity of local times with sum condition} it is shown that,
for any $G \in \check{\mathbb{F}}$,
the associated stochastic process $X_{G}$ admits jointly continuous local times 
$L_{G}=(L_{G}(x,t))_{t \geq 0,\, x \in F}$ satisfying the occupation density formula (see \eqref{eq: the occupation density formula}).

\begin{dfn} [{The space $\check{\mathbb{F}}$ and $\check{\mathbb{F}}_{c}$}]
  We define the subspace $\check{\mathbb{F}}$ of $\mathbb{F}$
  to be the collection of $(F, R, \rho, \mu) \in \mathbb{F}$
  such that, for any $r >0$, there exists $\alpha_{r} \in (0,1/2)$ satisfying
  \begin{equation} \label{eq: definition of vF}
    \sum_{k \geq 1} N_{R^{(r)}}(F^{(r)},2^{-k})^{2} \exp(-2^{\alpha_{r} k})  < \infty,
  \end{equation}
  and we also define $\check{\mathbb{F}}_{c}=\mathbb{F}_{c} \cap \check{\mathbb{F}}$.
  Again,
  we equip $\check{\mathbb{F}}$ and $\check{\mathbb{F}}_{c}$
  with the relative topologies induced from $\mathbb{F}$ and $\mathbb{F}_{c}$, respectively.
\end{dfn}

\begin{rem} \label{rem: about vF} \leavevmode
  \begin{enumerate} [label = (\roman*)]
    \item
          If $G$ is such that $F$ is a finite set, then it is obvious that $G$ belongs to $\check{\mathbb{F}}_{c}$.
    \item \label{rem item: a connection between Gaussian processes and vF}
          Condition \eqref{eq: definition of vF} on the metric entropy of the underlying space
          is natural from the point of view of the theory of Gaussian processes,
          being closely related to the condition of Dudley
          that is known to be sufficient for the continuity of a Gaussian process
          (\cite{Dudley_67_The_sizes}, \cite{Dudley_73_Sample}).
          In fact, in Theorem \ref{thm: equivalence of continuity of local times and a Gaussian process},
          we confirm that the joint continuity of the local times and continuity of the corresponding Gaussian process are equivalent
          (the corresponding Gaussian process is a mean zero Gaussian process
          whose covariance function is given by the $1$-potential density of the Hunt process).
          Hence it is to be expected that a similar condition is useful in the study of local times.
          (See Proposition \ref{prop: dudley condition} and
          Corollary \ref{cor: joint continuity of local times with sum condition}
          for further details of the connection between Dudley's condition and local times,
          and also Remark \ref{rem: difference of dudley and mine}
          for a discussion of the difference between Dudley's condition and ours.)
  \end{enumerate}
\end{rem}

For $G=(F,R,\rho,\mu) \in \check{\mathbb{F}}$,
we define 
\begin{equation} \label{eq: def of P_G}
  P_{G}(\cdot)
  \coloneqq
  P_{\rho} 
  \left(
    (X_{G}, L_{G}) \in \cdot
  \right),
\end{equation}
which is a probability measure on $D(\mbRp, F) \times C(F \times \mbRp, \mbR)$,
where $D(\mbRp, F)$ is the space of cadlag functions with values in $F$ equipped with the usual $J_{1}$-Skorohod topology
and $C(F \times \mbRp, \mbR)$ is the space of continuous functions from $F \times \mbRp$ to $\mbR$ 
equipped with the compact-convergence topology.
(Note that we say that functions $f_{n}$ on a topological space $S$
converge to $f$ in the compact-convergence topology 
if and only if the functions $f_{n}$ converge to $f$ on every compact subset of $S$.)
Set 
\begin{equation}
  \cX_{G}\coloneqq (F,R,\rho,\mu, P_{G}).
\end{equation}
In this work, we will consider this object as an element of $\mbM_{L}$,
defined below in Section \ref{sec: the space M_L}.
In particular, $\mbM_{L}$ will be a Polish space 
that provides a framework for discussing convergence of stochastic processes and local times.
The introduction of $\mbM_{L}$, using the recent framework of \cite{Noda_pre_Metrization}, is an important contribution of this work.
Indeed, in the earlier works of \cite{Croydon_Hambly_Kumagai_17_Time-changes, Croydon_18_Scaling},
no framework was provided for studying the convergence of local times on different spaces,
and in \cite{Andriopoulos_23_Convergence}, the focus was restricted to compact spaces, which limited the applications.
Moreover,
convergence of stochastic processes with respect to our topology implies the convergence of those with respect to topology 
used in previous research on scaling limits of stochastic processes on different spaces
(e.g.\ \cite{Croydon_12_Scaling,Athreya_Lohr_Winter_17_Invariance})
(in these papers, 
convergence of distributions of $(F, R, \rho, \mu, X_{G})$ on a space 
consisting of rooted-and-measured metric spaces and cadlag curves was considered).

\begin{figure}[th]
  \centering
  \includegraphics[scale=0.86]{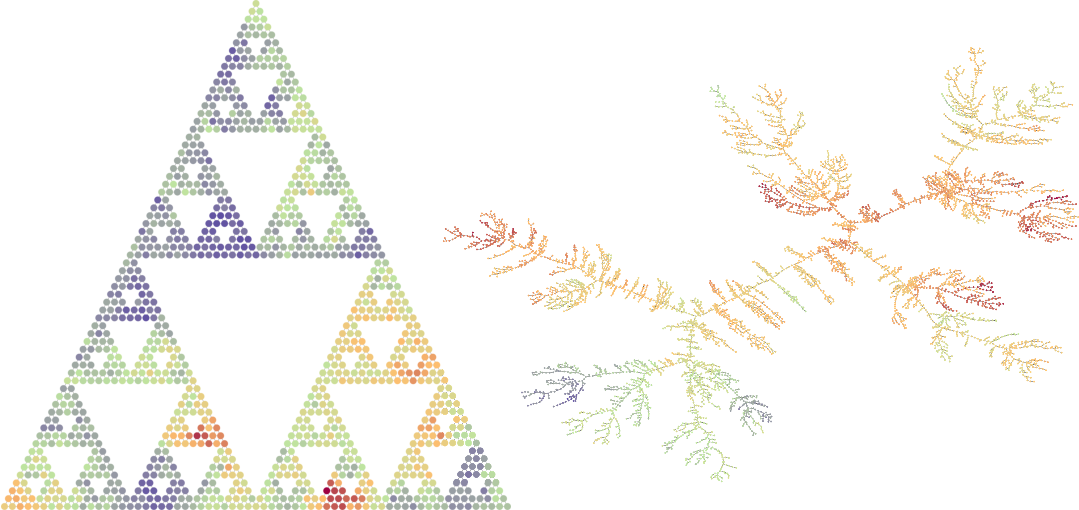}
  \caption{Simulations of local times of constant speed random walks on graphs.
  From least to most visited vertices, colors blend from blue to red.
  The left figure shows the local times of 
  a random walk on a graph approximating the Sierpi\'{n}ski gasket after $5^{6}$ steps.
  The right figure shows those on a Galton-Watson tree with $80^2$ vertices after $1000 \cdot 80^{3}$ steps.}
\end{figure}

In our first main result,
we will assume that we have a sequence $G_{n}=(F_{n}, R_{n},\rho_{n}, \mu_{n})$ in $\check{\mathbb{F}}$
that converges with respect to the local Gromov-Hausdorff-vague topology to an element $G=(F, R, \rho, \mu)$ in $\mathbb{F}$.
Moreover, we will assume the non-explosion condition of \cite{Croydon_18_Scaling}
and a regularity condition on the metric entropies of the spaces in the sequence.
We show that, under such assumptions, it is the case that $\cX_{G_{n}}$ converges to $\cX_{G}$
(see Theorem \ref{1. thm: main result for deterministic spaces}  below).

\begin{assum} \label{1. assum: deterministic spaces}
  \leavevmode
  \begin{enumerate}  [label=(\roman*)]
    \item
          The sequence $G_{n} = (F_{n},R_{n},\rho_{n},\mu_{n})$ in $\check{\mathbb{F}}$ satisfies
          \begin{equation}
            (F_{n},R_{n},\rho_{n},\mu_{n}) \rightarrow (F, R, \rho, \mu)
          \end{equation}
          in the local Gromov-Hausdorff-vague topology for some $G = (F,R,\rho,\mu) \in \mathbb{F}$.
          \label{1. assum item: deterministic, convergence of spaces}
    \item
          It holds that
          \begin{equation}
            \lim_{r \to \infty} \liminf_{n \to \infty} R_{n}(\rho_{n}, B_{R_{n}}(\rho_{n},r)^{c}) =\infty.
          \end{equation}
          \label{1. assum item: deterministic, non-explosion condition}
    \item
          For every $r >0$, there exists $\alpha_{r} \in (0,1/2)$ satisfying
          \begin{equation}
            \lim_{m \to \infty} \limsup_{n \to \infty}
            \sum_{k \geq m} N_{R_{n}^{(r)}}(F_{n}^{(r)}, 2^{-k})^{2} \exp(-2^{\alpha_{r} k})=0.
          \end{equation}
          \label{1. assum item: deterministic, metric-entropy condition}
  \end{enumerate}
\end{assum}

\begin{thm} \label{1. thm: main result for deterministic spaces}
  Under Assumption \ref{1. assum: deterministic spaces},
  the limiting space $G$ belongs to $\check{\mathbb{F}}$, 
  and $\cX_{G_{n}}$ converges to $\cX_{G}$ in $\mbM_{L}$.
\end{thm}

Our second theorem is an annealed version of Theorem \ref{1. thm: main result for deterministic spaces} which is suitable for sequences of random spaces.
Suppose that $G$ is a random element of $\check{\mathbb{F}}$ 
built on a probability space equipped with a complete probability measure $\mathbf{P}$.
It is then the case that $\cX_{G}$ is a random element of $\mbM_{L}$ (see Proposition \ref{prop: measurability of XG}).
To state the main result,
we begin with random elements $G_{n}$ of $\check{\mathbb{F}}$ 
built on a probability space equipped with a complete probability measure $\mathbf{P}_{n}$.
If $G_{n}$ converges to a random element $G$ of $\mathbb{F}$ in distribution
and the sequence $G_{n}$ satisfies annealed versions of 
Assumption \ref{1. assum: deterministic spaces}\ref{1. assum item: deterministic, non-explosion condition}, \ref{1. assum item: deterministic, metric-entropy condition},
then we show that $\mathbf{P}_{n}(\cX_{G_{n}} \in \cdot)$ converges to $\mathbf{P}(\cX_{G} \in \cdot)$ 
(see Theorem \ref{1. thm: main result for random spaces}).

\begin{assum} \label{1. assum: random spaces} \leavevmode
  \begin{enumerate} [label=(\roman*), series=assumptions in random case]
    \item
          The sequence of random elements $G_{n} = (F_{n},R_{n},\rho_{n},\mu_{n})$ of $\check{\mathbb{F}}$ satisfies
          \begin{equation}
            (F_{n},R_{n},\rho_{n},\mu_{n})
            \xrightarrow{\mathrm{d}}
            (F, R, \rho, \mu)
          \end{equation}
          in the local Gromov-Hausdorff-vague topology for some random element $G = (F,R,\rho,\mu)$ of $\mathbb{F}$
          built on a probability space with a complete probability measure $\mathbf{P}$.
          \label{1. assum item: random spaces, space convergence}
    \item
          It holds that
          \begin{equation}
            \lim_{r \to \infty} \liminf_{n \to \infty}
            \mathbf{P}_{n}
            \left( R_{n}(\rho_{n},B_{R_{n}}(\rho_{n},r)^{c}) \geq \lambda \right)=1,
            \quad \forall \lambda >0.
          \end{equation}
          \label{1. assum item: random spaces, non-explosion}
    \item
          For every $r>0$, there exists $\alpha_{r} \in (0,1/2)$ such that
          \begin{equation}
            \lim_{m \to \infty} \limsup_{n \to \infty}
            \mathbf{P}_{n}
            \left(
            \sum_{k \geq m} N_{R_{n}^{(r)}}(F_{n}^{(r)}, 2^{-k})^{2} \exp(-2^{\alpha_{r} k}) \geq \varepsilon
            \right)
            =0,
            \quad \forall \varepsilon >0.
          \end{equation}
          \label{1. assum item: random, metric-entropy condition}
  \end{enumerate}
\end{assum}

\begin{thm} \label{1. thm: main result for random spaces}
  Under Assumption \ref{1. assum: random spaces},
  $G$ is a random element of $\check{\mathbb{F}}$,
  and $\cX_{G_{n}} \xrightarrow{\mathrm{d}} \cX_{G}$ 
  as random elements of $\mbM_{L}$.
\end{thm}

\begin{rem} \leavevmode
  \begin{enumerate} [label=(\roman*)]
    \item 
          In Assumption \ref{1. assum: deterministic spaces} and \ref{1. assum: random spaces},
          we assume that the limiting space belongs to $\mathbb{F}$.
          However,
          it is possible to remove this condition on the limiting space.
          For example,
          in the case of Theorem \ref{1. thm: main result for random spaces},
          if $(F_{n}, R_{n}, \rho_{n}, \mu_{n})$ converges to some random element $G$ of $\mathbb{G}$
          in the local Gromov-Hausdorff-vague topology 
          and Assumption \ref{1. assum: random spaces}\ref{1. assum item: random spaces, non-explosion} 
          and \ref{1. assum item: random, metric-entropy condition} are satisfied,
          then the assertion of Theorem \ref{1. thm: main result for random spaces} still holds.
          For details, see \cite{Noda_pre_Scaling}.
    \item \label{rem item: metric-entropy via volume estimates}
          Checking the metric-entropy condition of 
          Assumption \ref{1. assum: deterministic spaces}\ref{1. assum item: deterministic, metric-entropy condition}
          or Assumption \ref{1. assum: random spaces}\ref{1. assum item: random, metric-entropy condition} 
          directly can be challenging.
          However, it can be verified via suitable volume estimates of balls.
          For example, suppose that we have a sequence $G_{n}=(F_{n}, a_{n}^{-1}R_{n}, \rho_{n}, b_{n}^{-1}\mu_{n})$ in $\check{\mathbb{F}}_{c}$
          where we assume that $F_{n}$ is a finite set, $R_{n}$ is a resistance metric on this, $\mu_{n}$ is the counting measure on it, and $a_{n}, b_{n}$ are scaling factors.
          Convergence of $G_{n}$ in the local Gromov-Hausdorff-vague topology implies a lower estimate on volumes of balls with radius of order at least $a_{n}$ in the metric $R_{n}$.
          The metric-entropy condition requires a lower bound of volumes of balls with radius $a_{n}/(\log b_{n})^{2+\varepsilon}$.
          The details are left to Section \ref{sec: metric entropy and volume estimates}.
    \item 
          A typical example of resistance metric spaces is an electrical network. 
          In this case, 
          the associated stochastic process considered in our main theorems 
          is a continuous-time Markov chain on the electrical network. 
          (The holding times are determined by the weights put on the electrical network.) 
          However, a discrete-time Markov chain is sometimes a more natural object.
          It is possible to establish the same convergence results for discrete-time Markov chains on electrical networks,
          but there are some non-trivial technical issues with proofs.
          The readers interested in discrete-time cases
          refer to \cite{Noda_pre_Scaling}.
    \item
          As might be expected given Remark \ref{rem: about vF}\ref{rem item: a connection between Gaussian processes and vF},
          the metric-entropy condition is also useful for showing convergence of Gaussian processes.
          We provide a precise presentation in Appendix \ref{application to Gaussian processes},
          but the situation is roughly described as follows.
          For each $n \in \mathbb{N} \cup \{ \infty \}$, let $(G_{n}(x))_{x \in F_{n}}$ be a mean zero Gaussian process with a covariance function $\Sigma_{n}$.
          Let $d_{G_{n}}$ be the natural pseudometric on $F_{n}$, that is, $d_{G_{n}}(x,y)=E((G_{n}(x)-G_{n}(y))^{2})^{1/2}$.
          We assume that $d_{G_{n}}$ is a metric on $F_{n}$ and $(F_{n},d_{G_{n}})$ is compact.
          If $(F_{n}, d_{G_{n}}, \Sigma_{n})$ converges to $(F_{\infty}, d_{G_{\infty}}, \Sigma_{\infty})$
          and $(F_{n}, d_{G_{n}})_{n \in \mathbb{N}}$ satisfies the metric-entropy condition,
          then $(F_{n}, d_{G_{n}}, \nu_{n})$ converges to $(F_{\infty},d_{G_{\infty}}, \nu_{\infty})$,
          where $\nu_{n}$ and $\nu_{\infty}$ denote the laws of $G_{n}$ and $G_{\infty}$, respectively.
          This is because 
          the convergence of the covariance functions gives the convergence of the finite dimensional distributions
          and the metric-entropy condition gives the tightness of the processes.
    \item
          As can be seen in the proofs of the main theorems and the above-mentioned argument about convergence of Gaussian processes,
          the metric-entropy condition gives the tightness of the relevant processes.
          Since the condition only depends on the geometries of the index sets of processes,
          we believe that the condition can be used for arguments about scaling limits of other random fields.
  \end{enumerate}
\end{rem}

Via suitable volume estimates, we have confirmed that
the metric-entropy condition holds for a lot of (random-graph) models
and hence the convergence of Markov processes and their local times holds.
The following is the list of such examples and the details are found in Section \ref{sec: examples}.

\begin{itemize}
  \item \textit{Models in the UVD regime}:
        One can check that if spaces satisfy the UVD condition,
        then the metric-entropy condition is satisfied.
        Hence the present work includes all the examples of \cite{Croydon_Hambly_Kumagai_17_Time-changes},
        such as scaling limits of Liouville Brownian motions
        defined on resistance metric spaces (hence, not on $\mbR^{2}$) satisfying the UVD condition.
  \item \textit{A random recursive Sierpi\'{n}ski gasket}:
        The Sierpi\'{n}ski gasket is a deterministic space
        obtained by repeating the operation of removing several smaller equilateral triangles from an equilateral triangle.
        A random recursive gasket is
        a random space obtained by randomizing such removals, which was introduced by Hambly \cite{Hambly_97_Brownian}.
        In Section \ref{subsec: a random recursive SG},
        we provide a uniform (polynomial) volume estimate of electrical networks convergent to a random recursive Sierpi\'{n}ski gasket.
        Although it turns out that this model is not in the UVD regime, 
        the volume estimate enables us to apply our main result.
  \item \textit{Critical Galton-Watson trees}:
        The convergence of Galton-Watson trees was proven by Aldous \cite{Aldous_93_The_continuum}
        in the case of the finite-variance offspring distribution,
        and this convergence result was extended
        to Galton-Watson trees with infinite-variance offspring distribution
        by Duquesne \cite{Duquesne_03_A_limit}.
        In \cite{Andriopoulos_23_Convergence},
        the convergence of stochastic processes and local times was established
        in the case of finite variance.
        In Section \ref{subsubsec: critical GW trees},
        we show that
        the H\"{o}lder continuity of height functions gives a suitable volume estimate,
        and hence the convergence of stochastic processes and local times
        still holds in the case of infinite variance.
  \item \textit{Uniform spanning trees}:
        In \cite{Barlow_Croydon_Kumagai_17_Subsequential,Angel_Croydon_Hernandez-Torres_Shiraishi_21_Scaling},
        the tightnesses of scaled uniform spanning trees in two and three dimensions were deduced respectively,
        and in \cite{Archer_Nachmias_Shalev_pre_The_GHP},
        the scaling limit of uniform spanning trees in five and higher dimensions was established.
        In Section \ref{subsec: UST in high dimensions} and \ref{subsec: UST in 2 and 3 dimensions},
        we provide a volume estimate for those uniform spanning trees,
        which implies that they satisfy the metric-entropy condition.
        As a consequence,
        we show the convergence of the sequence (or subsequence) of stochastic processes and their local times
        on those uniform spanning trees.
  \item \textit{The critical \erdosrenyi\ random graph}:
        In \cite{Berry_Broutin_Goldschmidt_12_The_continuum},
        the scaling limit of critical \erdosrenyi\ random graphs equipped with the graph distance was established.
        The key to this convergence was
        that the sequence of  \erdosrenyi\ random graphs
        has the same asymptotic behavior as
        the sequence of random graphs
        obtained by fusing random points in tilted trees.
        In \cite[Section 8.3]{Croydon_18_Scaling},
        a theory for convergence of such fused resistance metric spaces was developed.
        Using this theory,
        the convergence of critical \erdosrenyi\ random graphs equipped with the resistance metric is deduced.
        Moreover,
        the volume estimate of Galton-Watson trees mentioned above essentially
        implies a volume estimate of the \erdosrenyi\ random graphs.
        The details are given in Section \ref{subsec: critical ER random grpah}.
  \item \textit{The critical configuration model}:
        In \cite{Bhamidi_Sen_20_Geometry},
        the scaling limit of critical configuration models equipped with the graph distance was established.
        In Section \ref{subsec: the critical configuration model},
        using the above-mentioned theory of \cite[Section 8.3]{Croydon_18_Scaling},
        we prove that the convergence holds even when the graph metric is replaced by the resistance metric.
        Furthermore,
        we provide a volume estimate by a similar argument to critical \erdosrenyi\ random graphs,
        which yields the convergence of stochastic processes and their local times on critical configuration models.
\end{itemize}
We also mention that the random conductance model on unbounded fractals is considered 
in the follow-up paper \cite{Noda_pre_Scaling}.

The remainder of the article is organized as follows.
In Section \ref{sec: GH-type topologies},
we introduce the Gromov-Hausdorff-type topologies which are used to describe convergence of resistance metric spaces
equipped with measures and laws of stochastic processes and associated local times.
In Section \ref{sec: uniform continuity of stochastic processes},
we provide a preliminary result regarding uniform continuity of a stochastic process where metric entropy plays an important role.
In Section \ref{sec: resistance},
we recall some fundamental results 
about the theory of resistance forms and study joint continuity of local times on resistance metric spaces.
In Section \ref{sec: proof of the main theorem for deterministic spaces} and Section \ref{sec: proof of main results in random case}, respectively,
the main results Theorem \ref{1. thm: main result for deterministic spaces} and Theorem \ref{1. thm: main result for random spaces} are proved.
Methods for checking the metric-entropy conditions via volume estimates are provided in Section \ref{sec: metric entropy and volume estimates},
before finally, in Section \ref{sec: examples}, we present some examples to which our main results are applicable.


\section{Gromov-Hausdorff-type topologies} \label{sec: GH-type topologies}
\newcommand{\cC}{\mathcal{C}}
\newcommand{\cCc}{\mathcal{C}_{\mathrm{cpt}}}

When one considers convergence of compact metric spaces,
the Gromov-Hausdorff metric (see \cite{Burago_Burago_Ivanov_01_A_course}) gives a topology
(called the \textit{Gromov-Hausdorff topology})
that is useful for many applications.
From the viewpoint of probability theory,
it is natural to also ask about convergence of additional structures on the spaces, such as processes and measures.
Furthermore, in some cases, the spaces of interest are not compact and one needs to extend the topology accordingly.
We call the topology suitable for such arguments Gromov-Hausdorff-type topologies.
For example in \cite{Croydon_Hambly_Kumagai_17_Time-changes,Croydon_18_Scaling},
the (local) Gromov-Hausdorff-vague topology, which is a topology on tuples consisting of a boundedly-compact metric space, a root and a Radon measure,
is used to discuss the convergence of the underlying state spaces of random walks
(the details about this topology are given in Section \ref{subsec: the GHV}).
For the application of the probability theory,
it is important to have a metric for the Gromov-Hausdorff-type topology
and recently a unified method for constructing a metric for Gromov-Hausdorff-type topologies has been proposed 
in \cite{Khezeli_23_A_unified,Noda_pre_Metrization}.
(The differences in methods of those papers are discussed in \cite[Section 1]{Noda_pre_Metrization}.)
In this section, following the method given in \cite{Noda_pre_Metrization},
we introduce two Gromov-Hausdorff-type topologies used in our main results.
Henceforth, we set $a \wedge b \coloneqq  \min \{a, b\}$ 
and $a \vee b \coloneqq  \max \{ a, b \}$ for $a, b \in \mbR \cup \{ \infty \}$.


\subsection{The local Gromov-Hausdorff-vague topology} \label{subsec: the GHV}

In this section,
we define the local Gromov-Hausdorff-vague topology,
which is a generalization of the Gromov-Hausdorff-Prohorov topology 
introduced in \cite{Abraham_Delmas_Hoscheit_13_A_note},
and it is used to discuss the convergence of resistance metric spaces equipped with Radon measures
in our main results.

Let $(S, d, \rho)$ be a rooted boundedly-compact metric space.
For a subset $A \subseteq S$,
the \textit{(closed) $\varepsilon$-neighborhood} of $A$ in $(S, d)$ is given by 
\begin{equation}
  A^{\varepsilon} 
  \coloneqq
  \{
    x \in S : \exists y \in A\ \text{such that}\ d(x,y) \leq \varepsilon
  \}.
\end{equation} 
Let $\cC(S)$ be the set of closed subsets in $S$
and $\cCc(S)$ be the set of compact subsets in $S$
(containing the empty set).
The \textit{Hausdorff metric} $d_{H}$ on $\cCc(S)$ is defined by setting
\begin{equation}
  d_{H}(A, B)
  \coloneqq
  \inf\{
    \varepsilon \geq 0 : A \subseteq B^{\varepsilon}, B \subseteq A^{\varepsilon}
  \},
\end{equation}
where the infimum over the empty set is defined to be $\infty$.
It is known that $d_{H}$ is indeed a metric (allowed to take the value $\infty$ due to the empty set) on $\cCc(S)$
(see \cite[Section 7.3.1]{Burago_Burago_Ivanov_01_A_course}),
and the induced topology is called the \textit{Hausdorff topology}.
To deal with non-compact sets,
we introduce a metric on $\cC(S)$.
One candidate is of course the Hausdorff metric,
but it is not a good metric for non-compact sets in the following sense:
if $A$ is compact and $B$ is non-compact,
then it is the case that $d_{H}(A, B) = \infty$;
this means that it is impossible to approximate a non-compact set by a sequence of compact sets
as long as one uses the Hausdorff metric.
Hence, we introduce another metric $d_{\bar{H}, \rho}$,
which is the same spirit as metrics 
introduced in \cite{Abraham_Delmas_Hoscheit_13_A_note,Athreya_Lohr_Winter_16_The_gap,Khezeli_20_Metrization,Whitt_80_Some}.

\begin{dfn} [{The local Hausdorff metric $d_{\bar{H}, \rho}$}]
  For $A \in \cC(S)$ and $r>0$,
  we write 
  \begin{equation}  \label{eq: restriction of a set to a closed ball}
    A^{(r)} 
    \coloneqq 
    \mathrm{cl}(A \cap B_{d^{S}}(\rho, r)), 
  \end{equation}
  where we recall that $\mathrm{cl}(\cdot)$ denotes the closure of a set.
  We then define
  \begin{equation}
    d_{\bar{H}, \rho} (A, B) 
    \coloneqq 
    \int_{0}^{\infty} e^{-r} \left( 1 \wedge d_{H}(A^{(r)}, B^{(r)}) \right) dr,
    \quad 
    A,B \in \cC(S).
  \end{equation}
  We call $d_{\bar{H}, \rho}$ the \textit{local Hausdorff metric (with root $\rho$)}.
\end{dfn}

The function $d_{\bar{H}, \rho}$ is indeed a metric on $\cC(S)$
and a natural extension of the Hausdorff metric for non-compact sets.
The following is a basic property of $d_{\bar{H}, \rho}$.

\begin{thm} [{\cite[Section 2.2.1]{Noda_pre_Metrization}}]
  The function $d_{\bar{H}, \rho}$ is a metric 
  on $\cC(S)$
  and the metric space $(\cC(S), d_{\bar{H}, \rho})$ is compact.
  A sequence $(A_{n})_{n \geq 1}$ converges to $A$ with respect to 
  the local Hausdorff metric $d_{\bar{H}, \rho}$
  if and only if 
  $A_{n}^{(r)}$ converges to $A^{(r)}$
  in the Hausdorff topology
  for all but countably many $r>0$.
  Moreover,
  the topology on $\cC(S)$ induced from $d_{\bar{H}, \rho}$ 
  is independent of the root $\rho$.
\end{thm}

\begin{dfn} [{The local Hausdorff topology}]
  We call the topology induced from the local Hausdorff metric the \textit{local Hausdorff topology}. 
\end{dfn}

For convergence of measures, 
we use the vague topology.
So, we next introduce a metric inducing the vague topology.
Write $\mathcal{M}_{\mathrm{fin}}(S)$
for the set of finite Borel measures on $S$,
which we equip with the weak topology.
Recall that the weak topology is induced from the \textit{Prohorov metric} $d_{P}$ 
given by 
\begin{equation}
  d_{P}(\mu, \nu) 
  \coloneqq
  \inf\{
    \varepsilon : \mu(A) \leq \nu(A^{\varepsilon}) + \varepsilon, \nu(A) \leq \mu(A^{\varepsilon}) + \varepsilon, \forall A \subseteq S
  \}.
\end{equation}

\begin{dfn} [{The vague metric $d_{V, \rho}$}]
  We denote the set of Radon measures on $S$ by $\mathcal{M}(S)$.
  For $\mu \in \mathcal{M}(S)$,
  we denote the restriction of $\mu$ to $\mathrm{cl} (B_{d^{S}}(\rho, r))$ by $\mu^{(r)}$, 
  that is,
  $\mu^{(r)}$ is a finite Borel measure given by
  \begin{equation}  \label{eq: restriction of a measure to a closed ball}
    \mu^{(r)}(\cdot) 
    \coloneqq
    \mu \left( \cdot \cap \mathrm{cl} (B_{d^{S}}(\rho, r)) \right).
  \end{equation} 
  We then define
  \begin{equation}
    d_{V, \rho}(\mu, \nu)
    \coloneqq
    \int_{0}^{\infty} e^{-r} \left( 1 \wedge d_{P}(\mu^{(r)}, \nu^{(r)}) \right) dr,
    \quad 
    \mu, \nu \in \mathcal{M}(S).
  \end{equation}
  We call $d_{V, \rho}$ the \textit{vague metric (with root $\rho$)}.
\end{dfn}

\begin{thm} [{\cite[Section 2.2.2]{Noda_pre_Metrization}}]
  The function $d_{V, \rho}$ is a metric on $\mathcal{M}(S)$.
  The metric space $(\mathcal{M}(S), d_{V, \rho})$ is complete and separable.
  Let $\mu, \mu_{1}, \mu_{2}, \ldots$ be 
  Radon measures on $S$.
  Then these conditions are equivalent:
  \begin{enumerate}
    \item $\mu_{n}$ converges to a Radon measure $\mu$ with respect to $d_{V, \rho}$;
    \item $\mu_{n}^{(r)}$ converges weakly to $\mu^{(r)}$ for all but countably many $r>0$;
    \item $\mu_{n}$ converges vaguely to $\mu$, that is, for all continuous functions $f : S \to \mbR$ with compact support,
      it holds that 
      \begin{equation}
        \lim_{n \to \infty} \int_{S} f(x)\, \mu_{n}(dx) = \int_{S} f(x)\, \mu(dx).
      \end{equation}
  \end{enumerate}
\end{thm}

Now, we introduce the Gromov-Hausdorff-Prohorov topology 
and the local Gromov-Hausdorff-vague topology.
Let $\mathbb{G}^{\circ}$ be the collection of rooted-and-measured boundedly-compact metric spaces,
i.e., the collection of $G=(S, d, \rho, \mu)$
such that $(S, d)$ is a boundedly-compact metric space,
$\rho$ is a distinguished element of $S$ called the \textit{root}
and $\mu$ is a Radon measure on $S$.
For $G_{i} = (S_{i}, d_{i}, \rho_{i}, \mu_{i}) \in \mathbb{G}^{\circ},\, i=1,2$,
we say that $G_{1}$ and $G_{2}$ are \textit{GHV-equivalent} 
if and only if there exists a root-preserving isometry $f : S_{1} \to S_{2}$ such that 
$\mu_{2} = \mu_{1} \circ f^{-1}$.
Note that $f$ being an \textit{isometry} means that $f$ is distance-preserving and surjective (and hence bijective)
and $f$ being \textit{root-preserving} means that $f(\rho_{1})=\rho_{2}$.

\begin{dfn} [{The set $\mathbb{G}$ and $\mathbb{G}_{c}$}]
  We denote by $\mathbb{G}$ the collection of GHV-equivalence classes of elements in $\mathbb{G}^{\circ}$.
  We define $\mathbb{G}_{c}$ to be the collection of $(S, d, \rho, \mu) \in \mathbb{G}$
  such that $(S, d)$ is compact. 
\end{dfn}

\begin{rem} \label{rem: how to regard G as a set}
  From the rigorous point of view of set theory,
  neither $\mathbb{G}_{c}$ nor $\mathbb{G}$ is a set.
  However, it is possible to think of $\mathbb{G}$ as a set
  and introduce a metric structure.
  This is because one can construct a legitimate set $\mathscr{G}$ of rooted-and-measured boundedly-compact spaces 
  such that any rooted-and-measured boundedly-compact space is GHV-equivalent to an element of $\mathscr{G}$
  (see \cite[Section 3.1]{Noda_pre_Metrization}).
  Therefore,
  in this article,
  we will proceed with the discussion by treating $\mathbb{G}$ as a set 
  to avoid complications.
\end{rem}

Recall that 
the \textit{(pointed) Gromov-Hausdorff-Prohorov metric} $d_{GHP}$ on $\mathbb{G}_{c}$ is given 
by setting,
for $G_{i} = (S_{i}, d_{i}, \rho_{i}, \mu_{i}) \in \mathbb{G}_{c},\, i=1,2$,
\begin{equation}
  d_{\mathbb{G}_{c}}(G_{1}, G_{2})
  \coloneqq 
  \inf_{f_{1}, f_{2}, M}
  \{
    d(f_{1}(\rho_{1}), f_{2}(\rho_{2})) 
    \vee 
    d_{H}(f_{1}(S_{1}), f_{2}(S_{2})) 
    \vee 
    d_{P}(\mu_{1} \circ f_{1}^{-1}, \mu_{2}  \circ f_{2}^{-1})
  \},
\end{equation}
where the infimum is taken over all compact metric spaces $(M, d)$ 
and all distance-preserving maps $f_{i} : S_{i} \to M,\, i=1,2$.
The induced topology on $\mathbb{G}_{c}$ is called the \textit{(pointed) Gromov-Hausdorff-Prohorov topology}.
It is known that this topology is Polish 
and further details of this metric are found in \cite{Abraham_Delmas_Hoscheit_13_A_note,Khezeli_20_Metrization}.
We define a metric on $\mathbb{G}$ in a similar way
(with a slight change in the handling of roots).

\begin{dfn} [{The metric $d_{\mathbb{G}}$}]
  For $G_{i} = (S_{i}, d_{i}, \rho_{i}, \mu_{i}) \in \mathbb{G},\, i=1,2$,
  we set 
  \begin{equation}
    d_{\mathbb{G}}(G_{1}, G_{2})
    \coloneqq
    \inf_{f_{1}, f_{2}, M} 
    \left\{
      d_{\bar{H}, \rho} (f_{1}(S_{1}), f_{2}(S_{2})) 
      \vee 
      d_{V, \rho} (\mu_{1} \circ f_{1}^{-1}, \mu_{2} \circ f_{2}^{-1})
    \right\},
  \end{equation}
  where the infimum is taken over all rooted boundedly-compact metric spaces $(M, d, \rho)$ 
  and all root-and-distance-preserving maps $f_{i} : S_{i} \to M,\, i=1,2$.
\end{dfn}

\begin{thm} [{\cite[Corollary 3.30 and Proposition 4.11]{Noda_pre_Metrization}}]
  The function $d_{\mathbb{G}}$ is a well-defined metric on $\mathbb{G}$,
  and the metric space $(\mathbb{G}, d_{\mathbb{G}})$ is complete and separable.
\end{thm}

\begin{dfn} [{The local Gromov-Hausdorff-vague topology}]
  We call the topology on $\mathbb{G}$ induced by the metric $d_{\mathbb{G}}$ 
  the \textit{local Gromov-Hausdorff-vague topology}.
\end{dfn}

Regarding convergence in $\mathbb{G}$,
we have the following result.

\begin{thm} [{\cite[Theorem 4.13]{Noda_pre_Metrization}}] \label{2. thm: characterization of convergence in the local. GHV}
  Let $G=(S, d, \rho, \mu)$ and $G_{n} = (S_{n}, d^{n}, \rho_{n}, \mu_{n}),\, n \in \mathbb{N}$
  be elements in $\mathbb{G}$.
  Then, the following statements are equivalent:
  \begin{enumerate} [label = (\roman*)]
    \item 
      $G_{n}$ converges to $G$ in the local Gromov-Hausdorff-vague topology;
    \item 
      $G_{n}^{(r)}$ converges to $G^{(r)}$ in the Gromov-Hausdorff-Prohorov topology 
      for all but countably many $r>0$;
    \item 
      there exist a rooted boundedly-compact metric space $(M, d^{M}, \rho_{M})$ 
      and root-and-distance-preserving maps $f_{n} : S_{n} \to M$ and $f : S \to M$ such that 
      $f_{n}(S_{n}) \to f(S)$ in the local Hausdorff topology in $M$ 
      and 
      $\mu_{n} \circ f_{n}^{-1} \to \mu \circ f^{-1}$ vaguely as measures on $M$.
  \end{enumerate}  
\end{thm}

\begin{rem} 
  The local Gromov-Hausdorff-vague topology is an extension of the one 
  in \cite[Section 2.2]{Croydon_Hambly_Kumagai_17_Time-changes},
  and this is a little different from the Gromov-Hausdorff-vague topology introduced in \cite{Athreya_Lohr_Winter_16_The_gap}.
  This is because
  the topology in \cite{Athreya_Lohr_Winter_16_The_gap} deals with the convergence of the supports of measures instead of the whole spaces.
  However since we assume that all measures contained in $\mathbb{F}$ are of full support,
  the topology induced into $\mathbb{F}$ is same
  whether one uses the local Gromov-Hausdorff-vague topology
  or the Gromov-Hausdorff-vague topology.
\end{rem}

\begin{rem}
  The local Gromov-Hausdorff-vague topology is strictly coarser than the Gromov-Hausdorff-Prohorov topology.
  Indeed, one can check this 
  by noting that the subsets $\{0, n\}$ in $\mbR$ converge to ${0}$ in the local Hausdorff topology in $\mbR$
  while the diameter of $\{0, n\}$ diverges,
  which implies that $\{0, n\}$ does not converge in the Gromov-Hausdorff topology.
  In this paper,
  we equip $\mathbb{G}_{c}$ with the Gromov-Hausdorff-Prohorov topology 
  and $\mathbb{G}$ with the local Gromov-Hausdorff-vague topology.
\end{rem}

The following results concern the (joint) continuity of the operation $G^{(\cdot)}$
(recall this from \eqref{1. eq: dfn of restriction operator}),
which are used in Section \ref{sec: proof of main results in random case}.

\begin{lem} \label{2. lem: continuity of the restriction of bcms equipped with a measure}
  For every $G\in \mathbb{G}$,
  the map $(0,\infty)\ni r \mapsto G^{(r)} \in \mathbb{G}_{c}$ is left-continuous with right limits.
  In particular,
  it is continuous for all but countably many $r$.
\end{lem}

\begin{proof}
  Write $G = (S, d, \rho, \mu)$.
  Consider the maps 
  $f_{1}: (0, \infty) \ni r \mapsto S^{(r)} \in \cCc(S)$
  and 
  $f_{2}: (0, \infty) \ni r \mapsto \mu^{(r)} \in \mathcal{M}_{\mathrm{fin}}(S)$.
  It is not difficult to check that both maps are left-continuous with left limits
  with respect to the Hausdorff topology and the weak topology, respectively
  (the right limits at $r$ are $D_{d}(\rho, r)$ and $\mu(\cdot \cap D_{d}(\rho, r))$, respectively).
  Hence, we obtain the desired result.
\end{proof}

\begin{lem} \label{2. lem: joint continuity of the restriction operator}
  Let $G, G_{1}, G_{2}, \ldots$ be elements of $\mathbb{G}$
  such that $G_{n} \to G$ in the Gromov-Hausdorff-Prohorov topology
  and $r>0$ be a continuity point of $f_{1}$ and $f_{2}$ 
  defined in the proof of Lemma \ref{2. lem: continuity of the restriction of bcms equipped with a measure}.
  Then,
  for any sequence $(r_{n})_{n \geq 1}$ of positive numbers converging to $r$,
  it holds that $G_{n}^{(r_{n})} \to G^{(r)}$ in the Gromov-Hausdorff-Prohorov topology.
\end{lem}

\begin{proof}
  Write $G_{n}=(S_{n}, d^{n}, \rho_{n}, \mu_{n})$ and $G=(S, d, \rho, \mu)$.
  Fix $\varepsilon \in (0,1)$.
  By assumption,
  there exists $\delta \in (0, \varepsilon)$ such that,
  for any $r'>0$ with $|r-r'| \leq \delta$,
  it holds that $d_{H}(S^{(r)}, S^{(r')}) < \varepsilon$
  and $d_{P}(\mu^{(r)}, \mu^{(r')}) < \varepsilon$.
  From Theorem \ref{2. thm: characterization of convergence in the local. GHV},
  for some $a>1$,
  we have that $G_{n}^{(r+a)} \to G^{(r+a)}$ in the Gromov-Hausdorff-Prohorov topology,
  and hence we may assume that
  $S_{n}^{(r+a)}$ and $S^{(r+a)}$ are embedded isometrically
  into a common rooted compact metric space $(K, d^{K}, \rho_{K})$
  in such a way that 
  $S_{n}^{(r+a)} \to S^{(r+a)}$ in the Hausdorff topology in $K$,
  $\rho_{n} = \rho = \rho_{K}$ as elements of $K$,
  and $\mu_{n}^{(r+a)} \to \mu^{(r+a)}$ weakly as measures on $K$.
  It is then the case that
  $|r_{n} - r| < \delta/2$,
  $d_{H}^{K}(S_{n}^{(r+a)}, S^{(r+a)}) < \delta/2$,
  and $d_{P}^{K}(\mu_{n}^{(r+a)}, \mu^{(r+a)}) < \delta/2$
  (at least, for all sufficiently large $n$).
  For $x_{n} \in S_{n}^{(r_{n})}$,
  choose $y \in S^{(r+a)}$ satisfying $d^{K}(x_{n}, y) < \delta/2$.
  Then, 
  \begin{equation}
    d(\rho, y) 
    \leq 
    d^{K}(\rho, x_{n}) + d^{K}(x_{n}, y)
    <r_{n}+\delta/2
    < r+ \delta,
  \end{equation}
  which implies that $y \in S^{(r+\delta)}$.
  Thus, we can find $x \in S^{(r)}$ such that $d(x, y) < \varepsilon$.
  It then follows that $d^{K}(x_{n}, x) < 2 \varepsilon$.
  Similarly,
  for any $x \in S^{(r)}$,
  one can find $x_{n} \in S_{n}^{(r_{n})}$ satisfying $d^{K}(x, x_{n}) < 2 \varepsilon$.
  Therefore,
  we deduce that $d_{H}^{K}(S_{n}^{(r_{n})}, S^{(r)}) < 2 \varepsilon$.
  Moreover,
  using $d_{P}^{K}(\mu_{n}^{(r+a)}, \mu^{(r+a)}) < \delta/2$
  and $|r_{n} - r| < \delta/2$,
  it is not difficult to show that $d^{K}(\mu_{n}^{(r_{n})}, \mu^{(r)}) < 2 \varepsilon$.
  Hence,
  we complete the proof.
\end{proof}


\subsection{The space $\mbM_{L}$} \label{sec: the space M_L}
\newcommand{\hatC}{\widehat{C}}
\newcommand{\dhatCrho}{d_{\widehat{C}, \rho}}
\newcommand{\dhatCc}{d_{\widehat{C}_{c}}}

In this section,
we define an extended version of the local Gromov-Hausdorff-vague topology
on tuples consisting of 
a rooted-and-measured boundedly-compact metric space 
equipped with a probability measure on the set of cadlag functions and local-time-type functions.

Let $C(\mbRp, \mbR)$ be the space of continuous functions from $\mbRp$ to $\mbR$
equipped with the compact-convergence topology.
This topology is induced from the metric $d^{C(\mbRp, \mbR)}$ given by
\begin{equation} \label{2. eq: a metric on C(R_+, R)}
  d^{C(\mbRp, \mbR)} (f, g)\coloneqq \sum_{n \geq 1} 2^{-n} \max_{0 \leq t \leq n}( |f(t)-g(t)| \wedge 1).
\end{equation}
Let $(S, d, \rho)$ be a rooted boundedly-compact metric space. 

\begin{dfn} [{The sets $\hatC_{c}(S \times \mbRp, \mbR)$ and $\hatC(S \times \mbRp, \mbR)$}]
  We define 
  \begin{equation}
    \hatC (S \times \mbRp, \mbR) 
    \coloneqq 
    \bigcup_{X \in \cC(S)} 
    C(X \times \mbRp, \mbR),
  \end{equation} 
  Note that $\hatC(S \times \mbRp, \mbR)$ contains 
  the empty map $\emptyset_{\mbR}: \emptyset \to \mbR$.
  For each $L \in \hatC(S \times \mbRp, \mbR)$,
  if $L \in C(X \times \mbRp, \mbR)$,
  then we write $\dom_{1}(L) \coloneqq X$.
  We define $\hatC_{c}(S \times \mbRp, \mbR)$ 
  to be the subset of $\hatC(S \times \mbRp, \mbR)$ consisting of 
  $L$ such that $\dom_{1}(L)$ is compact in $S$.
\end{dfn}

\begin{dfn} [{The metrics $\dhatCc$ and $\dhatCrho$}]
  For $L_{1}, L_{2} \in \hatC_{c}(S \times \mbRp, \mbR)$ and $\varepsilon>0$, 
  consider the following condition.
  \begin{enumerate} [label = ($\hat{C}_{c}$)]
    \item \label{2. dfn item: epsilon condition for metric on hatC_c}
      For any $x \in \dom_{1}(L_{1})$, there exists an element $y \in \dom_{1}(L_{2})$ such that 
      \begin{equation}
        d(x, y) \vee d^{C(\mbRp, \mbR)}(L_{1}(x, \cdot), L_{2}(y, \cdot)) \leq \varepsilon.
      \end{equation}
      Similarly,
      for any $y \in \dom_{1}(L_{2})$, there exists an element $x \in \dom_{1}(L_{1})$ such that 
      the above inequality holds.
  \end{enumerate}
  We then define
  \begin{equation}
    \dhatCc(L_{1}, L_{2}) 
    \coloneqq
    \inf 
    \{
      \varepsilon > 0 \mid \varepsilon\ \text{satisfies \ref{2. dfn item: epsilon condition for metric on hatC_c}}
    \},
  \end{equation}
  where the infimum over the empty set is defined to be $\infty$.
  For $L \in \hatC(S \times \mbRp, \mbR)$,
  we set 
  \begin{equation}
    L^{(r)} \coloneqq L|_{\dom_{1}(L)^{(r)} \times \mbRp},
  \end{equation}
  where we recall that $\dom_{1}(L)^{(r)} = \dom_{1}(L) \cap S^{(r)}$.
  Obviously, $\dom_{1}(L^{(r)}) = \dom_{1}(L)^{(r)}$.
  We then define 
  \begin{equation}
    \dhatCrho(L_{1}, L_{2}) 
    \coloneqq 
    \int_{0}^{\infty} 
    e^{-r} 
    \left(
      1 \wedge \dhatCc(L_{1}^{(r)}, L_{2}^{(r)})
    \right)\,
    dr, 
    \quad 
    L_{1}, L_{2} \in \hatC(S \times \mbRp, \mbR).
  \end{equation}
\end{dfn}

\begin{thm} [{\cite[Theorem 2.45]{Noda_pre_Metrization}}]
  The function $\dhatCrho$ is a well-defined metric on $\hatC(S \times \mbRp, \mbR)$.
  The induced topology is Polish and independent of the root $\rho$.
\end{thm}

\begin{dfn} [{The compact-convergence topology with variable domains}]
  We call the topology on $\hatC(S \times \mbRp, \mbR)$ induced from $\dhatCrho$ 
  the \textit{compact-convergence topology with variable domains}.
\end{dfn}

\begin{thm}  [{Convergence in $\hatC(S \times \mbRp, \mbR)$}]  \label{2. thm: convergence in hatC}
  Let $L, L_{1}, L_{2}, \ldots$ be elements of $\hatC(S \times \mbRp, \mbR)$.
  The following conditions are equivalent.
  \begin{enumerate} [label = (\roman*)]
    \item \label{2. thm item: convergence in hatC, f_n converges to f} 
      The functions $L_{n}$ converge to $L$ in the compact-convergence topology with variable domains.
    \item \label{2. thm item: convergence in hatC, domain and value convergence}
      The sets $\dom_{1}(L_{n})$ converge to $\dom_{1}(L)$ in the local Hausdorff topology in $S$,
      and it holds that, for all $T>0$ and $r>0$, 
      \begin{equation}  \label{2. eq: convergence in hatC, uniform convergence on compact subsets}
        \lim_{\delta \to 0}
        \limsup_{n \to \infty}
        \sup_{\substack{ x_{n} \in \dom_{1}(L_{n})^{(r)}, \\ x \in \dom_{1}(L)^{(r)}, \\ d(x_{n},x) < \delta}}
        \sup_{0 \leq t \leq T}
        |L_{n}(x_{n}, t) - L(x, t)| 
        =0.
      \end{equation}
    \item \label{2. thm item: convergence in hatC, Tietze extension}
      The sets $\dom_{1}(L_{n})$ converge to $\dom_{1}(L)$ in the local Hausdorff topology in $S$,
      and there exist continuous functions $L'_{n}, L' \in C(S \times \mbRp, \mbR)$ such that 
      $L'_{n}|_{\dom_{1}(L_{n}) \times \mbRp} = L_{n}$,
      $L'|_{\dom_{1}(L) \times \mbRp} = L$ and 
      $L'_{n} \to L'$ in the compact-convergence topology in $C(S \times \mbRp, \mbR)$. 
  \end{enumerate}
\end{thm}

\begin{proof}
  This is an immediate consequence of \cite[Theorem 2.59]{Noda_pre_Metrization}.
\end{proof}

\begin{cor} \label{2. cor: embedding of C(S times R) to hatC}
  The following map is a topological embedding, 
  i.e., a homeomorphism onto its image:
  \begin{equation}
    C(S \times \mbRp, \mbR) \ni 
    L \mapsto L 
    \in \hatC(S \times \mbRp, \mbR).
  \end{equation}
  (NB.\ The space $C(S \times \mbRp, \mbR)$ is equipped with the compact-convergence topology,
  while $\hatC(S \times \mbRp, \mbR)$ is equipped with the compact-convergent topology with variable domains.)
\end{cor}

\begin{dfn} [{The embedding of $C(S \times \mbRp, \mbR)$ into $\hatC(S \times \mbRp, \mbR)$}]
  Via the topological embedding given in Corollary \ref{2. cor: embedding of C(S times R) to hatC},
  we always regard $C(S \times \mbRp, \mbR)$ as a subspace of $\hatC(S \times \mbRp, \mbR)$.
\end{dfn}

\begin{thm} [{Precompactness in $\hatC(S \times \mbRp, \mbR)$}] \label{2. thm: precompactness in hatC}
  A non-empty subset $\{L_{\alpha} \mid \alpha \in \mathscr{A}\}$ of $\hatC(S \times \mbRp, \mbR)$ 
  is precompact if and only if the following conditions are satisfied.
  \begin{enumerate} [label=(\roman*)]
    \item \label{2. thm item: precompactness in hatC, boundedness} 
      For each $r>0$, it holds that 
      $\displaystyle \sup_{\alpha \in \mathscr{A}} 
        \sup_{x \in \dom_{1}(L_{\alpha})^{(r)}} 
        |L_{\alpha}(x, 0)| 
        < \infty
      $. 
    \item \label{2. thm item: precompactness in hatC, equicontinuity}
      For each $r>0$ and $T>0$, it holds that
      \begin{equation}
        \lim_{\delta \downarrow 0} 
        \sup_{\alpha \in \mathscr{A}} 
        \sup_{\substack{x, y \in \dom_{1}(L_{\alpha})^{(r)},\\ d(x, y) < \delta}} 
        \sup_{\substack{0 \leq s, t \leq T, \\ |t -s| < \delta}}
        |L_{\alpha}(x, t) - L_{\alpha}(y, s)|
        = 0.
      \end{equation}
  \end{enumerate}
\end{thm}

\begin{proof}
  By \cite[Theorem 2.62]{Noda_pre_Metrization}, 
  $\{L_{\alpha} \mid \alpha \in \mathscr{A} \}$ is precompact if and only if the following conditions are satisfied.
  \begin{enumerate} [label = (\roman*')]
    \item \label{2. pr item: precompactness in C(Rp, R)}
      For each $r>0$, the set $\{L_{\alpha}(x, \cdot) \mid x \in \dom_{1}(L_{\alpha})^{(r)},\, \alpha \in \mathscr{A}\}$
      is precompact in $C(\mbRp, \mbR)$.
    \item \label{2. pr item: equicontinuity on S} 
      For each $r>0$, it holds that 
      \begin{equation}
        \lim_{\delta \downarrow 0} 
        \sup_{\alpha \in \mathscr{A}} 
        \sup_{\substack{x, y \in \dom_{1}(L_{\alpha})^{(r)},\\ d(x, y) < \delta}} 
        d^{C(\mbRp, \mbR)}(L_{\alpha}(x,\cdot), L_{\alpha}(y, \cdot))
        = 0.
      \end{equation}
  \end{enumerate}
  Using the Arzel\`{a}-Ascoli theorem (c.f.\ \cite[Chapter VI. Theorem 1.5]{Jacod_Shiryaev_03_Limit}),
  we deduce that \ref{2. pr item: precompactness in C(Rp, R)}
  is equivalent to the following conditions.
  \begin{enumerate} [label = (i'' -\alph*), leftmargin = *]
    \item \label{2. pr item: boundedness of L(x, 0)}
      For each $r>0$, it holds that 
      $\displaystyle  
        \sup_{\alpha \in \mathscr{A}} 
        \sup_{x \in \dom_{1}(L_{\alpha})^{(r)}}
        L_{\alpha}(x, 0)
        < \infty 
      $.
    \item \label{2. pr item: equicontinuity of L(x, dot)} 
      For each $r>0$ and $T>0$, it holds that 
      \begin{equation}
        \lim_{\delta \downarrow 0} 
        \sup_{\alpha \in \mathscr{A}} 
        \sup_{x \in \dom_{1}(L_{\alpha})^{(r)}} 
        \sup_{\substack{0 \leq s, t \leq T,\\ |t - s| < \delta}}
        |L_{\alpha}(x, s) - L_{\alpha}(x, t)|
        = 0.
      \end{equation}
  \end{enumerate}
  From the definition of $d^{C(\mbRp, \mbR)}$ given in \eqref{2. eq: a metric on C(R_+, R)},
  we also deduce that \ref{2. pr item: equicontinuity on S} is equivalent to the following condition.
  \begin{enumerate} [label = (ii''), leftmargin = *]
    \item \label{2. pr item: equicontinuity of L(dot, t)} 
    For each $r>0$ and $T>0$, it holds that 
    \begin{equation}
      \lim_{\delta \downarrow 0} 
      \sup_{\alpha \in \mathscr{A}} 
      \sup_{\substack{x, y \in \dom_{1}(L_{\alpha})^{(r)},\\ d(x, y) < \delta}} 
      \sup_{0 \leq t \leq T}
      |L_{\alpha}(x,t) - L_{\alpha}(y, t)|
      = 0.
    \end{equation}
  \end{enumerate}
  It is easy to check that \ref{2. pr item: equicontinuity of L(x, dot)} and \ref{2. pr item: equicontinuity of L(dot, t)} 
  are equivalent to \ref{2. thm item: precompactness in hatC, equicontinuity}.
  Now, the desired result is immediate.
\end{proof}

The precompactness criterion of Theorem \ref{2. thm: precompactness in hatC} yields 
the following tightness criterion in $\hatC(S \times \mbRp, \mbR)$.

\begin{thm} [{Tightness in $\hatC(S \times \mbRp, \mbR)$}]  \label{2. thm: tightness in hatC}
  Let $(L_{n})_{n \geq 1}$ be a sequence of 
  random elements of $\hatC(S \times \mbRp, \mbR)$.
  We write the underlying probability measure of $L_{n}$ by $P_{n}$.
  Then, $(L_{n})_{n \geq 1}$ is tight if and only if the following conditions are satisfied.
  \begin{enumerate} [label = (\roman*)]
    \item \label{2. thm item: tightness in hatC, bouncedness}
      For each $r>0$, it holds that
      $\displaystyle 
        \lim_{M \to \infty} 
        \limsup_{n \to \infty} 
        P_{n} \left( \sup_{x \in \dom_{1}(L_{n})^{(r)}} L_{n}(x, 0) > M \right) = 0.
      $
    \item \label{2. thm item: tightness in hatC, equicontinuity}
      For each $r>0$ and $T > 0$,
      it holds that, for all $\varepsilon > 0$,
      \begin{equation}
        \lim_{\delta \downarrow 0} 
        \limsup_{n \to \infty} 
        P_{n} 
        \left(
          \sup_{\substack{ x, y \in \dom_{1}(L_{n})^{(r)}, \\ d(x, y) < \delta}}
          \sup_{\substack{0 \leq s, t \leq T, \\ |t-s| < \delta}} 
          |L_{n}(x,t) - L_{n}(y,s)| 
          > 
          \varepsilon
        \right)
        = 
        0.
    \end{equation}
  \end{enumerate}
\end{thm}

\begin{proof}
  Suppose that $(L_{n})_{n \geq 1}$ is tight.
  Fix a sequence $(\varepsilon_{l})_{l \geq 1}$ of positive numbers with $\varepsilon_{l} \downarrow 0$.
  Let $\mathcal{K}_{l}$ be a compact subset of $\hatC(S \times \mbRp, \mbR)$ satisfying 
  $\sup_{n} P_{n}(L_{n} \notin \mathcal{K}_{l}) < \varepsilon_{l}$.
  Set, for each $l \geq 1,r>0$ and $T>0$,
  \begin{align}
    M_{l}^{(r)} 
    \coloneqq 
    \sup\{ |L(x,0)| : x \in \dom_{1}(L)^{(r)}, L \in \mathcal{K}_{l} \}, \\
    \varepsilon_{\delta}^{(r), T}
    \coloneqq 
    \sup_{L \in \mathcal{K}_{l}} 
    \sup_{\substack{x, y \in \dom_{1}(L_{\alpha})^{(r)},\\ d(x, y) < \delta}} 
    \sup_{\substack{0 \leq s, t \leq T, \\ |t -s| < \delta}}
    |L(x, t) - L(y, s)|.
  \end{align}
  By Theorem \ref{2. thm: precompactness in hatC},
  we have that $M_{l}^{(r)} < \infty$ and $\varepsilon_{\delta}^{(r), T} \to  0$ 
  as $\delta \downarrow 0$.
  We then deduce that 
  \begin{equation}
    \sup_{n \geq 1}
    P_{n}
    \left(
      \sup_{x \in \dom_{1}(L_{n})^{(r)}} |L_{n}(x, 0)| > M_{l}^{(r)} 
    \right)
    \leq 
    \sup_{n} 
    P_{n}(L_{n} \notin \mathcal{K}_{l})
    < \varepsilon_{l},
  \end{equation}
  and by choosing $\delta$ small so that $\varepsilon_{\delta}^{(r), T} < \varepsilon$,
  \begin{align}
    \sup_{n \geq 1} 
    P_{n} 
    \left(
      \sup_{\substack{ x, y \in \dom_{1}(L_{n})^{(r)}, \\ d(x, y) < \delta}}
      \sup_{\substack{0 \leq s, t \leq T, \\ |t-s| < \delta}} 
      |L_{n}(x,t) - L_{n}(y,s)| 
      > 
      \varepsilon
    \right)
    \leq 
    \sup_{n \geq 1} P_{n}(L_{n} \notin \mathcal{K}_{l}) 
    \leq 
    \varepsilon_{l}.
  \end{align}
  Therefore, we obtain \ref{2. thm item: tightness in hatC, bouncedness} and \ref{2. thm item: tightness in hatC, equicontinuity}.
  Conversely, assume \ref{2. thm item: tightness in hatC, bouncedness} and \ref{2. thm item: tightness in hatC, equicontinuity}.
  Since $\hatC(S \times \mbRp, \mbR)$ is Polish,
  each random element $L_{n}$ is tight by itself.
  Hence,
  $\limsup_{n \to \infty}$ that appears in both conditions can be replaced by $\sup_{n \geq 1}$.
  Fix $\varepsilon>0$.
  Let $(r_{l})_{l \geq 1}$ and $(T_{l})_{l \geq 1}$ 
  be increasing sequences of positive numbers with $r_{l} \wedge T_{l} \uparrow \infty$.
  We choose $(M_{l})_{l \geq 1}$ and $(\delta_{l})_{l \geq 1}$ with $M_{l} < \infty$ and $\delta_{l} \downarrow 0$
  satisfying, for each $l \geq 1$,
  \begin{gather}
    \sup_{n \geq 1} 
    P_{n} 
    \left(
      \sup_{x \in \dom_{1}(L_{n})^{(r_{l})}} 
      |L_{n}(x, 0)|
      > M_{l}
    \right) 
    < \frac{\varepsilon}{2^{l}}, 
    \label{2. eq: tightness in hatC, choice of M_l}\\ 
    \sup_{n \geq 1} 
    P_{n} 
    \left(
      \sup_{\substack{ x, y \in \dom_{1}(L_{n})^{(r_{l})}, \\ d(x, y) < \delta_{l}}}
      \sup_{\substack{0 \leq s, t \leq T_{l}, \\ |t-s| < \delta_{l}}} 
      |L_{n}(x,t) - L_{n}(y,s)| 
      > 
      \frac{1}{2^{l}}
    \right) 
    < \frac{\varepsilon}{2^{l}}.
    \label{2. eq: tightness in hatC, choice of delta_l}
  \end{gather}
  Define $\mathcal{K}$ to be the collection of $L \in \hatC(S \times \mbRp, \mbR)$ 
  such that, for each $l \geq 1$, 
  \begin{gather}
    \sup_{x \in \dom_{1}(L)^{(r_{l})}} 
    |L(x, 0)|
    \leq M_{l},\\
    \sup_{\substack{ x, y \in \dom_{1}(L)^{(r_{l})}, \\ d(x, y) < \delta_{l}}}
    \sup_{\substack{0 \leq s, t \leq T_{l}, \\ |t-s| < \delta_{l}}} 
    |L(x,t) - L(y,s)| 
    \leq
    \frac{1}{2^{l}}.
  \end{gather}
  By Theorem \ref{2. thm: precompactness in hatC}, $\mathcal{K}$ is precompact in $\hatC(S \times \mbRp, \mbR)$,
  and by \eqref{2. eq: tightness in hatC, choice of M_l} and \eqref{2. eq: tightness in hatC, choice of delta_l},
  we deduce that 
  \begin{equation}
    \sup_{n \geq 1} P_{n}(L_{n} \notin \mathcal{K})
    \leq 
    \sum_{l \geq 1} \left( \frac{\varepsilon}{2^{l}} + \frac{\varepsilon}{2^{l}} \right) 
    \leq 
    2 \varepsilon.
  \end{equation}
  Hence, $(L_{n})_{n \geq 1}$ is tight.
\end{proof}

Recall that $D(\mbRp, S)$ denotes the space of cadlag functions with values in $S$
equipped with the usual $J_{1}$-Skorohod topology.
We write $d_{J_{1}}$ for the complete, separable metric on $D(\mbRp, S)$
(see \cite{Billingsley_99_Convergence} or \cite{Whitt_80_Some}
for such a metric).
We then equip $D(\mbRp, S) \times \hatC(S \times \mbRp, \mbR)$ 
with the \textit{max product metric} given by 
\begin{equation}
  d_{J_{1} \times \hatC, \rho} ( (f_{1}, L_{1}), (f_{2}, L_{2})) 
  \coloneqq 
  d_{J_{1}}(f_{1}, f_{2}) 
  \vee 
  \dhatCrho(L_{1}, L_{2}).
\end{equation}
Given two maps $f: A \to B$ and $f': A' \to B'$,
we define $f \times f': A \times A' \to B \times B'$ by setting
\begin{equation}
  (f \times f') (a, a')
  \coloneqq
  (f(a), f'(a')).
\end{equation}
We write $\idty_{A}$ for the identity map from $A$ to itself.

Finally,
it is possible to define the space $\mbM_{L}$.
Let $\mbM_{L}^{\circ}$ be the collection of $(S, d, \rho, \mu, \pi)$ 
such that $(S, d, \rho, \mu) \in \mathbb{G}^{\circ}$ and $\pi$ 
is a probability measure on $D(\mbRp, S) \times \hatC(S \times \mbRp, \mbR)$.
To introduce an equivalence relation on $\mbM_{L}^{\circ}$,
we need a preparation.
For a distance-preserving map $f: S_{1} \to S_{2}$,
we define
\begin{gather}  \label{2. eq: functor for cadlag functions and local-time functions}
\begin{split}
  \tau_{f}^{J_{1}}: D(\mbRp, S_{1}) \ni 
    X &\mapsto f \circ X 
  \in D(\mbRp, S_{2}),\\
  \tau_{f}^{\hatC}: \hatC(S_{1} \times \mbRp, \mbR) \ni 
    L &\mapsto L \circ (f^{-1} \times \idty_{\mbRp}) 
  \in \hatC(S_{2} \times \mbRp, \mbR),
\end{split}
\end{gather}
where $f^{-1}$ is restricted to $f(\dom_{1}(L))$ so that $L \circ (f^{-1} \times \idty_{\mbRp})$ is well-defined.
We then define $\tau_{f}^{J_{1} \times \hatC} \coloneqq \tau_{f}^{J_{1}} \times \tau_{f}^{\hatC}$.
For $\cX_{i} = (S_{i}, d_{i}, \rho_{i}, \mu_{i}, \pi_{i}) \in \mbM_{L}^{\circ},\, i=1,2$,
we say that $\cX_{1}$ is $(\tau^{J_{1} \times \hatC})^{-1}$-equivalent to $\cX_{2}$ 
if and only if there exists a root-preserving isometry $f: S_{1} \to S_{2}$ 
such that $\mu_{2} = \mu_{1} \circ f^{-1}$ and $\pi_{2} = \pi_{1} \circ (\tau_{f}^{J_{1} \times \hatC})^{-1}$.

\begin{dfn} [{The space $\mbM_{L}$}]
  We define $\mbM_{L}$ to be the collection of $(\tau^{J_{1} \times \hatC})^{-1}$-equivalence classes of elements in $\mbM_{L}^{\circ}$.
\end{dfn}

\begin{rem}
  By the same reason given in Remark \ref{rem: how to regard G as a set}, 
  we can safely regard $\mbM_{L}$ as a set and introduce a metric structure.
\end{rem}

\begin{dfn} [{The metric $d_{\mbM_{L}}$}]
  For $\cX_{i} = (S_{i}, d_{i}, \rho_{i}, \mu_{i}, \pi_{i}) \in \mbM_{L},\, i=1,2$,
  we set 
  \begin{equation}
    \begin{split}
      d_{\mbM_{L}}(\cX_{1}, \cX_{2}) 
      \coloneqq
      \inf_{f_{1}, f_{2}, M} 
      &\Bigl\{ 
        d_{\bar{H}, \rho} (f_{1}(S_{1}), f_{2}(S_{2})) 
        \vee 
        d_{V, \rho} (\mu_{1} \circ f_{1}^{-1}, \mu_{2} \circ f_{2}^{-1})
      \Bigr.\\
      &\vee 
        \tilde{d}_{P}\left(
          \pi_{1} \circ (\tau_{f_{1}}^{J_{1} \times \hatC})^{-1},
          \pi_{2} \circ (\tau_{f_{2}}^{J_{1} \times \hatC})^{-1}
        \right)
      \Bigr\},
    \end{split}
  \end{equation}
  where the infimum is taken over all rooted boundedly-compact metric spaces $(M, d, \rho)$ 
  and all root-and-distance-preserving maps $f_{i} : S_{i} \to M,\, i=1,2$,
  and $\tilde{d}_{P}$ denotes the Prohorov metric on the set of probability measures 
  on the metric space $(D(\mbRp, M) \times \hatC(M \times \mbRp, \mbR), d_{J_{1} \times \hatC, \rho})$.
\end{dfn}

Using the theory established in \cite{Noda_pre_Metrization},
we obtain the following result readily.

\begin{thm} \label{2. thm: metric on M_L}
  The function $d_{\mbM_{L}}$ is a well-defined metric on $\mbM_{L}$,
  and the induced topology is Polish.
  (NB. The metric $d_{\mbM_{L}}$ is not necessarily complete.)
\end{thm}

\begin{proof}
  By setting $\tau^{J_{1}}(S) \coloneqq D(\mbRp, S)$ and $\tau^{\hat{C}}(S) \coloneqq \hat{C}(S \times \mbRp, \mbRp)$
  for each rooted boundedly-compact metric space,
  we obtain functors $\tau^{J_{1}}$ and $\tau^{\hat{C}}$,
  where morphisms are given by \eqref{2. eq: functor for cadlag functions and local-time functions}
  (for the notion of functors here, see \cite[Definition 3.14]{Noda_pre_Metrization}).
  By \cite[Proposition 4.15 and 4.27]{Noda_pre_Metrization},
  both functors are Polish functors
  and hence the product functor $\tau \coloneqq \tau^{J_{1}} \times \tau^{\hat{C}}$ is also a Polish functor
  by \cite[Proposition 3.38]{Noda_pre_Metrization}
  (see \cite[Definition 3.34 and 3.37]{Noda_pre_Metrization} for Polish functors and product functors, respectively).
  Let $\sigma^{\mathcal{P}(\tau)}$ be the functor introduced in \cite[Section 4.8]{Noda_pre_Metrization}.
  It is then from \cite[Theorem 4.32]{Noda_pre_Metrization} that $\sigma^{\mathcal{P}(\tau)}$ is a Polish functor.
  By taking the product of $\sigma^{\mathcal{P}(\tau)}$ and the functor for measures $\tau^{m}$,
  which is introduced in \cite[Section 4.4]{Noda_pre_Metrization},
  and applying \cite[Theorem 3.23, 3.36 and Proposition 3.38]{Noda_pre_Metrization},
  we deduce the desired result.
\end{proof}

Regarding convergence in $\mbM_{L}$,
we have the following result.

\begin{thm} [{\cite[Theorem 3.24]{Noda_pre_Metrization}}] \label{2. thm: space M_L, convergence}
  Let $\cX=(S, d, \rho, \mu, \pi)$ 
  and $\cX_{n} = (S_{n}, d^{n}, \rho_{n}, \mu_{n}, \pi_{n}),\, n \in \mathbb{N}$
  be elements in $\mbM_{L}$.
  Then, $\cX_{n}$ converges to $\cX$ with respect to $d_{\mbM_{L}}$
  if and only if there exist a rooted boundedly-compact metric space $(M, d, \rho)$ 
  and root-and-distance-preserving maps $f_{n} : S_{n} \to M$ and $f : S \to M$ such that 
  $f_{n}(S_{n}) \to f(S)$ in the local Hausdorff topology,
  $\mu_{n} \circ f_{n}^{-1} \to \mu \circ f^{-1}$ vaguely as measures on $M$
  and $\pi_{n} \circ (\tau_{f_{n}}^{J_{1} \times \hatC})^{-1} \to \pi \circ (\tau_{f}^{J_{1} \times \hatC})^{-1}$
  weakly as probability measures on $D(\mbRp, M) \times \hatC(M \times \mbRp, \mbR)$.
\end{thm}

Let us prepare to describe a precompactness criterion.
For $\xi \in D(\mbRp, S)$, 
we define  
\begin{equation}  \label{2. eq: modulus continuity for cadlag curves}
  \tilde{w}_{S}(\xi, h, t) 
  \coloneqq 
  \inf_{(I_{k}) \in \Pi_{t}} \max_{k} \sup_{r, s \in I_{k}} d(\xi(r), \xi(s)),
  \quad 
  t, h > 0,
\end{equation}
where $\Pi_{t}$ denotes the set of all partitions of the interval $[0, t)$
into subintervals $I_{k} = [u, v)$ with $v - u \geq h$ when $v < t$.

\begin{thm} [{Precompactness in $\mbM_{L}$}]  \label{2. thm: precompactness in M_L}
  Fix a sequence $((S_{n}, d^{n}, \rho_{n}, \mu_{n}, \pi_{n}))_{n \geq 1}$ of elements of $\mbM_{L}$.
  For each $n \in \mathbb{N}$,
  let $(X_{n}, L_{n})$ be a random element of $D(\mbRp, S_{n}) \times \hatC(S_{n} \times \mbRp, \mbR)$
  whose law coincides with $\pi_{n}$.
  We denote the underlying probability measure of $(X_{n}, L_{n})$ by $P_{n}$.
  Fix a dense set $I \subseteq \mbRp$.
  Then,
  the sequence $((S_{n}, d^{n}, \rho_{n}, \pi_{n}))_{n \geq 1}$ is precompact 
  if and only if the following conditions are satisfied.
  \begin{enumerate} [label = (\roman*), series = precompactness in M_L]
    \item \label{2. thm item: precompactness in M_L, spaces are precompact}
      The sequence $((S_{n}, d^{n}, \rho_{n}, \mu_{n}))_{n \geq 1}$ is precompact in the local Gromov-Hausdorff-vague topology. 
    \item \label{2. thm item: precompactness in M_L, values of processes are precompact} 
      For each $t \in I$,
      it holds that 
      $\displaystyle 
        \lim_{r \to \infty} 
        \limsup_{n \to \infty} 
        P_{n} \left( X_{n}(t) \notin S_{n}^{(r)} \right) = 0
      $.
    \item \label{2. thm item: precompactness in M_L, equicontinuity of processes} 
      For each $t > 0$,
      it holds that, for all $\varepsilon > 0$, 
      $\displaystyle 
        \lim_{h \downarrow 0} 
        \limsup_{n \to \infty} 
        P_{n} \left( \tilde{w}_{S_{n}} (X_{n}, h, t) > \varepsilon \right) = 0
      $.
    \item \label{2. thm item: precompactness in M_L, boundedness of local times} 
      For each $r>0$, 
      it holds that 
      $\displaystyle 
        \lim_{M \to \infty} 
        \limsup_{n \to \infty} 
        P_{n} \left( \sup_{x \in \dom_{1}(L_{n})^{(r)}} L_{n}(x, 0) > M \right) = 0.
      $
    \item \label{2. thm item: precompactness in M_L, equicontinuity of local times} 
      For each $r>0$ and $T > 0$,
      it holds that, for all $\varepsilon > 0$,
      \begin{equation}
        \lim_{\delta \downarrow 0} 
        \limsup_{n \to \infty} 
        P_{n} 
        \left(
          \sup_{\substack{ x, y \in \dom_{1}(L_{n})^{(r)}, \\ d^{n}(x, y) < \delta}}
          \sup_{\substack{0 \leq s, t \leq T, \\ |t-s| < \delta}} 
          |L_{n}(x,t) - L_{n}(y,s)| 
          > 
          \varepsilon
        \right)
        = 
        0.
      \end{equation}
  \end{enumerate}
  In that case, the following result holds.
    \begin{enumerate} [resume* = precompactness in M_L]
      \item \label{2. thm item: precompactness in M_L, uniform boundedness of processes}
        For each $t \geq 0$, 
        it holds that 
        $\displaystyle
          \lim_{r \to \infty} 
          \limsup_{n \to \infty} 
          P_{n}\left( X_{n}(s) \notin S_{n}^{(r)}\ \text{for some}\ s \leq t \right) 
          = 0.
        $
    \end{enumerate}
\end{thm}

\begin{proof}
  We have a tightness criterion in $\hatC(S \times \mbRp, \mbR)$ by Theorem \ref{2. thm: tightness in hatC}.
  A tightness criterion in $D(\mathbb{R}_{+}, S)$ is well-known (see \cite[Lemma 4.36]{Noda_pre_Metrization} for exmaple).
  These immediately yield a tightness criterion in $D(\mathbb{R}_{+}, S) \times \hatC(S \times \mbRp, \mbR)$.
  Now, following the proof of \cite[Theorem 4.37]{Noda_pre_Metrization},
  we obtain the desired result readily.
\end{proof}

The precompactness criterion yields the following tightness criterion in $\mbM_{L}$.
It is proven similarly to Theorem \ref{2. thm: tightness in hatC},
and so we omit the proof.

\begin{thm} [{Tightness in $\mbM_{L}$}]  \label{2. thm: tightness in M_L}
  For each $n \in \mathbb{N}$,
  let $(S_{n}, d^{n}, \rho_{n}, \mu_{n}, \pi_{n})$ 
  be a random element of $\mbM_{L}$ 
  built on a probability space $(\Omega_{n}, \mathcal{F}_{n}, \mathbf{P}_{n})$.
  For each $\omega \in \Omega_{n}$,
  let $(X_{n}^{\omega}, L_{n}^{\omega})$ 
  be a random element of $D(\mbRp, S_{n}) \times \hatC(S_{n} \times \mbRp, \mbR)$
  whose law coincides with $\pi_{n}(\omega)$.
  We denote the underlying probability measure of $(X_{n}^{\omega}, L_{n}^{\omega})$ by $P_{n}^{\omega}$.
  Fix a dense set $I \subseteq \mbRp$.
  Then,
  the sequence $((S_{n}, d^{n}, \rho_{n}, \pi_{n}))_{n \geq 1}$ is tight
  if and only if the following conditions are satisfied.
  \begin{enumerate} [label = (\roman*)]
    \item \label{2. thm item: tightness in M_L, spaces are tight}
      The sequence $((S_{n}, d^{n}, \rho_{n}, \mu_{n}))_{n \geq 1}$ is tight in the local Gromov-Hausdorff-vague topology. 
    \item \label{2. thm item: tightness in M_L, values of processes are precompact} 
      For each $t \in T$,
      it holds that, for all $\varepsilon > 0$,
      $\displaystyle 
        \lim_{r \to \infty} 
        \limsup_{n \to \infty} 
        \mathbf{P}_{n}
        \left(
          P_{n}^{\omega} \left( X_{n}^{\omega}(t) \notin S_{n}^{(r)} \right) > \varepsilon
        \right)
        = 0
      $.
    \item \label{2. thm item: tightness in M_L, equicontinuity of processes} 
      For each $t > 0$,
      it holds that, for all $\varepsilon, \delta >0$,
      \begin{equation}
        \lim_{h \downarrow 0} 
        \limsup_{n \to \infty} 
        \mathbf{P}_{n}
        \left(
          P_{n}^{\omega} \left( \tilde{w}_{S_{n}} (X_{n}^{\omega}, h, t) > \varepsilon \right)
          > 
          \delta
        \right)
         = 0.
      \end{equation}
    \item \label{2. thm item: tightness in M_L, boundedness of local times} 
      For each $r>0$, 
      it holds that, for all $\varepsilon > 0$,
      \begin{equation}
        \lim_{M \to \infty} 
        \limsup_{n \to \infty} 
        \mathbf{P}_{n}
        \left(
          P_{n}^{\omega} \left( \sup_{x \in \dom(L_{n})^{(r)}} L_{n}^{\omega}(x, 0) > M \right) 
          > 
          \varepsilon
        \right)
         = 0.
      \end{equation}
    \item \label{2. thm item: tightness in M_L, equicontinuity of local times} 
      For each $r>0$ and $T > 0$,
      it holds that, for all $\varepsilon, \delta > 0$,
      \begin{equation}
        \lim_{\delta \downarrow 0} 
        \limsup_{n \to \infty} 
        \mathbf{P}_{n}
        \left(
          P_{n}^{\omega} 
          \left(
            \sup_{\substack{ x, y \in \dom_{1}(L_{n})^{(r)}, \\ d^{n}(x, y) < \delta}}
            \sup_{\substack{0 \leq s, t \leq T, \\ |t-s| < \delta}} 
            |L_{n}^{\omega}(x,t) - L_{n}^{\omega}(y,s)| 
            > 
            \varepsilon
          \right)
          > 
          \delta
        \right)
        = 
        0.
      \end{equation}
  \end{enumerate}
\end{thm}

Although it is not a space of our main interest,
we introduce another space $\mbM$ for convenience.
This space is used in the proofs of our main results.
Roughly speaking,
it is the space of measured metric spaces equipped with laws of stochastic processes,
and precisely defined as follows.
Let $\mbM^{\circ}$ be the collection of $(S, d, \rho, \mu, \pi')$ 
such that $(S, d, \rho, \mu) \in \mathbb{G}^{\circ}$ and $\pi'$ 
is a probability measure on $D(\mbRp, S)$.
Recall $\tau^{J_{1}}$ from \eqref{2. eq: functor for cadlag functions and local-time functions}.
For $\cX_{i} = (S_{i}, d_{i}, \rho_{i}, \mu_{i}, \pi'_{i}) \in \mbM^{\circ},\, i=1,2$,
we say that $\cX_{1}$ is $(\tau^{J_{1}})^{-1}$-equivalent to $\cX_{2}$ 
if and only if there exists a root-preserving isometry $f: S_{1} \to S_{2}$ 
such that $\mu_{2} = \mu_{1} \circ f^{-1}$ and $\pi'_{2} = \pi'_{1} \circ (\tau_{f}^{J_{1}})^{-1}$.

\begin{dfn} [{The space $\mbM$}]
  We define $\mbM$ to be the collection of $(\tau^{J_{1}})^{-1}$-equivalence classes of elements in $\mbM^{\circ}$.
\end{dfn}

\begin{dfn} [{The metric $d_{\mbM}$}]
  For $\cX_{i} = (S_{i}, d_{i}, \rho_{i}, \mu_{i}, \pi'_{i}) \in \mbM_{L},\, i=1,2$,
  we set 
  \begin{equation}
    \begin{split}
      d_{\mbM_{L}}(\cX_{1}, \cX_{2}) 
      \coloneqq
      \inf_{f_{1}, f_{2}, M} 
      &\Bigl\{ 
        d_{\bar{H}, \rho} (f_{1}(S_{1}), f_{2}(S_{2})) 
        \vee 
        d_{V, \rho} (\mu_{1} \circ f_{1}^{-1}, \mu_{2} \circ f_{2}^{-1})
      \Bigr.\\
      &\vee 
        \tilde{d}_{P}\left(
          \pi'_{1} \circ (\tau_{f_{1}}^{J_{1}})^{-1},
          \pi'_{2} \circ (\tau_{f_{2}}^{J_{1}})^{-1}
        \right)
      \Bigr\},
    \end{split}
  \end{equation}
  where the infimum is taken over all rooted boundedly-compact metric spaces $(M, d, \rho)$ 
  and all root-and-distance-preserving maps $f_{i} : S_{i} \to M,\, i=1,2$,
  and $\tilde{d}_{P}$ denotes the Prohorov metric on the set of probability measures 
  on the metric space $(D(\mbRp, M), d_{J_{1}})$.
\end{dfn}

\begin{thm} 
  The function $d_{\mbM}$ is a well-defined complete, separable metric on $\mbM$.
\end{thm}

\begin{proof}
  Recall the functors $\tau^{J_{1}}$ and $\tau^{m}$ introduced in the proof of Theorem \ref{2. thm: metric on M_L}.
  Let $\sigma^{\mathcal{P}(\tau^{J_{1}})}$ be the functor introduced in \cite[Section 4.8]{Noda_pre_Metrization}.
  It is then from \cite[Theorem 4.31]{Noda_pre_Metrization} that $\sigma^{\mathcal{P}(\tau^{J_{1}})}$ is a 
  complete, separable, continuous functor
  (see \cite[Section 3.2]{Noda_pre_Metrization} for the notion of 
  completeness, separability and continuity for functors).
  By taking the product of $\sigma^{\mathcal{P}(\tau^{J_{1}})}$ and $\tau^{m}$
  and applying \cite[Corollary 3.30 and Proposition 3.38]{Noda_pre_Metrization},
  we deduce the desired result.
\end{proof}

We have the following results similar to $d_{\mbM_{L}}$.
The proofs are omitted because they are proved similarly.

\begin{thm} [{Convergence in $\mbM$}] \label{2. thm: space M, convergence}
  Let $\cX=(S, d, \rho, \mu, \pi')$ 
  and $\cX_{n} = (S_{n}, d^{n}, \rho_{n}, \mu_{n}, \pi'_{n}),\, n \in \mathbb{N}$
  be elements in $\mbM$.
  Then, $\cX_{n}$ converges to $\cX$ in $\mbM$
  if and only if there exist a rooted boundedly-compact metric space $(M, d^{M}, \rho_{M})$ 
  and root-and-distance-preserving maps $f_{n} : S_{n} \to M$ and $f : S \to M$ such that 
  $f_{n}(S_{n}) \to f(S)$ in the local Hausdorff topology,
  $\mu_{n} \circ f_{n}^{-1} \to \mu \circ f^{-1}$ vaguely as measures on $M$
  and $\pi'_{n} \circ (\tau_{f_{n}}^{J_{1}})^{-1} \to \pi' \circ (\tau_{f}^{J_{1}})^{-1}$
  weakly as probability measures on $D(\mbRp, M)$.
\end{thm}

\begin{thm} [{Precompactness in $\mbM$}]  \label{2. thm: precompactness in M}
  Fix a sequence $((S_{n}, d^{n}, \rho_{n}, \mu_{n}, \pi'_{n}))_{n \geq 1}$ of elements of $\mbM$.
  Then,
  the sequence $((S_{n}, d^{n}, \rho_{n}, \pi'_{n}))_{n \geq 1}$ is precompact 
  if and only if 
  the conditions \ref{2. thm item: precompactness in M_L, spaces are precompact},
  \ref{2. thm item: precompactness in M_L, values of processes are precompact} 
  and \ref{2. thm item: precompactness in M_L, equicontinuity of processes} 
  of Theorem \ref{2. thm: precompactness in M_L} are satisfied.
\end{thm}

\begin{thm} [{Tightness in $\mbM$}]  \label{2. thm: tightness in M}
  For each $n \in \mathbb{N}$,
  let $(S_{n}, d^{n}, \rho_{n}, \mu_{n}, \pi'_{n})$ 
  be a random element of $\mbM$.
  Then,
  the sequence $((S_{n}, d^{n}, \rho_{n}, \pi'_{n}))_{n \geq 1}$ is tight
  if and only if 
  the conditions \ref{2. thm item: tightness in M_L, spaces are tight},
  \ref{2. thm item: tightness in M_L, values of processes are precompact} 
  and \ref{2. thm item: tightness in M_L, equicontinuity of processes} 
  of Theorem \ref{2. thm: tightness in M_L}
  are satisfied.
\end{thm}


\section{Uniform continuity of stochastic processes} \label{sec: uniform continuity of stochastic processes}

In this section,
we provide a version of the Kolmogorov-Chentsov continuity theorem (cf.\ \cite[Theorem 3.23]{Kallenberg_02_Foundations}).
One can find similar results and detailed arguments for random fields on $\mbR^{m}$ in  \cite{Potthoff_09_SampleII}.
The difference of our result is that
the condition we assume for an index set of a stochastic process is weaker than that of \cite{Potthoff_09_SampleII},
and moreover, we provide a quantitative estimate for the continuity of a process.

Let $(F,R)$ be a compact metric space
and $(M,d)$ be a separable metric space.
Suppose that $X=(X(x))_{x \in F}$ is a family of random elements of $M$ 
built on a probability space equipped with a probability measure $P$.
Consider the following assumption.

\begin{assum} \label{assum: functions for tail bound for the rv}
  There exist non-decreasing functions $r,q: (0, \infty) \to \mbRp$ 
  satisfying 
  \begin{equation}  \label{eq: a tail bound for the rv}
    P \bigl( d(X(x), X(y)) > r(R(x,y)) \bigr)\leq q(R(x,y)), \quad \forall x,y \in F.
  \end{equation}
\end{assum}

\begin{prop} \label{prop: equicontinuity of stochastic process with compact index}
  Suppose that Assumption \ref{assum: functions for tail bound for the rv} is satisfied.
  Then there exists a dense subset $D \subseteq F$ such that, for each $n \geq 0$,
  \begin{equation}  
    P
    \left(
    \sup_{\substack{x,y \in D \\ R(x,y)<2^{-n+1}}}
    d(X(x),X(y))
    >
    2
    \sum_{k \geq n} r(2^{-k+3})
    \right)
    \leq \sum_{k\geq n} (k+1)^{2}N_{R}(F,2^{-k})^{2} q(2^{-k+3}).
  \end{equation}
  If $X$ is continuous almost-surely,
  then the supremum over $x,y \in D$ can be extended to $x, y \in F$.
\end{prop}

\begin{proof}
  For each $n \geq 0$,
  we let $D'_{n}$ be a minimal $2^{-n}$-covering of $F$ and set
  \begin{equation}
    \delta_{n} \coloneqq 2^{-n+1}, \quad 
    \delta_{n}' \coloneqq 2 \sum_{k \geq n} \delta_{k} = 2^{-n+3}, \quad 
    D_{n} \coloneqq\bigcup_{k=0}^{n} D'_{k}.
  \end{equation}
  Observe that the increasing sequence $(D_{n})_{n}$ satisfies the following condition:
  \begin{enumerate} [label=(U)]
    \item \label{cond: chaining-set condition}
      For any $n \geq 0$ and $x \in D_{n+1}$, there exists a $y \in D_{n}$ such that $R(x,y)\leq \delta_{n+1}$.
  \end{enumerate}
  We define 
  \begin{equation}
    \pi_{n} 
    \coloneqq 
    \{ (x,y) \in D_{n} \times D_{n}: R(x,y) \leq \delta_{n}'\}.
  \end{equation}
  If $(x,y) \in \pi_{n}$,
  then by monotonicity,
  we have $r(R(x,y)) \leq r(\delta'_{n})$ and $q(R(x,y)) \leq q(\delta'_{n})$.
  Thus using \eqref{eq: a tail bound for the rv},
  we obtain that
  \begin{equation}
    \begin{split}
      P\left(\max_{(x,y) \in \pi_{n}} d(X(x),X(y)) > r(\delta'_{n})\right)
      &\leq \sum_{(x,y) \in \pi_{n}} P\left(d(X(x),X(y)) > r(R(x,y))\right) \\
      &\leq  \sum_{(x,y) \in \pi_{n}} q(R(x,y))\\
      &\leq  |\pi_{n}| q(\delta'_{n}) 
      \leq (n+1)^{2} N_{R}(F, 2^{-n})^{2} 1(\delta'_{n}),
    \end{split}
  \end{equation}
  where we use that $|\pi_{n}| \leq |D_{n}|^{2} \leq (n+1)^{2}N_{R}(F,2^{-n})^{2}$ at the last inequality.
  This immediately yields that
  \begin{equation}  \label{eq: upper bound for k geq n}
    P\left(\bigcup_{k \geq n}
    \left\{
    \max_{(x,y) \in \pi_{k}} d(X(x),X(y)) > r(\delta'_{k})
    \right\} \right)
    \leq
    \sum_{k \geq n} (k+1)^{2} N_{R}(F, 2^{-k})^{2} q(\delta'_{k}).
  \end{equation}
  Define $D \coloneqq \bigcup_{n \geq 0} D_{n}$,
  which is dense in $F$.
  Fix $n \geq 0$.
  Assume that 
  \begin{gather} \label{eq: chaining 1}
    \sup_{\substack{x,y \in D \\ R(x,y)<\delta_{n}}}
    d(X(x),X(y)) > 2\sum_{k \geq n} r(\delta'_{k}),
  \end{gather}
  and also
  \begin{gather} \label{eq: chaining 2}
    \max_{(x,y) \in \pi_{k}} d(X(x),X(y))
    \leq
    r(\delta'_{k}),\quad
    \forall k > n.
  \end{gather}
  By \eqref{eq: chaining 1},
  we can find $x, y \in D_{m}$ for some $m > n$
  such that $R(x,y)<\delta_{n}$ and $d(X(x),X(y)) > 2\sum_{k \geq n} r(\delta'_{k})$.
  By \ref{cond: chaining-set condition},
  we have $x_{i} \in D_{i},i=n,n+1, \ldots ,m$
  such that $x_{m}=x,$ and $R(x_{i},x_{i+1}) \leq \delta_{i+1}$ for $i=n, \ldots ,m-1$, i.e., $(x_{i}, x_{i+1}) \in \pi_{i+1}$.
  Similarly, we choose $(y_{i})_{i=n}^{m}$.
  Then, by \eqref{eq: chaining 2} and the triangle inequality,
  we deduce that
  \begin{align}
    d(X(x_{n}), X(y_{n}))
     & \geq
    d(X(x_{m}), X(y_{m}))
    -\sum_{i=n}^{m-1} d(X(x_{i}), X(x_{i+1}))
    -\sum_{i=n}^{m-1} d(X(y_{i}), X(y_{i+1})) \\
     & \geq 2r(\delta'_{n}).
  \end{align}
  By the choice of $(x_{i})_{i=n}^{m}$ and $(y_{i})_{i=n}^{m}$,
  we also have $(x_{n}, y_{n}) \in \pi_{n}$.
  Therefore we obtain that
  \begin{equation}
    \left\{
    \sup_{
      \substack{x,y \in D \\ R(x,y)<\delta_{n}}}
    d(X(x),X(y)) > 2\sum_{k \geq n} r(\delta'_{k})
    \right\}
    \subseteq
    \bigcup_{k \geq n}
    \left\{ \max_{(x,y) \in \pi_{k}} d(X(x),X(y)) > r(\delta'_{k}) \right\}.
  \end{equation}
  This, combined with \eqref{eq: upper bound for k geq n}, yields the desired inequality.
  The last assertion about the replacement of $D$ by $F$ is straightforward.
\end{proof}


\section{Resistance forms and local times} \label{sec: resistance}


\subsection{Resistance forms and associated processes} \label{sec: introduction to resistance metrics}

Following \cite{Croydon_18_Scaling},
in this section we recall some basic properties of resistance forms, starting with their definition.
The reader is referred to \cite{Kigami_12_Resistance} for further background.
Also,
for further study of resistance forms and their extended Dirichlet spaces,
see \cite[Section 3]{Noda_pre_Scaling}.

\begin{dfn} [{Resistance form and resistance metric, \cite[Definition 3.1]{Kigami_12_Resistance}}] 
  \label{4. dfn: resistance forms}
  Let $F$ be a non-empty set.
  A pair $(\mathcal{E}, \mathcal{F})$ is called a \textit{resistance form} on $F$ if it satisfies the following conditions.
  \begin{enumerate} [label=(RF\arabic*), leftmargin = *]
    \item
          The symbol $\mathcal{F}$ is a linear subspace of the collection of functions $\{ f : F \to \mbR \}$ 
          containing constants,
          and $\mathcal{E}$ is a non-negative symmetric bilinear form on $\mathcal{F}$
          such that $\mathcal{E}(f,f)=0$ if and only if $f$ is constant on $F$.
    \item \label{4. dfn cond: the quotient space of resistance form is Hilbert}
          Let $\sim$ be the equivalence relation on $\mathcal{F}$ defined by saying $f \sim g$ if and only if $f-g$ is constant on $F$.
          Then $(\mathcal{F}/\sim, \mathcal{E})$ is a Hilbert space.
    \item
          If $x \neq y$, then there exists an $f \in \mathcal{F}$ such that $f(x) \neq f(y)$.
    \item \label{4. dfn cond: resistance forms condition 4}
          For any $x, y \in F$,
          \begin{equation} \label{eq: definition of resistance metric}
            R(x,y)
            \coloneqq
            \sup
            \left\{
            \frac{|f(x) - f(y)|^{2}}{\mathcal{E}(f,f)}
            :
            f \in \mathcal{F},\
            \mathcal{E}(f,f) > 0
            \right\}
            < \infty.
          \end{equation}     
    \item
          If $\bar{f}\coloneqq  (f \wedge 1) \vee 0$,
          then $\bar{f} \in \mathcal{F}$ and $\mathcal{E}(\bar{f}, \bar{f}) \leq \mathcal{E}(f,f)$ for any $f \in \mathcal{F}$.
  \end{enumerate}
  The function $R: F \times F \to \mbRp$ is a metric on $F$ (see \cite[Proposition 3.3]{Kigami_12_Resistance})
  and called a \textit{resistance metric} (associated with the resistance form $(\mathcal{E}, \mathcal{F})$).
\end{dfn}

For the following definition,
recall the effective resistance on an electrical network with a finite vertex set 
from \cite[Section 9.4]{Levin_Peres_17_Markov}
(see also \cite[Section 2.1]{Kigami_01_Analysis}).

\begin{dfn} [{Resistance metric, \cite[Definition 2.3.2]{Kigami_01_Analysis}}]
  \label{3. dfn: resistance metrics}
  A metric $R$ on a non-empty set $F$ is called a \textit{resistance metric}
  if and only if,
  for any non-empty finite subset $V \subseteq F$,
  there exists an electrical network $G$ with the vertex set $V$ 
  such that the effective resistance on $G$ coincides with $R|_{V \times V}$.
\end{dfn}

\begin{thm} [{\cite[Theorem 2.3.6]{Kigami_01_Analysis}}]  \label{3. thm: one-to-one correspondence of forms and metrics}
  There exists a one-to-one correspondence between resistance forms $(\mathcal{E}, \mathcal{F})$ on $F$ 
  and resistance metrics $R$ on $F$ via $R = R_{(\mathcal{E}, \mathcal{F})}$.
  In other words,
  a resistance form $(\mathcal{E}, \mathcal{F})$ is characterized by $R_{(\mathcal{E}, \mathcal{F})}$ 
  given in \ref{4. dfn cond: resistance forms condition 4}.
\end{thm}

In Assumption \ref{1. assum: deterministic spaces}\ref{1. assum item: deterministic, non-explosion condition} 
and Assumption \ref{1. assum: random spaces}\ref{1. assum item: random spaces, non-explosion},
we consider effective resistance between sets.
This is precisely defined below.

\begin{dfn} [{Effective resistance between sets}]  \label{3. dfn: effective resistance between sets}
  Fix a resistance form $(\mathcal{E}, \mathcal{F})$ on $F$
  and write $R$ for the corresponding resistance metric.
  For sets $A, B \subseteq F$,
  we define
  \begin{equation}
    R(A,B)
    \coloneqq
    \left(
    \inf\{
    \mathcal{E}(f,f):
    f \in \mathcal{F},\
    f|_{A}=1,\
    f|_{B}=0
    \}
    \right)^{-1},
  \end{equation}
  which is defined to be zero if the infimum is taken over the empty set.
  Note that by \ref{4. dfn cond: resistance forms condition 4} 
  we clearly have $R( \{x\}, \{y\}) = R(x,y)$.
\end{dfn}

By \ref{4. dfn cond: resistance forms condition 4},
we have that
\begin{equation}
  | f(x) - f(y) |^{2}
  \leq
  \mathcal{E}(f,f) R(x,y),
\end{equation}
which implies that any function in $\mathcal{F}$ is continuous on the metric space $(F,R)$.

We will henceforth assume that we have a non-empty set $F$ equipped with a resistance form $(\mathcal{E},\mathcal{F})$,
and denote the corresponding resistance metric $R$.
Furthermore,
we assume that $(F, R)$ is locally compact and separable,
and the resistance form $(\mathcal{E}, \mathcal{F})$ is regular, as described by the following.

\begin{dfn} [{Regular resistance form, \cite[Definition 6.2]{Kigami_12_Resistance}}] \label{dfn: regular resistance forms}
  Let $C_{c}(F)$ be the collection of compactly supported, continuous functions on $(F,R)$,
  and $\| \cdot \|$ be the supremum norm for functions on $F$.
  A resistance form $(\mathcal{E}, \mathcal{F})$ on $F$ is called \textit{regular}
  if and only if $\mathcal{F} \cap C_{c}(F)$ is dense in $C_{c}(F)$ with respect to $\| \cdot \|$.
\end{dfn}

We next introduce related Dirichlet forms and stochastic processes.
First, suppose that we have a Radon measure $\mu$ of full support on $(F,R)$.
Let $\mathcal{B}(F)$ be the Borel $\sigma$-algebra on $(F,R)$
and $\mathcal{B}^{\mu}(F)$ be the completion of $\mathcal{B}(F)$ with respect to $\mu$.
Two extended real-valued functions are said to be $\mu$-equivalent
if they coincide outside a $\mu$-null set.
The $L^{2}$-space $L^{2}(F,\mu)$ consists of $\mu$-equivalence classes of square-integrable $\mathcal{B}^{\mu}(F)$-measurable extended real-valued functions on $F$.
Now, we define a bilinear form $\mathcal{E}_{1}$ on $\mathcal{F} \cap L^{2}(F, \mu)$
(where we regard $\mathcal{F}$ as a subspace of $L^{2}(F, \mu)$)
by setting
\begin{equation}
  \mathcal{E}_{1}(f,\, g)
  \coloneqq
  \mathcal{E}(f,\, g)
  +
  \int_{F}fg\, d\mu.
\end{equation}
Then $(\mathcal{F} \cap L^{2}(F, \mu), \mathcal{E}_{1})$ is a Hilbert space (see \cite[Theorem 2.4.1]{Kigami_01_Analysis}).
We write $\mathcal{D}$ to be the closure of $\mathcal{F} \cap C_{c}(F)$ with respect to $\mathcal{E}_{1}$.
Under the assumption that $(\mathcal{E},\mathcal{F})$ is regular,
we then have from \cite[Theorem 9.4]{Kigami_12_Resistance}
that $(\mathcal{E}, \mathcal{D})$ is a regular Dirichlet form on $L^{2}(F,\mu)$
(see \cite{Fukushima_Oshima_Takeda_11_Dirichlet} for the definition of a regular Dirichlet form).
Moreover, standard theory gives us the existence of an associated Hunt process $((X_{t} )_{t \geq 0}, (P_{x} )_{x \in F} )$
(e.g. \cite[Theorem 7.2.1]{Fukushima_Oshima_Takeda_11_Dirichlet}).
Note that such a process is, in general, only specified uniquely for starting points outside a set of zero capacity.
However, in this setting,
every point has strictly positive capacity (see \cite[Theorem 9.9]{Kigami_12_Resistance}),
and so the process is defined uniquely everywhere.

\begin{rem}
  In \cite[Chaper 9]{Kigami_12_Resistance},
  in addition to the above assumptions,
  $(F, R)$ is assumed to be complete but it is easy to remove this assumption.
\end{rem}

The extended Dirichlet space $\mathcal{D}_{e}$ is the totality of $\mu$-equivalence classes of $\mathcal{B}^{\mu}(F)$-measurable functions $f$ on $F$
such that $|f| < \infty$, $\mu$-a.e.\
and there exists an $\mathcal{E}$-Cauchy sequence $(f_{n})_{n \geq 0}$ in $\mathcal{D}$ with $f_{n}(x) \to f(x)$, $\mu$-a.e.\
(see \cite[Definition 1.1.4]{Chen_Fukushima_12_Symmetric}).
For $(F,R,\rho,\mu) \in \mathbb{F}$,
it is an assumption that $(\mathcal{E}, \mathcal{D})$ is recurrent.
We recall that a Dirichlet form $(\mathcal{E}, \mathcal{D})$ being recurrent is equivalent to
$1 \in \mathcal{D}_{e}$ and $\mathcal{E}(1,1)=0$ both holding (see \cite[Theorem 1.6.3]{Fukushima_Oshima_Takeda_11_Dirichlet}).
In the present setting, the recurrence of a Dirichlet form is characterized by a resistance form as follows.

\begin{lem} [{\cite[Lemma 2.3]{Croydon_18_Scaling}}] \label{lem: characterization of recurrence}
  In the above setting,
  the associated regular Dirichlet form $(\mathcal{E}, \mathcal{D})$ is recurrent
  if and only if $\displaystyle \lim_{r \to \infty} R(\rho, B_{R} (\rho, r)^{c}) =\infty$ 
  for some (or equivalently, any) $\rho \in F$.
\end{lem}

\begin{rem} \label{4. rem: recurrence implies regularity}
  By \cite[Corollary 3.22]{Noda_pre_Scaling},
  if a resistance metric satisfies $\displaystyle \lim_{r \to \infty} R(\rho, B_{R} (\rho, r)^{c}) =\infty$
  for some (or equivalently, any) $\rho \in F$,
  then the corresponding resistance form is regular.
\end{rem}


\subsection{Joint continuity of local times}
In this section,
we give a sufficient condition for the existence of jointly continuous local times of a Hunt process associated with a resistance form.
Let $(\mathcal{E},\mathcal{F})$ be a regular resistance form on a non-empty set $F$ and $R$ be the corresponding resistance metric.
We assume that $(F, R)$ is separable and locally compact.
Let $\mu$ be a Radon measure of full support on $(F, R)$
and $(\mathcal{E}, \mathcal{D})$ be the regular Dirichlet form associated with $(F, R, \mu)$.
Let $(X=(X_{t})_{t \geq 0},(P_{x} )_{x \in F})$ be the associated Hunt process.

For discussing the joint continuity of local times,
we begin by introducing the $\alpha$-potential density, following \cite{Barlow_98_Diffusions}.
In particular, as in \cite[Theorem 7.20]{Barlow_98_Diffusions},
one can check that the Hunt process has a continuous $\alpha$-potential density with respect to $\mu$.
(NB.\ in \cite[Theorem 7.20]{Barlow_98_Diffusions}, the space is assumed to be compact, but it is not difficult to remove this assumption.)

\begin{lem} \label{lem: continuous potential density of a resistance form}
  The Hunt process has a continuous and symmetric $\alpha$-potential density $u_{\alpha}(x,y)$ with respect to $\mu$ for each $\alpha > 0$
  such that, for any $x \in F$, $u_{\alpha}(x, \cdot) \in \mathcal{D}$ and
  \begin{equation}
    \mathcal{E}_{\alpha}(u_{\alpha}(x,\cdot),f)=f(x)
  \end{equation}
  for any $f \in \mathcal{D}$,
  where $\mathcal{E}_{\alpha}(g,h)=\mathcal{E}(g,h)+\alpha \int gh\,  d\mu$ for each $g,h \in \mathcal{D}$.
  Moreover, it holds that
  \begin{equation} \label{eq: inequality for alpha potential density}
    (u_{\alpha}(x,y)-u_{\alpha}(x,z))^{2}
    \leq
    u_{\alpha}(x,x)R(y,z)
  \end{equation}
  for any $x,y,z \in F$.
\end{lem}

We recall the definition of local times.
Let $(\Omega, \mathcal{F})$ be the measurable space where the probability measures $(P_{x})_{x \in F}$ are defined.
We denote the minimum completed admissible filtration of $X$ by $(\mathcal{F}_{t})_{t \geq 0}$,
the family of the translation (shift) operators for $X$ by $(\theta_{t})_{t \geq 0}$
and the lifetime of $X$ by $\zeta$
(see \cite{Fukushima_Oshima_Takeda_11_Dirichlet} for these definitions).

\begin{dfn} [{PCAF and local time}]
  A non-decreasing, continuous, $(\mathcal{F}_{t})_{t \geq 0}$-adapted process $A = (A_{t})_{t \geq 0}$ on $(\Omega, \mathcal{F})$
  is called a \textit{positive continuous additive functional (PCAF)} of $X$ 
  if for all $x \in F$ it holds $P_{x}$-a.s.\ that $A_{0} = 0$, $A_{t} = A_{\zeta}$ for all $t \geq \zeta$ 
  and $A_{s+t} = A_{s} + A_{t} \circ \theta_{s}$ for all $s,t \geq 0$.
  A PCAF $A = (A_{t})_{t \geq 0}$ of $X$ is called a \textit{local time} of $X$ at $x \in F$ 
  if $P_{x}(R_{A} = 0) = 1$ and $P_{y}(R_{A} = 0) = 0$ for all $y \neq x$,
  where we set $R_{A}(\omega) \coloneqq \inf \{t \geq 0 : A_{t}(\omega) > 0\}$.
\end{dfn}

\begin{rem}
  Lemma \ref{lem: continuous potential density of a resistance form} implies that the Hunt process $X$ is strongly symmetric.
  Note that, by saying $X$ is strongly symmetric,
  we mean that $X$ is symmetric with respect to $\mu$ and the measure $U_{\alpha}(\cdot)=U_{\alpha}(x, \cdot)$ given by
  \begin{equation}
    U_{\alpha}(x,\cdot)
    =
    \int_{0}^{\infty} e^{-\alpha t} P_{t}(x, \cdot) dt
  \end{equation}
  is absolutely continuous with respect to $\mu$, where $P_{t}$ is the transition function of $X$.
\end{rem}

Throughout the article,
we write
\begin{equation}
  \sigma_{A} \coloneqq  \inf \{t > 0 : X_{t} \in A\}
\end{equation}
for the hitting time of a set $A$ by $X$, and abbreviate $\sigma_{x} \coloneqq  \sigma_{\{x\}}$.
Since the $1$-potential density $u_{1}$ is finite,
we have that
\begin{equation}\label{eq: expectation of hitting time}
  E_{x}(e^{-\sigma_{y}})
  =\frac{u_{1}(x,y)}{u_{1}(y,y)},\quad
  \forall x, y \in F
\end{equation}
(see \cite[Theorem 3.6.5]{Marcus_Rosen_06_Markov}).
In particular,
$E_{x}(e^{-\sigma_{y}})$ is jointly measurable,
and thus, by \cite[Theorem 1]{Getoor_Kesten_72_Continuity},
there exists a local time $(L_{t}(x))_{t \geq 0}$ at each $x \in F$
such that $L=(L_{t}(x))_{t \geq 0,\, x \in F}$ is jointly measurable.
Joint continuity of local times of a strongly symmetric Markov process was studied intensively in \cite{Marcus_Rosen_92_Sample},
and a strong connection between sample path properties of local times and those of Gaussian processes was obtained.
Let $(G(y))_{y \in F}$ be a mean zero Gaussian process with covariance
\begin{equation} \label{eq: definition of corresponding Gaussian process}
  E(G(x)G(y))=u_{1}(x,y),\quad
  \forall x,y \in F.
\end{equation}
Then a psudometric $d_{G}$ on $F$ is defined by setting
\begin{equation} \label{eq: metric for Gaussian process}
  d_{G}(x,y)
  =
  (E(G(x)-G(y))^{2})^{1/2}
  =
  (u_{1}(x,x)+u_{1}(y,y)-2u_{1}(x,y))^{1/2}.
\end{equation}
By \eqref{eq: expectation of hitting time} (see also \cite[Lemma 3.6]{Marcus_Rosen_92_Sample}),
we have $u_{1}(x,y) < u_{1}(x,x) \wedge u_{1}(y,y)$ for $x \neq y$,
which implies that $d_{G}$ is a metric on $F$.
Now we have two metrics $R$ and $d_{G}$ on $F$, and the following result concerns the topologies on $F$ induced by these metrics.

\begin{lem} \label{lem: equivalence of R and dG}
  The identity map $\mathrm{id}_{F} : (F,R) \to (F,d_{G})$ is a homeomorphism.
\end{lem}

\begin{proof}
  Let $(U_{n})_{n}$ be an increasing sequence of open subsets of $(F, R)$
  such that the closure $\overline{U}_{n}$ with respect to $R$ is compact in $(F, R)$.
  It is enough to show
  that the identity map $\mathrm{id}_{\overline{U}_{n}} : (\overline{U}_{n}, R|_{\overline{U}_{n}}) \to (\overline{U}_{n}, d_{G}|_{\overline{U}_{n}})$
  is a homeomorphism, where $R|_{\overline{U}_{n}}$ and $d_{G}|_{\overline{U}_{n}}$ are the restrictions of $R$ and $d_{G}$ to $\overline{U}_{n}$, respectively.
  Set $c_{n}\coloneqq \sup_{x \in \overline{U}_{n}} u_{1}(x,x)^{1/4} <\infty$.
  By \eqref{eq: inequality for alpha potential density},
  we have
  \begin{equation} \label{eq: relation of two metrics for Gaussian process}
    d_{G}(x,y)
    \leq (u_{1}(x,x)-u_{1}(x,y))^{1/2}+ (u_{1}(y,y)-u_{1}(x,y))^{1/2}
    \leq 2c_{n}R(x,y)^{1/4}
  \end{equation}
  for all $x,y \in \overline{U}_{n}$.
  By \eqref{eq: relation of two metrics for Gaussian process}, $\mathrm{id}_{\overline{U}_{n}}$ is continuous.
  Since $(\overline{U}_{n}, R|_{\overline{U}_{n}})$ is a compact space and $(\overline{U}_{n}, d_{G}|_{\overline{U}_{n}})$ is a Hausdorff space,
  $\mathrm{id}_{\overline{U}_{n}}$ is a homeomorphism, which completes the proof.
\end{proof}

By \cite[Theorem 2]{Marcus_Rosen_92_Sample},
the local time $L$ is jointly continuous on $\mbRp \times F$ almost-surely
if and only if $G$ is continuous on $(F,R)$ almost-surely.
Combining this with Lemma \ref{lem: equivalence of R and dG}, we obtain the following result.

\begin{thm} \label{thm: equivalence of continuity of local times and a Gaussian process}
  The local time $L$ is jointly continuous on $\mbRp \times F$ almost-surely
  if and only if the Gaussian process $G$ defined by \eqref{eq: definition of corresponding Gaussian process} is continuous almost-surely on $(F, R)$,
  which is equivalent to saying that it is continuous almost-surely on $(F, d_{G})$,
  where $d_{G}$ is given by \eqref{eq: metric for Gaussian process}.
\end{thm}

\begin{rem}
  Whenever we say that the local time $L$ is jointly continuous almost-surely,
  we mean that
  we can find a stochastic process $\bar{L} = (\bar{L}_{t}(y))_{t \geq 0, y \in F}$
  that is jointly continuous on $\mbRp \times F,\ P_{x}$-a.s.\
  for all $x \in F$
  and which satisfies
  \begin{equation}
    \bar{L}_{t}(y)=L_{t}(y), \quad
    \forall t \in \mbRp, \quad
    P_{x} \text{-a.s.}
  \end{equation}
  (Note that $(\bar{L}_{t}(x))_{t \geq 0}$ is also a local time of $X$ at $x$.)
  Also, whenever we say that the Gaussian process $G$ is continuous almost-surely,
  we mean that we can find a stochastic process $\bar{G}=(\bar{G}(x))_{x \in F}$
  such that $\bar{G}$ is continuous on $F$ almost-surely and $G(x)=\bar{G}(x)$ almost-surely for each $x \in F$.
\end{rem}

Necessary and sufficient conditions for the continuity of the Gaussian process $G$ with respect to $d_{G}$ were obtained by Talagrand \cite{Talagrand_87_Regularity},
and Theorem \ref{thm: equivalence of continuity of local times and a Gaussian process} implies
that the conditions are also necessary and sufficient conditions for the joint continuity of the local time $L$.
However, for our arguments,
a simple sufficient condition in terms of metric entropy due to Dudley \cite{Dudley_67_The_sizes}
(see also \cite{Dudley_73_Sample,Marcus_Rosen_06_Markov} ) is enough.

\begin{prop} \label{prop: dudley condition}
  Let $(K_{n})_{n}$ be an increasing sequence of compact subsets of $(F,R)$ such that their union is $F$.
  Assume that
  \begin{equation} \label{eq: covering condition for dudley}
    \int_{0}^{1} (\log N_{R^{1/4}}(K_{n},r))^{1/2} dr < \infty, \quad
    \forall n.
  \end{equation}
  Then the Hunt process $X$ admits jointly continuous local times $L=(L_{t}(x))_{t \geq 0,\, x \in F}$.
  Moreover, the local times satisfy the occupation density formula,
  i.e.\ it holds that, for all $x \in F$, $t \geq 0$ and all non-negative measurable functions $f: F \to \mbRp$,
  \begin{equation} \label{eq: the occupation density formula}
    \int_{0}^{t} f(X_{s})ds = \int_{F}f(y) L_{t}(y) \mu(dy), \quad
    P_{x} \text{-a.s.}
  \end{equation}
\end{prop}

\begin{proof}
  By Theorem \ref{thm: equivalence of continuity of local times and a Gaussian process} and \cite[Theorem 6.1.2]{Marcus_Rosen_06_Markov},
  the Hunt process admits jointly continuous local times if it holds that
  \begin{equation} \label{eq: covering condition wrt dG}
    \int_{0}^{1} (\log N_{d_{G}}(K_{n},r))^{1/2} dr < \infty, \quad
    \forall n.
  \end{equation}
  Set $c_{n}=\sup_{x \in K_{n}} u_{1}(x,x)^{1/4} <\infty$.
  Then the inequality \eqref{eq: relation of two metrics for Gaussian process} holds for all $x,y \in K_{n}$.
  This yields that $N_{d_{G}}(K_{n},2c_{n}r) \leq N_{R^{1/4}}(K_{n},r)$.
  Therefore if the condition \eqref{eq: covering condition for dudley} is satisfied,
  then the condition \eqref{eq: covering condition wrt dG} is satisfied,
  and thus the Hunt process admits jointly continuous local times.
  The second part of the assertion follows from \cite[Theorem 3.7.1]{Marcus_Rosen_06_Markov}.
\end{proof}

\begin{cor} \label{cor: joint continuity of local times with sum condition}
  Let $(K_{n})_{n}$ be an increasing sequence of compact subsets of $(F,R)$ such that their union is $F$.
  If, for each $n$,
  there exists $\alpha_{n} \in (0,1/2)$ such that
  \begin{equation}\label{eq: covering condition for mine}
    \sum_{k} N_{R}(K_{n}, 2^{-k})^{2} \exp (-2^{\alpha_{n}k})
    <
    \infty,
  \end{equation}
  then the Hunt process admits jointly continuous local times satisfying the occupation density formula \eqref{eq: the occupation density formula}.
\end{cor}

\begin{proof}
  From \eqref{eq: covering condition for mine}, we deduce that
  \begin{equation}
    \begin{split}
      \infty
      &>
      \sum_{k \geq 1} N_{R}(K_{n}, 2^{-k})^{2} \exp (-2^{\alpha_{n} k}) \\
      &=
      \sum_{k \geq 1} \int_{2^{-k}}^{2^{-k+1}} N_{R}(K_{n}, 2^{-k})^{2}
      \exp (-2^{\alpha_{n} k}) 2^{k} dr\\
      &\geq
      \sum_{k \geq 1} \int_{2^{-k}}^{2^{-k+1}} N_{R}(K_{n}, r)^{2}
      \exp (-(2r^{-1})^{\alpha_{n}}) r^{-1} dr\\
      &=
      \int_{0}^{1} N_{R}(K_{n}, r)^{2}
      \exp (-(2r^{-1})^{\alpha_{n}}) r^{-1} dr\\
      &=
      \int_{0}^{1} N_{R^{1/4}}(K_{n}, u)^{2}
      \exp (-(2u^{-4})^{\alpha_{n}}) 4u^{-1} du.
    \end{split}
  \end{equation}
  Observe that there exists a constant $a > 0$ such that $a ((\log x) \lor 0)^{1/2} \leq x$ for all $x >0$.
  Using this inequality yields that
  \begin{equation} \label{eq: int of log plus g is finite}
    a \int_{0}^{1} ((2\log N_{R^{1/4}}(K_{n}, u)+g(u)) \lor 0)^{1/2}\ du
    <
    \infty,
  \end{equation}
  where we set $g(u)=-(2u^{-4})^{\alpha_{n}}+\log (4u^{-1})$.
  Note that
  \begin{equation} \label{eq: int of g is finite}
    \int_{0}^{1} |g(u)|^{1/2} du
    <
    \infty
  \end{equation}
  since $\alpha_{n} \in (0, 1/2)$.
  By the inequalities \eqref{eq: int of log plus g is finite}, \eqref{eq: int of g is finite} and the triangle inequality,
  we obtain that
  \begin{equation}
    \begin{split}
      \int_{0}^{1} (2\log N_{R^{1/4}}(K_{n}, u))^{1/2} du
      &=
      \int_{0}^{1} (2\log N_{R^{1/4}}(K_{n}, u) \lor 0)^{1/2} du \\
      &\leq
      \int_{0}^{1} ((2\log N_{R^{1/4}}(K_{n}, u)+g(u)) \lor 0
      +|g(u)|)^{1/2} du \\
      &\leq
      \int_{0}^{1} ((2\log N_{R^{1/4}}(K_{n}, u)+g(u)) \lor 0)^{1/2}du
      +\int_{0}^{1} |g(u)|^{1/2} du \\
      &<
      \infty.
    \end{split}
  \end{equation}
  Therefore the condition \eqref{eq: covering condition for dudley} in Proposition \ref{prop: dudley condition} is satisfied.
\end{proof}

\begin{rem} \label{rem: difference of dudley and mine}
  The conditions \eqref{eq: covering condition for dudley} and \eqref{eq: covering condition for mine} look quite similar
  and actually when one considers the case $N_{R}(K_{n}, r)=\exp(r^{-\beta}), r>0$,
  then the left-hand sides of \eqref{eq: covering condition for dudley} and \eqref{eq: covering condition for mine} are each finite
  if and only if $\beta < 1/2$.
  However the condition \eqref{eq: covering condition for dudley} is slightly weaker because if $N_{R}(K_{n},r)=\exp(f(r))$,
  where $f$ is given by
  \begin{equation}
    f(r)=r^{\frac{1}{\sqrt{k}}-\frac{1}{2}} \quad \text{if} \quad 2^{-k-1} < r \leq 2^{-k},
  \end{equation}
  then one can check that \eqref{eq: covering condition for dudley} holds,
  but \eqref{eq: covering condition for mine} fails.
\end{rem}


\subsection{Equicontinuity of local times} \label{sec: equicontinuity of local times}

In this section, we study equicontinuity of local times,
which plays a crucial role in the proof of the main results of this article.
We proceed in the same setting as the previous section.

\begin{prop} \label{prop: stochastic continuity in non-cpt case}
  If $(F, R)$ is compact,
  then it holds that 
  \begin{equation} \label{4. eq: pairwise continuity inequality of local times in cpt case}
    P_{z}
    \left(
    \sup_{0 \leq t \leq T}
    |L_{t}(x)-L_{t}(y)|
    >
    2\delta
    \right)
    \leq
    2 e^{T} \exp \left(-\frac{\delta}{\sqrt{2\mu (F) R(x,y)}}\right)
  \end{equation}
  for any $\delta > 0, T \geq 0$ and $x,y,z \in F$.
\end{prop}

\begin{proof}
  By \cite[Chapter\, V. Proposition 3.28]{Blumenthal_Getoor_68_Markov},
  we have that,
  for any $x,y,z \in F,\, T \geq 0$ and $\delta >0$,
  \begin{equation} \label{4. eq: local time inequality by Blumenthal and Getoor}
    P_{z}
    \left(
    \sup_{0 \leq t \leq T}
    |L_{t}(x)-L_{t}(y)|
    >
    2\delta
    \right)
    \leq
    2 e^{T} e^{-\delta / \gamma},
  \end{equation}
  where we set $\gamma \coloneqq  (1-E_{x}(e^{-\sigma_{y}})E_{y}(e^{-\sigma_{x}}))^{1/2}$.
  By the inequality $e^{-x} \geq 1 -x$,
  it follows that 
  \begin{equation}
    \theta(x,y)
    \coloneqq  
    1-E_{x}(e^{-\sigma_{y}})
    \leq 
    E_{x}(\sigma_{y}).
  \end{equation}
  Then, the commute time identity (see \cite[Lemma 2.4]{Croydon_18_Scaling}) yields that 
  \begin{equation}  \label{4. eq: estimate of gamma}
    \gamma =(\theta(x,y)+\theta(y,x)-\theta(x,y) \theta(y,x))^{1/2}
    \leq
    (\theta(x,y)+\theta(y,x))^{1/2}
    \leq
    \sqrt{2 \mu(F) R(x,y)}.
  \end{equation}
  Combining \eqref{4. eq: local time inequality by Blumenthal and Getoor} with \eqref{4. eq: estimate of gamma},
  we obtain \eqref{4. eq: pairwise continuity inequality of local times in cpt case}.
\end{proof}

Now it is possible to show the main results of this section.

\begin{thm} \label{thm: tightness inequality of local times in cpt case}
  Assume that $(F, R)$ is compact 
  and $X$ admits jointly continuous local times $L=(L_{t}(x))_{t \geq 0,\, x \in F}$.
  Then for every $\alpha \in (0,1/2)$,
  there exists a constant $c_{\alpha} \in (0, \infty)$ depending only on $\alpha$ such that
  \begin{align}
    &P_{z}
    \left(
    \sup_{\substack{ x,y \in F \\ R(x,y) < 2^{-n+1}}}
    \sup_{0 \leq t \leq T}
    |L_{t}(x)-L_{t}(y)|
    >
    c_{\alpha} \sqrt{\mu(F)}\, 2^{-(\frac{1}{2}-\alpha)n}
    \right) \\
    \leq
    &2e^{T}
    \sum_{k \geq n} (k+1)^{2} N_{R}(F,2^{-k})^{2} \exp \left(-2^{\alpha(k-3)}
    \right)
  \end{align}
  for any $z \in F,\, T \geq 0,\, n \geq 0$.
\end{thm}

\begin{proof}
  We write $d^{C([0,T])}$ for the uniform metric on $C([0,T],\mbR)$.
  We have that
  \begin{equation}
    d^{C([0,T])}(L_{\cdot}(x),L_{\cdot}(y))
    =
    \sup_{0 \leq t \leq T} |L_{t}(x)-L_{t}(y)|.
  \end{equation}
  Setting $\delta = \sqrt{2\mu(F)} R(x,y)^{\frac{1}{2} - \alpha}$ 
  in \eqref{4. eq: pairwise continuity inequality of local times in cpt case} yields that
  \begin{equation}
    P_{z}
    \left(
    d^{C([0,T])} (L_{\cdot}(x),L_{\cdot}(y))
    >
    2\sqrt{2\mu(F)} R(x,y)^{\frac{1}{2} - \alpha}
    \right)
    \leq
    2e^{T} \exp \left( -R(x,y)^{-\alpha} \right).
  \end{equation}
  Therefore we can apply Proposition \ref{prop: equicontinuity of stochastic process with compact index}
  to random elements $(L_{\cdot}(x))_{x \in K}$ of $C([0,T],\mbR)$
  with functions $r(u)= 2\sqrt{2\mu(F)} u^{\frac{1}{2}-\alpha}$ and $q(u)=2e^{T} \exp(-u^{-\alpha})$
  to obtain that, for some dense subset $D$ in $F$,
  \begin{align}
    &P_{z}
    \left(
    \sup_{\substack{ x,y \in D \\ R(x,y) < 2^{-n+1}}}
    d^{C([0,T])} (L_{\cdot}(x),L_{\cdot}(y))
    >
    c_{\alpha} \sqrt{\mu(F)} 2^{-(\frac{1}{2}-\alpha)n}
    \right) \\
    \leq
    &2e^{T}\sum_{k \geq n} (k+1)^{2} N_{R}(F,2^{-k})^{2}
    \exp \left(-2^{\alpha(k-3)} \right),
  \end{align}
  where $c_{\alpha} \in (0,\infty)$ is a constant depending on $\alpha$.
  By the joint continuity of the local times, we can replace $D$ by $F$ in the above inequality.
\end{proof}


\section{Proof of Theorem \ref{1. thm: main result for deterministic spaces}}  
\label{sec: proof of the main theorem for deterministic spaces}

The first assertion of Theorem \ref{1. thm: main result for deterministic spaces} follows from the next proposition.

\begin{prop} \label{5. prop: limit belongs to vF}
  Assume that a sequence $G_{n}=(F_{n},R_{n},\rho_{n},\mu_{n})$ in $\check{\mathbb{F}}$ converges to $G=(F,R,\rho,\mu)$ in $\mathbb{F}$
  in the local Gromov-Hausdorff-vague topology and,
  for every $r >0$,
  there exists $\alpha_{r} \in (0, 1/2)$ such that
  \begin{equation}
    \liminf_{n \to \infty}
    \sum_{k} N_{R_{n}}(F_{n}^{(r)},2^{-k})^{2} \exp(-2^{\alpha_{r} k})
    <
    \infty.
  \end{equation}
  It is then the case that $G$ belongs to $\check{\mathbb{F}}$.
\end{prop}

\begin{proof}
  Choose $r' > r$ so that $G_{n}^{(r')}$ converges to $G^{(r')}$ in the Gromov-Hausdorff-Prohorov topology.
  By \cite[Theorem 3.12]{Noda_pre_Metrization},
  we can find $r_{k} \in [2^{-k-2},2^{-k-1}]$ such that
  \begin{equation}
    \lim_{n \to \infty}N_{R^{(r')}_{n}}(F^{(r')}_{n},r_{k})
    =
    N_{R^{(r')}}(F^{(r')},r_{k}).
  \end{equation}
  Choose $\alpha_{r'}' \in (\alpha_{r'}, 1/2)$.
  Using that $N_{R^{(r)}}(F^{(r)}, 2^{-k}) \leq N_{R^{(r')}}(F^{(r')}, r_{k})$ and Fatou's lemma,
  we deduce that
  \begin{align}
    \sum_{k \geq 1}
    N_{R^{(r)}}(F^{(r)},2^{-k})^{2}
    \exp(-2^{\alpha_{r'}' k })
     & \leq
    \sum_{k \geq 1}
    N_{R^{(r')}}(F^{(r')},r_{k})^{2}
    \exp(-2^{\alpha_{r'} ' k }) \notag \\
     & \leq
    \liminf_{n \to \infty}
    \sum_{k \geq 1 }
    N_{R^{(r')}_{n}}(F^{(r')}_{n},r_{k})^{2}
    \exp(-2^{\alpha_{r'} ' k }) \notag \\
     & \leq
    \liminf_{n \to \infty}
    \sum_{k \geq 3}
    N_{R^{(r')}_{n}}(F^{(r')}_{n},2^{-k})^{2}
    \exp(-2^{\alpha_{r'} ' (k-2) })
    <
    \infty.
  \end{align}
\end{proof}

We recall important results from \cite{Croydon_18_Scaling}.

\begin{lem}[{\cite[Lemma 4.2]{Croydon_18_Scaling}}] \label{5. lem: exit time estimate by croydon}
  For every $G=(F,R,\rho,\mu) \in \mathbb{F},\, \delta \in (0, R(\rho, B_{R}(\rho,r)^{c}))$ and $T \geq 0$,
  it holds that
  \begin{equation}
    P_{\rho}^{G}
    (\sigma_{B_{R}(\rho,r)^{c}} \leq T)
    \leq\frac{4\delta}{R(\rho,B_{R}(\rho,r)^{c})}
    +
    \frac{4T}
    {\mu(B_{R}(\rho,\delta))(R(\rho,B_{R}(\rho,r)^{c})-\delta)}.
  \end{equation}
\end{lem}

\begin{lem} \label{5. lem: uniform exit time estimate}
  Under Assumption \ref{1. assum: deterministic spaces}, 
  it holds that 
  \begin{equation}  \label{5. lem eq: uniform exit time estimate}
    \lim_{r \to \infty}
    \limsup_{n \to \infty}
    P_{\rho_{n}}^{G_{n}} \left( \sigma_{B_{R_{n}}(\rho, r)^{c}} \leq T \right) 
    = 0,
    \quad 
    \forall T >0.
  \end{equation}
\end{lem}

\begin{proof}
  Using the convergence $G_{n} \to G$ in the local Gromov-Hausdorff-vague topology 
  and Theorem \ref{2. thm: characterization of convergence in the local. GHV},
  we deduce that $\liminf_{n \to \infty} \mu_{n}(B_{R_{n}}(\rho_{n}, 1)) \geq \mu(B_{R}(\rho, 1)) > 0$.
  This, combined with Lemma \ref{5. lem: exit time estimate by croydon} 
  and Assumption \ref{1. assum: deterministic spaces}\ref{1. assum item: deterministic, non-explosion condition},
  yields the desired result.
\end{proof}

\begin{prop}  \label{5. prop: convergence of processes in M}
  If Assumption \ref{1. assum: deterministic spaces}\ref{1. assum item: deterministic, convergence of spaces} 
  is satisfied and 
  \eqref{5. lem eq: uniform exit time estimate} holds,
  then 
  it holds that 
  \begin{equation}
    \left( F_{n}, R_{n}, \rho_{n}, \mu_{n}, P_{\rho_{n}}^{G_{n}}(X_{G_{n}} \in \cdot) \right) 
    \to 
    \left( F, R, \rho, \mu, P_{\rho}^{G}(X_{G} \in \cdot) \right)
  \end{equation}
  as elements in $\mbM$
  (recall this space from Section \ref{sec: the space M_L}).
  In particular,
  under Assumption \ref{1. assum: deterministic spaces}\ref{1. assum item: deterministic, convergence of spaces}
  and \ref{1. assum item: deterministic, non-explosion condition},
  the above convergence holds.
\end{prop}

\begin{proof}
  By Theorem \ref{2. thm: characterization of convergence in the local. GHV},
  there exist a boundedly-compact metric space $(Z, d, \rho)$ 
  and root-and-distance-preserving maps $f_{n}: F_{n} \to Z$ and $f: F \to Z$ 
  such that $f_{n}(F_{n}) \to f(F)$ in the local Hausdorff topology in $Z$ and
  $\mu_{n} \circ f_{n}^{-1} \to \mu \circ f^{-1}$ vaguely as measures on $Z$.
  By Lemma \ref{5. lem: uniform exit time estimate},
  if Assumption \ref{1. assum: deterministic spaces}\ref{1. assum item: deterministic, non-explosion condition} is satisfied,
  then we have \eqref{5. lem eq: uniform exit time estimate}.
  In the proof of \cite[Theorem 1.2]{Croydon_18_Scaling},
  using \eqref{5. lem eq: uniform exit time estimate},
  it is proven that 
  \begin{equation}
    P_{\rho_{n}}^{G_{n}}(f_{n} \circ X_{G_{n}} \in \cdot)
    \rightarrow
    P_{\rho}^{G}(f \circ X_{G} \in \cdot)
  \end{equation}
  weakly as probability measures on $D(\mbRp,Z)$.
  Therefore, by Theorem \ref{2. thm: space M, convergence}, we obtain the desired result.
\end{proof}

For a topological space $E$,
we denote by $\mathcal{P}(E)$ the totality of Borel probability measures on $E$
which we equip with the weak topology.
Let $G=(F,R,\rho, \mu)$ be an element of $\check{\mathbb{F}}$.
Since the map
\begin{equation}
  \mathcal{P}
  \left(
    D(\mbRp, F)
    \times
    \hatC(F \times \mbRp, \mbR)
  \right)
  \ni P
  \mapsto
  (F,R,\rho, \mu, P)
  \in
  \mbM_{L}
\end{equation}
is continuous,
$\cX_{G}$ is a random element of $\mbM_{L}$.
Given a rooted boundedly-compact metric space $G'=(S, d, \rho, \mu)$
such that $\mu$ is of full support and $\delta > 0$,
we define two functions $f_{\delta}^{S} : S \times S \rightarrow  \mbRp$
and $g_{\delta}^{G'} : D(\mbRp, S) \rightarrow C(S \times \mbRp, \mbRp)$
by setting
\begin{equation}
  f_{\delta}^{S}(x,y)=0 \vee (\delta - d(x,y)),
  \quad
  g_{\delta}^{G'}(Y)(x,t)
  =
  \frac{\int_{0}^{t} f_{\delta}^{S}(x,Y_{s})\,ds}
  {\int_{S} f_{\delta}^{S}(x,y)\, \mu(dy)}.
\end{equation}
In particular,
$f_{\delta}$ approximates a delta function,
and $g_{\delta}$ is an approximation of the local time, as made precise in the next lemma.

\begin{lem} \label{5. lem: delta approximation of local times}
  For every $G=(F,R,\rho, \mu) \in \check{\mathbb{F}}$,
  $g_{\delta}^{G}(X_{G}) \to L_{G}$ in $C(F \times \mbRp, \mbRp)$
  as $\delta \downarrow 0$, $P_{\rho}^{G}$-a.s.
\end{lem}

\begin{proof}
  By the occupation density formula \eqref{eq: the occupation density formula}
  and the continuity of $f_{\delta}^{F}(x,y)$ with respect to $(\delta,x,y)$,
  we may assume that $P_{\rho}^{G}$-a.s.\
  it holds that
  \begin{equation}
    \int_{0}^{t}f_{\delta}^{F}(x,X_{s})ds
    =
    \int_{F} f_{\delta}^{F}(x,y)L_{G}(y,t) \mu(dy)
  \end{equation}
  for all $t \geq 0,\, x \in F$ and $\delta >0$.
  Using this identity,
  we obtain that
  \begin{align} 
    \sup_{x \in F^{(r)}}
    \sup_{ 0 \leq t \leq T}
    |g_{\delta}^{G}(X_{G})(x,t)-L_{G}(x,t)|
     & \leq
    \sup_{x \in F^{(r)}}
    \sup_{0 \leq t \leq T}
    \frac{\int_{F}f_{\delta}^{F}(x,y)|L_{G}(y,t)-L_{G}(x,t)| \mu(dy)}
    {\int_{F} f_{\delta}^{F}(x,y)\mu(dy)} \notag \\
     & \leq
    \sup_{0 \leq t \leq T}\,
    \sup_{\substack{(x,y) \in F^{(r+2\delta)}    \\ R(x,y) < \delta}}
    |L_{G}(y,t)-L_{G}(x,t)|
    \label{5. eq: error of delta approximation for local times}
  \end{align}
  for all $\delta,\, r,\, T >0$.
  The joint continuity of local times yields the desired result.
\end{proof}

For $\delta >0$ and $G=(F, R, \rho, \mu) \in \check{\mathbb{F}}$, 
set 
\begin{equation}
  P_{G, \delta}(\cdot)
  \coloneqq
  P_{\rho}^{G} 
  \left(
    ( X_{G}, g_{\delta}^{G}(X_{G}) ) \in \cdot
  \right), \quad
  \cX_{G, \delta} 
  \coloneqq  
  (F,R,\rho, \mu, P_{G, \delta}).
\end{equation}
The following is an immediate consequence of Lemma \ref{5. lem: delta approximation of local times}.

\begin{lem} \label{5. lem: X_G_delta convergence}
  For every $G=(F,R,\rho,\mu) \in \check{\mathbb{F}}$,
  $\cX_{G, \delta} \to \cX_{G}$ in $\mbM_{L}$ as $\delta \downarrow 0$.
\end{lem}

In what follows,
we recall that 
$\mathbb{F}_{c}$ and $\check{\mathbb{F}}_{c}$
are equipped with the Gromov-Hausdorff-Prohorov topology,
and $\check{\mathbb{F}}$ and $\mathbb{F}$ are equipped with the local Gromov-Hausdorff-vague topology.

\begin{prop} \label{5. prop: continuity of delta approximation}
  The map $\check{\mathbb{F}}_{c} \ni G \mapsto \cX_{G,\delta} \in \mathcal{P}(\mbM_{L})$ is continuous.
\end{prop}

\begin{proof}
  Assume that $G_{n}=(F_{n},R_{n},\rho_{n},\mu_{n}) \in \check{\mathbb{F}}_{c}$ 
  converges to $G=(F,R,\rho,\mu) \in \check{\mathbb{F}}_{c}$
  in the Gromov-Hausdorff-vague topology
  (and hence in the local Gromov-Hausdorff-vague topology).
  Lemma \ref{lem: characterization of recurrence}
  implies that
  the sequence $(G_{n})_{n \geq 1}$ satisfies 
  Assumption \ref{1. assum: deterministic spaces}\ref{1. assum item: deterministic, non-explosion condition}.
  Thus, by Proposition \ref{5. prop: convergence of processes in M},
  we may assume that all the spaces $(F_{n}, R_{n}, \rho_{n})$ and $(F, R, \rho)$
  are isometrically embedded into a common rooted compact metric space $(M, d^{M}, \rho_{M})$
  in such a way that $\rho_{n} = \rho = \rho_{M}$ as elements in $M$,
  $F_{n} \to F$ in the local Hausdorff topology on $M$,
  $\mu_{n} \to \mu$ vaguely as measures on $M$ 
  and $P_{\rho_{n}}^{G_{n}}(X_{G_{n}} \in \cdot) \to P_{\rho}^{G}(X_{G} \in \cdot)$ 
  as probability measures on $D(\mbRp, M)$.
  By the Skorohod representation theorem,
  we may assume that $X_{G_{n}}$ started from $\rho_{n}$ and $X_{G}$ started from $\rho$ are coupled
  so that $X_{G_{n}} \to X_{G}$ almost-surely.
  Fix $r > 0$. Choose $r' > r+1$ satisfying $\mu_{n}^{(r')} \to \mu^{(r)}$
  weakly as probability measures on $M$.
  Set 
  $\cC_{n}^{\eta} 
    \coloneqq
    \left\{
    (x,y) \in F_{n}^{(r)} \times F^{(r)}
    : d^{M}(x,y) < \eta \right\}
  $.
  Note that 
  \begin{equation}  \label{5. eq: Lipsitz continuity of delta function}
    |f_{\delta}^{M}(x, y) - f_{\delta}^{M}(x, z)| \leq d^{M}(y, z),
  \end{equation}
  We then deduce that, for any $\eta \in (0,1)$,
  \begin{align}
    &\sup_{(x,y) \in \cC_{n}^{\eta}}
    \left| 
      \int_{F_{n}} f_{\delta}^{F_{n}}(x,z)\, \mu_{n}(dz)
      -\int_{F} f_{\delta}^{F}(y,z)\, \mu(dz) 
    \right|\\
    =
    &\sup_{(x,y) \in \cC_{n}^{\eta}}
    \left| 
      \int_{M} f_{\delta}^{M}(x,z)\, \mu_{n}^{(r')}(dz)
      -\int_{M} f_{\delta}^{M}(y,z)\, \mu^{(r')}(dz) 
    \right|\\
    \leq
    &\sup_{x \in F_{n}^{(r)}}
    \left| 
      \int_{M} f_{\delta}^{M}(x,z)\, \mu_{n}^{(r')}(dz)
      -
      \int_{M} f_{\delta}^{M}(x,z)\, \mu^{(r')}(dz)
    \right|
    +
    \sup_{(x,y) \in \cC_{n}^{\eta}}
    \left|
      \int_{M} 
      \bigl( f_{\delta}^{M}(x,z) - f_{\delta}^{M}(y,z) \bigr)
      \, \mu^{(r')}(dz) 
    \right|\\
    \leq
    &\sup_{x \in M^{(r)}}
    \left| 
      \int_{M} f_{\delta}^{M}(x,z)\, \mu_{n}^{(r')}(dz)
      -
      \int_{M} f_{\delta}^{M}(x,z)\, \mu^{(r')}(dz)
    \right|
    +
    2 \eta
    \mu(F^{(r')}),
  \end{align}
  where we use \eqref{5. eq: Lipsitz continuity of delta function} at the last inequality.
  By \eqref{5. eq: Lipsitz continuity of delta function},
  the family $\{f_{\delta}^{M}(x, \cdot) \mid x \in M^{(r)}\}$ is equicontinuous.
  This, combined with $\mu_{n}^{(r')} \to \mu^{(r')}$,
  we obtain that 
  \begin{equation}  \label{5. eq: convergence of denominator}
    \lim_{\eta \to 0}
    \limsup_{n \to \infty}
    \sup_{(x,y) \in \cC_{n}^{\eta}}
    \left| 
      \int_{F_{n}} f_{\delta}^{F_{n}}(x,z)\, \mu_{n}(dz)
      -\int_{F} f_{\delta}^{F}(y,z)\, \mu(dz) 
    \right|
    =0.
  \end{equation}
  Since $X_{G_{n}} \rightarrow X_{G}$ in the usual Skorohod $J_{1}$-topology,
  there exists a strictly increasing continuous function $\lambda_{n}:\mbRp \rightarrow \mbRp$ with $\lambda(0)=0$
  such that
  \begin{gather} \label{5. eq: time-change function for tilde process}
    \sup_{t \in \mbRp}
    |\lambda_{n}(t)-t|
    \rightarrow 0,
    \quad
    \sup_{0 \leq t \leq T}
    d^{M}(X_{G_{n}}(t),X_{G}(\lambda_{n}(t)))
    \rightarrow 0,
    \quad
    \forall T \geq0
  \end{gather}
  (cf.\ \cite[Theorem 16.1]{Billingsley_99_Convergence}).
  We have that 
  \begin{align}
    &\sup_{(x, y) \in \cC_{n}^{\eta}} 
    \sup_{0 \leq t \leq T} 
    \left|
      \int_{0}^{t} f_{\delta}^{F_{n}}(x, X_{G_{n}}(s))\, ds 
      - 
      \int_{0}^{t} f_{\delta}^{F}(y, X_{G}(s))\, ds
    \right|\\
    \leq 
    &\sup_{x \in M^{(r)}} 
    \int_{0}^{T} 
    \left| 
      f_{\delta}^{M}(x, X_{G_{n}}(s)) 
      - 
      f_{\delta}^{M}(x, X_{G} \circ \lambda_{n}(s)) 
    \right| \, ds \\
    &\quad
    + 
    \sup_{(x, y) \in \cC_{n}^{\eta}} 
    \int_{0}^{T}
    \left| 
      f_{\delta}^{M}(x, X_{G} \circ \lambda_{n} (s)) 
      - 
      f_{\delta}^{M}(y, X_{G}(s)) 
    \right|\, ds\\
    \leq 
    &\,
    T \sup_{0 \leq t \leq T} d^{M}(X_{G_{n}}(s), X_{G} \circ \lambda_{n}(s)) 
    +
    \sup_{x \in M^{(r)}} 
    \int_{0}^{T}
    \left| 
      f_{\delta}^{M}(x, X_{G} \circ \lambda_{n} (s)) 
      - 
      f_{\delta}^{M}(x, X_{G}(s)) 
    \right|\, ds \\
    &\quad 
    +
    \sup_{(x, y) \in \cC_{n}^{\eta}} 
    \int_{0}^{T}
    \left| 
      f_{\delta}^{M}(x, X_{G}(s)) 
      - 
      f_{\delta}^{M}(y, X_{G}(s)) 
    \right|\, ds\\ 
    \leq 
    &\,
    2T \sup_{0 \leq t \leq T} d^{M}(X_{G_{n}}(s), X_{G} \circ \lambda_{n}(s)) 
    + 
    \eta T,
  \end{align}
  where we use \eqref{5. eq: Lipsitz continuity of delta function} at the last inequality.
  Using \eqref{5. eq: time-change function for tilde process},
  we obtain that 
  \begin{equation}  \label{5. eq: convergence of numerator}
    \lim_{\eta \to 0}
    \limsup_{n \to \infty}
    \sup_{(x, y) \in \cC_{n}^{\eta}} 
    \sup_{0 \leq t \leq T} 
    \left|
      \int_{0}^{t} f_{\delta}^{F_{n}}(x, X_{G_{n}}(s))\, ds 
      - 
      \int_{0}^{t} f_{\delta}^{F}(y, X_{G}(s))\, ds
    \right|
    = 0.
  \end{equation}
  From \eqref{5. eq: convergence of denominator} and \eqref{5. eq: convergence of numerator},
  it follows that,
  for all $r , T >0$,
  \begin{equation}
    \lim_{\eta \to 0}
    \limsup_{n \to \infty}
    \sup_{(x, y) \in \cC_{n}^{\eta}}
    \sup_{0 \leq t \leq T}
    \left|
      g_{\delta}^{G_{n}}(X_{G_{n}})(x,t)
      -
      g_{\delta}^{G}(X_{G})(y,t)
    \right|
    =0.
  \end{equation}
  Therefore,
  by Theorem \ref{2. thm: convergence in hatC},
  $g_{\delta}^{G_{n}}(X_{G_{n}})$ converges to $g_{\delta}^{G}(X_{G})$
  in $\hatC(M \times \mbRp, \mbR)$ almost-surely.
  Now the result is immediate.
\end{proof}

\begin{lem} \label{5. lem: for precompactness of local times}
  Under Assumption \ref{1. assum: deterministic spaces}, 
  it holds that
  \begin{equation}  \label{5. lem eq: for precompactness, equicontinuity of local times}
    \lim_{\delta \to 0}
    \limsup_{n \to \infty}
    P_{\rho_{n}^{(r)}}^{G_{n}^{(r)}} 
    \left( 
      \sup_{\substack{x,y \in F_{n}^{(r)},\\ R_{n}^{(r)}(x,y) < \delta }}
      \sup_{ 0 \leq t \leq T}
      |L_{G_{n}^{(r)}}(x,t) - L_{G_{n}^{(r)}}(y,t)|
      >
      \varepsilon
    \right) 
    = 0,
    \quad 
    \forall \varepsilon, r, T >0.
  \end{equation}
\end{lem}

\begin{proof}
  By Assumption  \ref{1. assum: deterministic spaces} \ref{1. assum item: deterministic, metric-entropy condition},
  we can find $\alpha \in (0,1/2)$ such that
  \begin{equation} \label{5. eq: tightness of metric entropies}
    \lim_{m \to \infty}
    \limsup_{n \to \infty}
    \sum_{k \geq m}
    (k+1)^{2} N_{R_{n}^{(r)}}(F_{n}^{(r)},2^{-k})^{2}
    \exp \left(-2^{\alpha (k-3)} \right)
    =0.
  \end{equation}
  Let $c_{\alpha} \in (0,\infty)$ be a constant
  satisfying the conditions in Theorem \ref{thm: tightness inequality of local times in cpt case} corresponding to $\alpha$.
  Then, by Theorem \ref{thm: tightness inequality of local times in cpt case},
  we deduce that
  \begin{align} \label{5. eq: equicontinuity at same time, basic inequality}
     & P_{\rho_{n}^{(r)}}^{G_{n}^{(r)}}
    \left(
      \sup_{\substack{x,y \in F_{n}^{(r)} \\ R_{n}^{(r)}(x,y)<\delta_{m}}}
      \sup_{0 \leq t \leq T}
      |L_{G_{n}^{(r)}}(x,t)-L_{G_{n}^{(r)}}(y,t)|
      >
      \varepsilon
    \right)                             \\
    \leq
    &2e^{T}
    \sum_{k \geq m}
    (k+1)^{2} N_{R_{n}^{(r)}}(F_{n}^{(r)},2^{-k})^{2}
    \exp \left(-2^{\alpha(k-3)} \right)
    +
    P_{\rho_{n}^{(r)}}^{G_{n}^{(r)}}
    \left(
      c_{\alpha} 2^{-(\frac{1}{2} - \alpha)m} \sqrt{\mu_{n}^{(r)}(F_{n}^{(r)})} \geq \varepsilon
    \right).
  \end{align}
  Assumption \ref{1. assum: deterministic spaces}\ref{1. assum item: deterministic, convergence of spaces} 
  implies that 
  $\sup_{n} \mu_{n}^{(r)}(F_{n}^{(r)}) < \infty$.
  Therefore, from \eqref{5. eq: tightness of metric entropies},
  we obtain the desired result.
\end{proof}

To approximate the laws for non-compact spaces by compact ones,
we introduce the notion of the trace of a process onto a subset.
For $G \in \check{\mathbb{F}}$,
we set $A_{G}^{(r)}$ to be the PCAF of $X_{G}$ given by $A_{G}^{(r)}(t) = \int_{0}^{t}1_{F^{(r)}}(X_{G}(s))ds$,
and $\gamma_{G}^{(r)}$ to be the right-continuous inverse of $A_{G}^{(r)}$,
i.e.\ $\gamma_{G}^{(r)}(t)\coloneqq  \inf \{ s >0 : A_{G}^{(r)}(s)>t \}$.
Then the trace of $X_{G}$ onto $F^{(r)}$ is defined by setting $\mathrm{tr}_{(r)}X_{G} \coloneqq X_{G} \circ \gamma_{G}^{(r)}$.
By \cite[Lemma 2.6]{Croydon_18_Scaling},
$\mathrm{tr}_{(r)}X_{G}$ is a strong Markov process associated with $G^{(r)}$.
In other words,
$P_{x}^{G}(\mathrm{tr}_{(r)}X_{G} \in \cdot)=P_{x}^{G^{(r)}}(X_{G^{(r)}} \in \cdot)$
as probability measures on $D(\mbRp,F^{(r)})$ for every $x \in F^{(r)}$.
This, combined with Lemma \ref{5. lem: delta approximation of local times},
ensures the existence of the almost-sure limit of $g_{\delta}^{G^{(r)}}(\mathrm{tr}_{(r)}X_{G})$ in $C(F^{(r)} \times \mbRp, \mbRp)$
under $P_{\rho}^{G}$ as $\delta \to 0$.
We write $\mathrm{tr}_{(r)}L_{G}$ for the limit.
By the construction,
we have that $P_{x}^{G}(\mathrm{tr}_{(r)}L_{G} \in \cdot)=P_{x}^{G^{(r)}}(L_{G^{(r)}} \in \cdot)$
as probability measures on $C(F^{(r)} \times \mbRp, \mbRp)$ for every $x \in F^{(r)}$.
We set
\begin{equation}
  \mathrm{tr}_{(r)}\cX_{G}
  \coloneqq (F^{(r)},R^{(r)},\rho^{(r)},\mu^{(r)},
  \mathrm{tr}_{(r)}X_{G}, \mathrm{tr}_{(r)}L_{G}).
\end{equation}
Note that $P_{\rho}^{G}( \mathrm{tr}_{(r)}\cX_{G} \in \cdot)=P_{\rho^{(r)}}^{G^{(r)}}(\cX_{G^{(r)}} \in \cdot)$
as probability measures on $\mbM_{L}$.
The following two lemmas are important for the argument of approximation.

\begin{lem} \label{5. lem: coincidence of trace and original up to exit time}
  Let $G=(F,R,\rho,\mu)$ be an element of $\check{\mathbb{F}}$.
  Then, $P_{\rho}$-a.s., it holds that
  \begin{gather}
    \mathrm{tr}_{(r)}X_{G}(t) =X_{G}(\gamma_{G}^{(r)}(t)), \quad
    \mathrm{tr}_{(r)}L_{G}(x,t) = L_{G}(x,\gamma_{G}^{(r)}(t))
  \end{gather}
  for all $t \geq 0,\, r >0$ and $x \in F^{(r)}$.
  In particular,
  on the event $\{ \sigma_{B_{R}(\rho, r)^{c}} > T \}$,
  it holds that
  \begin{gather}
    \mathrm{tr}_{(r)}X_{G}(t) =X_{G}(t),\quad
    \mathrm{tr}_{(r)}L_{G}(x,t) = L_{G}(x,t)
  \end{gather}
  for all $t \leq T$ and $x \in F^{(r)}$.
\end{lem}

\begin{proof}
  Let $\mathcal{R}$ be a countable dense set in $(0,\infty)$ 
  containing all $r \in \mathcal{R}$ such that $F^{(r)} \neq D_{R}(\rho, r)$.
  By the occupation density formula 
  and the continuity of $f_{\delta}^{F}(x,y)$ with respect to $(x, y, \delta)$,
  we deduce that $P_{\rho}$-a.s.,
  \begin{equation} \label{5. eq: trace relation, occupation density formula}
    \int_{0}^{t}
    1_{F^{(r)}}(X_{G}(s))
    f_{\delta}^{F}(x, X_{G}(s))
    ds
    =
    \int_{G}
    1_{F^{(r)}}(y) L_{G}(y,t)
    f_{\delta}^{F}(x,y)
    \mu(dy)
  \end{equation}
  for all $r \in \mathcal{R},\, t \geq 0,\, \delta>0$ and $x \in F$.
  For $r \notin \mathcal{R}$,
  since we have that $F^{(r)} = D_{R}(\rho, r)$,
  it holds that $1_{F^{r'}} \downarrow 1_{F^{(r)}}$ as $r' \downarrow r$.
  Therefore,
  it is the case that, $P_{\rho}$-a.s.,
  the occupation density formula given in \eqref{5. eq: trace relation, occupation density formula} 
  holds for all $r>0,\, t \geq 0,\, \delta >0$ and $x \in F$.
  Then, by \eqref{5. eq: trace relation, occupation density formula},
  we deduce that,
  for all $r>0,\, t \geq 0,\, \delta>0$ and $x \in F^{(r)}$,
  \begin{align} \label{eq: odf for time changes}
    \int_{0}^{t} f_{\delta}^{F^{(r)}}(x, X_{G}(\gamma_{G}^{(r)}(s)))ds
     & =
    \int_{\mbRp}
    1_{(0, \gamma_{G}^{(r)}(t)]}(\gamma_{G}^{(r)}(s))
    f_{\delta}^{F}(x, X_{G}(\gamma_{G}^{(r)}(s)))
    ds
    \notag \\
     & =
    \int_{\mbRp}
    1_{(0, \gamma_{G}^{(r)}(t)]}(s)
    f_{\delta}^{F}(x, X_{G}(s))
    dA_{G}^{(r)}(s)
    \notag \\
     & =
    \int_{0}^{\gamma_{G}^{(r)}(t)}
    f_{\delta}^{F}(x, X_{G}(s))
    1_{F^{(r)}}(X_{G}(s))
    ds
    \notag \\
     & =
    \int_{G}
    1_{F^{(r)}}(y) L_{G}(y,\gamma_{G}^{(r)}(t))
    f_{\delta}^{F}(x,y)
    \mu(dy).
  \end{align}
  Using \eqref{eq: odf for time changes},
  we deduce that
  \begin{align}
    |\mathrm{tr}_{(r)}L_{G}(x,t) - L_{G}(x,\gamma_{G}^{(r)}(t))|
     & \leq
    |\mathrm{tr}_{(r)}L_{G}(x,t)
    -g_{\delta}^{G^{(r)}}(\mathrm{tr}_{(r)}X_{G})(x,t)| \\
     & \hspace{9pt}
    +\left| \frac{\int_{0}^{t} f_{\delta}^{F^{(r)}}(x, X_{G}(\gamma_{G}^{(r)}(s)))ds}
    {\int_{F^{(r)}}f_{\delta}^{F^{(r)}}(x,y)\mu(dy)}
    -L_{G}(x,\gamma_{G}^{(r)}(t))\right|                \\
     & \leq
    |\mathrm{tr}_{(r)}L_{G}(x,t)
    -g_{\delta}^{G^{(r)}}(\mathrm{tr}_{(r)}X_{G})(x,t)| \\
     & \hspace{9pt}
    + \sup_{\substack{y \in F^{(r)}                     \\ R(x,y) \leq \delta}}
    |L_{G}(x,\gamma_{G}^{(r)}(t)) - L_{G}(y,\gamma_{G}^{(r)}(t))|,
  \end{align}
  and letting $\delta \to 0$ yields that $\mathrm{tr}_{(r)}L_{G}(x,t) = L_{G}(x,\gamma_{G}^{(r)}(t))$.
  If $\sigma_{B_{R}(\rho, r)^{c}} > T$,
  then we have that $A_{G}^{(r)}(t)=t$ for $t \leq T+\eta$
  for some $\eta >0$,
  which yields that $\gamma_{G}^{(r)}(t)=t$ for $t \leq T$.
  Now the second assertion is immediate.
\end{proof}

\begin{lem}  \label{5. lem: precompactness of cX_n for deterministic spaces}
  If Assumption \ref{1. assum: deterministic spaces}\ref{1. assum item: deterministic, convergence of spaces},
  \eqref{5. lem eq: uniform exit time estimate} and \eqref{5. lem eq: for precompactness, equicontinuity of local times}
  are satisfied,
  then $(\cX_{G_{n}})_{n \geq 1}$ is precompact in $\mbM_{L}$.
  In particular,
  under Assumption \ref{1. assum: deterministic spaces},
  $(\cX_{G_{n}})_{n \geq 1}$ is precompact in $\mbM_{L}$.
\end{lem}

\begin{proof}
  We check that $(\cX)_{n \geq 1}$ satisfies all the conditions 
  \ref{2. thm item: precompactness in M_L, spaces are precompact}
  -
  \ref{2. thm item: precompactness in M_L, equicontinuity of local times}
  of Theorem \ref{2. thm: precompactness in M_L}.
  By Theorem \ref{2. thm: precompactness in M} 
  and Proposition \ref{5. prop: convergence of processes in M},
  we obtain 
  \ref{2. thm item: precompactness in M_L, spaces are precompact},
  \ref{2. thm item: precompactness in M_L, values of processes are precompact} 
  and \ref{2. thm item: precompactness in M_L, equicontinuity of processes}.
  Since we have that, $L_{n}(x,0)=0$ for all $x \in F_{n}$,
  we obtain \ref{2. thm item: precompactness in M_L, boundedness of local times}.
  Fix $\varepsilon, \delta, T, r_{0} > 0$.
  Lemma \ref{5. lem: coincidence of trace and original up to exit time} yields that,
  for all $r>r_{0}$,
  \begin{align} 
    &P_{\rho_{n}}^{G_{n}} 
    \left(
      \sup_{\substack{x, y \in F_{n}^{(r_{0})}, \\ R_{n}(x,y) < \delta}} 
      \sup_{\substack{0 \leq s,t \leq T,\\ | t-s | < \delta}} 
      \left| L_{G_{n}}(x,t) - L_{G_{n}}(y, s) \right| 
      > 
      \varepsilon
    \right)\\
    \leq
    &P_{\rho_{n}}^{G_{n}} 
    \left(
      \sup_{\substack{x, y \in F_{n}^{(r)}, \\ R_{n}(x,y) < \delta}} 
      \sup_{\substack{0 \leq s,t \leq T,\\ | t-s | < \delta}} 
      \left| L_{G_{n}}(x,t) - L_{G_{n}}(y, s) \right| 
      > 
      \varepsilon
    \right)\\
    \leq 
    &P_{\rho_{n}}^{G_{n}} 
    \left(
      \sigma_{B_{R_{n}}(\rho_{n}, r)^{c}} \leq  T
    \right)
    +
    P_{\rho_{n}^{(r)}}^{G_{n}^{(r)}} 
    \left(
      \sup_{\substack{x, y \in F_{n}^{(r)}, \\ R_{n}(x,y) < \delta}} 
      \sup_{\substack{0 \leq s,t \leq T,\\ | t-s | < \delta}} 
      \left| L_{G_{n}^{(r)}}(x,t) - L_{G_{n}^{(r)}}(y, s) \right| 
      > 
      \varepsilon
    \right).
    \label{5. eq: precompactness, reducing to traces}
  \end{align}
  By the triangle inequality, we deduce that, for each $\eta > 0$, 
  \begin{align}
    &\sup_{\substack{x, y \in F_{n}^{(r)}, \\ R_{n}(x,y) < \delta}} 
    \sup_{\substack{0 \leq s,t \leq T,\\ | t-s | < \delta}} 
    \left| L_{G_{n}^{(r)}}(x,t) - L_{G_{n}^{(r)}}(y, s) \right| \\
    \leq 
    &2\sup_{x \in F_{n}^{(r)}} 
    \sup_{0 \leq t \leq T} 
    \left| L_{G_{n}^{(r)}}(x, t) - g_{\eta}^{G_{n}^{(r)}}(X_{G_{n}^{(r)}}(x,t)) \right| \\
    &\quad
    + 
    \sup_{\substack{x, y \in F_{n}^{(r)}, \\ R_{n}(x,y) < \delta}} 
    \sup_{\substack{0 \leq s,t \leq T,\\ | t-s | < \delta}} 
    \left| g_{\eta}^{G_{n}^{(r)}}(X_{G_{n}^{(r)}}(x,t)) - g_{\eta}^{G_{n}^{(r)}}(X_{G_{n}^{(r)}})(y,s) \right|\\
    \leq 
    &2\sup_{\substack{x, y \in F_{n}^{(r)}, \\ R_{n}(x,y) < \eta}} 
    \sup_{0 \leq t \leq T} 
    \left| L_{G_{n}^{(r)}}(x, t) - L_{G_{n}^{(r)}}(y, t) \right| \\
    &\quad 
    + 
    \sup_{\substack{x, y \in F_{n}^{(r)}, \\ R_{n}(x,y) < \delta}} 
    \sup_{\substack{0 \leq s,t \leq T,\\ | t-s | < \delta}} 
    \left| g_{\eta}^{G_{n}^{(r)}}(X_{G_{n}^{(r)}}(x,t)) - g_{\eta}^{G_{n}^{(r)}}(X_{G_{n}^{(r)}})(y,s) \right|,
    \label{5. eq: precompactness, triangle inequality for equicontinuity of local times}
  \end{align}
  where we use \eqref{5. eq: error of delta approximation for local times} at the last inequality.
  By Proposition \ref{5. prop: continuity of delta approximation},
  $(\cX_{G_{n}, \eta})_{n \geq 1}$ is precompact in $\mbM_{L}$ for each $\eta > 0$.
  Thus, it follows from Theorem \ref{2. thm: precompactness in M_L} that, for each $r>0$ and $\eta >0$, 
  \begin{equation}
    \lim_{\delta \downarrow 0}
    \limsup_{n \to \infty} 
    P_{\rho_{n}^{(r)}}^{G_{n}^{(r)}} 
    \left(
      \sup_{\substack{x, y \in F_{n}^{(r)}, \\ R_{n}(x,y) < \delta}} 
      \sup_{\substack{0 \leq s,t \leq T,\\ | t-s | < \delta}} 
      \left| g_{\eta}^{G_{n}^{(r)}}(X_{G_{n}^{(r)}}(x,t)) - g_{\eta}^{G_{n}^{(r)}}(X_{G_{n}^{(r)}})(y,s) \right|
      > 
      \frac{\varepsilon}{2}
    \right)
    =0.
  \end{equation}
  This, combined with 
  Lemma \ref{5. lem: for precompactness of local times}
  and \eqref{5. eq: precompactness, triangle inequality for equicontinuity of local times},
  yields that, for each $r_{0}>0$,
  \begin{equation}  \label{5. eq: precompactness, local times on traces are precompact}
    \lim_{\delta \downarrow 0}
    \limsup_{n \to \infty}
    P_{\rho_{n}^{(r)}}^{G_{n}^{(r)}} 
    \left(
      \sup_{\substack{x, y \in F_{n}^{(r)}, \\ R_{n}(x,y) < \delta}} 
      \sup_{\substack{0 \leq s,t \leq T,\\ | t-s | < \delta}} 
      \left| L_{G_{n}^{(r)}}(x,t) - L_{G_{n}^{(r)}}(y, s) \right| 
      > 
      \varepsilon
    \right)
    =
    0.
  \end{equation}
  By \eqref{5. lem eq: uniform exit time estimate},
  \eqref{5. eq: precompactness, reducing to traces} and \eqref{5. eq: precompactness, local times on traces are precompact},
  we obtain 
  the condition \ref{2. thm item: precompactness in M_L, equicontinuity of local times} of Theorem \ref{2. thm: precompactness in M_L}.
\end{proof}

We complete the proof of Theorem \ref{1. thm: main result for deterministic spaces} 
by characterizing the limit of $(\cX_{G_{n}})_{n \geq 1}$.

\begin{proof} [Proof of Theorem \ref{1. thm: main result for deterministic spaces}]
  Since we have the precompactness of $(\cX_{G_{n}})_{n \geq 1}$ from Lemma \ref{5. lem: precompactness of cX_n for deterministic spaces},
  it remains to show that the limit of any convergent subsequence of $(\cX_{G_{n}})_{n \geq 1}$ is $\cX_{G}$.
  To simplify notation,
  we suppose that $\cX_{G_{n}}$ converges to $\cX$ in $\mbM_{L}$,
  and show that $\cX = \cX_{G}$.
  By Assumption \ref{1. assum: deterministic spaces}\ref{1. assum item: deterministic, convergence of spaces},
  we can write $\cX = (F, R, \rho, \mu, \pi)$.
  Let $(X, L)$ be a random element of $D(\mbR, F) \times \hatC(F \times \mbRp, \mbR)$
  whose law coincides with $\pi$.
  From Proposition \ref{5. prop: convergence of processes in M},
  it follows that $X \stackrel{\mathrm{d}}{=} X_{G}$.
  Using Theorem \ref{2. thm: space M_L, convergence},
  we may assume that all the spaces $(F_{n}, R_{n}, \rho_{n})$ and $(F, R, \rho)$
  are embedded isometrically into a common rooted boundedly-compact metric space $(M, d^{M}, \rho_{M})$
  in such a way that $\rho_{n} = \rho = \rho_{M}$ as elements in $M$,
  $F_{n} \to F$ in the local Hausdorff topology in $M$,
  $\mu_{n} \to \mu$ vaguely as measures on $M$
  and $(X_{G_{n}}, L_{G_{n}})$ started at $\rho_{n}$ converges to $(X, L)$ in distribution 
  as random elements of $D(\mbRp, M) \times \hatC(M \times \mbRp, \mbR)$.
  By the Skorohod representation theorem,
  we may assume that $(X_{G_{n}}, L_{G_{n}})$ started at $\rho_{n}$
  converges to $(X, L)$ almost-surely
  on some probability space.
  Note that Theorem \ref{2. thm: convergence in hatC} implies that $\dom_{1}(L) = F$.
  By the occupation density formula,
  with probability $1$,
  it holds that,
  for all $t \geq 0,\, x \in M,\, \delta > 0$ and $n \geq 1$,
  \begin{equation}  \label{5. eq: main proof, ODF for sequence}
    \int_{0}^{t} f_{\delta}^{M}(x, X_{G_{n}}(s))\, ds 
    = 
    \int_{F_{n}} f_{\delta}^{M}(x, y) L_{G_{n}}(y, t)\, \mu_{n}(dy),
  \end{equation}
  where we recall that $f_{\delta}^{M}(x,y) = 0 \vee (\delta - d^{M}(x,y))$.
  By Theorem \ref{2. thm: convergence in hatC},
  we can extend the domains of $L_{G_{n}}$ and $L$ to 
  $M \times \mbRp$ continuously
  in such a way that $L_{G_{n}} \to L$ 
  in the compact-convergence topology on $M \times \mbRp$.
  It is then the case that 
  $L_{G_{n}}(y, t) \mu_{n}(dy)$ converges to $L(y, t) \mu(dy)$ vaguely 
  as measures on $M$.
  Since $f_{\delta}^{M}(x, \cdot)$ is compactly supported,
  we deduce that 
  the right-hand side of \eqref{5. eq: main proof, ODF for sequence} 
  converges to $\int_{F} f_{\delta}^{M}(x,y)L(y,t)\, \mu(dy)$
  as $n \to \infty$.
  Since the left-hand side of \eqref{5. eq: main proof, ODF for sequence} 
  converges to $\int_{0}^{t} f_{\delta}^{M}(x, X(s))\, ds$,
  we obtain that, for all $t \geq 0,\, x \in M,\, \delta > 0$ and $n \geq 1$,
  \begin{equation}  \label{5. eq: main proof, ODF for limit}
    \int_{0}^{t} f_{\delta}^{M}(x, X(s))\, ds 
    = 
    \int_{F} f_{\delta}^{M}(x, y) L(y, t)\, \mu(dy),
  \end{equation}
  A similar argument to Lemma \ref{5. lem: delta approximation of local times} yields that 
  $g_{\delta}^{G}(X) \to L$ in $C(F \times \mbRp, \mbR)$
  as $\delta \to 0$ almost-surely.
  Therefore,
  we deduce that $(X, g_{\delta}^{G}(X)) \to (X, L)$ almost-surely.
  However, we have that $(X, g_{\delta}^{G}(X))$ has the same distribution as $(X_{G}, g_{\delta}^{G}(X_{G}))$,
  which converges to $(X_{G}, L_{G})$ almost-surely by Lemma \ref{5. lem: delta approximation of local times}.
  Hence, we establish that $(X, L) \stackrel{\mathrm{d}}{=} (X_{G}, L_{G})$,
  which completes the proof.
\end{proof}


\section{Proof of Theorem \ref{1. thm: main result for random spaces} } \label{sec: proof of main results in random case}
  Before the proof of Theorem \ref{1. thm: main result for random spaces}, 
  we show the following claim 
  which ensures that if $G$ is a random element of $\check{\mathbb{F}}$,
  then $\cX_{G}$ is a random element of $\mbM_{L}$. 

  \begin{prop} \label{prop: measurability of XG}
    The map $\check{\mathbb{F}} \ni G \mapsto \cX_{G} \in \mbM_{L}$ is measurable.
  \end{prop}

  We need two lemmas below to prove Proposition \ref{prop: measurability of XG}.

  \begin{lem} \label{lem: measurability of X^r}
    For every $G=(F, R,\rho,\mu) \in \check{\mathbb{F}}$, the following maps are 
    left-continuous at every $r$, and continuous
    except for at most a countable number of $r$:
    \begin{align}
      &(0,\infty) \ni r \mapsto \cX_{G^{(r)},\delta} \in \mbM_{L};\\
      &(0,\infty) \ni r \mapsto \cX_{G^{(r)}} \in \mbM_{L}
    \end{align}
  \end{lem}

  \begin{proof}
    The left-continuity of the first map follows 
    from Lemma \ref{2. lem: continuity of the restriction of bcms equipped with a measure} 
    and Proposition \ref{5. prop: continuity of delta approximation}. 
    Fix $r>0$. 
    By the definition of $\check{\mathbb{F}}$, we can find $\alpha \in (0,1/2)$ such that
    \begin{equation}
      \sum_{k} N_{R^{(r+1)}}(F^{(r+1)},2^{-k})^{2} 
        \exp(-2^{\alpha (k-1)}) 
      < \infty.
    \end{equation}
    Since we have that, 
    for any $s < r+1$,
    \begin{equation}
      \sum_{k \geq m} N_{R^{(s)}}(F^{(s)},2^{-k})^{2} 
        \exp(-2^{\alpha k}) 
      \leq 
      \sum_{k \geq m+1} N_{R^{(r+1)}}(F^{(r+1)},2^{-k})^{2} 
        \exp(-2^{\alpha (k-1)}),
    \end{equation}
    we obtain that $\cX_{G^{(s)}} \to \cX_{G^{(r)}}$ as $s \uparrow r$
    by applying Theorem \ref{1. thm: main result for deterministic spaces}.
  \end{proof}

  \begin{lem} \label{6. lem: measurability of X_G^r wtr G and r}
    The map $(0,\infty) \times \check{\mathbb{F}} \ni (r, G) \mapsto \cX_{G^{(r)}} \in \mbM_{L}$
    is measurable.
  \end{lem}

  \begin{proof}
    Fix a measurable function $f: \mbM_{L} \to \mbR$ assumed to be bounded and continuous. 
    It suffices to show that the map 
    $F: (0,\infty) \times \check{\mathbb{F}} \ni (r, G) \mapsto f(\cX_{G^{(r)}}) \in \mbR$
    is measurable.
    By Lemma \ref{lem: measurability of X^r},
    we can define maps $F_{\varepsilon}, F_{\varepsilon, \delta} : (0,\infty) \times \check{\mathbb{F}} \to \mbR$
    by setting 
    \begin{equation}
      F_{\varepsilon}(r, G)
      \coloneqq
      \frac{1}{\varepsilon \wedge (r/2)}
      \int_{0}^{\varepsilon \wedge (r/2)} f(\cX_{G^{(r-s)}})\, ds,
      \quad 
      F_{\varepsilon, \delta} (r, G)
      \coloneqq
      \frac{1}{\varepsilon \wedge (r/2)}
      \int_{0}^{\varepsilon \wedge (r/2)} f(\cX_{G^{(r-s)}, \delta})\, ds.
    \end{equation}
    Suppose that $(r_{n}, G_{n}) \to (r, G)$.
    By Lemma \ref{2. lem: joint continuity of the restriction operator},
    for all but countably many $s$,
    we have that $G_{n}^{(r_{n}-s)} \to G^{(r-s)}$.
    This immediately yields that the map $F_{\varepsilon, \delta}$ is continuous. 
    Since $F_{\varepsilon, \delta} \to F_{\varepsilon}$
    as $\delta \to 0$ pointwise by Lemma \ref{5. lem: X_G_delta convergence},
    the map $F_{\varepsilon}$ is measurable.
    Furthermore, by Lemma \ref{lem: measurability of X^r},
    we have that $F_{\varepsilon} \to F$ as $\varepsilon \to 0$ pointwise,
    which completes the proof.
  \end{proof}

  \begin{lem} \label{5. lem: convergence of traces to original}
    For every $G=(F, R, \rho, \mu) \in \mathbb{F}$,
    $\cX_{G^{(r)}} \to \cX_{G}$
    in $\mbM_{L}$ as $r \to \infty$.
  \end{lem}
  
  \begin{proof}
    Fix an increasing sequence $(r_{n})_{n \geq 1}$ with $r_{n} \uparrow \infty$. 
    Using \cite[Theorem 8.4]{Kigami_12_Resistance},
    we deduce that 
    \begin{equation}
      R^{(r_{n})}
      \left(
        \rho^{(r_{n})}, B_{R^{(r_{n})}}(\rho^{(r_{n})}, r)^{c}
      \right)
      =
      R
      \bigl(
        \rho, F^{(r_{n})} \cap B_{R}(\rho, r)^{c}
      \bigr)
      \geq 
      R(\rho, B_{R}(\rho, r)^{c}).
    \end{equation}
    This, combined with Lemma \ref{lem: characterization of recurrence}, yields that 
    the sequence $(G^{(r_{n})})_{n \geq 1}$ satisfies 
    Assumption \ref{1. assum: deterministic spaces}\ref{1. assum item: deterministic, non-explosion condition}.
    It is easy to check that $(G^{(r_{n})})_{n \geq 1}$ satisfies 
    Assumption \ref{1. assum: deterministic spaces}\ref{1. assum item: deterministic, convergence of spaces} 
    and \ref{1. assum item: deterministic, metric-entropy condition}.
    Hence, 
    the desired result follows from Theorem \ref{1. thm: main result for deterministic spaces}.
  \end{proof}

  \begin{proof}[Proof of Proposition \ref{prop: measurability of XG}]
    Lemma \ref{6. lem: measurability of X_G^r wtr G and r}
    and Lemma \ref{5. lem: convergence of traces to original}
    immediately yield the desired result.
  \end{proof}

  The first part of Theorem \ref{1. thm: main result for random spaces} follows from the next proposition.

  \begin{prop}
    If Assumption \ref{1. assum: random spaces}\ref{1. assum item: random spaces, space convergence} 
    and \ref{1. assum item: random, metric-entropy condition} hold, 
    then $\mathbf{P}(G \in \check{\mathbb{F}})=1$.
    In particular, $G$ is a random element of $\check{\mathbb{F}}$.
  \end{prop}

  \begin{proof}
    By Assumption \ref{1. assum: random spaces}\ref{1. assum item: random spaces, space convergence}, 
    we may assume that $G_{n}$ and $G$ are coupled 
    so that $G_{n} \to G$ in the local Gromov-Hausdorff-vague topology almost-surely under a complete probability measure $P$. 
    Fix $r>0$ and $\varepsilon > 0$. 
    By Assumption \ref{1. assum: random spaces}\ref{1. assum item: random, metric-entropy condition}, 
    there exist $m$ and $\alpha \in (0,1/2)$  such that 
    \begin{equation} \label{eq: limiting object argument 1}
      \limsup_{n \to \infty}
      P 
        \left( 
          \sum_{k \geq m} 
            N_{R_{n}^{(r+1)}}(F_{n}^{(r+1)}, 2^{-k})^{2} 
            \exp(-2^{\alpha (k-3)}) 
          \geq
          1 
        \right) 
      <
      \varepsilon.
    \end{equation}
    Choose $s \in (r, r+1)$ such that $G_{n}^{(s)} \to G^{(s)}$ in the Gromov-Hausdorff-Prohorov topology. 
    By a similar argument to the proof of Proposition \ref{5. prop: limit belongs to vF}, 
    we deduce that 
    \begin{align} \label{eq: limiting object argument 2}
      \sum_{k \geq m} 
        N_{R^{(r)}}(F^{(r)},2^{-k})^{2} 
        \exp(-2^{\alpha k}) 
      &\leq 
      \sum_{k \geq m} 
        N_{R^{(s)}}(F^{(s)},2^{-k-1})^{2}  
        \exp(-2^{\alpha k}) \notag \\
      &\leq 
      \liminf_{n \to \infty} 
      \sum_{k \geq m+3} 
        N_{R_{n}^{(r+1)}}(F_{n}^{(r+1)},2^{-k})^{2}  
        \exp(-2^{\alpha (k-3)}) .
    \end{align}
    By \eqref{eq: limiting object argument 1}, \eqref{eq: limiting object argument 2} and (reverse) Fatou's lemma, 
    we obtain that
    \begin{align}
      1-\varepsilon
      &\leq
      \limsup_{n \to \infty}
      P
        \left( 
          \sum_{k \geq m} 
            N_{R_{n}^{(r+1)}}(F_{n}^{(r+1)}, 2^{-k})^{2} 
            \exp(-2^{\alpha (k-3)}) 
          <
          1 
        \right) \\
      &\leq
      P
        \left( 
          \liminf_{n \to \infty}
          \sum_{k \geq m} 
            N_{R_{n}^{(r+1)}}(F_{n}^{(r+1)}, 2^{-k})^{2} 
            \exp(-2^{\alpha (k-3)}) 
          \leq
          1 
        \right) \\
      &\leq
      P 
        \left( 
          \sum_{k} 
            N_{R_{n}^{(r)}}(F_{n}^{(r)}, 2^{-k})^{2} 
            \exp(-2^{\alpha k}) 
          < 
          \infty 
        \right).
    \end{align}
    Letting $\varepsilon \to 0$ in the above inequality yields the desired result.
  \end{proof}

  \begin{lem} \label{6. lem: annealed exit time estimate}
    Under Assumption \ref{1. assum: random spaces}\ref{1. assum item: random spaces, space convergence} 
    and \ref{1. assum item: random spaces, non-explosion},
    it holds that
    \begin{equation}  \label{6. lem eq: annealed, exit time estimate}
      \lim_{r \to \infty}
      \limsup_{n \to \infty}
      \mathbf{P}_{n}
      \left(
        P_{\rho_{n}}^{G_{n}} \left( \sigma_{B_{R_{n}}(\rho, r)^{c}} \leq T \right) > \varepsilon
      \right)
      = 0,
      \quad 
      \forall 
      T, \varepsilon >0.
    \end{equation}
  \end{lem}

  \begin{proof}
    We may assume that $G_{n}$ and $G$ are coupled 
    so that $G_{n} \to G$ in the local Gromov-Hausdorff-vague topology almost-surely under a complete probability measure $P$.
    Fix $\varepsilon_{1} > 0$ arbitrarily.
    By Assumption \ref{1. assum: random spaces}\ref{1. assum item: random spaces, space convergence},
    we can find $\varepsilon_{2} > 0$ such that 
    \begin{equation} \label{eq: random fractals lower bound1}
      P 
      \left( 
        \inf_{n} 
        \mu_{n} (B_{R_{n}}(\rho_{n},1)) 
        > 
        \varepsilon_{2} 
      \right) 
      \geq 
      1 - \varepsilon_{1}
    \end{equation}
    (see also \cite[Corollary 5.7]{Athreya_Lohr_Winter_16_The_gap}).
    If we have that 
    \begin{equation}
      R_{n}(\rho_{n}, B_{R_{n}}(\rho_{n},r)^{c}) 
      \geq 
      \lambda, 
      \quad  
      \mu_{n}(B_{R_{n}}(\rho_{n},1)) 
      > 
      \varepsilon_{2}.
    \end{equation}
    for some $\lambda > 1$,
    then by Lemma \ref{5. lem: exit time estimate by croydon}
    it holds that 
    \begin{equation}
      P_{\rho_{n}}^{G_{n}}
      (\sigma_{B_{R_{n}}(\rho_{n}, r)^{c}} \leq T)
      \leq 
      \frac
      {4}{\lambda}
      + 
      \frac{4T}{\varepsilon_{2}(\lambda - 1)}.
    \end{equation}
    Since, for all sufficiently large $\lambda$,
    the right-hand side of the above inequality is bounded above by $\varepsilon$,
    we deduce that 
    \begin{equation}
      P
      \left(
        P_{\rho_{n}}^{G_{n}} \left( \sigma_{B_{R_{n}}(\rho, r)^{c}} \leq T \right) > \varepsilon
      \right)
      \leq 
      P 
      \left( 
        \mu_{n} (B_{R_{n}}(\rho_{n},1)) 
        \leq 
        \varepsilon_{2} 
      \right) 
      +
      P(
        R_{n}(\rho_{n}, B_{R_{n}}(\rho_{n},r)^{c}) 
        \geq 
        \lambda
      ).
    \end{equation}
    By \eqref{eq: random fractals lower bound1} and
    Assumption \ref{1. assum: random spaces} \ref{1. assum item: random spaces, non-explosion},
    we obtain the desired result.
  \end{proof}

  \begin{prop}  \label{6. prop: convergence of processes for random spaces}
    If Assumption \ref{1. assum: random spaces}\ref{1. assum item: random spaces, space convergence} 
    and \eqref{6. lem eq: annealed, exit time estimate} are satisfied,
    then 
    \begin{equation}
      \cY_{G_{n}}
      \coloneqq
      \left( F_{n}, R_{n}, \rho_{n}, \mu_{n}, P_{\rho_{n}}^{G_{n}}(X_{G_{n}} \in \cdot) \right) 
      \xrightarrow{\mathrm{d}} 
      \left( F, R, \rho, \mu, P_{\rho}^{G}(X_{G} \in \cdot) \right)
      \eqqcolon
      \cY_{G}
    \end{equation}
    as random elements of $\mbM$
    (recall this space from Section \ref{sec: the space M_L}).
  \end{prop}

  \begin{proof}
    Fix a continuous function $f$ on $\mbM$.
    Set 
    \begin{gather}
      \cY_{G_{n}}^{(r)} 
      \coloneqq 
      \left(
        F_{n}^{(r)}, R_{n}^{(r)}, \rho_{n}^{(r)}, \mu_{n}^{(r)}, P_{\rho_{n}^{(r)}}^{G_{n}^{(r)}}(X_{G_{n}^{(r)}} \in \cdot)
      \right),\\
      \cY_{G}^{(r)}
      \coloneqq
      \left(
        F^{(r)}, R^{(r)}, \rho^{(r)}, \mu^{(r)}, P_{\rho^{(r)}}^{G^{(r)}}(X_{G^{(r)}} \in \cdot)
      \right).
    \end{gather}
    We then define probability measures on $\mbM$ by 
    \begin{equation}
      Q_{G_{n}}^{(r)} (\cdot)
      \coloneqq 
      \int_{r}^{r+1} 
      \mathbf{P}_{n}
      (\cY_{G_{n}}^{(s)} \in \cdot)\, ds, 
      \quad 
      Q_{G}^{(r)} (\cdot)
      \coloneqq 
      \int_{r}^{r + 1} 
      \mathbf{P}
      (\cY_{G}^{(s)} \in \cdot)\, ds.
    \end{equation}
    (By a similar argument to Lemma \ref{6. lem: measurability of X_G^r wtr G and r},
    one can check the measurability of the integrands.)
    Using the Skorohod representation theorem,
    we may assume that $G_{n}$ converges to $G$ in the local Gromov-Hausdorff-vague topology 
    almost-surely on some probability space.
    For $r>0$ such that $G_{n}^{(r)} \to G^{(r)}$ in the Gromov-Hausdorff-Prohorov topology,
    we have from Proposition \ref{5. prop: convergence of processes in M} that 
    $\cY_{G_{n}}^{(r)} \to \cY_{G}^{(r)}$ in $\mbM$.
    Thus, we deduce that $Q_{G_{n}}^{(r)} \to Q_{G}^{(r)}$ weakly as 
    probability measures on $\mbM$.
    By Lemma \ref{5. lem: coincidence of trace and original up to exit time},
    we obtain that,
    for each $h,t,\varepsilon, \delta >0$ and $n \geq 1$,
    \begin{align}
      \mathbf{P}_{n}
      \left(
        P_{\rho_{n}}^{G_{n}}( \tilde{w}_{F_{n}}(X_{G_{n}}, h, t) > \varepsilon) > \delta
      \right)
      &\leq 
      \int_{r}^{r+1}
      \mathbf{P}_{n}
      \left(
        P_{\rho_{n}}^{G_{n}}( \sigma_{B_{R_{n}}(\rho_{n}, s)^{c}} \leq t) > \delta /2
      \right)\, ds\\
      &\quad 
      +
      \int_{r}^{r+1}
      \mathbf{P}_{n}
      \left(
        P_{\rho_{n}^{(s)}}^{G_{n}^{(s)}}( \tilde{w}_{F_{n}^{(s)}}(X_{G_{n}^{(s)}}, h, t) > \varepsilon) > \delta/2
      \right)\, ds,
    \end{align}
    where we recall $\tilde{w}$ from \eqref{2. eq: modulus continuity for cadlag curves}.
    Since $(Q_{G_{n}}^{(r)})_{n \geq 1}$ is tight in $\mbM$,
    by Theorem \ref{2. thm: tightness in M} and \eqref{6. lem eq: annealed, exit time estimate},
    we deduce that, for all $\varepsilon, \delta >0$,
    \begin{equation}
      \lim_{h \downarrow 0}
      \limsup_{n \to \infty}
      \mathbf{P}_{n}
      \left(
        P_{\rho_{n}}^{G_{n}}( \tilde{w}_{F_{n}}(X_{G_{n}}, h, t) > \varepsilon) > \delta
      \right)
      =
      0.
    \end{equation}
    This, combined with \eqref{6. lem eq: annealed, exit time estimate} and Theorem \ref{2. thm: tightness in M},
    yields that $(\cY_{G_{n}})_{n \geq 1}$ is tight in $\mbM$.
    It remains to show that the limit of any convergent subsequence of $(\cY_{G_{n}})_{n \geq 1}$ is $\cY_{G}$.
    To simplify notation,
    we suppose that $\cY_{G_{n}}$ converges to $\cY = (F', R', \rho', \mu', \pi')$
    weakly as random elements of $\mbM$,
    and show that $\cY \stackrel{\mathrm{d}}{=} \cY_{G}$.
    By the Skorohod representation theorem,
    we may assume that $\cY_{G_{n}}$ converges to $\cY$ almost-surely 
    on a probability space $(\Omega, \mathcal{F}, P)$.
    Fix $\omega \in \Omega$.
    Since $(\cY_{G_{n}})_{n \geq 1}$ is precompact in $\mbM$,
    it follows from Theorem \ref{2. thm: precompactness in M} that 
    \begin{equation}
      \lim_{r \to \infty} 
      \limsup_{n \to \infty} 
      P_{\rho_{n}}^{G_{n}}
      \left(
        X_{G_{n}}(s) \in F_{n}^{(r)},\ \forall s \leq T
      \right)
      =
      0,
    \end{equation}
    which implies \eqref{5. lem eq: uniform exit time estimate}.
    Hence, by Proposition \ref{5. prop: convergence of processes in M},
    we deduce that 
    $\cY_{G_{n}}$ converges to $\cY_{G}$,
    which completes the proof.    
  \end{proof}

  \begin{lem} \label{6. lem: for tightness of local times}
    Under Assumption \ref{1. assum: random spaces}, 
    it holds that, for every $T, r, \eta, \varepsilon >0$,
    \begin{equation}  \label{6. lem eq: for tightness of local times}
      \lim_{\delta \to 0}
      \limsup_{n \to \infty}
      \mathbf{P}_{n}
      \left(
        P_{\rho_{n}^{(r)}}^{G_{n}^{(r)}} 
        \left( 
          \sup_{\substack{x,y \in F_{n}^{(r)},\\ R_{n}^{(r)}(x,y) < \delta }}
          \sup_{ 0 \leq t \leq T}
          |L_{G_{n}^{(r)}}(x,t) - L_{G_{n}^{(r)}}(y,t)|
          >
          \eta
        \right) 
        > 
        \varepsilon
      \right)
      = 0.
    \end{equation}
  \end{lem}

  \begin{proof}
    By Assumption \ref{1. assum: random spaces}\ref{1. assum item: random, metric-entropy condition},
    there exists $\alpha \in (0,1/2)$ such that 
    \begin{equation}
      \lim_{m \to \infty} 
      \limsup_{n \to \infty} 
      \mathbf{P}_{n}
      \left( 
        \sum_{k \geq m} 
        N_{R_{n}^{(r)}}(F_{n}^{(r)}, 2^{-k})^{2} 
        \exp(-2^{\alpha k}) 
        \geq 
        \varepsilon' 
      \right)
      =0, \quad 
      \forall \varepsilon' >0.
    \end{equation}
    Fix $\alpha' \in (\alpha , 1/2)$.
    Then, one can check that 
    \begin{equation} \label{eq: cardinality condition stronger ver 1}
      \lim_{m \to \infty} 
      \limsup_{n \to \infty} 
      \mathbf{P}_{n} 
      \left( 
        \sum_{k \geq m} 
        (k+1)^{2} 
        N_{R_{n}^{(r)}}(F_{n}^{(r)}, 2^{-k-1})^{2} 
        \exp(-2^{\alpha' (k-3)}) 
        \geq 
        \varepsilon' 
      \right)
      =0, 
      \quad
      \forall \varepsilon' >0.
    \end{equation}
    Assumption \ref{1. assum: random spaces}\ref{1. assum item: random spaces, space convergence}
    implies that $(\mu_{n}^{(r)}(F_{n}^{(r)}))_{n \geq 1}$ is tight for each $r>0$.
    This, combined with \eqref{eq: cardinality condition stronger ver 1} and \eqref{5. eq: equicontinuity at same time, basic inequality},
    yields the desired result.
  \end{proof}

  Now it is possible to complete the proof of Theorem \ref{1. thm: main result for random spaces}

  \begin{proof} [Proof of the second part of Theorem \ref{1. thm: main result for random spaces}]
    By Theorem \ref{2. thm: tightness in M_L},
    Proposition \ref{6. prop: convergence of processes for random spaces},
    Lemma \ref{6. lem: annealed exit time estimate}
    and Lemma \ref{6. lem: for tightness of local times},
    one can check that $(\cX_{G_{n}})_{n \geq 1}$ is tight in $\mbM_{L}$
    following the proof of Lemma \ref{5. lem: precompactness of cX_n for deterministic spaces}.
    Thus, 
    it remains to show that the limit of any weakly-convergent subsequence of $(\cX_{G_{n}})_{n \geq 1}$ is $\cX_{G}$.
    To simplify notation,
    we suppose that $\cX_{G_{n}}$ converges to $\cX$ 
    as random elements of $\mbM_{L}$,
    and show that $\cX \stackrel{\mathrm{d}}{=} \cX_{G}$.
    Write $\cX = (F', R', \rho', \mu', \pi')$.
    Note that $G' \coloneqq (F', R', \rho', \mu') \stackrel{\mathrm{d}}{=} G$.
    By the Skorohod representation theorem,
    we may assume that $\cX_{G_{n}}$ converges to $\cX$ 
    almost-surely on some probability space.
    In the same way as the proof of Theorem \ref{1. thm: main result for deterministic spaces},
    we obtain that $\pi' = P_{G'}$, which completes the proof.
  \end{proof}

\section{Metric entropy and volume estimates} \label{sec: metric entropy and volume estimates}

In general,
estimating metric entropy directly is not easy.
In this section,
we provide sufficient conditions
for Assumption \ref{1. assum: random spaces}\ref{1. assum item: random, metric-entropy condition}
in terms of volume estimates of balls of the underlying spaces,
which are useful for applications.

\begin{lem} \label{lem: coverings and volume lemma}
  Let $(S,d, \rho)$ be a rooted boundedly-compact metric space and $\mu$ be a Radon measure on $S$.
  Assume that there exist a positive constant $r$ and a function $v:(0,\infty) \to (0,\infty)$ such that
  \begin{equation}
    \inf_{x \in S^{(r)}}
    \mu \left( D_{d} (x, u) \right)
    \geq v(u),
    \quad
    \forall u>0.
  \end{equation}
  Then, for every $r' < r$, it holds that
  \begin{equation}
    N_{d^{(r')}}(S^{(r')}, u)
    \leq
    \mu(S^{(r)}) v(u / 4)^{-1},
    \quad
    \forall u \in (0, r-r').
  \end{equation}
\end{lem}

\begin{proof}
  Fix $u \in (0, r-r')$.
  For a finite subset $(x_{i})_{i=1}^{N}$ in $S^{(r)}$,
  we consider the following condition.
  \begin{enumerate} [label=(D)]
    \item
          The collection $(D_{d}(x_{i}, u/4))_{i=1}^{N}$ is disjoint.
          \label{cond: disjoint-balls condition}
  \end{enumerate}
  If $(x_{i})_{i=1}^{N}$ satisfies \ref{cond: disjoint-balls condition},
  then
  \begin{equation}
    \mu(S^{(r+u)})
    \geq
    \mu
    \left(
    \bigcup_{i=1}^{N}
    D_{d}(x_{i}, u/4)
    \right)
    \geq
    N v(u/4).
  \end{equation}
  Thus
  \begin{equation}
    \left \{
    N \in \mathbb{N}:
    \text{
      there exists a finite subset
    }
    (x_{i})_{i=1}^{N}
    \text{
      of
    }
    S^{(r)}
    \text{
      satisfying (D)
    }
    \right \}
  \end{equation}
  is bounded.
  Choose the maximum $N$ of this set and a corresponding subset $(x_{i})_{i=1}^{N}$ satisfying \ref{cond: disjoint-balls condition}.
  Then $(x_{i})_{i=1}^{N}$ is a $u/2$-covering of $(S^{(r)}, d^{(r)})$.
  If $D_{d}(x_{i}, u/2)$ intersects $S^{(r')}$,
  then we choose an element $y_{i}$ from the intersection.
  Without loss of generality,
  we may assume that we obtain $y_{1},\ldots,y_{N'}$ for some $N' \leq N$.
  It is easy to check that $(y_{i})_{i=1}^{N'}$ is a $u$-covering of $(S^{(r')},d^{(r')})$,
  and thus $N_{d^{(r')}}(S^{(r')},u) \leq N'$.
  For every $y \in D_{d}(x_{i},u/2)$ with $i=1,\ldots,N'$,
  we have that
  \begin{align}
    d(\rho,y)
    \leq
    d(\rho,y_{i})+d(y_{i},x_{i})+d(x_{i},y) < r,
  \end{align}
  which implies that $D_{d}(x_{i},u/2) \subseteq S^{(r)}$ for $i=1,\ldots,N'$.
  Since $(D_{d}(x_{i}, u/4))_{i=1}^{N'}$ is disjoint,
  we obtain that
  \begin{align}
    \mu(S^{(r)})
     & \geq
    \mu
    \left(
    \bigcup_{i=1}^{N'}
    D_{d}(x_{i},u/4)
    \right) \\
     & \geq
    \sum_{i=1}^{N'}
    \mu
    \left(
    D_{d} (x_{i},u/4)
    \right) \\
     & \geq
    N_{d^{(r')}}(S^{(r')},u)
    v(u/4),
  \end{align}
  which completes the proof.
\end{proof}

\begin{prop}  \label{prop: sufficient condition from volume estimate}
  Suppose that $G_{n}=(F_{n},R_{n},\rho_{n},\mu_{n})$ and $G=(F,R,\rho,\mu)$ satisfy
  Assumption \ref{1. assum: random spaces}\ref{1. assum item: random spaces, space convergence}.
  Then Assumption \ref{1. assum: random spaces}\ref{1. assum item: random, metric-entropy condition} is implied
  by the following.
  \begin{enumerate} [resume*=assumptions in random case]
    \item
          There exist a sequence $(r_{k})$ of positive numbers with $r_{k} \to \infty$ and $\alpha_{k} \in (0,1/2)$
          such that, for every $\varepsilon > 0$ and $n$,
          there exists a measurable function $v_{n,k}^{\varepsilon} : (0,\infty) \to (0,\infty)$ such that
          \begin{equation}
            \liminf_{n \to \infty}
            \mathbf{P}_{n}
            \left(
            \inf_{x \in B_{R_{n}}
              (\rho_{n}, r_{k})} \mu_{n} (D_{R_{n}}(x,u))
            \geq v_{n,k}^{\varepsilon}(u),\
            \forall u
            \right)
            \geq
            1-\varepsilon,
            \quad
            \forall k,
          \end{equation}
          and $v_{n,k}^{\varepsilon}$ satisfies
          \begin{equation}
            \lim_{m \to \infty}
            \limsup_{n \to \infty}
            \sum_{l \geq m} v_{n,k}^{\varepsilon} (2^{-l})^{-2}
            \exp (-2^{\alpha_{k} l})=0,
            \quad
            \forall k.
          \end{equation}
          \label{assum item: volume condition for random spaces}
  \end{enumerate}
\end{prop}

\begin{proof}
  We assume that the condition \ref{assum item: volume condition for random spaces} is satisfied.
  Without loss of generality,
  we may assume that $r_{k+1} - r_{k} > 1$.
  Given $r>0$,
  we choose $k$ such that $r_{k} > r$ and $\alpha_{k+1}' \in (\alpha_{k+1},1/2)$.
  We may assume that $G_{n}$ and $G$ are coupled
  so that $G_{n} \to G$ almost-surely under a complete probability measure $P$.
  Fix $\varepsilon >0$,
  and choose $M$ such that
  \begin{equation}
    P
    \left(
    \sup_{m}
    \mu_{m} (F^{(r_{k+1})})
    < M
    \right)
    \geq
    1- \varepsilon.
  \end{equation}
  We define an event $E_{n}$ by setting
  \begin{align}
    E_{n}
    \coloneqq
    \left\{
    \sup_{m}
    \mu_{m} (F^{(r_{k+1})})
    < M
    \right\}
    \cap
    \left\{
    \inf_{x \in F_{n}^{(r_{k}+2^{-1})}}
    \mu_{n} (D_{R_{n}}(x,u))
    \geq
    v_{n,k+1}^{\varepsilon}(u),\quad
    \forall u
    \right\}.
  \end{align}
  Then by condition \ref{assum item: volume condition for random spaces},
  we have that
  \begin{equation} \label{eq: prob of the complement E}
    \limsup_{n \to \infty}
    P(E_{n}^{c})
    \leq
    2\varepsilon.
  \end{equation}
  Given $\delta > 0$,
  by condition \ref{assum item: volume condition for random spaces},
  there exists $m_{0}$ such that,
  for any $m \geq m_{0}$,
  there exists $n(m)$ such that
  \begin{align} \label{eq: volume inequality in proof}
    \sum_{l \geq m}
    v_{n, k+1}^{\varepsilon}(2^{-l-3})^{-2}
    \exp (-2^{\alpha'_{k+1} l })
    <
    \delta/M^{2},\ \
    \forall
    n \geq n(m).
  \end{align}
  Fix $m \geq m_{0}$ for a while.
  Assume that the event $E_{n}$ occurs for some $n \geq n(m)$.
  Then by Lemma \ref{lem: coverings and volume lemma} and \eqref{eq: volume inequality in proof},
  we have that
  \begin{align}
    \sum_{l \geq m}
    N_{R_{n}^{(r)}}(F_{n}^{(r)},2^{-l})^{2}
    \exp(-2^{\alpha_{k+1}' l})
     & \leq
    \sum_{l \geq m}
    N_{R_{n}^{(r_{k})}}(F_{n}^{(r_{k})},2^{-l-1})^{2}
    \exp(-2^{\alpha_{k+1}' l}) \\
     & \leq \sum_{l \geq m}
    \mu_{n}(F_{n}^{(r_{k+1})})^{2}
    v_{n,k+1}^{\varepsilon}(2^{-l-3})^{-2}
    \exp(-2^{\alpha_{k+1}' l})
    < \delta.                  \\
  \end{align}
  Therefore, by \eqref{eq: prob of the complement E},
  we deduce that
  \begin{equation}
    \limsup_{n \to \infty}
    P
    \left(
    \sum_{l \geq m}
    N_{R_{n}^{(r)}}(F_{n}^{(r)},2^{-l})^{2}
    \exp(-2^{\alpha_{k+1}' l})
    \geq
    \delta
    \right)
    \leq
    2\varepsilon.
  \end{equation}
  Letting $m \to \infty$ and $\varepsilon \to 0$ in the above inequality yields the desired result.
\end{proof}

In the next result,
we give a simple version of Proposition \ref{prop: sufficient condition from volume estimate}
which will be useful for several examples.
We consider a finite (or countably infinite) set $F_{n}$ with a distinguished point $\rho_{n}$.
Let $R_{n}$ be a resistance metric on $F_{n}$
inducing the discrete topology on $F_{n}$
and $\mu_{n}$ be the counting measure on $F_{n}$.
Suppose that there exist deterministic sequences $(a_{n})_{n \geq 1}$ and $(b_{n})_{n \geq 1}$ 
of positive numbers with $a_{n} \to \infty,\, b_{n} \to \infty$
and a random element $G=(F,R,\rho,\mu)$ of $\mathbb{F}$ such that
\begin{equation}
  G_{n}\coloneqq (F_{n},a_{n}^{-1}R_{n},\rho_{n},b_{n}^{-1}\mu_{n})
  \xrightarrow{\mathrm{d}}
  G.
\end{equation}

\begin{cor} \label{eq: discrete graph version of volume condition}
  In the above setting,
  Assumption \ref{1. assum: random spaces}\ref{1. assum item: random, metric-entropy condition} is implied
  by the following condition:\
  for all sufficiently large $L >0$ and for all sufficiently small $\varepsilon>0$,
  there exist positive constants $c,\, c_{n}$ and a measurable function $v: (0, \infty) \to (0,\infty)$,
  where $c,\, c_{n}$ and $v$ are allowed to depend on $L$ and $\varepsilon$,
  such that
  \begin{equation} \label{eq: volume estimates for discrete graph cases}
    \liminf_{n \to \infty}
    \mathbf{P}_{n}
    \left(
    \inf_{ x \in B_{R_{n}}(\rho_{n},a_{n}L) }
    b_{n}^{-1}\mu_{n}(D_{R_{n}}(x,a_{n}r))
    \geq v(r),\
    \forall r \in (c_{n}^{-1},c)
    \right)
    \geq 1-\varepsilon,
  \end{equation}
  and
  \begin{gather} \label{eq: condition for scaling factors and volume function}
    \sum_{l}
    v(2^{-l})^{-2}
    \exp(-2^{\alpha l}) < \infty,\quad
    b_{n}^{2}\exp
    \left(
    -c_{n}^{\alpha}
    \right)
    \xrightarrow{n\to\infty}
    0
  \end{gather}
  for some $\alpha \in (0, 1/2)$,
  which depends only on $L$.
\end{cor}

\begin{proof}
  Set
  \begin{equation}
    v_{n}^{\varepsilon}(r)\coloneqq
    \begin{cases}
      v(r),       & (r > c_{n}^{-1}),    \\
      b_{n}^{-1}, & (r \leq c_{n}^{-1}).
    \end{cases}
  \end{equation}
  Then we deduce that for all sufficiently large $n$,
  \begin{align}
    \sum_{l \geq \lceil \log_{2} c_{n} \rceil} v_{n}^{\varepsilon}(2^{-l})^{-2} \exp(-2^{\alpha l})
     & = b_{n}^{2} \sum_{l \geq \lceil \log_{2} c_{n} \rceil} \exp(-2^{\alpha l})                          \\
     & \leq b_{n}^{2}\int_{\log_{2} c_{n}}^{\infty} \exp(-2^{\alpha s}) ds                                 \\
     & \leq b_{n}^{2}\int_{\log_{2} c_{n}}^{\infty} \alpha \log 2 \cdot 2^{\alpha s}\exp(-2^{\alpha s}) ds \\
     & = b_{n}^{2} \exp(-c_{n}^{\alpha}).
  \end{align}
  Using the above inequality,
  it is not difficult
  to check condition \ref{assum item: volume condition for random spaces} in Proposition \ref{prop: sufficient condition from volume estimate}
  under the condition given at \eqref{eq: condition for scaling factors and volume function},
  and hence the desired result follows.
\end{proof}

\begin{rem}
  When showing the Gromov-Hausdorff-Prohorov convergence of $(F_{n}, a_{n}^{-1}R_{n}, \rho_{n}, b_{n}^{-1}\mu_{n})$,
  one needs a lower estimate on volumes of balls with radius of order at least $a_{n}$ in the metric $R_{n}$
  (see \cite[Theorem 6.5]{Archer_Nachmias_Shalev_pre_The_GHP} for a sufficient condition for convergence in the Gromov-Hausdorff-Prohorov topology).
  In \eqref{eq: volume estimates for discrete graph cases},
  we need a lower estimate of volumes of balls with radius of order $a_{n}c_{n}^{-1}$.
  The condition \eqref{eq: condition for scaling factors and volume function} requires
  that $c_{n}$ increases at least as $(\log b_{n})^{2+\varepsilon}$.
  Therefore, roughly speaking,
  if one has a lower estimate on volumes of balls with radius of order $a_{n}/(\log b_{n})^{2+\varepsilon}$,
  then our results are applicable.
  Indeed as we see in Section \ref{sec: examples},
  in many cases $a_{n}, b_{n}$ and $c_{n}$ are polynomial functions of $n$
  and Corollary \ref{eq: discrete graph version of volume condition} can be applied.
\end{rem}

\section{Examples} \label{sec: examples}

\subsection{Real trees and plane trees}

\subsubsection{Gromov-Hausdorff-Prohorov distance between trees} \label{subsubsec: introduction of trees}

In this section,
we recall some results about the Gromov-Hausdorff-Prohorov distance between trees.
For further details about trees the reader should refer to \cite{LeGall_06_Random}, for example.
Write $\mathscr{E}$ for the space of excursions,
that is,
\begin{equation}
  \mathscr{E}
  \coloneqq
  \{
  f \in C(\mbRp, \mbRp)
  :
  f(0)=0,
  \exists \sigma^{f} < \infty\
  \text{such that}\
  f(x) > 0\ \forall x \in (0, \sigma^{f}),
  f(x)=0\ \forall x \geq \sigma^{f}
  \}.
\end{equation}
Given a function $f \in C([0,\sigma], \mbRp)$ 
with $f(0)=f(\sigma)=0$ and $f(x) > 0$ for all $x \in (0, \sigma)$,
we will abuse notation by identifying $f$ with the function $g \in \mathscr{E}$
which has $g(x)=f(x),\, 0 \leq x \leq \sigma$ and $g(x)=0, \, x \geq \sigma$.
We equip $\mathscr{E}$ with the metric induced by the supremum norm $\| \cdot \|$. 
Given an excursion $f \in \mathscr{E}$, 
we define a pseudometric $\bar{d}^{f}$ on $[0, \sigma^{f}]$ by setting
\begin{equation}
  \bar{d}^{f}(s,\, t)
  \coloneqq
  f(s)+f(t)-2 \inf_{u \in [s \wedge t, s \vee t]} f(u).
\end{equation}
Then,
we use the equivalence
\begin{equation}
  s \sim t \qquad
  \Leftrightarrow \qquad
  \bar{d}^{f}(s,\, t)=0
\end{equation}
to define $T^{f} \coloneqq  [0, \sigma^{f}]/ \sim$.
Let $p^{f} : [0,\sigma^{f}] \to T^{f}$ be the canonical projection.
It is then elementary to check that
\begin{equation}
  d^{f}(p^{f}(s), p^{f}(t))
  \coloneqq
  \bar{d}^{f}(s,\, t)
\end{equation}
defines a metric on $T^{f}$.
The metric space $(T^{f}, d^{f})$ is called a \textit{real tree} coded by $f$.
We have a Radon measure $m^{f} \coloneqq  \Leb \circ (p^{f})^{-1}$,
where $\Leb$ stands for the one-dimensional Lebesgue measure.
We define the root $\rho^{f}$ by setting $\rho^{f} \coloneqq  p^{f}(0)$.
For $a,\, b >0$, we set
\begin{equation}
  \mathcal{T}^{f}_{a,b}
  \coloneqq
  (T^{f}, ad^{f}, \rho^{f}, b m^{f}),
\end{equation}
and we abbreviate $\mathcal{T}^{f} \coloneqq  \mathcal{T}_{1,1}^{f}$,
which is the rooted real tree coded by $f$ with the canonical measure.
The following is a basic property of real trees regarding scaling
and we omit the proof.

\begin{lem} \label{lem: invariance wrt scaling of real trees}
  Fix $f \in \mathscr{E}$.
  For $\alpha,\, \beta >0$,
  we set $g(t)\coloneqq \alpha f(\beta t)$.
  Then $\mathcal{T}^{g}$ is Gromov-Hausdorff-Prohorov isometric to $\mathcal{T}^{f}_{\alpha, \beta^{-1}}$,
  i.e.\ $d_{GHP}(\mathcal{T}^{f}_{\alpha, \beta^{-1}}, \mathcal{T}^{g})=0$.
\end{lem}

\begin{prop} [{\cite[Proposition 3.3]{Abraham_Delmas_Hoscheit_13_A_note}}] \label{prop: ghp distance between real trees}
  Let $f,\, g \in \mathscr{E}$.
  Then it holds that
  \begin{equation}
    d_{\mathbb{G}_{c}}
    \left(
    \mathcal{T}^{f},
    \mathcal{T}^{g}
    \right)
    \leq
    6 ||f-g||+|\sigma^{f} - \sigma^{g}|.
  \end{equation}
\end{prop}

Next,
we consider a plane tree.
Let $\mathbf{I}= \bigcup_{n=0}^{\infty} \mathbb{N}^{n}$,
where we set $\mathbb{N}^{0}\coloneqq \{ 0 \}$ and $0$ serves as the root.
If $u=(u_{1}, \ldots, u_{m})$ and $v=(v_{1},\ldots,v_{n})$ belong to $\mathbf{I}$,
then we write $uv=(u_{1}, \ldots., u_{m}, v_{1},\ldots,v_{n})$ for the concatenation of $u$ and $v$.
In particular, $0u=u0=u$. A \textit{plane tree} $\tau$ is a finite subset of $\mathbf{I}$ such that:
\begin{itemize}
  \item
        $\rho_{\tau} \coloneqq0 \in \tau$.
  \item
        If $v \in \tau$ and $v=uj$ for some $u \in \mathbf{I}$ and $j \in \mathbb{N}$,
        then $u \in \tau$
        ($u$ is said to be a \textit{parent} of $v$ 
        and $v$ is said to be a \textit{child} of $u$).
  \item
        For every $u \in \tau$,
        there exists a number $k_{u}\coloneqq k_{u}(\tau) \geq 0$ such that $uj \in \tau$ if and only if $1 \leq j \leq k_{u}$.
\end{itemize}
We denote the set of all plane trees by $\mathbf{T}$.
Note that we can regard a plane tree $\tau$ as a graph 
by declaring $\tau$ is a vertex set and 
$\{u, v\} \subseteq \tau$ is an edge if and only if 
one of $u$ and $v$ is a parent of the other.
Given a plane tree $\tau \in \mathbf{T}$,
we write $d^{\tau}$ and $m^{\tau}$ for the graph distance and the counting measure on $\tau$.

We introduce coding functions of plane trees.
Let $\tau$ be a plane tree with $n+1$ nodes and the root $\rho^{\tau} \coloneqq 0$.
For children $u=vj$ and $u'=vj'$ of $v \in \tau$,
$u$ is said to be older than $u'$ if $j < j'$.
We set $\llbracket a,b \rrbracket \coloneqq  [a,b] \cap \mathbb{Z}$.
We define a map $f_{\tau}:\llbracket 0, 2n \rrbracket \to \tau$ on $\tau$ as follows:
Set $f_{\tau}(0)\coloneqq \rho^{\tau}$.
Given $f_{\tau}(i)=u$,
if there are unvisited children of $u$,
then $f_{\tau}(i+1)$ is the oldest children among them,
and otherwise $f_{\tau}(i+1)$ is the parent of $u$.
The \textit{contour function} $C^{\tau}: [0, 2n] \to \mbRp$ is defined
by setting $C^{\tau}(k)=d^{\tau}(f_{\tau}(k))$ for $k \in \llbracket 0, 2n \rrbracket$ and linear interpolation.
Let $u_{0} (=0), u_{1}, \ldots, u_{n}$ be the vertices of $\tau$ in lexicographical order
(i.e., $u_{i}$ is the $i$-th new node visited by the function $f_{\tau}$).
The \textit{height function} $H^{\tau}:[0,n] \to \mbRp$ is defined
by setting $H^{\tau}(i) \coloneqq  h_{\tau}(u_{i})$ for $i \in \llbracket 0, n \rrbracket$ and linear interpolation.
The \textit{depth-first walk} $X^{\tau} : [0, n+1] \to \mbRp$ is defined by setting
\begin{gather}
  X^{\tau}(0) \coloneqq 0,\\
  X^{\tau}(i) \coloneqq \sum_{j =0}^{i-1} (k_{u_{j}}(\tau)-1)
  \ \text{for}\  i=1,2, \ldots, n+1.
\end{gather}

For $\tau \in \mathbf{T}$ and $a,\, b >0$,
we set
\begin{equation}
  \mathcal{T}^{\tau}_{a,b}
  \coloneqq
  (\tau, ad^{\tau}, \rho^{\tau}, b m^{\tau}).
\end{equation}

\begin{prop} \label{prop: distance between real and plane trees}
  For every $\tau \in \mathbf{T}$ and $a,\, b >0$,
  we have that
  \begin{equation}
    d_{\mathbb{G}_{c}}
    \left(
    \mathcal{T}_{a,b}^{\tau},
    \mathcal{T}_{a,\frac{b}{2}}^{C^{\tau}}
    \right)
    \leq
    \frac{3a}{2}+b.
  \end{equation}
\end{prop}

\begin{proof}
  Define a map $\tilde{f}_{\tau}:[0,2n] \to \tau$ 
  by setting $\tilde{f}_{\tau}(t)$ to be $f_{\tau}(\lceil t \rceil)$ or $f_{\tau}(\lfloor t \rfloor)$,
  whichever is most distant from the root.
  Define a correspondence $\mathcal{R}$ between $\tau$ and $T^{C^{\tau}}$
  by setting
  \begin{equation}
    \mathcal{R}
    \coloneqq
    \left\{
    (u,x) \in \tau \times T^{C^{\tau}} :
    \exists t \in [0,2n]\
    \text{such that} \
    \tilde{f}_{\tau}(t)=u, \,
    p^{C^{\tau}}(t)=x
    \right\}.
  \end{equation}
  (Recall the definition of a correspondence between sets from \cite[Definition 7.3.17]{Burago_Burago_Ivanov_01_A_course}.)
  Then,
  we define a metric $d$ on the disjoint union $\tau \sqcup T^{C^{\tau}}$, extending $ad^{\tau}$ and $ad^{C^{\tau}}$,
  by setting
  \begin{align}
    d(u,x)
    \coloneqq
    \inf
    \left\{
    ad^{\tau}(u,u')
    +
    \frac{1}{2}\, \mathrm{dis}\mathcal{R}
    +
    \delta
    +
    a d^{C^{\tau}}(x', x)
    :
    (u',x') \in \mathcal{R}
    \right\},
  \end{align}
  where $\delta$ is a positive constant and $\mathrm{dis} \mathcal{R}$ is the distortion of the correspondence $\mathcal{R}$.
  Note that the distortion is given by
  \begin{equation}
    \mathrm{dis} \mathcal{R}
    \coloneqq
    \sup
    \left\{
    |ad^{\tau}(u,u')-ad^{C^{\tau}}(x,x')| \,
    :
    (u,x),(u',x') \in \mathcal{R}
    \right\}.
  \end{equation}
  Given $s,\, t \in [0,2n]$ with $s<t$,
  we can find $s',\, t' \in \llbracket 0, 2n \rrbracket$ such that $|s'-s| \vee |t'-t| \leq 1$,
  $\tilde{f}_{\tau}(s)=f_{\tau}(s')$ and $\tilde{f}_{\tau}(t)=f_{\tau}(t')$ hold.
  This yields that
  \begin{align}
     &
    \left|
    d^{\tau}(\tilde{f}(t),\tilde{f}(s))
    -
    d^{C^{\tau}}(p^{C^{\tau}}(t), p^{C^{\tau}}(s))
    \right| \\
    =
     &
    \left|
    C^{\tau}(t')
    +
    C^{\tau}(s')
    -
    \inf_{s' \leq u \leq t'} C^{\tau}(u)
    -
    C^{\tau}(t)
    -
    C^{\tau}
    +
    \inf_{s\leq u \leq t} C^{\tau}(u)
    \right|
    \leq
    3,
  \end{align}
  which implies that $\mathrm{dis} \mathcal{R} \leq 3a$.
  Thus we obtain that
  \begin{align} \label{eq: Hausdorff and root distance for trees}
    d_{H}(\tau, T^{C^{\tau}})
    \vee
    d(\rho^{\tau},\rho^{C^{\tau}})
    \leq
    \frac{1}{2}\mathrm{dis}\mathcal{R}
    +
    \delta
    \leq
    \frac{3a}{2}
    +
    \delta
  \end{align}
  (c.f.\ \cite[Theorem 7.3.25]{Burago_Burago_Ivanov_01_A_course}).
  Next,
  we consider the Prohorov distance.
  For every $u \in \tau \setminus \{\rho^{\tau} \}$,
  there are two intervals $I_{1}(u)$ and $I_{2}(u)$ of unit length satisfying
  \begin{gather}
    I_{1}(u)
    \cap
    I_{2}(u)
    =
    \emptyset,
    \ I_{1}(u)
    \cup
    I_{2}(u)
    =
    \tilde{f}_{\tau}^{-1} (\{ u \}) .
  \end{gather}
  Let $A$ be a subset of $\tau$,
  and define a subset $B$ of $T^{C^{\tau}}$ by setting
  \begin{equation}
    B
    \coloneqq
    p^{C^{\tau}}
    \left(
    \bigcup_{a \in A \setminus \{ \rho^{\tau} \} }
    ( I_{1}(a) \cup I_{2}(a) )
    \right).
  \end{equation}
  It is then the case that
  \begin{align}
    m^{C^{\tau}}(B)
    \geq
    \Leb
    \left(
    \bigcup_{a \in A \setminus \{ \rho^{\tau} \} }
    ( I_{1}(a) \cup I_{2}(a) )
    \right)
    \geq
    2(m^{\tau}(A)-1),
  \end{align}
  and
  $B \subseteq A^{\frac{1}{2} \mathrm{dis} \mathcal{R}+2\delta}
    \coloneqq
    \{ z \in \tau \sqcup T^{C^{\tau}} : d(z, A) <\frac{1}{2} \mathrm{dis} \mathcal{R} + 2\delta \}$.
  Therefore it follows that
  \begin{equation} \label{eq: prohorov for trees 1}
    m^{\tau}(A)
    \leq
    \frac{1}{2}
    m^{C^{\tau}}
    (A^{\frac{1}{2} \mathrm{dis} \mathcal{R}+2\delta})
    +
    1.
  \end{equation}
  Let $B'$ be  a closed set of $T^{C^{\tau}}$,
  and define a subset $A'$ of $\tau$ by setting
  \begin{equation}
    A'
    \coloneqq
    \tilde{f}_{\tau}
    \left(
    (p^{C^{\tau}})^{-1}(B')
    \right).
  \end{equation}
  Then we have that
  \begin{align}
    (p^{C^{\tau}})^{-1} (B')
    \subseteq
    \tilde{f}_{\tau}^{-1}(A')
    \subseteq \{ 0 \}
    \cup
    \bigcup _{a \in A' \setminus \{ \rho^{\tau} \} }
    (I_{1}(a) \cup I_{2}(a)),
  \end{align}
  which yields that
  \begin{align}
    m^{C^{\tau}}(B')
    =
    \Leb
    \left(
    (p^{C^{\tau}})^{-1}(B')
    \right)
    \leq
    2m^{\tau}(A').
  \end{align}
  Since we also have that
  $A' \subseteq (B')^{\frac{1}{2} \mathrm{dis} \mathcal{R}+2\delta}
    \coloneqq  \{ z \in \tau \sqcup T^{C^{\tau}} : d(z, B') < \frac{1}{2} \mathrm{dis} \mathcal{R} + 2\delta \}$,
  we deduce that
  \begin{equation} \label{eq: prohorov for trees 2}
    \frac{1}{2}
    m^{C^{\tau}}(B')
    \leq
    m^{\tau}( (B')^{\frac{1}{2} \mathrm{dis} \mathcal{R}+2\delta}).
  \end{equation}
  By \eqref{eq: prohorov for trees 1} and \eqref{eq: prohorov for trees 2},
  we obtain that
  \begin{equation} \label{eq: prohorov for trees}
    d_{P}
    (
    bm^{\tau},
    \frac{b}{2}m^{C^{\tau}}
    )
    \leq
    \frac{1}{2} \mathrm{dis} \mathcal{R}
    +
    2\delta
    +
    b
    \leq \frac{3}{2}a
    +
    2\delta
    +
    b.
  \end{equation}
  By \eqref{eq: Hausdorff and root distance for trees}, \eqref{eq: prohorov for trees} and letting $\delta \to 0$,
  we obtain the desired result.
\end{proof}

\subsubsection{Critical Galton-Watson trees} \label{subsubsec: critical GW trees}
Let $W=(W_{t})_{0 \leq t \leq 1}$ be the normalized Brownian excursion.
The random real tree $\mathcal{T}^{2W}$ is the continuum random tree (CRT) introduced by Aldous.
In \cite{Aldous_93_The_continuum},
he showed that the properly scaled contour functions of Galton-Watson trees with the finite variance offspring distribution converge to $2W$,
and this result is extended to Galton-Watson trees with possibly infinite variance offspring distribution in \cite{Duquesne_03_A_limit}.
The convergence of the contour functions leads to the convergence of the Galton-Watson trees
in the Gromov-Hausdorff-Prohorov topology (see Corollary \ref{cor: convergence of trees}),
and the convergence of trees yields the convergence of random walks on the trees.
Moreover,
in \cite[Section 4.1]{Andriopoulos_23_Convergence},
the convergence of associated local times is proved through detailed calculations about the equicontinuity of the local times.
In this section,
we provide another proof of this convergence by estimating volumes of balls in the Galton-Watson trees.

\begin{assum} \label{assum: assumption for offspring distribution}
  Let $p$ be a probability measure on $\mathbb{Z}_{+}$
  such that the mean of the distribution $p$ is $1$, $p(0)>0$ and $p$ is aperiodic,
  that is, the greatest common divisor of the set $\{ n : p(n) >0 \}$ is $1$.
  Moreover there exists $\alpha \in (1,2]$
  such that $p$ belongs to the domain of attraction of an $\alpha$-stable law,
  which means that there exists an increasing sequence $(B_{n})$ of positive numbers such that
  \begin{equation}
    B_{n}^{-1} (\xi_{1}+ \cdots+\xi_{n}-n)
    \xrightarrow{\mathrm{d}}
    X^{(\alpha)},
  \end{equation}
  where $(\xi_{n})$ is a sequence of i.i.d.\ random variables with distribution $p$
  and $X^{(\alpha)}$ is a random variable
  whose law is given by the Laplace transform $\displaystyle E(e^{-\lambda X^{(\alpha)}})=\exp(\lambda^{\alpha})$.
\end{assum}

\begin{rem}
  It is elementary to check that,
  for any $\gamma \in [0, 1-\alpha^{-1})$,
  it holds that
  \begin{equation} \label{eq: divergence of coefficient}
    \lim_{n \to \infty}
    \frac{n}{B_{n}\, n^{\gamma}}
    =
    \infty
  \end{equation}
  (cf.\ \cite[Theorem 8.3.1]{Bingham_Goldie_Teugels_87_Regular} and the comment below the theorem).
  This is used in the proof of Proposition \ref{prop: volume estimates for trees}.
\end{rem}

Let $Y=(Y_{t})_{t \geq 0}$ be a L\'{e}vy process with $Y_{1}\overset{\mathrm{d}}{=} X^{(\alpha)}$
and $H^{\mathrm{exc}}=(H^{\mathrm{exc}};0\leq t \leq 1)$ be the normalized excursion of the height process of $Y$
built on a probability measure $\mathbf{P}$
(see \cite[Section 3]{Duquesne_03_A_limit} for their definitions).
We consider a Galton-Watson tree $\tau$ with offspring distribution $p$
satisfying Assumption \ref{assum: assumption for offspring distribution}.
We denote by $\mathbf{P}_{n}$ the underlying probability measure conditioning $\tau$ to have $n+1$ nodes,
which is well-defined for all sufficiently large $n$ by the aperiodicity of $p$.

\begin{thm} [{\cite[Theorem 3.1]{Duquesne_03_A_limit}}] \label{thm: convergence of contour function}
  It holds that
  \begin{equation}
    \mathbf{P}_{n} \left( \left(\frac{B_{n}}{n}C^{\tau}_{2nt}\,  ; 0\leq t \leq1 \right) \in \cdot \right)
    \to
    \mathbf{P} (H^{\mathrm{exc}} \in \cdot)
  \end{equation}
  weakly as probability measures on $C([0,1],\mbR)$ equipped with the uniform convergence topology.
\end{thm}

Set
\begin{align}
  \mathcal{T}_{n}
   & \coloneqq
  \mathcal{T}_{\frac{B_{n}}{n},
  \frac{1}{n}}^{\tau}
  =(\tau, \frac{B_{n}}{n}\, d^{\tau},
  \rho^{\tau}, \frac{1}{n}\, m^{\tau}),        \\
  \mathcal{T}
   & \coloneqq  \mathcal{T}^{H^{\mathrm{exc}}}
  =(T^{H^{\mathrm{exc}}}, d^{H^{\mathrm{exc}}},
  \rho^{H^{\mathrm{exc}}}, m^{H^{\mathrm{exc}}}).
\end{align}

\begin{cor} \label{cor: convergence of trees}
  In the above setting,
  $\mathbf{P}_{n}(\mathcal{T}_{n} \in \cdot) \xrightarrow{\mathrm{d}} \mathbf{P}(\mathcal{T} \in \cdot)$
  weakly as probability measures on $\mathbb{G}_{c}$.
\end{cor}

\begin{proof}
  We may assume that $C^{\tau}_{2n \cdot}$ and $H^{\mathrm{exc}}$ are coupled
  so that the convergence in Theorem \ref{thm: convergence of contour function} is the almost-sure convergence.
  By Proposition \ref{prop: ghp distance between real trees},
  we deduce that $\mathcal{T}^{\frac{B_{n}}{n}\, C^{\tau}_{2n \cdot}}$ converges to $\mathcal{T}$ almost surely.
  By Lemma \ref{lem: invariance wrt scaling of real trees} and Proposition \ref{prop: distance between real and plane trees},
  we obtain the desired convergence.
\end{proof}

By \cite[Proposition 5.1]{Kigami_95_Harmonic},
plane trees and real trees are resistance metric spaces,
and it is easy to check that $\mathcal{T}_{n}$ belongs to $\check{\mathbb{F}}_{c}$.
To apply Theorem \ref{1. thm: main result for random spaces},
we need volume estimates of balls in $\mathcal{T}_{n}$,
and these can be obtained from the H\"{o}lder continuity of the height function.

\begin{lem} [{\cite[Lemma 1]{Marzouk_20_Scaling}}]  \label{lem: holder continuity of height function}
  For every $\gamma \in (0, 1-\alpha^{-1})$ and every $\varepsilon >0$,
  there exists $C_{\gamma, \varepsilon} >0$ such that
  \begin{equation}  \label{eq: tightness of holder continuity of height functions}
    \liminf_{n \to \infty}
    \mathbf{P}_{n}
    \left(
    | H^{\tau}(t)-H^{\tau}(s)|
    \leq
    \frac{n}{B_{n}n^{\gamma}}\, C_{\gamma ,\varepsilon}|t-s|^{\gamma},\quad
    \forall t,s \in [0,n]
    \right)
    \geq
    1-\varepsilon.
  \end{equation}
\end{lem}

The following technical lemma is used
to estimate the distance between two vertices of a plane tree by the height function.

\begin{lem} [{\cite[Lemma 17]{Berry_Broutin_Goldschmidt_12_The_continuum}}] \label{lem: distance estimate by the height function}
  Suppose that $\tau$ is a plane tree with vertices $v_{0}, v_{1}, \ldots, v_{n}$ labeled in lexicographical order.
  Write $u \wedge v$ for the common ancestor of vertices $u$ and $v$ furthest from $v_{0}$.
  Then, for $0 \leq i \leq j \leq n$,
  \begin{equation}
    \left|
    d^{\tau}(v_{0}, v_{i} \wedge v_{j})
    -
    \min_{i \leq k \leq j} H^{\tau}(k)
    \right|
    \leq
    1.
  \end{equation}
\end{lem}

\begin{prop} \label{prop: volume estimates for trees}
  For every $\gamma \in (0,1-\alpha^{-1})$ and every $\varepsilon >0$,
  there exists $c_{\gamma,\varepsilon}>0$ such that
  \begin{equation}
    \liminf_{n \to \infty}
    \mathbf{P}_{n}
    \left(
    \inf_{x \in \tau} \frac{1}{n} m^{\tau}
    \left(D_{d^{\tau}}\left(x, \frac{n}{B_{n}}r\right)\right)
    \geq
    (c_{\gamma, \varepsilon} r^{\gamma^{-1}}) \wedge 1,\quad \forall r>0
    \right)
    \geq
    1-\varepsilon.
  \end{equation}
\end{prop}

\begin{proof}
  Let $C_{\gamma, \varepsilon} >0$ be the constant of Lemma \ref{lem: holder continuity of height function}.
  By \eqref{eq: divergence of coefficient},
  we have that, for all sufficiently large $n$,
  \begin{equation} \label{eq: coefficient is larger than 2}
    \frac{C_{\gamma, \varepsilon} n}{B_{n} n^{\gamma}} >2.
  \end{equation}
  Assume that the Galton-Watson tree $\tau$ with vertices $v_{0}, v_{1}, \ldots, v_{n}$ in lexicographical order satisfies
  \begin{equation} \label{eq: holder continuity of height function in proof}
    | H^{\tau}(s)-H^{\tau}(t)|
    \leq
    \frac{n}{B_{n}n^{\gamma}}\, C_{\gamma ,\varepsilon}|s-t|^{\gamma},\quad
    \forall s,\, t \in [0,n].
  \end{equation}
  Fix $v_{i} \in \tau$ and $k \in \{1, 2, \ldots, n\}$.
  If $v_{j}$ satisfies $| i-j | \leq k$,
  then, by Lemma \ref{lem: distance estimate by the height function}, \eqref{eq: coefficient is larger than 2}
  and \eqref{eq: holder continuity of height function in proof},
  we deduce that
  \begin{align}
    d^{\tau}(v_{i}, v_{j})
     & =
    d^{\tau}(v_{i}, v_{0}) + d^{\tau}(v_{j}, v_{0}) - 2d^{\tau}(v_{0}, v_{i} \wedge v_{j}) \\
     &
    \leq
    H^{\tau}(i) + H^{\tau}(j) - 2 \min_{i \wedge j \leq k \leq i \vee j} H^{\tau}(k) + 2   \\
     &
    \leq
    \frac{2 C_{\gamma, \varepsilon} n } {B_{n}}
    \left( \frac{k}{n} \right)^{\gamma}
    +2
    \leq
    \frac{3 C_{\gamma, \varepsilon} n } {B_{n}}
    \left( \frac{k}{n} \right)^{\gamma}.
  \end{align}
  There are, at least, $k+1$ many $v_{j}$ satisfying $|i -j| \leq k$,
  and hence we deduce that
  \begin{equation} \label{eq: volume of trees with non-scaling}
    m^{\tau}
    \left(
    D_{d^{\tau}}
    \left(
    v_{i},
    \frac{3 C_{\gamma, \varepsilon} n } {B_{n}}
    \left( \frac{k}{n} \right)^{\gamma}
    \right)
    \right)
    \geq
    (k+1) \wedge n
  \end{equation}
  for $k \in \{1,2, \ldots, n\}$.
  The above inequality holds for $k=0$.
  If the integer $k$ is larger than $n$,
  then, by \eqref{eq: volume of trees with non-scaling},
  it follows that
  \begin{equation}
    m^{\tau}
    \left(
    D_{d^{\tau}}
    \left(
    v_{i},
    \frac{3 C_{\gamma, \varepsilon} n } {B_{n}}
    \left( \frac{k}{n} \right)^{\gamma}
    \right)
    \right)
    \geq
    m^{\tau}
    \left(
    D_{d^{\tau}}
    \left(
    v_{i},
    \frac{3 C_{\gamma, \varepsilon} n } {B_{n}}
    \left( \frac{n}{n} \right)^{\gamma}
    \right)
    \right)
    \geq
    (k+1) \wedge n.
  \end{equation}
  Thus, the inequality \eqref{eq: volume of trees with non-scaling} holds for any $k \in \mathbb{Z}_{+}$.
  For $r \in \mbRp$,
  choose $l \in \mathbb{Z}_{+}$ with $k \leq r < k+1$.
  Then it holds that
  \begin{equation}
    m^{\tau}
    \left(
    D_{d^{\tau}}
    \left(
    v_{i},
    \frac{3 C_{\gamma, \varepsilon} n } {B_{n}}
    \left( \frac{r}{n} \right)^{\gamma}
    \right)
    \right)
    \geq
    m^{\tau}
    \left(
    D_{d^{\tau}}
    \left(
    v_{i},
    \frac{3 C_{\gamma, \varepsilon} n } {B_{n}}
    \left( \frac{k}{n} \right)^{\gamma}
    \right)
    \right)
    \geq
    (k+1) \wedge n
    \geq
    r \wedge n
  \end{equation}
  From the above inequality and Lemma \ref{lem: holder continuity of height function},
  the desired result follows.
\end{proof}

By Proposition \ref{prop: volume estimates for trees} and Corollary \ref{eq: discrete graph version of volume condition},
we obtain the convergence of stochastic processes and local times on Galton-Watson trees
as essentially established in \cite{Andriopoulos_23_Convergence} .
\begin{cor}
  The limiting space $\mathcal{T}$ belongs to $\check{\mathbb{F}}_{c}$ with probability $1$,
  and $\cX_{\mathcal{T}_{n}} \xrightarrow{\mathrm{d}} \cX_{\mathcal{T}}$ 
  as random elements of $\mbM_{L}$.
\end{cor}

\subsection{Uniform spanning trees in high-dimensional tori} \label{subsec: UST in high dimensions}
Now we provide new results about convergence of local times of random walks on uniform spanning trees.
A spanning tree of a connected finite graph $G$ is a connected subgraph whose edges touch every vertex of $G$ and contain no cycles.
The uniform spanning tree (UST) on $G$ is a uniformly drawn sample from the finite set of such objects.
In \cite{Archer_Nachmias_Shalev_pre_The_GHP},
it is shown that uniform spanning trees on graphs in high-dimensional tori converge to the CRT,
and in this section,
we apply our main results to this.

For a graph $G$,
two vertices $x,y$ and a non-negative integer $t$,
we write $p_{t}(x,y)$ for the probability that the lazy random walk starting at $x$ will be at $y$ at time $t$.
When $G$ is a finite connected regular graph on $n$ vertices,
we define the uniform mixing time of $G$ to be
\begin{equation}
  t_{\mathrm{mix}}(G)
  \coloneqq
  \min
  \left\{
  t \geq 0 :
  \max_{x,y \in G} |np_{t}(x,y)-1|\leq \frac{1}{2}
  \right\}.
\end{equation}

\begin{assum} \label{assum: ust assumption in high}
  Let $(G_{n})$ be a sequence of finite connected vertex-transitive graphs with $n$ vertices.
  \begin{enumerate}
    \item
          There exists $\theta < \infty$ such that $\displaystyle \sup_{n} \sup_{x \in G_{n}} \sum_{t=0}^{\sqrt{n}}(t+1) p_{t}(x,x) \leq \theta$.
    \item
          There exists $\alpha > 0$ such that $t_{\mathrm{mix}}(G_{n})=o(n^{\frac{1}{2}-\alpha})$ as $n \to \infty$.
  \end{enumerate}
\end{assum}

Let $(G_{n})_{n \geq 1}$ be a deterministic sequence of graphs satisfying Assumption \ref{assum: ust assumption in high} and,
for each $n$, fix an arbitrary vertex $\rho_{n}$ in $G_{n}$.
We consider the uniform spanning tree $T_{n}$ on $G_{n}$
built on a probability space equipped with a complete probability measure $\mathbf{P}_{n}$.
Let $d_{n}$ and $\mu_{n}$ be the graph metric and counting measure on $T_{n}$.

\begin{thm} [{\cite[Theorem 1.7]{Archer_Nachmias_Shalev_pre_The_GHP}}] \label{thm: ust theorem in high dimension}
  In the above setting, there exists a sequence $(\beta_{n})_{n \geq 1}$ 
  satisfying $0<\inf_{n} \beta_{n} \leq \sup_{n} \beta_{n} < \infty$ such that
  \begin{equation}
    \mathcal{T}_{n}
    \coloneqq
    \left(
    T_{n}, \frac{1}{\beta_{n} \sqrt{n}}d_{n}, \rho_{n}, \frac{1}{n} \mu_{n}
    \right) \xrightarrow{\mathrm{d}} \mathcal{T}^{2W}
  \end{equation}
  in the Gromov-Hausdorff-Prohorov topology,
  where $\mathcal{T}^{2W}$ is the rooted CRT with its canonical measure (see Section \ref{subsubsec: introduction of trees}).
\end{thm}

\begin{rem}
  As stated in \cite{Archer_Nachmias_Shalev_pre_The_GHP},
  the scaling limit of uniform spanning trees of the $d$-dimensional torus $\mathbb{Z}_{n}^{d}$ with $d >4$ is included
  in Theorem \ref{thm: ust theorem in high dimension}.
\end{rem}

Set $\beta\coloneqq \inf_{n}\beta_{n}$ and $\gamma\coloneqq \alpha/20$,
where $(\beta_{n})_{n \geq 1}$ is the sequence of Theorem \ref{thm: ust theorem in high dimension}
and $\alpha$ is the constant of Assumption \ref{assum: ust assumption in high}.

\begin{prop} \label{prop: volume estimate for ust in high}
  In the above setting, for every $\delta>0$,
  there exist positive constants $c,\, \varepsilon$ such that
  \begin{equation}
    \liminf_{n \to \infty}
    \mathbf{P}_{n}
    \left(
    \inf_{x \in T_{n}}
    \frac{1}{n} \mu_{n}
    \left( D_{d_{n}}(x, \beta_{n} \sqrt{n}r) \right)
    \geq
    \frac{\varepsilon \beta^{4}}{16c^{2}}r^{4},\quad
    \forall r
    \in
    \left[\frac{c}{\beta n^{\gamma}},\frac{2c}{\beta} \right]
    \right)
    \geq
    1-\delta.
  \end{equation}
\end{prop}

\begin{proof}
  For $c,\, \varepsilon >0,\, l \in \mathbb{Z}_{+}$ and $n$, we set
  \begin{gather}
    r_{l}\coloneqq \frac{c\sqrt{n}}{2^{l}},\quad
    \varepsilon_{l}\coloneqq \frac{\varepsilon}{4^{l}},\quad
    N_{n}\coloneqq 2\gamma \log_{2}n=\frac{\alpha}{10} \log_{2}n
  \end{gather}
  and define the events
  \begin{align}
    A_{n,l} & \coloneqq
    \left \{
    \text{there exists an}\
    x \in \mathcal{T}_{n}\
    \text{such that}\
    \mu_{n}(D_{d_{n}} (x, r_{l})) \leq \varepsilon_{l} r_{l}^{2}\
    \text{and}\
    \mu_{n}(D_{d_{n}}(x, r_{l+1})) \geq \varepsilon_{l+1} r_{l+1}^{2}
    \right \},          \\
    B_{n,l} & \coloneqq
    \left \{
    \text{there exists an}\
    x \in \mathcal{T}_{n}\
    \text{such that}\
    \mu_{n}(D_{d_{n}} (x, r_{l})) \leq \varepsilon_{l} r_{l}^{2}
    \right \},          \\
    V_{n}   & \coloneqq
    \left \{
    \text{there exists an}\
    x \in \mathcal{T}_{n}\
    \text{such that}\
    \mu_{n}(D_{d_{n}} (x, r_{l})) \leq \varepsilon_{l} r_{l}^{2}\
    \text{for some integer}\
    0 \leq l \leq \lfloor N_{n} \rfloor
    \right \}.
  \end{align}
  Then we have
  \begin{equation}
    V_{n}
    \subseteq
    \left( \bigcup_{l=0}^{\lfloor N_{n} \rfloor -1} A_{n,l} \right)
    \cup
    B_{n, \lfloor N_{n} \rfloor}.
  \end{equation}
  By the argument of the proof of \cite[Theorem 3.2]{Archer_Nachmias_Shalev_pre_The_GHP},
  given $\delta >0$,
  we can find $c,\, \varepsilon >0$ so that
  \begin{equation}
    \mathbf{P}_{n}
    \left(
    \left( \bigcup_{l=0}^{\lfloor N_{n} \rfloor -1} A_{n,l} \right)
    \cup
    B_{n, \lfloor N_{n} \rfloor}
    \right)
    < \delta
  \end{equation}
  for all sufficiently large $n$.
  (To check this, observe that the above probability is bounded above
  by the right-hand side of \cite[the inequality (17)]{Archer_Nachmias_Shalev_pre_The_GHP},
  which can be made as small as we want by the choice of $c$ and $\varepsilon$.)
  Therefore we deduce that
  \begin{equation}
    \liminf_{n \to \infty}
    \mathbf{P}_{n}
    \left(
    \inf_{x \in T_{n}}
    \mu_{n}
    \left(
    D_{d_{n}}\left(x, \frac{c\sqrt{n}}{2^{l}}
    \right)
    \right)
    >
    \frac{\varepsilon}{4^{l}}
    \left( \frac{c\sqrt{n}}{2^{l}} \right) ^{2},\quad
    \forall l=0,\ldots, \lfloor 2\gamma \log_{2}n \rfloor
    \right)
    \geq
    1-\delta.
  \end{equation}
  Suppose that the above event occurs.
  Then for every
  \begin{equation}
    r
    \in
    \left[
      \frac{c}{\beta n^{\gamma}}, \frac{2c}{\beta}
      \right],
  \end{equation}
  we can find an integer $l \in \{0,\ldots, \lceil \gamma \log_{2}n \rceil \}$ satisfying
  \begin{equation}
    \frac{c}{\beta 2^{l}}
    \leq
    r \leq \frac{c}{\beta 2^{l-1}},
  \end{equation}
  and we have that
  \begin{align}
    \mu_{n}
    \left(
    D_{d_{n}}(x, \beta_{n} \sqrt{n} r)
    \right)
     & \geq
    \mu_{n}
    \left(
    D_{d_{n}}
    \left(
    x, \frac{c\sqrt{n} }{2^{l}}
    \right)
    \right)     \\
     & \geq
    \frac{\varepsilon}{4^{l}}
    \left(
    \frac{c\sqrt{n}}{2^{l}}
    \right)^{2} \\
     & \geq
    \frac{\delta \beta^{4}}{16c^{2}}r^{4}n.
  \end{align}
  Therefore the desired result follows.
\end{proof}

By Proposition \ref{prop: volume estimate for ust in high} and Corollary \ref{eq: discrete graph version of volume condition},
we obtain the convergence of stochastic processes and local times on the uniform spanning trees in high-dimensional tori.

\begin{cor}
  The CRT $\mathcal{T}^{2W}$ belongs to $\check{\mathbb{F}}_{c}$ with probability $1$,
  and 
  $\cX_{\mathcal{T}_{n}} \xrightarrow{\mathrm{d}} \cX_{\mathcal{T}^{2W}}$
  as random elements of $\mbM_{L}$.
\end{cor}

\subsection{Uniform spanning trees in two and three dimensions} \label{subsec: UST in 2 and 3 dimensions}
Let $\mathcal{U}$ be the UST on $\mathbb{Z}^{3}$,
which is obtained as a local limit of the USTs on the finite boxes $[-n,n]^{3}\cap \mathbb{Z}^{3}$
(see \cite[Theorem 2.3]{Pemantle_91_Choosing} for the precise definition of this limit).
We denote the underlying probability measure by $\mathbf{P}$ and define $d_{\mathcal{U}},\, \mu_{\mathcal{U}}$ and $\rho_{\mathcal{U}}$
to be the graph metric on $\mathcal{U}$, the counting measure on $\mathcal{U}$ and the origin $0$ of $\mathbb{Z}^{3}$.
Let $\beta$ be the growth exponent defined as follows.
Let $M_{n}$ be the number of steps of the loop-erased random walk on $\mathbb{Z}^{3}$ until its first exit from a ball of radius $n$.
Then set
\begin{equation} \label{eq: growth exponent of Z2}
  \beta
  \coloneqq
  \lim_{n \to \infty}
  \frac{\log \mathbf{E}(M_{n})}{\log n}.
\end{equation}
The existence of this limit was proved in \cite{Shiraishi_18_Growth},
and this growth exponent determines the scaling of $d_{\mathcal{U}}$.
We set
\begin{equation}
  \mathcal{U}_{\delta}
  \coloneqq
  \left(
  \mathcal{U}, \delta^{\beta}d_{\mathcal{U}},
  \rho_{\mathcal{U}}, \delta^{3}\mu_{\mathcal{U}}
  \right).
\end{equation}
By \cite[Theorem 1.6]{Angel_Croydon_Hernandez-Torres_Shiraishi_21_Scaling},
we have  $\mathcal{U}_{\delta} \in \check{\mathbb{F}}$, $\mathbf{P}$-a.s.

\begin{thm}[{\cite[Theorem 1.1 and Proof of Theorem 1.9]{Angel_Croydon_Hernandez-Torres_Shiraishi_21_Scaling}}] \label{thm: scaling limit of ust in 3d}
  There exists a random element $\mathcal{T}$ of $\mathbb{F}$
  such that the random elements $\mathcal{U}_{2^{-n}}$ converge to $\mathcal{T}$ in distribution in the local Gromov-Hausdorff-vague topology.
  Moreover the sequence $(\mathcal{U}_{2^{-n}})_{n}$ satisfies
  Assumption \ref{1. assum: random spaces} \ref{1. assum item: random spaces, non-explosion}.
\end{thm}

We use the following ingredients to obtain a volume estimate.

\begin{lem} [{\cite[Proposition 4.1]{Angel_Croydon_Hernandez-Torres_Shiraishi_21_Scaling}}] \label{lem: comparison of balls}
  There exist positive constants $c_{1}$ and $C_{1}$ such that, for every $\tilde{\lambda} \geq 1$ and $\tilde{\delta} \in (0,1)$,
  \begin{equation}
    \mathbf{P}
    \left(
    B(\rho_{\mathcal{U}},\tilde{\lambda}^{-1}\tilde{\delta}^{-1})
    \subseteq
    B_{d_{\mathcal{U}}}(\rho_{\mathcal{U}}, \tilde{\delta}^{-\beta})
    \subseteq
    B(\rho_{\mathcal{U}},\tilde{\lambda} \tilde{\delta}^{-1})
    \right)
    \geq
    1-C_{1}\tilde{\lambda}^{-c_{1}},
  \end{equation}
  where $B(\rho_{\mathcal{U}},r)$ is the open ball in $\mathbb{Z}^{3}$ centered at $\rho_{\mathcal{U}}$ with radius $r$ with respect to the Euclidean metric.
\end{lem}

\begin{lem} [{\cite[Theorem 5.2]{Angel_Croydon_Hernandez-Torres_Shiraishi_21_Scaling}}] \label{lem: pointwise volume estimate of ust in 3d}
  There exist positive constants $c_{2},\, C_{2}$ and $b$
  such that, for every $R \geq 1$ and $\lambda \geq 1$,
  \begin{equation}
    \mathbf{P}
    \left(
    \inf_{x \in B(\rho_{\mathcal{U}},R^{1/\beta})}
    \mu_{\mathcal{U}}
    \left(
    B_{d_{\mathcal{U}}}(x, \lambda^{-b}R)
    \right)
    \leq
    \lambda^{-1}R^{3/\beta}
    \right)
    \leq
    C_{2} \exp(-c_{2}\lambda^{b}).
  \end{equation}
\end{lem}

\begin{prop} \label{prop: volume estimate for ust in 3d}
  For every $\varepsilon>0$ and $L >0$,
  there exist positive constants $c$ and $c'$ satisfying
  \begin{equation}
    \liminf_{\delta \to 0}
    \mathbf{P}
    \left(
    \inf_{x \in B_{d_{\mathcal{U}}}(\rho_{\mathcal{U}}, \delta^{-\beta}L)}
    \delta^{3}
    \mu_{\mathcal{U}}
    \left(
    B_{d_{\mathcal{U}}} (x, \delta^{-\beta} r )
    \right)
    >
    cr^{1/b}, \quad
    \forall r \in (0, c')
    \right)
    \geq
    1-\varepsilon,
  \end{equation}
  where $b$ is the constant in Lemma \ref{lem: pointwise volume estimate of ust in 3d}.
\end{prop}

\begin{proof}
  Let $C_{1}$ and $c_{1}$ be the constants of Lemma \ref{lem: comparison of balls}
  and $C_{2},\, c_{2}$ and $b$ be the constants of Lemma \ref{lem: pointwise volume estimate of ust in 3d}.
  We choose $a >1$ satisfying $C_{1}a^{-c_{1}}<\varepsilon/2$ and choose $\delta_{0} \in (0,1)$ satisfying
  \begin{equation}
    \delta^{-\beta}La^{\beta}>1,\
    \delta L^{-1/\beta}<1,\ \
    \forall \delta < \delta_{0}.
  \end{equation}
  Then,
  by setting $R=\delta^{-\beta}La^{\beta}$ in the inequality in Lemma \ref{lem: pointwise volume estimate of ust in 3d},
  we obtain that
  \begin{equation} \label{eq: ust in 3d estimate 1}
    \mathbf{P}
    \left(
    \inf_{x \in B(\rho_{\mathcal{U}},\delta^{-1}L^{1/\beta}a)}
    \mu_{\mathcal{U}}
    \left(
    B_{d_{\mathcal{U}}}
    (x, \lambda^{-b}\delta^{-\beta}La^{\beta})
    \right)
    \leq
    \lambda^{-1} \delta^{-3} L^{3/\beta} a^{3}
    \right)
    \leq
    C_{2} \exp(-c_{2}\lambda^{b})
  \end{equation}
  for every $\delta \in (0,\delta_{0})$ and $\lambda \geq 1$.
  By setting $\tilde{\delta}=\delta L^{-1/\beta}$ and $\tilde{\lambda}=a$ in the inequality in Lemma \ref{lem: comparison of balls},
  we also obtain that
  \begin{equation} \label{eq: ust in 3d estimate 2}
    \mathbf{P}
    \left(
    B_{d_{\mathcal{U}}}(\rho_{\mathcal{U}},\delta^{-\beta}L)
    \subseteq
    B(\rho_{\mathcal{U}}, a\delta^{-1}L^{1/\beta})
    \right)
    \geq
    1-\varepsilon / 2,
  \end{equation}
  for every $\delta \in (0,\delta_{0})$.
  Let $(\lambda_{k})$ be a sequence of positive numbers with $\lambda_{k} \geq 1$.
  From \eqref{eq: ust in 3d estimate 1} and \eqref{eq: ust in 3d estimate 2},
  it follows that
  \begin{align} \label{eq: ust in 3d estimate 3}
    \mathbf{P}
    \left(
    \inf_{x \in B_{d_{\mathcal{U}}}(\rho_{\mathcal{U}}, \delta^{-\beta}L)}
    \delta^{3}
    \mu_{\mathcal{U}}
    \left(
    B_{d_{\mathcal{U}}} (x, \delta^{-\beta} \lambda_{k}^{-b} L a^{\beta} )
    \right)
    \leq
    \lambda_{k}^{-1} L^{3/\beta} a^{3}\
    \mathrm{for\ some}\
    k
    \right)
    \leq
    \frac{\varepsilon}{2}
    +
    \sum_{k}C_{2}\exp(-c_{2}\lambda_{k}^{b})
  \end{align}
  for every $\delta \in (0,\delta_{0})$.
  We fix $A>1$ satisfying $\sum_{k=0}^{\infty}C_{2}\exp(-c_{2}2^{k}A^{b}) < \varepsilon /2$.
  Then by setting $\lambda_{k}=2^{k/b}A$ in \eqref{eq: ust in 3d estimate 3},
  we deduce that
  \begin{equation}
    \liminf_{\delta \to 0}
    \mathbf{P}
    \left(
    \inf_{x \in B_{d_{\mathcal{U}}}(\rho_{\mathcal{U}}, \delta^{-\beta}L)}
    \delta^{3}
    \mu_{\mathcal{U}}
    \left(
    B_{d_{\mathcal{U}}}
    (x, \delta^{-\beta} 2^{-k}A^{-b} L a^{\beta} )
    \right)
    >
    2^{-k/b}A^{-1} L^{3/\beta} a^{3}, \ \
    \forall k \in \mathbb{Z}_{+}
    \right)
    \geq
    1-\varepsilon.
  \end{equation}
  By a similar argument to that of the proof of Proposition \ref{prop: volume estimate for ust in high},
  we obtain the desired result.
\end{proof}

By Proposition \ref{prop: volume estimate for ust in 3d} and Corollary \ref{eq: discrete graph version of volume condition},
we obtain the convergence of stochastic processes and local times on the three-dimensional uniform spanning trees.
\begin{cor}
  The limiting space $\mathcal{T}$ belongs to $\check{\mathbb{F}}$ with probability $1$,
  and $\cX_{\mathcal{U}_{2^{-n}}} \xrightarrow{\mathrm{d}} \cX_{\mathcal{T}}$
  as random elements of $\mbM_{L}$.
\end{cor}

For the two-dimensional UST,
we use the same notation.
We now let $\mathcal{U}$ be the UST on $\mathbb{Z}^{2}$ built on a probability space equipped with a probability measure $\mathbf{P}$.
Then we define $d_{\mathcal{U}},\, \mu_{\mathcal{U}}$ and $\rho_{\mathcal{U}}$ to be the graph metric on $\mathcal{U}$,
the counting measure on $\mathcal{U}$ and the origin $0$ of $\mathbb{Z}^{2}$.
Let $\kappa$ be the growth exponent of the loop-erased random walk on $\mathbb{Z}^{2}$ defined in the same way as \eqref{eq: growth exponent of Z2}.
We set
\begin{equation}
  \mathcal{U}_{\delta}
  \coloneqq
  \left(
  \mathcal{U}, \delta^{\kappa}d_{\mathcal{U}},
  \rho_{\mathcal{U}}, \delta^{2}\mu_{\mathcal{U}}
  \right).
\end{equation}
Note that $\mathcal{U}_{\delta} \in \check{\mathbb{F}},\, \mathbf{P}$-a.s.
(The recurrence of the Dirichlet form on the two-dimensional UST follows from the recurrence of the simple random walk of $\mathbb{Z}^{2}$.)
We write $\mathbf{P}_{\delta}\coloneqq \mathbf{P}(\mathcal{U}_{\delta} \in \cdot)$.
The tightness of $(\mathbf{P}_{\delta})_{\delta \in (0,1)}$ as probability measures on $\mathbb{G}$ was established
in \cite{Barlow_Croydon_Kumagai_17_Subsequential},
where we recall from Section \ref{sec: the space M_L}
that $\mathbb{G}$ is the collection of rooted, measured boundedly-compact metric spaces 
equipped with the local Gromov-Hausdorff-vague topology.
This implies the existence of a convergent subsequence $(\mathbf{P}_{\delta_{n}})_{n}$ with $\delta_{n} \to 0$,
to which we apply our main results.

\begin{thm} [{\cite[Theorem 1.3 and 1.4]{Barlow_Croydon_Kumagai_17_Subsequential}}]
  If the random elements $\mathcal{U}_{\delta_{n}}$ converge to a random element $\mathcal{T}$ in distribution in $\mathbb{G}$,
  then $\mathcal{T} \in \mathbb{F},\, \tilde{\mathbf{P}}$-a.s.,
  where $\tilde{\mathbf{P}}$ denotes the underlying probability measure of $\mathcal{T}$.
\end{thm}

\begin{rem}
  In \cite[Remark 1.2]{Holden_Sun_18_SLE},
  the full convergence $\mathcal{U}_{\delta} \to \mathcal{T}$ is suggested.
\end{rem}

\begin{lem}
  The sequence $(\mathcal{U}_{\delta_{n}})_{n}$ satisfies
  Assumption \ref{1. assum: random spaces} \ref{1. assum item: random spaces, non-explosion}.
\end{lem}

\begin{proof}
  We write $R_{n} \coloneqq  \delta_{n}^{\kappa} d_{\mathcal{U}}$ and $R \coloneqq  d_{\mathcal{U}}$.
  By \cite[Proposition 3.6]{Barlow_Masson_11_Spectral},
  we have that
  \begin{equation}
    \mathbf{P}
    \left(
    R (\rho_{\mathcal{U}}, B_{R}(\rho_{\mathcal{U}}, r)^{c})
    <
    \lambda^{-1} r
    \right)
    \leq
    C \exp (-c \lambda^{2/11}),
    \quad
    \forall \lambda,\, r \geq 1
  \end{equation}
  for some positive constants $C,\, c >0$.
  Setting $r \coloneqq \delta^{-\kappa} L$ in the above inequality yields that
  \begin{equation}
    \mathbf{P}
    \left(
    R_{n} (\rho_{\mathcal{U}}, B_{R_{n}}(\rho_{\mathcal{U}}, L)^{c})
    <
    \lambda^{-1} L
    \right)
    \leq
    C \exp (-c \lambda^{2/11}).
  \end{equation}
  Now the result is immediate.
\end{proof}

We use the following ingredients to obtain a volume estimate.

\begin{lem} [{\cite[Theorem 2.1]{Barlow_Croydon_Kumagai_17_Subsequential}}] \label{lem: comparison of balls for 2d}
  There exist positive constants $c_{1}$ and $c_{2}$ such that,
  for every $r \geq 1$ and $\lambda \geq 1$,
  \begin{equation}
    \mathbf{P}
    \left(
    D_{d_{\mathcal{U}}}(\rho_{\mathcal{U}}, \lambda^{-1} r^{\kappa})
    \nsubseteq
    D(\rho_{\mathcal{U}},r)
    \right)
    \leq
    c_{1} \exp(-c_{2} \lambda^{2/3}),
  \end{equation}
  where $D(\rho_{\mathcal{U}},r)$ is the closed ball in $\mathbb{Z}^{2}$ centered at $\rho_{\mathcal{U}}$ with radius $r$
  with respect to the Euclidean metric.
\end{lem}

\begin{lem} [{\cite[Proposition 2.10]{Barlow_Croydon_Kumagai_17_Subsequential}}]  \label{lem: pointwise volume estimate of ust in 2d}
  There exist positive constants $c_{3}$ and $c_{4}$ such that,
  \begin{equation}
    \mathbf{P}
    \left(
    \inf_{x \in D(\rho_{\mathcal{U}}, n)}
    \mu_{\mathcal{U}} (D_{d_{\mathcal{U}}}(x, r))
    <
    \lambda^{-1} r^{2/\kappa}\
    \text{for some}\
    r \in [e^{-\lambda^{1/40}}n^{\kappa},  n^{\kappa}]
    \right)
    \leq
    c_{3} \exp(-c_{4} \lambda^{1/80})
  \end{equation}
  for all $n \geq e^{\lambda^{1/16}}$.
\end{lem}

\begin{prop} \label{prop: volume estimate for ust in 2d}
  For every $\varepsilon>0$ and $L >0$,
  there exist positive constants $C_{1},\, C_{2}$ and $C_{3}$ satisfying
  \begin{equation}
    \liminf_{\delta \to 0}
    \mathbf{P}
    \left(
    \inf_{x \in B_{d_{\mathcal{U}}}(\rho_{\mathcal{U}}, \delta^{-\kappa}L)}
    \delta^{2}
    \mu_{\mathcal{U}}
    \left(
    B_{d_{\mathcal{U}}} (x, \delta^{-\kappa} r )
    \right)
    >
    C_{1}r^{1+(2/\kappa)}, \quad
    \forall r \in [C_{2} (\log \delta^{-1})^{-16},C_{3}]
    \right)
    \geq
    1-\varepsilon.
  \end{equation}
\end{prop}

\begin{proof}
  Choose $a >1$ satisfying $c_{1} \exp(-c_{2}a^{2/3}) <\varepsilon/2$,
  where $c_{1}$ and $c_{2}$ are the constants of Lemma \ref{lem: comparison of balls for 2d},
  and then choose $\delta_{0} \in (0,1)$ so that $(aL)^{1/\kappa}\delta^{-1} >1$ for all $\delta < \delta_{0}$.
  If we set $\tilde{c}\coloneqq (aL)^{1/\kappa}$, then,
  by Lemma \ref{lem: comparison of balls for 2d}, we have that
  \begin{equation}
    \mathbf{P}
    \left(
    D_{d_{\mathcal{U}}}(\rho_{\mathcal{U}}, \delta^{-\kappa}L)
    \nsubseteq
    D(\rho_{\mathcal{U}}, \tilde{c} \delta^{-1})
    \right)
    <\varepsilon/2,
    \quad
    \forall \delta < \delta_{0}.
  \end{equation}
  Noting that $e^{-\lambda^{1/40}} \leq \lambda^{-1} \leq 1$ for all $\lambda \geq 1$,
  by Lemma \ref{lem: pointwise volume estimate of ust in 2d},
  there exist positive constants $c_{3}$ and $c_{4}$ such that
  \begin{equation} \label{eq: inequality for volume estimate of 2d ust}
    \mathbf{P}
    \left(
    \inf_{x \in D(\rho_{\mathcal{U}}, n)}
    \mu_{\mathcal{U}} (D_{d_{\mathcal{U}}}(x, \lambda^{-1} n^{\kappa}))
    <
    \lambda^{-1+(2/\kappa)} n^{2}
    \right)
    \leq
    c_{3} \exp(-c_{4} \lambda^{1/80}),
    \quad
    \forall n \geq e^{\lambda^{1/16}}.
  \end{equation}
  We choose $A >1$ so that $\sum_{l=1}^{\infty} c_{3} \exp(-c_{4}(Al)^{1/80}) < \varepsilon/2$
  and set $\lambda_{l}\coloneqq Al$ for $l \in \mathbb{N}$.
  Since $\lceil \delta^{-1} \tilde{c} \rceil \geq e^{\lambda_{l}^{1/16}}$ for $1 \leq l \leq \lfloor A^{-1}(\log \delta^{-1})^{16} \rfloor$,
  setting $n\coloneqq  \lceil \delta^{-1} \tilde{c} \rceil$ and $\lambda \coloneqq  \lambda_{l}$ in \eqref{eq: inequality for volume estimate of 2d ust}
  yields that
  \begin{equation}
    \mathbf{P}
    \left(
    \inf_{x \in D(\rho_{\mathcal{U}},\lceil \delta^{-1} \tilde{c} \rceil )}
    \mu_{\mathcal{U}} (D_{d_{\mathcal{U}}}(x, A^{-1} l^{-1} \lceil \delta^{-1} \tilde{c} \rceil^{\kappa}))
    <
    (Al)^{-1-(\kappa/2)} \lceil \delta^{-1} \tilde{c} \rceil^{2/\kappa}
    \right)
    \leq
    c_{3}\exp(-c_{4}(Al)^{1/80})
  \end{equation}
  for all $1 \leq l \leq \lfloor A^{-1}(\log \delta^{-1})^{16} \rfloor$.
  Now a similar argument to the proof of Proposition \ref{prop: volume estimate for ust in 3d} works and we deduce the desired result.
\end{proof}

By Proposition \ref{prop: volume estimate for ust in 2d} and Corollary \ref{eq: discrete graph version of volume condition},
we obtain the convergence of stochastic processes and local times on the two-dimensional uniform spanning trees.

\begin{cor}
  The limiting space $\mathcal{T}$ belongs to $\check{\mathbb{F}}$ with probability $1$,
  and $\cX_{\mathcal{U}_{\delta_{n}}} \xrightarrow{\mathrm{d}} \cX_{\mathcal{T}}$
  as random elements of $\mbM_{L}$.
\end{cor}

\subsection{A random recursive Sierpi\'{n}ski gasket} \label{subsec: a random recursive SG}

A random recursive Sierpi\'{n}ski gasket $G$ is a random fractal introduced in \cite{Hambly_97_Brownian},
and it is obtained as a limit of graphs $G_{n}$ which are generated randomly based on the Sierpi\'{n}ski gasket graph.
In this model, the resistance metric $R_{n}$ on $G_{n}$ is compatible with the resistance metric $R$ on $G$, i.e.\ $R|_{G_{n} \times G_{n}}=R_{n}$,
and this yields the convergence of local times without volume estimates of balls in $G_{n}$.

We begin with some operations on electrical networks.
Given an finite electrical network $G=(V, E, (c(e))_{e \in E})$,
we write $x \sim y$ if $\{ x, y\} \in E$, and $c_{G}(x,y)=c(\{x, y\})$.
We set $c_{G}(x)\coloneqq \sum_{y : x \sim y} c_{G}(x,y)$,
and then the probability measure $\mu_{G}$ on $V$ is defined
by setting $\mu_{G}(\{ x \}) \coloneqq  c_{G}(x) / \sum_{y} c_{G}(y)$.
We denote the effective resistance metric on $V$ by $R_{G}$ and the electrical energy by $\mathcal{E}_{G}$, i.e.\
\begin{equation}
  \mathcal{E}_{G}(f,\, g)
  =
  \sum_{ \{x, y \} \in E}
  c_{G}(x,y)(f(x)-f(y))(g(x)-g(y))
\end{equation}
for $f,\, g: V \to \mbR$.
By abuse of notation,
we often use the symbol $G$ of the electrical network as the vertex set $V$.
For example,
we write $x \in G$ instead of $x \in V$.
Assume that $G$ is a subset of $\mbR^{2}$.
Let $\psi : \mbR^{2} \to \mbR^{2}$ be an injective map and we associate a positive constanct $\gamma$ with $\psi$.
Then we define the electrical network $\psi(G)\coloneqq (\psi(V), \psi(E), (\gamma \, c(\psi^{-1}(e)))_{e \in \psi(E)})$.
We call $\gamma$ the conductance scaling factor of $\psi$.
The following lemma is elementary and we omit the proof.

\begin{lem} \label{lem: simple scaling property of psi}
  For any $x, y \in G$,
  it holds that $R_{\psi(G)}(x, y)=\gamma^{-1} R_{G}(\psi^{-1}(x),\psi^{-1}(y))$.
\end{lem}

Given another map $\psi'$ with a conductance scaling factor $\gamma'$,
we associate the conductance scaling factor $\gamma\, \gamma'$ with $\psi \circ \psi'$.
For two electrical networks $G_{i}=(V_{i}, E_{i}, (c_{i}(e))_{e \in E_{i}}),\, i=1,2$,
if $E_{1} \cap E_{2} = \emptyset$,
then we define the electrical network $G_{1} \cup G_{2}=(V, E, (c(e))_{e \in E})$
by setting $V \coloneqq  V_{1} \cup V_{2}, \, E \coloneqq  E_{1} \cup E_{2}$ and
\begin{equation}
  c(e)\coloneqq
  \begin{cases}
    c_{1}(e), & e \in E_{1}, \\
    c_{2}(e), & e \in E_{2}.
  \end{cases}
\end{equation}

We define contraction maps and conductances for a random recursive Sierpi\'{n}ski gasket.
Let $K_{0}$ be the unit equilateral triangle in $\mbR^{2}$,
and let $(V_{0}, E_{0})$ denote the complete graph with the vertices of $K_{0}$,
where $V_{0}=\{ x_{1}, x_{2}, x_{3} \} \subseteq \mbR^{2}$ is the vertex set and $E_{0}$ is the edge set.
We define $G_{0}$ to be the electrical network with the vertex set $V_{0}$ and conductance $1$ on each edge of $E_{0}$.
Fix a natural number $L \geq 2$.
For $\nu \in \{2,3,\ldots, L\}$,
let $(a_{i}, b_{i}, c_{i})_{i=1}^{\nu (\nu+1)/2}$ be the solutions of $a+b+c=\nu-1, a,b,c \in \mathbb{Z}_{+}$.
The family $(\psi_{i}^{\nu})_{i=1}^{\nu (\nu+1)/2}$ of contraction maps of type $\nu$ is defined
by setting $\psi_{i}^{\nu}(x)\coloneqq (x+a_{i}x_{1}+b_{i}x_{2}+c_{i}x_{3})/\nu$ for $x \in \mbR^{2}$.
We associate a positive constant $\gamma_{\nu}$,
which is specified later,
with each $\psi_{i}^{\nu}$.
We consider the electrical network $G^{\nu}\coloneqq \bigcup_{i=1}^{\nu(\nu+1)/2} \psi_{i}^{\nu}(G_{0})$
and define the conductance $\gamma_{\nu}$
so that
\begin{equation} \label{eq: def of conductance of type nu}
  \mathcal{E}_{G_{0}}(f,f)
  =
  \inf \left\{
  \mathcal{E}_{G^{\nu}}(g,g) :
  g : G^{\nu} \to \mbR\
  \text{such that}\
  g|_{G_{0}}=f
  \right\}.
\end{equation}
This is equivalent to defining $\gamma_{\nu}$
so that $G^{\nu}$ is compatible with $G_{0}$,
i.e.\ $R_{G^{\nu}}|_{G_{0} \times G_{0}} = R_{G_{0}}$ (see \cite[Corollary 2.1.13]{Kigami_01_Analysis}).
Note that $\gamma_{\nu} >1$.
We write $\gamma_{\mathrm{min}} \coloneqq  \min_{\nu} \gamma_{\nu}$ 
and $\gamma_{\mathrm{max}} \coloneqq  \max_{\nu} \gamma_{\nu}$.

A random recursive Sierpi\'{n}ski gasket is constructed
by associating a sequence of electrical networks with a randomly generated sequence of finite plane trees
(recall the definition of plane trees from Section \ref{subsubsec: introduction of trees}).
We explain how to generate a Sierpi\'{n}ski-gasket-type graph from a sequence of finite plane trees, putting aside randomness for a while.
Let $\mathbf{T}^{*}$ be the set of all plane trees $T$
such that $k_{u} \in \{\nu(\nu+1)/2 : \nu=2,3,\ldots,L \} \cup \{ 0 \}$ for each individual $u \in T$,
where we recall that $k_{u}\coloneqq k_{u}(T)$ is the number of the children of $u$,
and we write $\nu=\nu^{T}(u) \geq 0$ to be the number
such that $k_{u}=\nu(u) (\nu(u)+1) /2$.
Let $T$ be a plane tree of $\mathbf{T}^{*}$.
For an individual $u = (u_{0}, u_{1},\ldots, u_{n})$ of $T$ with $u_{0} =0$,
we write $[u]_{k}\coloneqq (u_{0},\ldots,u_{k}), |u|=n$ and
\begin{align}
  \psi_{u}^{T}
   & \coloneqq
  \psi_{u_{1}}^{\nu([u]_{0})} \circ \psi_{u_{2}}^{\nu([u]_{1})}
  \circ \dots \circ \psi_{u_{n}}^{\nu([u]_{n-1})}, \\
  \gamma_{u}^{T}
   & \coloneqq
  \gamma_{\nu([u]_{0})} \gamma_{\nu([u]_{1})} \cdots \gamma_{\nu([u]_{n-1})}
\end{align}
(if $u$ is the root $0$,
then we set $\psi_{0}^{T}$ to be the identity map on $\mbR^{2}$ and $\gamma_{0}^{T}\coloneqq 1$).
Note that $\gamma_{u}^{T}$ is the conductance scaling factor of $\psi_{u}^{T}$.
Now we assume that $T$ is finite.
Let $T^{p}$ be the totality of individuals of $T$ who have no children.
Then the electrical network $G(T)=(V(T), E(T))$ associated with the finite plane tree $T$ is defined by setting
\begin{equation}
  G(T) \coloneqq  \bigcup_{u \in T^{p}} \psi_{u}^{T}(G_{0}).
\end{equation}
Note that $G_{0}=G(\{ 0 \})$ and the electrical energy $\mathcal{E}_{G(T)}$ is given by
\begin{equation} \label{eq: the energy form on a finite graph of SG}
  \mathcal{E}_{G(T)}(f,\, g)
  =
  \sum_{u \in T^{p}}
  \gamma_{u}^{T}
  \sum_{ \{x, y\} \in \psi_{u}^{T}(E_{0})}
  (f(x)-f(y))(g(x)-g(y)).
\end{equation}

\begin{prop} \label{prop: energy forms are compatible}
  Let $T_{1},\, T_{2}$ be finite plane trees of $\mathbf{T}^{*}$.
  If $T_{1} \subseteq T_{2}$,
  then
  \begin{equation}
    \mathcal{E}_{G(T_{1})}(f,f)
    =
    \inf \left\{
    \mathcal{E}_{G(T_{2})}(g,g):
    g:G(T_{2}) \to \mbR\
    \text{ such that }\
    g|_{G(T_{1})}=f
    \right\},
  \end{equation}
  which is equivalent to that $G(T_{2})$ is compatible with $G(T_{1})$.
\end{prop}

\begin{proof}
  If $T_{1}=T_{2}$,
  then the result is trivial.
  Otherwise choose an individual $u=(u_{0},u_{1},\ldots,u_{n}) \in T_{2}^{p} \setminus T_{1}^{p}$.
  Set $u'=(u_{0}, \ldots, u_{n-1})$, $U\coloneqq  \{ u'i \in T_{2} : i=1, \ldots, k_{u'} \}$ and $T\coloneqq  T \setminus U$.
  Using that $\psi_{u'i}^{T_{2}}=\psi_{u'}^{T} \circ \psi_{i}^{\nu(u')}$ and $\gamma_{u'i}^{T_{2}}=\gamma_{u'}^{T}\, \gamma_{\nu(u')}$,
  we deduce that,
  for any $g : G(T_{2}) \to \mbR$,
  \begin{align}
    \sum_{v \in U}
    \sum_{ \{x, y\} \in \psi_{v}^{T_{2}}(E_{0})}
    \gamma_{v}^{T_{2}}
    (g(x)-g(y))^{2}
     & =
    \gamma_{u'}^{T} \,
    \sum_{i=1}^{k_{u'}}
    \sum_{\{x', y'\} \in \psi_{i}^{\nu(u')}(E_{0})}
    \gamma_{\nu(u')}
    (g(\psi_{u'}^{T}(x'))-g(\psi_{u'}^{T}(y')))^{2} \notag \\
     & =
    \gamma_{u'}^{T}
    \,
    \mathcal{E}_{G^{\nu(u')}}
    (g \circ \psi_{u'}^{T}, g \circ \psi_{u'}^{T}).
  \end{align}
  This yields that
  \begin{equation} \label{eq: one generation energy}
    \mathcal{E}_{G(T_{2})}(g,g)
    =
    \sum_{v \in T^{p}}
    \gamma_{v}^{T}
    \sum_{ \{ x, y \} \in \psi_{v}^{T}(E_{0})}
    (g(x)-g(y))^{2}
    +
    \gamma_{u'}^{T}
    \,
    \mathcal{E}_{G^{\nu(u')}}
    (g \circ \psi_{u'}^{T}, g \circ \psi_{u'}^{T}).
  \end{equation}
  By \eqref{eq: def of conductance of type nu}, \eqref{eq: the energy form on a finite graph of SG} and \eqref{eq: one generation energy},
  we deduce that for any $f : G(T) \to \mbR$,
  \begin{equation}
    \mathcal{E}_{G(T)}(f,f)
    =
    \inf\{
    \mathcal{E}_{G(T_{2})}(g,g):
    g : G(T_{2}) \to \mbR\
    \text{ such that}\
    g|_{G(T)}=f
    \},
  \end{equation}
  which is equivalent to that $R_{G(T_{2})}|_{G(T) \times G(T)} = R_{G(T)}$ (see \cite[Corollary 2.1.13]{Kigami_01_Analysis}).
  Inductively, the desired result is proved.
\end{proof}

For a plane tree $T$,
we define the shift of $T$ at $u \in T$
by setting $\theta_{u}T\coloneqq \{ v \in \mathbf{I} : uv \in T \}$.

\begin{lem} \label{lem: decomposition of energy forms}
  Let $T_{1},\, T_{2}$ be finite plane trees of $\mathbf{T}^{*}$.
  If $T_{1} \subseteq T_{2}$,
  then it holds that
  \begin{equation}
    \mathcal{E}_{G(T_{2})}(f,\, g)
    =
    \sum_{u \in T_{1}^{p}}
    \gamma_{u}^{T_{1}} \,
    \mathcal{E}_{G(\theta_{u} T_{2})} (f \circ \psi_{u}^{T_{1}}, g \circ \psi_{u}^{T_{1}}),
  \end{equation}
  for any $f,\, g : G(T_{2}) \to \mbR$.
\end{lem}

\begin{proof}
  Observe that every individual $u \in T_{2}^{p}$ is uniquely written as $u=vv'$ for some $v \in T_{1}^{p}$ and $v' \in (\theta_{u}T_{2})^{p}$,
  and it holds that $\gamma_{u}^{T_{2}}=\gamma_{v}^{T_{1}}\,  \gamma_{v'}^{\theta_{u}T_{2}}$
  and $\psi_{u}^{T_{2}}=\psi_{v}^{T_{1}} \circ \psi_{v'}^{\theta_{u}T_{2}}$.
  Using these and \eqref{eq: the energy form on a finite graph of SG},
  one can verify the desired identity.
\end{proof}

\begin{lem} \label{lem: uniform upper bound of resistance}
  There exists a constant $c_{1} >0$ such that,
  for any finite tree $T \in \mathbf{T}^{*}$,
  \begin{equation}
    R_{G(T)}(x,y) \leq c_{1},
    \quad
    \forall x,y \in G(T).
  \end{equation}
\end{lem}

\begin{proof}
  Set $c \coloneqq  \max \{ R_{G^{\nu}}(x,y) : x, y \in G^{\nu}, \nu = 2, \ldots , L \}$.
  Fix $x \in G(T)$.
  By definition,
  there exist $u \in T^{p}$ with $n\coloneqq |u|$ and $x_{*} \in G_{0}$ such that $x = \psi_{u}^{T}(x_{*})$.
  Define a sequence $(x_{l})_{l=0}^{n}$ of vertices in $G(T)$
  by setting $x_{l} \coloneqq  \psi_{[u]_{l}}^{T}(x_{*})$.
  Noting that $x_{0}=x_{*}$ and $x_{n}=x$,
  we have that
  \begin{equation} \label{eq: chaining inequality for rough resistance metric estimate}
    R_{G(T)}(x,x_{*})
    \leq
    \sum_{l=0}^{n-1} R_{G(T)}(x_{l}, x_{l+1}).
  \end{equation}
  We define a subtree $T_{l}$ of $T$ by setting $T_{l} \coloneqq  \{ v \in T : |v| \leq l \}$.
  By Proposition \ref{prop: energy forms are compatible}, we have that
  \begin{equation} \label{eq: changing compatible network}
    R_{G(T)}(x_{l}, x_{l+1})
    =
    R_{G(T_{l+1})}(x_{l}, x_{l+1}).
  \end{equation}
  Since $\psi_{[u]_{l}}^{T}(G^{\nu([u]_{l})})$ is a subgraph of $G(T_{l+1})$ containing $x_{l}$ and $x_{l+1}$,
  by Rayleigh's monotonicity law (see \cite[Theorem 9.12]{Levin_Peres_17_Markov}, for example) and Lemma \ref{lem: simple scaling property of psi},
  we obtain that
  \begin{equation} \label{eq: one level resistance metric estimate}
    R_{G(T_{l+1})}(x_{l},x_{l+1})
    \leq
    R_{\psi_{[u]_{l}}^{T}(G^{\nu([u]_{l})})}(x_{l}, x_{l+1})
    \leq
    (\gamma_{[u]_{l}}^{T})^{-1} c.
  \end{equation}
  Using \eqref{eq: chaining inequality for rough resistance metric estimate}, \eqref{eq: changing compatible network},
  \eqref{eq: one level resistance metric estimate} and that $\gamma_{[u]_{l}}^{T} \geq \gamma_{\mathrm{min}}^{l}$,
  we deduce that
  \begin{equation}
    R_{G(T)}(x,x_{*})
    \leq
    \sum_{l=0}^{n-1}
    \gamma_{\mathrm{min}}^{-l}c
    \leq
    c/(1-\gamma_{\mathrm{min}}^{-1}).
  \end{equation}
  For $y \in G(T)$,
  we define $y_{*} \in G_{0}$ in the same way.
  Using the triangle inequality and that $R_{G(T)}(x_{*}, y_{*})=R_{G_{0}}(x_{*},y_{*}) \leq 2/3$,
  we obtain the desired result.
\end{proof}

A random recursive gasket is defined via a sequence of finite trees generated by a general branching process.
Let $(\lambda_{u}, \xi_{u})_{u \in \mathbf{I}}$ be the general branching process,
built on the probability space $(\Omega, \mathcal{A}, \mathbf{P})$,
such that the joint distribution of the lifetime $\lambda_{u}$
and the reproduction process $\xi_{u}=(\xi_{u}(t))_{t \geq 0}$ of each $u \in \mathbf{I}$
are given by
\begin{equation}
  (\lambda_{u}, \xi_{u}( \cdot ))
  =
  \left(
  \log \gamma_{\nu},
  \frac{\nu(\nu+1)}{2} 1_{ \{ \log \gamma_{\nu} \}}(\cdot)
  \right) \quad
  \text{
    with probability
  }
  p_{\nu},\quad
  \nu \in \{2,3,\ldots,L\}.
\end{equation}
Let $T_{t}$ be the plane tree generated by the branching process up to time $t$,
i.e.\ $T_{t}\coloneqq  \{ u \in \mathbf{I}: \sigma_{u} \leq t \}$,
where $\sigma_{u}$ denotes the birth time of $u$ and $\sigma_{0}\coloneqq 0$.
We set $T \coloneqq  \bigcup_{t \geq 0} T_{t}$.
Noting that $\gamma_{u}^{T}=e^{\sigma_{u}}$ and $T_{t}^{p}$ is the set of individuals alive at $t$,
one can check that
\begin{equation} \label{eq: estimate of conductance on an edge}
  \gamma_{\mathrm{max}}^{-1}\, e^{t}
  \leq
  \gamma_{u}^{T}
  \leq
  e^{t},\quad
  \forall u \in T_{t}^{p}.
\end{equation}
We set $(G_{t}, R_{t}, \mu_{t})\coloneqq (G(T_{t}), R_{G(T_{t})}, \mu_{G(T_{t})})$.
By Proposition \ref{prop: energy forms are compatible},
the metric $R^{*}$ is defined on $G^{*} \coloneqq  \bigcup_{t \geq 0} G_{t}$ such that $R^{*}|_{G_{t} \times G_{t}} = R_{t}$.
We let $(G, R)$ be the completion of $(G^{*}, R^{*})$.
Note that there exists a resistance form $(\mathcal{E}_{G}, \mathcal{F}_{G})$
such that $R$ is the associated resistance metric on $G$ (see \cite[Theorem 3.13]{Kigami_12_Resistance}).
The resistance metric space $(G, R)$ is called a random recursive Sierpi\'{n}ski gasket.
In what follows, we regard $(G_{t}, R_{t})$ and $(G^{*}, R^{*})$ as subspaces of $(G,R)$,
and $\mu_{t}$ as a probability measure on $(G, R)$.
We can capture the metric space $(G, R)$ as a shape in the Euclidean space $\mbR^{2}$ from the following result.

\begin{prop} \label{prop: metric-comparison of rrg}
  There exist deterministic constants $c_{2},\, c_{3}>0$
  such that $\mathbf{P}$-a.s.\ it holds that
  \begin{equation}
    c_{2} d_{E}(x,y)^{\log_{2} \gamma_{\mathrm{max}}}
    \leq
    R(x,y)
    \leq
    c_{3} d_{E}(x,y)^{\log_{L} \gamma_{\mathrm{min}}},\quad
    \forall x,y \in G^{*},
  \end{equation}
  where $d_{E}$ denotes the Euclidean metric on $\mbR^{2}$.
\end{prop}

\begin{proof}
  Fix $x,y \in G^{*}$.
  Define
  \begin{equation}
    n\coloneqq
    \max\{
    m \in \mathbb{Z}_{+} :
    \exists u,v \in T_{m}^{p}\
    \text{such that}\
    x \in \psi_{u}^{T}(K_{0}),\
    y \in \psi_{v}^{T}(K_{0}),\
    \psi_{u}^{T}(G_{0}) \cap \psi_{v}^{T}(G_{0}) \neq \emptyset
    \}.
  \end{equation}
  Let  $u,v \in T_{n}^{p}$ be the corresponding individuals such that
  \begin{gather}
    x \in \psi_{u}^{T}(K_{0}),\quad
    y \in \psi_{v}^{T}(K_{0}),\quad
    \psi_{u}^{T}(G_{0}) \cap \psi_{v}^{T}(G_{0}) \neq \emptyset.
  \end{gather}
  Choose $u' \in T_{n+1}^{p}$ such that $x \in \psi_{u'}(K_{0})$.
  Let $\Lambda$ be the set of individuals $w$ of $T_{n+1}^{p}$
  such that the corresponding triangle $\psi_{w}^{T}(K_{0})$ is adjacent to the triangle $\psi_{u'}^{T}(K_{0})$, i.e.\
  \begin{equation}
    \Lambda\coloneqq
    \{ w \in T_{n+1}^{p} \setminus \{u'\} :
    \psi_{w}^{T}(K_{0}) \cap \psi_{u'}^{T}(K_{0}) \neq \emptyset
    \}
  \end{equation}
  Note that,
  by the definition of $n$ and $u'$,
  we have that $y \notin \psi_{w}^{T}(K_{0})$ for all $w \in \Lambda \cup \{ u' \}$.
  We choose $l > n+1$ so that $x,y \in G(T_{l})$.
  Noting that
  \begin{equation}
    G(T_{l})
    =
    \bigcup_{w \in T_{n+1}^{p}}
    \psi_{w}^{T}(G(\theta_{w}T_{l})),
  \end{equation}
  we define a function $f : G(T_{l}) \to \mbR$ with $f(x)=1,f(y)=0$
  by putting the values of $f$ on each graph $\psi_{w}^{T}(G(\theta_{w}T_{l}))$
  so that it is well-defined at the boundary $\psi_{w}^{T}(G_{0})$.
  We set $f|_{\psi_{u'}^{T}(G(\theta_{u'}T_{l}))} \equiv 1$
  and $f|_{\psi_{w}^{T}(G(\theta_{w}T_{l}))} \equiv 0$ for $w \in T_{n+1}^{p} \setminus (\Lambda \cup \{ u' \})$.
  For $w \in \Lambda$,
  we define $f$ on $\psi_{w}^{T}(G(\theta_{w}T_{l}))$
  so that $f$ is harmonic
  on the electrical network $\psi_{w}^{T}(G(\theta_{w}T_{l}))$ with $f \equiv 1$ on $\psi_{w}^{T}(G_{0}) \cap \psi_{u'}^{T}(G_{0})$
  and $f \equiv 0$ on $\psi_{w}^{T}(G_{0}) \setminus \psi_{u'}^{T}(G_{0})$,
  which implies that $\mathcal{E}_{G(\theta_{w}T_{l})}(f \circ \psi_{w}, f \circ \psi_{w})=2$.
  From Lemma \ref{lem: decomposition of energy forms}, \eqref{eq: estimate of conductance on an edge}, and that $|\Lambda| \leq 6$,
  it follows that
  \begin{equation} \label{eq: lower bound of resistance distance of x y}
    R(x,y)
    =
    R_{G(T_{l})}(x,y)
    \geq
    \left(
    \sum_{w \in \Lambda }
    2 \gamma_{w}(T)
    \right)^{-1}
    \geq
    \frac{1}{12}\, e^{-(n+1)}.
  \end{equation}
  Choose $z \in \psi_{u}^{T}(G_{0}) \cap \psi_{v}^{T}(G_{0})$.
  Since $x,z \in \psi_{u}^{T}(G(\theta_{u}T_{l}))$ and $\psi_{u}^{T}(G(\theta_{u}T_{l}))$ is a subgraph of $G(T_{l})$,
  by Rayleigh's monotonicity law,
  Lemma \ref{lem: simple scaling property of psi}, Lemma \ref{lem: uniform upper bound of resistance} and \eqref{eq: estimate of conductance on an edge},
  we obtain that
  \begin{equation} \label{eq: lower bound of resistance metric of a triangle}
    R(x,z)
    =
    R_{G(T_{l})}(x,z)
    \leq
    R_{\psi_{u}^{T}(G(\theta_{u}T_{l}))}(x,z)
    \leq
    (\gamma_{u}^{T})^{-1}
    c_{1}
    \leq
    c_{1}
    \gamma_{\mathrm{max}}
    e^{-n},
  \end{equation}
  where $c_{1}$ is the constant of Lemma \ref{lem: uniform upper bound of resistance}.
  This yields that
  \begin{equation} \label{eq: upper bound of resistance distance of x y}
    R(x,y) \leq 2c_{1}\gamma_{\mathrm{max}}e^{-n}.
  \end{equation}
  Next, we estimate the Euclidean distance $d_{E}(x,y)$.
  It is elementary to check that,
  for any subset $A \subseteq \mbR^{2}$,
  \begin{equation} \label{eq: euclidean metric scaling}
    \mathrm{diam}_{d_{E}} (\psi_{i}^{\nu}(A))
    =
    \nu^{-1}
    \mathrm{diam}_{d_{E}} (A),
  \end{equation}
  where $\mathrm{diam}_{d_{E}}( \cdot )$ denotes the diameter of a subset with respect to $d_{E}$.
  Fix $w \in \Lambda$ with $m\coloneqq |w|$. From \eqref{eq: euclidean metric scaling}, it follows that
  \begin{equation} \label{eq: diam of triangle in euclidean}
    \mathrm{diam}_{d_{E}} (\psi_{w}^{T}(K_{0}))
    =(\nu([w]_{0}) \cdots \nu([w]_{m-1}))^{-1}
    \geq
    L^{-m}.
  \end{equation}
  By \eqref{eq: estimate of conductance on an edge},
  we have that
  \begin{equation}
    e^{-(n+1)}
    \geq
    \gamma_{w}^{T}
    =
    \gamma_{\nu([w]_{0})} \cdots \gamma_{\nu([w]_{m-1})}
    \geq
    \gamma_{\mathrm{min}}^{m},
  \end{equation}
  This yields that $m \leq (n+1)/\log \gamma_{\mathrm{min}}$ and, by \eqref{eq: diam of triangle in euclidean},
  we obtain that
  \begin{equation}
    \mathrm{diam}_{d_{E}} (\psi_{w}^{T}(K_{0}))
    \geq
    \exp (-(n+1) \log_{\gamma_{\mathrm{min}}} L) .
  \end{equation}
  Therefore we deduce that
  \begin{equation} \label{eq: lower bound of Euclidean distance of x y}
    d_{E}(x,y) \geq \exp (-(n+1) \log_{\gamma_{\mathrm{min}}} L).
  \end{equation}
  Since $x, z \in \psi_{u}^{T}(K_{0})$,
  a similar argument shows that
  \begin{equation}
    d_{E}(x,z)
    \leq
    2 \exp(-n \log_{\gamma_{\mathrm{max}}}2).
  \end{equation}
  The same inequality holds for $d_{E}(y, z)$,
  and thus we obtain that
  \begin{equation} \label{eq: upper bound of Euclidean distance of x y}
    d_{E}(x,y) \leq 4 \exp(-n \log_{\gamma_{\mathrm{max}}}2).
  \end{equation}
  By \eqref{eq: lower bound of resistance distance of x y}, \eqref{eq: upper bound of resistance distance of x y},
  \eqref{eq: lower bound of Euclidean distance of x y} and \eqref{eq: upper bound of Euclidean distance of x y},
  we deduce the desired result.
\end{proof}

Define the random subset $K \subseteq \mbR^{2}$ by setting
\begin{equation}
  K= \bigcap_{n \geq 0} \bigcup_{u \in T_{n}^{p}} \psi_{u}^{T}(K_{0}).
\end{equation}
This random set $K$ is also called a random recursive gasket as we have the following result.

\begin{cor}
  The inclusion map $\iota : (G^{*}, R^{*}) \to (K, d_{E})$ is naturally extended to the homeomorphism $\iota : (G,R) \to (K, d_{E})$.
  As a consequence, the resistance form $(\mathcal{E}_{G}, \mathcal{F}_{G})$ on $G$ is regular.
\end{cor}

\begin{proof}
  It is elementary to show that the inclusion map is extended to the homeomorphism
  by using Proposition \ref{prop: metric-comparison of rrg}.
  Since $(K, d_{E})$ is obviously compact,
  the second assertion follows from \cite[Corollary 6.4]{Kigami_12_Resistance}.
\end{proof}

Henceforth,
we identify $G$ with $K$.
We consider an application of our results to a sequence $G_{n}$ and its limit $G$.
In \cite{Hambly_97_Brownian},
it was shown that the probability measures $\mu_{n}$ on $G_{n}$ converge to a probability measure $\mu$ on $G$,
and volume estimates of triangles in $G$ were also obtained.

\begin{thm} [{\cite[Theorem 5.4 and Theorem 5.5]{Hambly_97_Brownian}}]  \label{thm: convergence of measures and volume estimates of triangles}
  With probability $1$, the probability measures $\mu_{n}$ converge weakly to a probability measure $\mu$ on $G$,
  and there exist a random constant $c(\omega)>0$ and deterministic constants $\alpha,\, \beta >0$
  such that
  \begin{equation} \label{eq: hambly's volume estimates for triangles}
    \inf_{u \in T_{n}^{p}}
    \mu(\psi_{u}^{T}(K_{0}) \cap G)
    \geq
    c(\omega) n^{-\beta} e^{-\alpha n},
    \quad
    \forall n.
  \end{equation}
\end{thm}

Given a realization of $G$,
we let $\rho_{n} \in G_{n}$ be a sequence convergent to an element $\rho \in G$, and we set
\begin{equation}
  \tilde{G}_{n} \coloneqq  (G_{n}, R_{n}, \rho_{n}, \mu_{n}), \quad
  \tilde{G} \coloneqq  (G, R, \rho, \mu).
\end{equation}
It is obvious that $\tilde{G}_{n} \in \check{\mathbb{F}}_{c}$, $\mathbf{P}$-a.s.

\begin{cor} \label{cor: convergence of rrg graphs to rrg in ghp}
  The spaces $\tilde{G}_{n}$ converge to $\tilde{G}$ in the Gromov-Hausdorff-Prohorov topology, $\mathbf{P}$-a.s.
\end{cor}

\begin{proof}
  Recall that every $(G_{n}, R_{n})$ is embedded into $(G, R)$.
  By Proposition \ref{prop: metric-comparison of rrg},
  we deduce that the spaces $G_{n}$ converge to $G$ with respect to the Hausdorff metric on $(G,R)$.
  Combining this, with Theorem \ref{thm: convergence of measures and volume estimates of triangles},
  we obtain the desired result.
\end{proof}

\begin{cor} \label{cor: volume estimates of balls of rrg}
  There exist a deterministic constant $c_{4} \in (0,1)$ and a random constant $c^{*}(\omega) >0$
  such that
  \begin{equation} \label{eq: volume estimates of balls of rrg 2}
    \inf_{x \in G}
    \mu
    \left(
    D_{R}(x,r)
    \right)
    \geq
    c^{*}(\omega)
    (\log r^{-1})^{-\beta}
    r^{\alpha},
    \quad
    \forall r \in(0, c_{4}),
    \quad
    \mathbf{P}
    \text{-a.s.,}
  \end{equation}
  where $\alpha$ and $\beta$ are the deterministic constants of Theorem \ref{thm: convergence of measures and volume estimates of triangles}.
  As a consequence,
  \begin{equation} \label{eq: entropy condition for rrg}
    \sum_{k} N_{R}(G, 2^{-k-1})^{2} \exp(-2^{q k}) < \infty
  \end{equation}
  for any $q \in (0, 1/2), \, \mathbf{P}$-a.s.
  In particular, $\tilde{G} \in \check{\mathbb{F}}_{c}, \, \mathbf{P}$-a.s.
\end{cor}

\begin{proof}
  A similar argument to \eqref{eq: lower bound of resistance metric of a triangle} yields
  that there exists a deterministic constant $c >0$ such that
  \begin{equation} \label{eq: upper bound of diam of a triangle of rrg}
    \mathrm{diam}_{R}(\psi_{u}^{T}(G(\theta_{u}T_{m})))
    \leq
    c e^{-n}
  \end{equation}
  for all $u \in T_{n}^{p},\, m \geq n$,
  where $\mathrm{diam}_{R}(\cdot)$ denotes the diameter of a subset of $G$ with respect to $R$.
  Fix $u \in T_{n}^{p}$ for a while.
  For any $m \geq n$,
  we have that
  \begin{equation}
    G(T_{m})
    =
    \bigcup_{v \in T_{n}^{p}}
    \psi_{v}^{T}(G(\theta_{v}T_{m})).
  \end{equation}
  This yields that
  \begin{equation} \label{eq: triangle cap G* decomposition}
    \psi_{u}^{T}(K_{0}) \cap G^{*}
    =
    \bigcup_{m \geq n}
    \left(
    \psi_{u}^{T}(K_{0}) \cap G(T_{m})
    \right)
    =
    \bigcup_{m \geq n}
    \psi_{u}^{T}(G(\theta_{u}T_{m})).
  \end{equation}
  Noting that $(\psi_{u}^{T}(G(\theta_{u}T_{m})))_{m \geq n}$ is an increasing sequence,
  by \eqref{eq: upper bound of diam of a triangle of rrg} and \eqref{eq: triangle cap G* decomposition},
  we obtain that
  \begin{equation} \label{eq: diam estimate of triangle with G*}
    \mathrm{diam}_{R}(\psi_{u}^{T}(K_{0}) \cap G^{*})
    \leq
    c e^{-n}.
  \end{equation}
  For any $x \in \psi_{u}^{T}(K_{0}) \cap G$,
  there exist elements $x_{k} \in \psi_{u}^{T}(K_{0}) \cap G^{*}$
  such that $x_{k} \to x$ in $(G, R)$.
  (To check this, recall that $G$ is identified with $K$ and use the Euclidean metric $d_{E}$ instead of $R$ to show the convergence.)
  This, combined with \eqref{eq: diam estimate of triangle with G*},
  yields that
  \begin{equation} \label{eq: diam estimate of triangle with G}
    \mathrm{diam}_{R}(\psi_{u}^{T}(K_{0}) \cap G)
    \leq
    c e^{-n}.
  \end{equation}
  Fix $x \in G$ and $r \in (0, c)$. Choose $n \geq 1$ such that
  \begin{equation} \label{eq: choice of n for rrg volume estimate}
    c e^{-n}
    \leq
    r
    \leq
    c e^{-n+1},
  \end{equation}
  which is equivalent to that
  \begin{equation}	\label{eq: choice of n for rrg volume estimate 2}
    \log cr^{-1}
    \leq
    n
    \leq
    \log cer^{-1}.
  \end{equation}
  Choose $u \in T_{n}^{p}$ such that $x \in \psi_{u}^{T}(K_{0})$.
  By \eqref{eq: diam estimate of triangle with G} and \eqref{eq: choice of n for rrg volume estimate},
  we have that $\psi_{u}^{T}(K_{0}) \cap G \subseteq D_{R}(x,r)$.
  This, combined with \eqref{eq: hambly's volume estimates for triangles} and \eqref{eq: choice of n for rrg volume estimate 2},
  yields that
  \begin{align}
    \mu(D_{R}(x,r))
    \geq
    \mu \left( \psi_{u}^{T}(K_{0}) \cap G \right)
    \geq
    c(\omega) e^{-\alpha \log ce} (\log cer^{-1})^{-\beta} r^{\alpha}.
  \end{align}
  Therefore we obtain \eqref{eq: volume estimates of balls of rrg 2},
  and, by using Lemma \ref{lem: coverings and volume lemma} and \eqref{eq: volume estimates of balls of rrg 2},
  we deduce \eqref{eq: entropy condition for rrg}.
\end{proof}

So far,
we have checked that $\tilde{G}_{n}$ and $\tilde{G}$ satisfy
Assumption \ref{1. assum: deterministic spaces}\ref{1. assum item: deterministic, convergence of spaces}
and \ref{1. assum item: deterministic, metric-entropy condition}.
Assumption \ref{1. assum: deterministic spaces}\ref{1. assum item: deterministic, non-explosion condition} is satisfied
since $R|_{G_{n} \times G_{n}}=R_{n}$,
and so we obtain the convergence of random walks and local times on $G_{n}$.

\begin{thm} \label{thm: convergence result for rrg}
  It holds that $\cX_{\tilde{G}_{n}} \to \cX_{\tilde{G}}$ in $\mbM_{L}$,$\mathbf{P}$-a.s.
\end{thm}

\begin{proof}
  Corollary \ref{cor: convergence of rrg graphs to rrg in ghp} shows
  that $\tilde{G}_{n}$ and $\tilde{G}$ satisfy 
  Assumption \ref{1. assum: deterministic spaces}\ref{1. assum item: deterministic, convergence of spaces}, 
  and since $(G_{n}, R_{n})$ and $(G,R)$ are compact, 
  Assumption \ref{1. assum: deterministic spaces}\ref{1. assum item: deterministic, non-explosion condition} is satisfied.
  Since it holds that $R|_{G_{n} \times G_{n}}=R_{n}$,
  a similar argument to the proof of Lemma  \ref{lem: coverings and volume lemma} yields
  that $N_{R_{n}}(G_{n}, r) \leq N_{R}(G, r/2)$ for any $r >0$.
  Therefore, we obtain that,
  for any $q \in (0,1/2)$,
  \begin{equation}
    \sum_{l \geq k}
    N_{R_{n}}(G_{n}, 2^{-l}) \exp(-2^{q l})
    \leq
    \sum_{l \geq k}
    N_{R}(G, 2^{-l-1}) \exp(-2^{q l}).
  \end{equation}
  The right-hand side in the above inequality converges to $0$ as $k \to \infty$ by \eqref{eq: entropy condition for rrg},
  and thus Assumption \ref{1. assum: deterministic spaces}\ref{1. assum item: deterministic, non-explosion condition} is satisfied.
  From Theorem \ref{1. thm: main result for random spaces}, the desired result follows.
\end{proof}

\begin{rem} \label{rem: other proofs of convergence for rrg} \leavevmode
  There are other two ways to verify Theorem \ref{thm: convergence result for rrg}.
  \begin{enumerate}
    \item \label{rem item: proofs of convergence on rrg via local times estimate}
          To show the convergence of local times,
          we need to show the tightness of local times,
          and once Lemma \ref{5. lem: for precompactness of local times}
          is established,
          then the tightmess is obtained following the proof of Lemma \ref{5. lem: precompactness of cX_n for deterministic spaces}.
          In the random recursive Sierpi\'{n}ski gasket,
          the resistance metrics $R_{n}$ on $G_{n}$ are compatible with the resistance metric $R$ on $G$,
          and this enables us to estimate the equicontinuity of the local times $L_{\tilde{G}_{n}}$ of the random walk $X_{\tilde{G}_{n}}$ on $G_{n}$
          by the joint continuity of the limiting local time $L_{\tilde{G}}$ of the stochastic process $X_{\tilde{G}}$ on $G$.
          To do this, we assume that the starting points are fixed, i.e.\ $\rho_{n} = \rho$.
          We define a PCAF $A_{\tilde{G}_{n}}$ of $X_{\tilde{G}}$
          by setting $A_{\tilde{G}_{n}}(t)\coloneqq \int_{G} L_{\tilde{G}}(x,t) \mu_{n}(dx)$
          and its right-continuous inverse $\gamma_{\tilde{G}_{n}}$
          by setting $\gamma_{\tilde{G}_{n}}(t) \coloneqq  \inf \{ s >0 : A_{\tilde{G}_{n}} (s) > t \}$.
          Then, one can check that
          \begin{equation} \label{eq: time-change description of local times}
            P_{\rho}^{\tilde{G}_{n}}(L_{\tilde{G_{n}}}(x, t) \in \cdot )
            =
            P_{\rho}^{\tilde{G}}(L_{\tilde{G}} (x, \gamma_{\tilde{G}_{n}}(t)) \in \cdot)
          \end{equation}
          as probability measures on $C(G_{n} \times \mbRp, \mbRp)$ (cf.\ \cite[Lemma 3.4]{Croydon_08_Convergence}),
          where we recall the notations $ P_{\rho}^{\tilde{G}_{n}}$ and $P_{\rho}^{\tilde{G}}$ from Section \ref{sec: introduction}.
          From \eqref{eq: time-change description of local times},
          it follows that
          \begin{align}
             & P_{\rho}^{\tilde{G}_{n}}
            \left(
            \sup_{0\leq t \leq T}
            \sup_{\substack{
            x,y \in G_{n}               \\
                R_{n}(x,y) < \delta}}
            |L_{\tilde{G}_{n}}(x,t) - L_{\tilde{G}_{n}}(y,t)|
            >\varepsilon
            \right) \notag              \\
            \leq
             & P_{\rho}^{\tilde{G}}
            \left(
            \sup_{0\leq t \leq T+1}
            \sup_{\substack{
            x,y \in G                   \\
                R(x,y) < \delta}}
            |L_{\tilde{G}}(x,t) - L_{\tilde{G}}(y,t)|
            >\varepsilon
            \right)
            +
            P_{\rho}^{\tilde{G}}
            \left(\gamma_{\tilde{G}_{n}}(T) > T+1 \right).
          \end{align}
          Using the joint continuity of $L_{\tilde{G}}$ and that $A_{\tilde{G}_{n}}(T) \to T$,
          one can check that the right-hand side of the above inequality converges to $0$ as $n \to \infty$ and then $\delta \to 0$,
          which verifies the condition of 
          Lemma \ref{5. lem: for precompactness of local times}.
    \item
          Following the proof of \cite[Theorem 5.5]{Hambly_97_Brownian},
          it is possible to show that,
          for every $\varepsilon >0$, there exist constants $c,\, c',\, c'' >0$
          such that, for all $n$,
          \begin{equation} \label{eq: volume estimates of balls of Gn of rrg}
            \mathbf{P}
            \left(
            \inf_{x \in G_{n}}
            \mu_{n} (D_{R_{n}}(x,r))
            \geq
            c (\log r^{-1})^{-\beta} r^{\alpha},
            \quad
            \forall r \in [c' e^{-n},c'']
            \right)
            \geq
            1-\varepsilon.
          \end{equation}
          Thus, by Theorem \ref{1. thm: main result for random spaces},
          we obtain the convergence of the annealed laws, i.e.\ $\mathbb{P}_{\tilde{G}_{n}} \to \mathbb{P}_{\tilde{G}}$.
          Although this result is weaker than Theorem \ref{thm: convergence result for rrg},
          the approach via volume estimates is stable under fluctuations on the spaces $G_{n}$.
          For example,
          multiplying the conductances on the electrical networks $G_{n}$ by $a_{n} >0$ with $a_{n} \to 1$
          easily destroys that $R|_{G_{n} \times G_{n}}=R_{n}$.
          Thus the argument of 
          Remark \ref{rem: other proofs of convergence for rrg}\eqref{rem item: proofs of convergence on rrg via local times estimate}
          does not work.
          However the proof of Theorem \ref{thm: convergence result for rrg} is still valid.
          We believe
          that the volume estimates \eqref{eq: volume estimates of balls of Gn of rrg} of balls in $G_{n}$ will be useful
          when $a_{n}$ is random and the fluctuations are not simple like this.
  \end{enumerate}
\end{rem}

\subsection{The critical Erd\H{o}s-R\'{e}nyi random graph} \label{subsec: critical ER random grpah}

\newcommand{\Msigma}{M^{(\sigma)}}
\newcommand{\ERlc}{\cC_{1}^{n}}
\newcommand{\excs}{e^{(\sigma)}}
\newcommand{\tildexcs}{\tilde{e}^{(\sigma)}}
\newcommand{\tildTs}{\tilde{T}^{(\sigma)}}
\newcommand{\Mz}{M^{(Z_{1})}}
\newcommand{\Gmnp}{G_{m_{n}}^{p_{n}}}
\newcommand{\Gmp}{G_{m}^{p}}
\newcommand{\DFS}{\mathbf{DFS}}
\newcommand{\tildT}{\tilde{T}}
\newcommand{\tildG}{\tilde{G}}
\newcommand{\tildX}{\tilde{X}}
\newcommand{\tildH}{\tilde{H}}

In this section,
we consider the Erd\H{o}s-R\'{e}nyi random graph $G(n,p)$,
which is a graph on $n$ labeled vertices $[n] \coloneqq  \{1,2, \ldots, n\}$
chosen randomly by joining any two distinct vertices by an edge with probability $p$,
independently for different pairs of vertices.
This model exhibits a phase transition in its structure for large $n$.
Let $p=c/n$ for some $c >0$.
When $c < 1$, the largest connected component has size $O(\log n)$.
On the other hand, when $c > 1$, we see the emergence of a giant component that contains
a positive proportion of the vertices.
In the critical case $c=1$, the largest connected components have sizes of order $n^{2/3}$.
We will focus here on the critical case $c = 1$, and more specifically, on the critical window $p = n^{-1} + \lambda n^{-4/3}$, $\lambda \in \mbR$.
We fix $\lambda \in \mbR$ and write $p_{n}=n^{-1}+\lambda n^{-4/3}$.
Let $\cC_{i}^{n}$ be the $i$-th largest connected component of $G(n,p_{n})$.
In \cite{Berry_Broutin_Goldschmidt_12_The_continuum},
it was proven
that $\cC_{1}^{n}$ equipped with the graph metric converges to a random metric space with respect to the Gromov-Hausdorff topology,
and in \cite{Croydon_12_Scaling},
the convergence of a random walk on the component was established.
Using Theorem \ref{1. thm: main result for random spaces},
we will show the convergence of a random walk on $\cC_{1}^{n}$ and its local time.

One of the most significant results about random graphs in the above-mentioned critical regime was proved by Aldous \cite{Aldous_97_Brownian}.
Write $Z_{i}^{n}$ and $S_{i}^{n}$ for the size (that is, the number of vertices) and surplus
(that is, the number of edges that would need to be removed in order to obtain a tree)
of $\cC_{i}^{n}$.
Set $\bm{Z}^{n} \coloneqq  (Z_{1}^{n}, Z_{2}^{n}, \ldots)$ and $\bm{S}^{n}\coloneqq (S_{1}^{n}, S_{2}^{n}, \ldots)$.

\begin{thm} [{\cite[Folk Theorem 1, Corollary 2]{Aldous_97_Brownian}}]  \label{thm: convergence of sizes and surpluses of ER}
  As $n \to \infty$, it holds that
  \begin{equation}
    (n^{-2/3} \bm{Z}^{n}, \bm{S}^{n}) \to (\bm{Z}, \bm{S})
  \end{equation}
  in distribution, where the convergence of the first coordinate
  takes place in $l^{2}_{\searrow}$,
  the set of infinite sequences $(x_{1}, x_{2}, \ldots)$
  with $x_{1} \geq x_{2} \geq \cdots \geq 0$ and $\sum_{i} x_{i}^{2} < \infty$,
  equipped with the usual $l^{2}$-norm.
\end{thm}

The limit $\bm{Z}=(Z_{1}, Z_{2}, \ldots)$ and $\bm{S}=(S_{1}, S_{2}, \ldots)$ is constructed as follows.
Consider a Brownian motion with parabolic drift, $(W_{t}^{\lambda})_{t \geq 0}$, where
\begin{equation}
  W_{t}^{\lambda} \coloneqq  W_{t} + \lambda t - \frac{t^{2}}{2}
\end{equation}
and $(W_{t})_{t \geq 0}$ is a standard Brownian motion.
Then, the limit $\bm{Z}$ has the distribution of the ordered sequence of lengths of excursions
of the reflected process $W^{\lambda}_{t} - \min_{0 \leq s \leq t} W^{\lambda}_{s}$ above $0$,
while $\bm{S}$ is the sequence of numbers of points of a Poisson point process with rate one in $\mbRp^{2}$
lying under the corresponding excursions,
where the Poisson point process is assumed to be independent of $(W^{\lambda}_{t})_{t \geq 0}$.

Recall the space $\mathscr{E}$ of excursions from Section \ref{subsubsec: introduction of trees},
where we equip $\mathscr{E}$ with the metric induced by the supremum norm $\| \cdot \|$.
Let $\excs=(\excs(t), 0 \leq t \leq \sigma)$ be a Brownian excursion of length $\sigma$.
The \textit{tilted excursion} of length $\sigma$,
$\tildexcs=(\tildexcs(t), 0 \leq t \leq \sigma) \in \mathscr{E}$,
is defined to be an excursion whose distribution is characterized by
\begin{equation}
  \mathbf{P} \left( \tildexcs \in \mathcal{B} \right)
  =
  \frac
  {\mathbf{E}
    \left(
    1_{\{\excs \in  \mathcal{B}\}}
    \exp( \int_{0}^{\sigma} \excs(t) dt)
    \right)}
  { \mathbf{E}
    \left(
    \exp( \int_{0}^{\sigma} \excs(t) dt)
    \right)}
\end{equation}
for $\mathcal{B} \subseteq \mathscr{E}$ a Borel set.
For $f \in \mathscr{E}$ and $S \subseteq \mbRp^{2}$,
define
\begin{equation}
  S \cap f
  \coloneqq
  \{ (x, y) \in S : 0 \leq y < f(x) \}.
\end{equation}
For $u=(u_{x},u_{y}) \in S \cap f$,
we define $u'=(u_{x}, u'_{x})$
by setting $u'_{x} \coloneqq  \inf \{x \geq u_{x} : f(x)=u_{x} \}$.
We write
\begin{equation}  \label{eq: def of time markers from excursion and pointset}
  \mathscr{T}(S, f) \coloneqq  \{ u' \in \mbRp^{2} : u \in S \cap f \}.
\end{equation}
Let $\mathcal{P} \subseteq \mbRp^{2}$ be a Poisson point process with rate one,
independent of $\tildexcs$.
Assume that $\mathcal{P} \cap \tildexcs$ is non-empty
and write $\mathscr{T}(\mathcal{P}, \tildexcs)=\{ (\xi_{l}^{x}, \xi_{l}^{y}) : 1 \leq l \leq s\}$.
Define $a_{l} \coloneqq  \tau(\xi_{l}^{x})$ and $b_{l} \coloneqq  \tau(\xi_{l}^{y})$,
where $\tau$ denotes the canonical projection from $[0, \sigma]$ onto the real tree $\tildTs$ coded by $2 \tildexcs$.
Let $d_{\tildTs}$ be the metric on $\tildTs$,
$\rho_{\tildTs}$ be the root of $\tildTs$
and $\mu_{\tildTs}$ be the canonical measure on $\tildTs$
(recall these from Section \ref{subsubsec: introduction of trees}).
Define $(\Msigma, R_{\Msigma}, \rho_{\Msigma}, \mu_{\Msigma})$
to be the resistance metric space
obtained by fusing $(\tildTs, d_{\tildTs}, \rho_{\tildTs}, \mu_{\tildTs})$ over $(\{ a_{l}, b_{l} \})_{l=1}^{s}$
(see \cite[Section 8.3]{Croydon_18_Scaling} for the construction of fused resistance metric spaces).
If $\mathcal{P} \cap \tildexcs$ is empty,
then we define  $(\Msigma, R_{\Msigma}, \rho_{\Msigma}, \mu_{\Msigma})$
to be equal to $(\tildTs, d_{\tildTs}, \rho_{\tildTs}, \mu_{\tildTs})$.
We assume that the family $((\Msigma, R_{\Msigma}, \rho_{\Msigma}, \mu_{\Msigma}), \sigma >0)$ is independent of $Z_{1}$.

Given a finite connected graph $G$ with labeled vertices,
we write $V_{G}$ for the vertex set of $G$,
$R_{G}$ for the resistance metric on $V_{G}$,
$\rho_{G}$ for the smallest-labeled vertex of $G$
and $\mu_{G}$ for the counting measure on $V_{G}$.
The following concerns
Assumption \ref{1. assum: random spaces}\ref{1. assum item: random spaces, space convergence}.

\begin{thm}  \label{thm: convergence of ER graphs wrt resistance metric}
  It holds that
  \begin{equation}
    (V(\ERlc), n^{-1/3} R_{\ERlc}, \rho_{\ERlc}, |V(\ERlc)|^{-1} \mu_{\ERlc})
    \xrightarrow{\mathrm{d}}
    (\Mz, R_{\Mz}, \rho_{\Mz}, Z_{1}^{-1}\mu_{\Mz})
  \end{equation}
  in the Gromov-Hausdorff-Prohorov topology.
\end{thm}

Let $G_{m}^{p}$ be a random graph with the distribution of $G(m,p)$ conditioned to be connected.
We assume that $Z_{1}^{n}$ and $(G_{m}^{p_{n}})_{m \geq 1}$ are all independent.
It is then an easy exercise to check
that the random graph $G_{Z_{1}^{n}}^{p_{n}}$ has the same distribution as the random graph $\cC_{1}^{n}$ with relabeled vertices.
Combining this with Theorem \ref{thm: convergence of sizes and surpluses of ER},
we obtain Theorem \ref{thm: convergence of ER graphs wrt resistance metric},
once the following lemma is established.

\begin{lem} \label{lem: convergence of conditioned ER graphs wrt resistance metric}
  If a sequence $(m_{n})_{n \geq 1}$ of natural numbers satisfies $n^{-2/3}m_{n} \to \sigma \in (0,\infty)$,
  then it holds that
  \begin{equation}
    (V(\Gmnp), n^{-1/3} R_{\Gmnp}, \rho_{\Gmnp}, m_{n}^{-1} \mu_{\Gmnp})
    \xrightarrow{\mathrm{d}}
    (\Msigma, R_{\Msigma}, \rho_{\Msigma}, \mu_{\Msigma}).
  \end{equation}
\end{lem}

To prove Lemma \ref{lem: convergence of conditioned ER graphs wrt resistance metric},
we describe how to generate $G_{m}^{p}$.
Let $G$ be a connected graph with the vertex set $[m]$.
We can associate a spanning tree $\DFS(G)$ on $G$, the depth-first tree,
by running a depth-first search on $G$,
using the rule
that whenever there is a choice of which vertex to explore,
the smallest-labeled vertex is always explored first.
More precisely,
we define $\DFS(G)$ by specifying edges on the vertex set $V(\DFS(G))\coloneqq [m]$
via the following algorithm.
At time $0$,
the depth-first search process is at vertex $1$ and declares it open.
At time $1 \leq k \leq m-1$,
the vertex $v$ that the process currently visits is declared as explored.
If there exist neighbors of $v$ on $G$ not opened yet,
then the neighbors are declared as open,
the process moves to the smallest-labeled vertex $v'$ among them and
we attach an edge between $v$ and $v'$ to $V(\DFS(G))$.
If there are no such neighbors of $v$ on $G$,
then the process moves to the smallest-labeled vertex among the most recently opened and non-explored vertices.
(Note that in this case no edge is attached.)

Let $T$ be a tree with the vertex set $[m]$.
If a connected graph $G$ with the vertex set $[m]$ have that $\DFS(G)=T$,
then $G$ can be obtained by adding some edges to $T$.
Such edges are called \textit{permitted edges}.
In other words,
a permitted edge of $T$ is an edge
for which $\DFS(G)=T$ holds
for the graph $G$ which is  obtained by adding that edge to $T$
(see \cite[Lemma 7]{Berry_Broutin_Goldschmidt_12_The_continuum}).
The cardinality of the permitted edges of $T$ is denoted by $a(T)$.
If we regard $T$ as a plane tree by using the depth-first search
and write $X^{T}=(X^{T}(i),\, 0\leq i \leq m)$ for the depth-first walk of $T$,
then the number $a(T)$ corresponds to the area of the depth-first walk $X^{T}$ of $T$.
Namely, we have that
\begin{equation}
  a(T) \coloneqq  \sum_{i=1}^{m-1} X^{T}(i)
\end{equation}
(see \cite[Lemma 6]{Berry_Broutin_Goldschmidt_12_The_continuum}).
We have $X^{T}(m)=-1$,
but we redefine $X^{T}(m) \coloneqq 0$
so that $X^{T} \in C([0,m], \mbRp)$.
For later use,
we define $H^{T}=(H^{T}(i),\, 0 \leq i \leq m-1)$
to be the height function of $T$.
Note that $H^{T}(i)$ is equal to the distance between vertex $1$ and the depth-first search process at time $i$.
We set $H^{T}(m)\coloneqq 0$ and regard $H^{T}$ as a function in $C([0,m], \mbRp)$ by linear interpolation.

\begin{lem} [{\cite[Proposition 8]{Berry_Broutin_Goldschmidt_12_The_continuum}}]  \label{lem: construction of Gmp from a tilted tree}
  Let $\mathbb{T}_{m}$ be the set of trees with the vertex set $[m]$.
  Fix $p \in (0,1)$.
  Pick a random tree $\tildT_{m}^{p}$
  that has a ``tilted'' distribution which is biased in favor of trees with large area.
  Namely,
  pick $\tildT_{m}^{p}$ in such a way that
  \begin{equation}
    \mathbf{P}(\tildT_{m}^{p}=T)
    \propto
    (1-p)^{-a(T)}, \quad
    T \in \mathbb{T}_{m}.
  \end{equation}
  Add to $\tildT_{m}^{p}$ each of the $a(\tildT_{m}^{p})$ permitted edges independently with probability $p$.
  Call the generated graph $\tildG_{m}^{p}$.
  Then, $\tildG_{m}^{p}$ has the same distribution as $G_{m}^{p}$ on $\mathsf{G}_{m}^{c}$,
  which denotes the set of connected graphs with the vertex set $[m]$.
\end{lem}

Lemma \ref{lem: construction of Gmp from a tilted tree} can be rephrased
as the correspondence of
the distribution of $G_{m}^{p}$
to that of the coding function of $\tildT_{m}^{p}$ and a \textit{Binomial pointset}
as precisely described in Lemma \ref{lem: construction of Gmp from coding functions and Binomial pointset} below,
and such an interpretation is useful for deriving scaling limits
(as seen in Theorem \ref{thm: convergence of contour function} and Corollary \ref{cor: convergence of trees}).

\begin{lem} [{\cite[Lemma 18]{Berry_Broutin_Goldschmidt_12_The_continuum}}] \label{lem: construction of Gmp from coding functions and Binomial pointset}
  Fix $p \in (0,1)$.
  Let $\mathcal{Q}^{p} \subseteq \mathbb{Z}_{+}^{2}$ be a Binomial pointset of intensity $p$,
  that is,
  a random subset of $\mathbb{Z}_{+}^{2}$
  in which each point is present independently with probability $p$.
  Let $\tildT_{m}^{p}$ be a tilted tree as in Lemma \ref{lem: construction of Gmp from a tilted tree}
  with vertices $v_{0}, v_{1}, \ldots, v_{m-1}$ in depth-first order,
  independent of $\mathcal{Q}^{p}$.
  Let $\tildX=(\tildX(t), 0 \leq t \leq m)$ be the depth-first walk of $\tildT_{m}^{p}$.
  Write $\mathscr{T}(\mathcal{Q}^{p}, \tildX) = \{(x_{i}, y_{i}) : 1 \leq i \leq s \}$
  (recall the definition from \eqref{eq: def of time markers from excursion and pointset}).
  We define a graph $G(\tildT_{m}^{p}, \mathcal{Q}^{p})$
  by attaching an edge between $v_{x_{i}}$ and $v_{y_{i}}$ on $\tildT_{m}^{p}$.
  (If $\mathscr{T}(\mathcal{Q}^{p}, \tildX)$ is empty, we define $G(\tildT_{m}^{p}, \mathcal{Q}^{p}) \coloneqq  \tildT_{m}^{p}$.)
  Then, $G(\tildT_{m}^{p}, \mathcal{Q}^{p})$ has the same distribution as $G_{m}^{p}$ on $\mathsf{G}_{m}^{c}$.
\end{lem}

\begin{lem} [{\cite[Lemma 19]{Berry_Broutin_Goldschmidt_12_The_continuum}}] \label{lem: convergence of coding functions and pointset for ER}
  Assume that a sequence $(m_{n})_{n \geq 1}$ satisfies $n^{-2/3} m_{n} \to \sigma \in (0,\infty)$.
  Write $\tildX^{n}=(\tildX^{n}(t), 0 \leq t \leq m_{n})$ and $\tildH^{n}=(\tildH^{n}(t), 0 \leq t \leq m_{n})$
  for the depth-first walk and the height function of a tilted tree $\tildT_{m_{n}}^{p_{n}}$, respectively.
  Let $\mathcal{Q}^{p_{n}}$ be a Binomial pointset of intensity $p_{n}$, independent of $\tildT_{m_{n}}^{p_{n}}$,
  and define $\mathcal{P}_{n} \coloneqq  \{ ( (\sigma/m_{n}) i, (\sigma/m_{n})^{1/2} j ) : (i,j) \in \mathcal{Q}^{p_{n}} \}$.
  Let $\tildexcs=(\tildexcs(t), 0 \leq t \leq \sigma)$ be a tilted excursion with length $\sigma$
  and $\mathcal{P}$ be a Poisson point process on $\mbRp^{2}$ with rate one, independent of $\tildexcs$.
  Then, it holds that
  \begin{align}
     & \left(
    \sqrt{\frac{\sigma}{m_{n}}} \tildH^{n} \left( \left\lfloor \frac{m_{n}}{\sigma} \cdot \right\rfloor \right),
    \sqrt{\frac{\sigma}{m_{n}}} \tildX^{n} \left( \left\lfloor \frac{m_{n}}{\sigma} \cdot \right\rfloor \right),
    \mathcal{P}_{n} \cap \left( \sqrt{\frac{\sigma}{m_{n}}} \tildX^{n} \left( \left\lfloor \frac{m_{n}}{\sigma} \cdot \right\rfloor \right) \right)
    \right) \notag                            \\
     & \xrightarrow[n \to \infty]{\mathrm{d}}
    (2 \tildexcs, \tildexcs, \mathcal{P} \cap \tildexcs),
    \label{eq: convergence of coding functions and pointset for ER}
  \end{align}
  where the convergence of the first and second coordinate takes place in $D([0,\sigma], \mbRp)$
  equipped with the usual $J_{1}$-Skorohod topology
  and the convergence of the third coordinate takes place with respect to the Hausdorff metric.
\end{lem}

\begin{rem}
  In \cite[Lemma 19]{Berry_Broutin_Goldschmidt_12_The_continuum},
  the convergence of $\tildH^{n}$ is not mentioned,
  but it is an immediate consequence of \cite[Lemma 16]{Berry_Broutin_Goldschmidt_12_The_continuum}.
\end{rem}

\begin{proof} [Proof of Lemma \ref{lem: convergence of conditioned ER graphs wrt resistance metric}]
  We proceed with the proof in the setting of Lemma \ref{lem: convergence of coding functions and pointset for ER}.
  By the Skorohod representation theorem,
  we may assume that the convergence \eqref{eq: convergence of coding functions and pointset for ER} takes place almost-surely
  on some probability space.
  Assume that $\mathcal{P} \cap \tildexcs$ is non-empty
  and write $\mathscr{T}(\mathcal{P}, \tildexcs) = \{(\xi_{l}^{x}, \xi_{l}^{y}) : 1 \leq \l \leq s \}$.
  For all sufficiently large $n$,
  we can write
  $\mathscr{T}(\mathcal{P}_{n}, \sqrt{\sigma/ m_{n}} \tildX^{n}  (\lfloor (m_{n}/ \sigma) \cdot \rfloor))
    = \{(i_{l}^{n}, j_{l}^{n}) : 1 \leq l \leq s\}$
  in such a way that
  \begin{equation}  \label{eq: convergence of time markers of coding functions}
    \varepsilon_{n}
    \coloneqq
    \max_{1 \leq l \leq s}
    ( |i_{l}^{n} - \xi_{l}^{x}| \vee |j_{l}^{n} - \xi_{l}^{y}| )
    \to
    0.
  \end{equation}
  This follows from \eqref{eq: convergence of coding functions and pointset for ER} and the fact that
  almost-surely, for every $\delta >0$ small enough,
  we have
  \begin{equation}
    \inf_{0 \leq t \leq \delta} \tildexcs(\xi_{l}^{y}+t)
    <
    \tildexcs(\xi_{l}^{y})
    <
    \sup_{0 \leq t \leq \delta} \tildexcs(\xi_{l}^{y}-t),  \quad
    \forall l=1,2,\ldots,s
  \end{equation}
  (see \cite[Proof of Theorem 22]{Berry_Broutin_Goldschmidt_12_The_continuum}).
  Let $v_{1}^{n}, v_{2}^{n}, \ldots, v_{m_{n}}^{n}$ be the vertices of $\tildT_{m_{n}}^{p_{n}}$ in depth-first order,
  and define $a_{l}^{n}\coloneqq v_{m_{n} i_{l}^{n} / \sigma}^{n},\, b_{l}^{n} \coloneqq  v_{m_{n} j_{l}^{n} / \sigma}^{n}$.
  Here, we note that we have
  \begin{equation}
    \{(m_{n} i_{l}^{n} / \sigma, m_{n} j_{l}^{n} / \sigma) : 1 \leq l \leq s\}
    =
    \mathscr{T}(\mathcal{Q}^{p_{n}}, \tildX^{n})
  \end{equation}
  and in particular the indices $m_{n} i_{l}^{n} / \sigma$ and $m_{n} j_{l}^{n} / \sigma$ are integers.
  Define $a_{l}, b_{l} \in \tildTs$ by setting
  $a_{l}\coloneqq  \tau(\xi_{l}^{x})$ and $b_{l} \coloneqq  \tau(\xi_{l}^{y})$,
  where we recall that
  $\tau$ is the canonical projection from $[0,\sigma]$ onto the real tree $\tildTs$ coded by $2 \tildexcs$.
  Define a correspondence $\mathcal{R}_{n}$ between $V(\tildT_{m_{n}}^{p_{n}})$ and $\tildTs$
  by setting
  \begin{equation}
    \mathcal{R}_{n}
    \coloneqq
    \{
    (v_{i}^{n}, v) \in V(\tildT_{m_{n}}^{p_{n}}) \times \tildTs
    :
    \exists t \in [0,\sigma]\
    \text{such that}\
    |(\sigma / m_{n})\cdot i - t| \leq \varepsilon_{n} \vee (\sigma/m_{n})
    \}.
  \end{equation}
  Embed $(V(\tildT_{m_{n}}^{p_{n}}), n^{-1/3} d_{\tildT_{m_{n}}^{p_{n}}})$ and $(\tildTs, d_{\tildTs})$
  into the disjoint union $V(\tildT_{m_{n}}^{p_{n}}) \sqcup \tildTs$
  equipped with the metric $d_{n}$ satisfying
  \begin{equation}
    d_{n}(v_{i}, v)
    =
    \inf
    \{
    n^{-1/3} d_{\tildT_{m_{n}}^{p_{n}}} (v_{i}, v_{j})
    + 2^{-1} \mathrm{dis} \mathcal{R}_{n}
    + n^{-1}
    +
    d_{\tildTs} (v, v')
    :
    (v_{j}, v) \in \mathcal{R}_{n}
    \}
  \end{equation}
  for $(v_{i}, v) \in V(\tildT_{m_{n}}^{p_{n}}) \times \tildTs$.
  It is then routine to show
  \begin{align}
     & (V(\tildT_{m_{n}}^{p_{n}}),
    n^{-1/3} d_{\tildT_{m_{n}}^{p_{n}}},
    \rho_{\tildT_{m_{n}}^{p_{n}}},
    m_{n}^{-1} \mu_{\tildT_{m_{n}}^{p_{n}}},
    a_{1}^{n},
    b_{1}^{n},
    \ldots,
    a_{s}^{n},
    b_{s}^{n}
    )                              \\
     & \to
    (\tildTs, d_{\tildTs}, \rho_{\tildTs}, \sigma^{-1} \mu_{\tildTs}, a_{1}, b_{1}, \ldots, a_{s}, b_{s})
  \end{align}
  in the Gromov-Hausdorff-Prohorov topology with additional points
  (this topology is introduced in \cite[Section 8.3]{Croydon_18_Scaling}
  and see also Remark \ref{8. rem: GH-type metric for cadlag curves with sup norm} below).
  Moreover,
  almost-surely,
  we have that $a_{1}, b_{1}, \ldots, a_{s}, b_{s}$ are all distinct.
  This follows from the fact that,
  given distinct times $(t_{i})_{i=1}^{J} \in [0, \sigma]$,
  almost-surely,
  $(e^{(\sigma)}(t_{i}))_{i=1}^{J}$ are distinct
  and from the absolute continuity of the law of $\tildexcs$ with respect to $e^{(\sigma)}$.
  Therefore,
  by \cite[Proposition 8.4]{Croydon_18_Scaling},
  we obtain that
  \begin{equation}  \label{eq: convergence of fused spaces for ER}
    (G'_{n}, n^{-1/3} R_{G'_{n}}, \rho_{G'_{n}}, m_{n}^{-1} \mu_{G'_{n}})
    \to
    (\Msigma, R_{\Msigma}, \rho_{\Msigma}, Z_{1}^{-1} \mu_{\Msigma})
  \end{equation}
  in the Gromov-Hausdorff-Prohorov topology,
  where $(G'_{n}, R_{G'_{n}}, \rho_{G'_{n}}, \mu_{G'_{n}})$ is a rooted-and-measured resistance metric space
  obtained by fusing $(V(\tildT_{m_{n}}^{p_{n}}), d_{\tildT_{m_{n}}^{p_{n}}}, \rho_{\tildT_{m_{n}}^{p_{n}}}, \mu_{\tildT_{m_{n}}^{p_{n}}})$
  over $(\{a_{l}^{n}, b_{l}^{n}\})_{l=1}^{s}$.
  Define $G_{n} \coloneqq  G(\tildT_{m_{n}}^{p_{n}}, \mathcal{Q}^{p_{n}})$
  (recall the notation from Lemma \ref{lem: construction of Gmp from coding functions and Binomial pointset}).
  It is then not difficult to check that
  the Gromov-Hausdorff-Prohorov distance
  between $(G'_{n}, n^{-1/3} R_{G'_{n}}, \rho_{G'_{n}}, m_{n}^{-1} \mu_{G'_{n}})$ and
  $(G_{n}, n^{-1/3} R_{G_{n}}, \rho_{G_{n}}, m_{n}^{-1})$
  converges to $0$ in probability
  (note that $G_{n}$ is constructed by adding an edge between $a_{l}^{n}$ and $b_{l}^{n}$ on $\tildT_{m_{n}}^{p_{n}}$,
  while $G_{n}'$ is constructed by fusing $a_{l}^{n}$ and $b_{l}^{n}$).
  Combining this with \eqref{eq: convergence of fused spaces for ER},
  we deduce that
  \begin{equation} \label{eq: convergence of fused spaces for ER 2}
    (G_{n}, n^{-1/3} R_{G_{n}}, \rho_{G_{n}}, m_{n}^{-1})
    \to
    (\Msigma, R_{\Msigma}, \rho_{\Msigma}, Z_{1}^{-1} \mu_{\Msigma}).
  \end{equation}
  Assume that $\mathcal{P} \cap \tildexcs$ is empty.
  Then, we have that $\mathcal{Q}^{p_{n}} \cap \tildX^{n}$ is empty for all sufficiently large $n$.
  Thus we have $G_{n}=\tildT_{m_{n}}^{p_{n}}$ by definition of $G(\tildT_{m_{n}}^{p_{n}}, \mathcal{Q}^{p_{n}})$,
  and the convergence \eqref{eq: convergence of fused spaces for ER 2} follows immediately from
  the convergence \eqref{eq: convergence of coding functions and pointset for ER}.
  Hence, the convergence \eqref{eq: convergence of fused spaces for ER 2} holds almost-surely,
  and,
  by Lemma \ref{lem: construction of Gmp from coding functions and Binomial pointset},
  the desired result is established.
\end{proof}

Now we set
\begin{gather}
  \mathcal{G}_{m}^{p}
  \coloneqq
  (V(\Gmp), n^{-1/3} R_{\Gmp}, \rho_{\Gmp}, m_{n}^{-1} \mu_{\Gmp})\\
  \mathcal{G}_{n}
  \coloneqq
  (V(\ERlc), n^{-1/3} R_{\ERlc}, \rho_{\ERlc}, |V(\ERlc)|^{-1} \mu_{\ERlc})\\
  \mathcal{M}_{\sigma}
  \coloneqq
  (\Msigma, R_{\Msigma}, \rho_{\Msigma}, \mu_{\Msigma})\\
  \mathcal{M}
  \coloneqq
  (\Mz, R_{\Mz}, \rho_{\Mz}, Z_{1}^{-1}\mu_{\Mz}).
\end{gather}
As we deduced Theorem \ref{thm: convergence of ER graphs wrt resistance metric}
from Lemma \ref{lem: convergence of conditioned ER graphs wrt resistance metric},
the convergence $\cX_{\mathcal{G}_{n}} \xrightarrow{\mathrm{d}} \cX_{\mathcal{M}}$ is obtained,
once we show that $\cX_{\mathcal{G}_{m_{n}}^{p_{n}}} \xrightarrow{\mathrm{d}} \cX_{\mathcal{M}_{\sigma}}$
when $n^{-2/3} m_{n} \to \sigma \in (0,\infty)$.
Concerning the volume estimate for $\mathcal{G}_{m_{n}}^{p_{n}}$,
since the mass of a ball in a graph equipped with a resistance metric increases
when new edges are attached,
the volume estimate is reduced to the volume estimate for the tilted trees $\tildT_{m}^{p}$.
Following \cite[Proof of Lemma 4.9]{Andriopoulos_23_Convergence}
and using Proposition \ref{prop: volume estimates for trees},
we obtain the following result.

\begin{lem} \label{lem: volume estimate for conditioned ER graphs}
  If a sequence $(m_{n})_{n \geq 1}$ of natural numbers satisfies $n^{-2/3}m_{n} \to \sigma \in (0,\infty)$,
  then, for every $\gamma \in (0,1/2)$ and $\varepsilon >0$,
  there exists $c_{\gamma, \varepsilon} >0$ such that
  \begin{equation}
    \liminf_{n \to \infty}
    \mathbf{P}
    \left(
    \inf_{x \in V(\Gmnp)}
    m_{n}^{-1}
    \mu_{\Gmnp}
    \left(D_{R_{\Gmnp}}(x,r)\right)
    \geq
    (c_{\gamma,\varepsilon} r^{\gamma^{-1}}) \wedge 1,
    \quad
    \forall r > 0
    \right)
    \geq
    1-\varepsilon.
  \end{equation}
\end{lem}

By Proposition \ref{lem: volume estimate for conditioned ER graphs}
and Corollary \ref{eq: discrete graph version of volume condition},
we obtain the convergence of stochastic processes and local times
on the critical \erdosrenyi\ random graphs.
\begin{cor}
  The limiting space $\mathcal{M}$ belongs to $\check{\mathbb{F}}_{c}$ with probability $1$,
  and $\cX_{\mathcal{G}_{n}} \xrightarrow{\mathrm{d}} \cX_{\mathcal{M}}$
  as random elements of $\mbM_{L}$.
\end{cor}

\subsection{The configuration model} \label{subsec: the critical configuration model}
\newcommand{\bd}{\bm{d}}
\newcommand{\bt}{\bm{t}}
\newcommand{\bdn}{\bm{d}^{(n)}}
\newcommand{\btildd}{\tilde{\bm{d}}}
\newcommand{\btilddm}{\tilde{\bm{d}}^{(\tilde{m})}}
\newcommand{\tildd}{\tilde{d}}
\newcommand{\mbGdcon}{\mathsf{G}_{\tilde{\bm{d}}}^{\mathrm{con}}}
\newcommand{\scrGdcon}{\mathscr{G}_{\tilde{\bm{d}}}^{\mathrm{con}}}
\newcommand{\mbG}{\mathbb{G}}
\newcommand{\mbGnd}{\mathsf{G}_{n, \bm{d}}}
\newcommand{\mbGndbar}{\bar{\mathsf{G}}_{n, \bm{d}}}
\newcommand{\scrGnd}{\mathscr{G}_{n, \bm{d}}}
\newcommand{\CMnd}{\mathrm{CM}_{n}(\bm{d})}
\newcommand{\bfP}{\mathbf{P}}
\newcommand{\bfE}{\mathbf{E}}
\newcommand{\tildm}{\tilde{m}}
\newcommand{\tildp}{\tilde{p}}
\newcommand{\tildk}{\tilde{k}}
\newcommand{\scrL}{\mathscr{L}}
\newcommand{\bA}{\bm{A}}
\newcommand{\bxi}{\bm{\xi}}
\newcommand{\bs}{\bm{s}}
\newcommand{\bstildm}{\bm{s}^{(\tilde{m})}}
\newcommand{\mbTs}{\mathbb{T}_{\bm{s}}}
\newcommand{\mbTsk}{\mathbb{T}_{\bm{s}}^{(k)}}
\newcommand{\bx}{\bm{x}}
\newcommand{\tildscrTs}{\tilde{\mathscr{T}}_{\bm{s}}}
\newcommand{\tildbX}{\tilde{\bm{X}}}
\newcommand{\scrGtildd}{\mathscr{G}_{\tilde{\bm{d}}}}
\newcommand{\scrTs}{\mathscr{T}_{\bm{s}}}
\newcommand{\scrT}{\mathscr{T}}
\newcommand{\tildmu}{\tilde{\mu}}
\newcommand{\Mk}{M^{(k)}}
\newcommand{\Ttildek}{\tilde{T}^{(k)}}

Fix a collection of $n$ vertices labeled by $[n] \coloneqq \{1,2,\ldots,n\}$
and an associated degree sequence $\bd=\bdn= (d_{v}^{(n)}, v \in [n])$
where $\ell_{n} \coloneqq \sum_{v \in [n]} d_{v}^{(n)}$ is assumed even.
We write $\mbGnd$ for the space of all simple graphs labeled by $[n]$
such that the $i$-th vertex has degree $d_{i}$.
Let $\scrGnd$ be a random graph uniformly distributed on $\mbGnd$.
Write $\mbGndbar$ for the space of all multigraphs on vertex set $[n]$ with degree sequence $\bd$
(recall that in a multigraph,
we allow self-loops as well as the occurrence of multiple edges between the same pair of vertices).
The \textit{configuration model} $\CMnd$ is a random multigraph on vertex set $[n]$
with degree sequence $\bd$ constructed sequentially as follows:
Equip each vertex $v \in [n]$ with $d_{v}^{(n)}$ half-edges.
Pick two half-edges uniformly from the set of half-edges that have not yet been paired,
and pair them to form a full edge.
Repeat until all half-edges have been paired.
The resulting graph is denoted by $\CMnd$.
For simplicity,
we will omit the superscript and write $d_{v},\, v \in [n]$.
We will work with degree sequences that satisfy the following assumption.

\begin{assum} [{\cite[Assumption 2.1]{Bhamidi_Sen_20_Geometry}}]
  \label{assum: degree sequence of configuration models}
  Let $D_{n}$ be a random variable with distribution given by
  \begin{equation}
    P (D_{n} = i)
    = \frac{1}{n} |\{ j : d_{j}^{(n)} =i \}|,
  \end{equation}
  i.e.\ $D_{n}$ has the law of the degree of a vertex selected uniformly at random from $[n]$.
  Assume that there exists a random variable $D$ with $P(D = 1) > 0$
  such that the following hold as $n \to \infty$.
  \begin{enumerate}
    \item  \label{assum item: the existence of a rv D}
          It holds that $D_{n} \xrightarrow{\mathrm{d}} D$.
    \item
          Convergence of third moments (and hence all lower moments):
          \begin{equation}
            E(D_{n}^{3})
            =\frac{1}{n} \sum_{v \in [n]} d_{v}^{3}
            \to
            E(D^{3}) < \infty.
          \end{equation}
    \item  \label{assum item: we are in the critical window}
          We are in the critical scaling window, i.e.,
          there exists $\lambda \in \mbR$ such that
          \begin{equation}
            \frac{\sum_{v \in [n]} d_{v} (d_{v}-1)}{\sum_{v \in [n]} d_{v}}
            =
            1 + \frac{\lambda}{n^{1/3}} + o(n^{-1/3}).
          \end{equation}
          In particular, it holds that $E(D^{2}) = 2  E(D)$.
  \end{enumerate}
\end{assum}

Suppose that $\mathscr{C}_{1}^{n}, \mathscr{C}_{2}^{n}, \ldots$ are the connected components of $\CMnd$
in decreasing order of size.
The starting point for establishing the convergence of connected components $\mathscr{C}_{i}^{n}$ of $\CMnd$
is understanding the behavior of the component sizes,
and this was done in \cite{Dhara_Hofstad_Leeuwaarden_Sen_17_Critical}.
We first set up some notation.
Let $\bd=\bd^{(n)}$ be a degree sequence satisfying Assumption \ref{assum: degree sequence of configuration models}
and $D$ be the random variable in
Assumption \ref{assum: degree sequence of configuration models}.
Define $\sigma_{r} \coloneqq E(D^{r})$ for $r=1,2,3$
and $\bm{\mu} = (\alpha, \eta, \beta)$ by setting
\begin{equation}  \label{eq: parameters for CR models}
  \alpha \coloneqq \sigma_{1}, \quad
  \eta \coloneqq \sigma_{3} \sigma_{1} - \sigma_{2}^{2}, \quad
  \beta \coloneqq \sigma_{1}^{-1}.
\end{equation}
Let $\lambda$ be the parameter
in Assumption \ref{assum: degree sequence of configuration models}\eqref{assum item: we are in the critical window}.
Write
\begin{equation}
  W^{\bm{\mu}, \lambda}(s)
  \coloneqq
  \frac{\sqrt{\eta}}{\alpha} W(s) + \lambda s - \frac{\eta s^{2}}{2 \alpha^{3}},
  \quad
  s \geq 0,
\end{equation}
where  $(W(s), s \geq 0)$ denotes a standard Brownian motion.
We then define $\bm{Z}=(Z_{1}, Z_{2}, \ldots)$
to be the ordered sequence of lengths of excursions
of the reflected process $W^{\bm{\mu}, \lambda}(s) - \min_{0 \leq u \leq s} W^{\bm{\mu}, \lambda}(u)$.
We further define $\bm{S}=(S_{1}, S_{2}, \ldots)$
by setting $S_{i}$ to be the number of points of a Poisson point process on $\mbRp^{2}$ with rate $\beta$
lying under the corresponding excursions,
where the Poisson point process is assumed to be independent of $(W^{\bm{\mu}, \lambda}(s))_{s \geq 0}$.

\begin{thm} [{\cite[Theorem 2]{Dhara_Hofstad_Leeuwaarden_Sen_17_Critical}}]
  Let $S_{i}^{n}$ be the number of surplus edges in $\mathscr{C}_{i}^{n}$, i.e.,
  \begin{equation}
    S_{i}^{n}
    \coloneqq
    |E(\mathscr{C}_{i}^{n})| - |V(\mathscr{C}_{i}^{n})| +1,
  \end{equation}
  where $V(\mathscr{C}_{i}^{n})$ and $E(\mathscr{C}_{i}^{n})$ denote
  the vertex set and the edge set of $\mathscr{C}_{i}^{n}$, respectively.
  If degree sequences $\bd=\bd^{(n)}$ satisfy Assumption \ref{assum: degree sequence of configuration models},
  then it holds that
  \begin{equation}
    \left(
    (n^{-2/3} |V(\mathscr{C}_{i}^{n})|)_{i \geq 1}, (S_{i}^{n})_{i \geq 1}
    \right)
    \xrightarrow{\mathrm{d}}
    (\bm{Z}, \bm{S}),
  \end{equation}
  where the convergence of the first sequence takes place in $l^{2}_{\searrow}$
  (recall this space from Theorem \ref{thm: convergence of sizes and surpluses of ER}).
\end{thm}

Conditional on being simple,
the configuration model $\CMnd$ has the same distribution as $\scrGnd$
(cf.\ \cite[Proposition 7.7]{Hofstad_17_Random}).
Under Assumption \ref{assum: degree sequence of configuration models},
it is verified that the conditioning is weak enough
to conclude that the asymptotic behavior of $\scrGnd$ and $\CMnd$ is almost the same,
in the sense that each connected component of $\CMnd$ is simple with probability tending to $1$
(see Lemma \ref{lem: behavior of degree sequences of CR models}).
Based on this fact,
in \cite{Bhamidi_Sen_20_Geometry},
the convergence of connected components of $\scrGnd$ and $\CMnd$
equipped with the graph distance and the uniform probability measure was established.
In this section,
we will show that the convergence still holds if the metric is replaced by the resistance metric,
and, moreover the convergence of stochastic processes and their local times on the configuration models
also holds.
Write
\begin{equation}
  \mathcal{G}_{n}
  \coloneqq
  (V(\mathscr{C}_{1}^{n}), n^{-1/3} R_{\mathscr{C}_{1}^{n}},
  \rho_{\mathscr{C}_{1}^{n}}, \mu_{\mathscr{C}_{1}^{n}}),
\end{equation}
where $R_{\mathscr{C}_{1}^{n}}$ is the resistance metric on $\mathscr{C}_{1}^{n}$,
$\rho_{\mathscr{C}_{1}^{n}}$ is the smallest-labeled vertex among the vertices with the smallest degree,
and $\mu_{\mathscr{C}_{1}^{n}}$ is the uniform probability measure on $V(\mathscr{C}_{1}^{n})$.

\begin{thm} \label{thm: main result for CR models}
  Assume that degree sequences $\bd=\bd^{(n)}$ satisfy Assumption \ref{assum: degree sequence of configuration models}.
  Then, there exists a random element $\mathcal{M}$ of $\check{\mathbb{F}}_{c}$
  such that $\mathcal{G}_{n} \xrightarrow{\mathrm{d}}\mathcal{M}$ in the Gromov-Hausdorff-Prohorov topology.
  Furthermore, it holds that $\cX_{\mathcal{G}_{n}} \xrightarrow{\mathrm{d}} \cX_{\mathcal{M}}$
  as random elements of $\mbM_{L}$.
  (See Construction \ref{const: the limiting space of CR models} for an explicit description of $\mathcal{M}$.)
\end{thm}

For each fixed $\tildm \geq 1$,
let $\btilddm =(\tildd_{1}^{(\tildm)}, \ldots, \tildd_{\tildm}^{(\tildm)})$
be a given degree sequence.
We will often suppress the superscript and write $\btildd,\, \tildd_{i}$ etc.
Let $\mbGdcon$ be the set of all connected, simple, labeled (by $[m]$) graphs
with degree sequence $\btildd$ where the vertex labeled $j$ has degree $\tildd_{j}$.
Write $\scrGdcon$ for a random graph sampled uniformly from $\mbGdcon$.
Since the distribution of a connected component of $\scrGnd$ conditional on the vertex set $V$
is the uniform distribution
on the space of all connected, simple random graph
on $V$ with degree sequence $\{d_{i}, i \in V\}$,
the main ingredient to obtain the desired convergence results
is the convergence result for $\scrGdcon$.
Consider the following assumption on the sequence $(\btilddm)_{\tildm \geq 1}$.

\begin{assum} [{\cite[Assumption 2.3]{Bhamidi_Sen_20_Geometry}}]
  \label{assum: degree sequence tildd} \leavevmode
  \begin{enumerate}
    \item
          It holds that $\tildd_{j} \geq 1$ for $1 \leq j \leq \tildm$ and $\tildd_{1}=1$.
    \item
          There exists a probability mass function (p.m.f.) $(\tildp_{1}, \tildp_{2}, \ldots)$ with
          \begin{equation}
            \tildp_{1} >0,
            \quad
            \sum_{i \geq 1} i \tildp_{i} =2,
            \quad \text{and} \quad
            \sum_{i \geq 1} i^{2} \tildp_{i} < \infty
          \end{equation}
          such that
          \begin{equation}
            \frac{1}{\tildm} | \{ j : \tildd_{j} =i \}| \to \tildp_{i}
            \ \text{for}\
            i \geq 1,
            \ \text{and} \quad
            \frac{1}{\tildm} \sum_{i \geq 1} \tildd_{i}^{2} \to \sum_{i \geq 1} i^{2} \tildp_{i}.
          \end{equation}
          (Note that these yield that $\max_{1 \leq j \leq \tildm} \tildd_{j} = o (\sqrt{\tildm})$.)
    \item \label{assum item: surplus of the given degree sequence, tilted}
          There exists a non-negative integer $k$
          such that $\sum_{i \geq 1} \tildd_{i} =2 (\tildm -1) + 2 k$
          for all sufficiently large $\tildm$.
  \end{enumerate}
  We will write $\sigma^{2} \coloneqq \sum_{i} i^{2} \tildp_{i} -4$
  for the variance associated with the p.m.f.\ $(p_{0}, p_{1}, \ldots)$.
\end{assum}

To give a construction of $\scrGdcon$,
we introduce some notions on plane trees.
Let $\bt$ be a plane tree with a root vertex $\rho$.
Write $\scrL(\bt)$ for the set of leaves of $\bt$, i.e.,
the vertices that have no children.
For each non-root vertex $u \in \bt$,
let $\parent(u)$ be the parent of $u$.
If $\parent(u)$ is not the root,
then we set $\gparent(u) := \parent(\parent(u))$,
i.e., $\gparent(u)$ is the grandparent of $u$.
Let $[\rho, u]$ (resp.\ $[\rho, u)$) denote the ancestral line of $u$
including (resp.\ excluding) $u$.
Write $\DFprec$ for the order on vertices of a plane tree
induced by a depth-first search,
i.e., $x \DFprec y$ if $x$ is explored strictly before $y$
in a depth-first search of the plane tree.

\begin{dfn} {(\cite[Definition 3.1]{Bhamidi_Sen_20_Geometry})}
  For leaves $u, v \in \scrL(\bt)$,
  we say that the ordered pair $(u,v)$ is \textit{admissible},
  if $\parent(v) \neq \rho$, $\gparent(v) \in [\rho, \parent(u))$,
  and $\parent(u) \DFprec \parent(v)$.
  Let $\bA(\bt)$ denote the set of admissible pairs of $\bt$.
\end{dfn}

We define an order $\ll$ on $\bA(\bt)$ as follows:
For $(u_{1}, v_{1}), (u_{2}, v_{2}) \in \bA(\bt)$,
we write $(u_{1}, v_{1}) \ll (u_{2}, v_{2})$
if either $\parent(u_{1}) \DFprec \parent(u_{2})$ or
if $\parent(u_{1}) = \parent(u_{2})$ and $\parent(v_{1}) \DFprec \parent(v_{2})$.
Then, for each integer $k \geq 1$,
we write $\bA_{k}(\bt)$
for the set of ordered pairs $( (u_{1},v_{1}), \ldots, (u_{k}, v_{k}) )$
such that
\begin{equation}
  (u_{j}, v_{j}) \in \bA(\bt),\
  (u_{1}, v_{1}) \ll \cdots \ll (u_{k}, v_{k}), \
  \text{and}\
  u_{1}, v_{1}, \ldots, u_{k}, v_{k}\
  \text{are}\ 2k\ \text{distinct leaves}.
\end{equation}
We let $\bA_{k}^{\mathrm{ord}}(\bt)$ be the collection of all ordered pairs
obtained by shuffling admissible pairs of an element in $\bA_{k}(\bt)$.
(Clearly, we have $|\bA_{k}^{\mathrm{ord}}(\bt)| = k! |\bA_{k}(\bt)|$.)
For a leaf $u \in \scrL(\bt)$,
we set
\begin{equation}
  \bA(\bt, u)
  \coloneqq
  \{v \in \scrL(\bt) : (u,v) \in \bA(\bt)\}.
\end{equation}

\begin{rem} \label{rem: difference of the definitions of admissible pairs set}
  The set $\bA_{k}(\bt)$ is used in the algorithm generating $\scrGdcon$
  (see Theorem \ref{thm: algorithm generating the uniform random tree with a prescribed degree}).
  In \cite{Bhamidi_Sen_20_Geometry},
  $\bA_{k}(\bt)$ is defined
  to be the totality of sets of pairs $\{(u_{1}, v_{1}), \ldots, (u_{k},v_{k})\}$
  such that $(u_{j}, v_{j}) \in \bA(\bt),\, j=1, \ldots, k$ and
  $u_{1}, v_{1}, \ldots, u_{k}, v_{k}$ are $2k$ distinct leaves.
  However, the proof that the algorithm generates $\scrGdcon$
  (\cite[Section 6]{Bhamidi_Sen_20_Geometry})
  shows that our definition incorporating the order is the appropriate one.
\end{rem}

Given a plane tree $\bt$,
write $\bxi(\bt)=(\xi_{v}(\bt), v \in \bt)$,
where $\xi_{v}(\bt)$ is the number of children of $v$ in $\bt$.
We set $s_{i}(\bt)$ to be the number of vertices that have $i$ children, i.e.,
\begin{equation}
  s_{i}(\bt)
  \coloneqq
  | \{ v \in \bt : \xi_{v}(\bt) = i \} |,
\end{equation}
and write $\bs(\bt) \coloneqq (s_{i}(\bt), i \geq 0)$
for the \textit{child frequency distribution (CFD)}.
Note that $s_{0}(\bt)$ is the number of leaves of $\bt$.

\begin{dfn}
  A sequence of integers $\bs=(s_{i}, i \geq 0)$ is called
  a \textit{tenable} CFD,
  if there exists a finite plane tree $\bt$ with $\bs(\bt) = \bs$.
  Given a tenable CFD $\bs$,
  let $\mbTs$ denote the set of all plane trees having CFD $\bs$.
\end{dfn}

The following result is elementary and we omit the proof.

\begin{lem} \label{lem: necessary and sufficient condition for being tenable}
  A sequence of integers $\bs=(s_{i}, i \geq 0)$ is a tenable CFD
  for a tree with $N$ vertices
  if and only if $s_{i} \geq 0$ for all $i$ with $s_{0} \geq 1$,
  and
  \begin{equation}
    \sum_{i \geq 0} s_{i}
    =
    1 + \sum_{i \geq 0} i s_{i}
    =
    N
    <\infty.
  \end{equation}
\end{lem}

Given a tenable CFD $\bs$ and an integer $k \geq 1$,
let $\mbTsk$ denote the set of all pairs $(\bt, \bx)$,
where $\bt \in \mbTs$ and $\bx \in \bA_{k}(\bt)$.
Further, we write $\mbTs^{(0)} \coloneqq \mbTs$.

Now we give an algorithm for constructing $\scrGdcon$
(see also \cite[Section 3.2]{Bhamidi_Sen_20_Geometry}).
Let $\btildd=\btilddm$ be as in Assumption \ref{assum: degree sequence tildd}.
Recall that $\tildd_{1}=1$ and
$\sum_{i \geq 1} \tildd_{i} = 2(\tildm-1) + 2k$
for some fixed non-negative integer $k \geq 0$.
We first define a CFD $\bs$ via $\bd$ as follows.
Write $\bxi=(\xi_{j}, 2 \leq j \leq \tildm+2k)$
where we set
\begin{equation}
  \xi_{j}
  \coloneqq
  \begin{cases}
    \tildd_{j}-1, & 2 \leq j \leq \tildm,       \\
    0,            & \tildm< j \leq \tildm + 2k.
  \end{cases}
\end{equation}
Then,
we let $\bs= \bs(\btilddm) = (s_{i}, i \geq 0)$ be the frequency distribution of $\bxi$, i.e.,
$s_{i} \coloneqq | \{ j : \xi_{j} = i\} |$.
Since we have $1 + \sum_{i \geq 0} is_{i}  = \tildm-1+2k = \sum_{i \geq 0} s_{i}$,
by Lemma \ref{lem: necessary and sufficient condition for being tenable},
$\bs$ is a tenable CFD for a tree with $\tildm-1+2k$ vertices.
Consider the following algorithm:
\begin{const} \label{const: uniformly random graph on Gdcon}
  Fix $k \geq 0$.
  \begin{enumerate} [label=\textbf{\arabic*.}, itemsep=0.05cm, parsep=0.06cm]
    \item
          Sample $(\tildscrTs, \tildbX)$ from $\mbTsk$ uniformly,
          where we write
          $\tildbX = ( (\tilde{U}_{1}^{\tildm}, \tilde{V}_{1}^{\tildm}),
            \ldots, (\tilde{U}_{k}^{\tildm}, \tilde{V}_{k}^{\tildm}) )$.
          (If $k=0$, then
          we only sample $\tildscrTs$ from $\mbTs^{(0)} = \mbTs$ uniformly.)\
    \item
          Lebel $\tildm-1$ vertices of
          $\tildscrTs \setminus
            \{\tilde{U}_{1}^{\tildm}, \tilde{V}_{1}^{\tildm}, \ldots,
            \tilde{U}_{k}^{\tildm}, \tilde{V}_{k}^{\tildm} \}$
          uniformly using labels $2, \ldots, \tildm$
          so that, in the resulting labeled plane tree,
          $j$ has $\tildd_{j}-1$ many children.
    \item
          Delete $\tilde{U}_{j}^{\tildm}, \tilde{V}_{j}^{\tildm}$ and two edges incident to them
          and add an edge between $\parent(\tilde{U}_{j}^{\tildm})$ and $\parent(\tilde{V}_{j}^{\tildm})$
          for $j=1, \ldots, k$.
    \item
          Attach a vertex labeled $1$ to the root of $\tildscrTs$,
          and then forget about the planer order and the root.
          Let $\scrGtildd$ be the resulting graph on the vertex set $[\tildm]$.
  \end{enumerate}
\end{const}

\begin{thm} [{\cite[Theorem 3.2]{Bhamidi_Sen_20_Geometry}}]
  \label{thm: algorithm generating the uniform random tree with a prescribed degree}
  The random graph $\scrGtildd$ resulting from the above construction
  is uniformly distributed on $\mbGdcon$, i.e.,
  $\scrGtildd \overset{\mathrm{d}}{=} \scrGdcon$.
\end{thm}

\begin{rem}
  As mentioned in \cite[Remark 8]{Bhamidi_Sen_20_Geometry},
  if degree sequences $\bd = \bd^{(\tildm)}$ satisfy Assumption \ref{assum: degree sequence tildd},
  then $\mbTsk$ is non-empty for all sufficiently large $\tildm$,
  and it is possible to execute the above-mentioned algorithm.
\end{rem}
Using the above construction,
it is possible to obtain the convergence of $\scrGdcon$
from that of tilted trees $\tildscrTs$.
We consider a sequence of CFDs $\bs=\bs^{(\tildm)}=(s_{i}^{(\tildm)}, i \geq 0)$ satisfying the following assumption.
(We often omit the superscript and write $\bs$ and $s_{i}$.)

\begin{assum} [{\cite[Assumption 7.1]{Bhamidi_Sen_20_Geometry}}]
  \label{assum: CFD sequence}
  There exists a p.m.f.\ $(p_{0}, p_{1}, \ldots)$ with
  \begin{equation}
    p_{0} >0, \quad \sum_{i \geq 1} ip_{i}=1, \quad \sum_{i \geq 0} i^{2} p_{i} < \infty
  \end{equation}
  such that
  \begin{equation}
    \frac{s_{i}}{\tildm} \to p_{i}\ \text{for}\ i \geq 0,
    \quad \text{and}\
    \frac{1}{\tildm} \sum_{i \geq 0} i^{2} s_{i} \to \sum_{i \geq 0} i^{2} p_{i}.
  \end{equation}
  We will write $\sigma^{2} \coloneqq \sum_{i} i^{2} p_{i} -1$
  for the variance associated with the p.m.f.\ $(p_{0}, p_{1}, \ldots)$.
\end{assum}

Given a function $f \in \mathscr{E}$
(recall the excursion space $\mathscr{E}$ from Section \ref{subsubsec: introduction of trees}),
we write $(T_{f}, d_{T_{f}})$ for the real tree coded by $f$.
Let $\rho_{T_{f}}$ be the root of $T_{f}$ and
$\mu_{T_{f}}$ be the canonical measure on $T_{f}$.
Given a tenable CFD $\bs$,
let $\mathscr{T}_{\bs}$ denote a random tree sampled from $\mbTs$ uniformly.
Write $\rho_{\scrTs}$ for the root of $\scrTs$,
$d_{\scrTs}$ for the graph metric on $\scrTs$,
$\mu_{\scrTs}$ for the uniform probability measure on $\scrTs$,
and $\mu_{\scrTs}^{\scrL}$ for the uniform probability measure on the set of leaves of $\scrTs$.

\begin{thm} [{\cite[Theorem 7.2 and Lemma 7.3]{Bhamidi_Sen_20_Geometry}}]
  \label{8. thm: convergence of uniform trees with given degrees}
  Under Assumption \ref{assum: CFD sequence},
  it holds that
  \begin{equation}   \label{8. thm eq: convergence of uniform trees with given degrees}
    (\scrTs, \frac{\sigma}{\sqrt{\tildm}} d_{\scrTs}, \rho_{\scrTs}, \mu_{\scrTs}, \mu_{\scrTs}^{\scrL})
    \xrightarrow{\mathrm{d}}
    (T_{2e}, d_{T_{2e}}, \rho_{T_{2e}}, \mu_{T_{2e}}, \mu_{T_{2e}})
  \end{equation}
  in the extended Gromov-Hausdorff-type topology with two measures
  (we use the Prohorov metric for measures),
  where $e=(e(t), 0 \leq t \leq 1)$ is the normalized Brownian excursion.
\end{thm}

Let $(T, d_{T}, \rho_{T})$ be a rooted real tree.
Fix $x \in T \setminus \{ \rho_{T} \}$.
Write $\hght_{T}(x)$ for the height of $x$, i.e., $\hght_{T}(x) \coloneqq d_{T}(\rho_{T}, x)$,
and $p_{T}^{x} : [0, \hght_{T}(x)] \to T$ for the unique distance-preserving map
with $p_{T}^{x}(0) = \rho_{T}$ and $p_{T}^{x}( \hght_{T}(x)) =x$.
Define a path measure $\mu_{T}^{x}$
by setting
\begin{equation}
  \mu_{T}^{x}
  \coloneqq
  \frac{1}{\hght_{T}(x)}
  \Leb \circ (p_{T}^{x})^{-1},
\end{equation}
where $\Leb$ stands for the Lebesgue measure on $\mbR$.

We introduce similar notions for plane trees.
Let $\tau$ be a finite plane tree with the graph distance $d_{\tau}$ and the root $\rho_{\tau}$.
Write $\hght_{\tau}(x)$ for the height of $x$, i.e., $\hght_{\tau}(x) \coloneqq d_{\tau}(\rho_{\tau}, x)$.
Fix $x \in \tau$.
Let $\rho_{\tau}=x_{0}, x_{1}, \ldots, x_{j} = x$ be the shortest path from $\rho_{\tau}$ to $x$.
Then, we define a map $p_{\tau}^{x} : [0, \hght_{\tau}(x)] \to \tau$ by setting
\begin{equation}
  p_{\tau}^{x} (t)
  \coloneqq
  \begin{cases}
    x_{i}, & i \leq t < i+1,\ i=0,1,\ldots, j-1, \\
    x_{j}, & j-1 \leq t \leq j.
  \end{cases}
\end{equation}
Note that, for every $t \in [0, \hght_{\tau}(x)]$,
\begin{equation} \label{eq: quasi-distance-preserving of path map of plane trees, CR models}
  t-1 \leq d_{\tau}(\rho_{\tau}, p_{\tau}^{x}(t)) \leq t.
\end{equation}

\begin{prop} \label{prop: convergence of path measures}
  Let $\tau_{n}$ be a finite plane tree
  with the graph metric $d_{n}$ and the root $\rho_{n}$.
  Let $(T, d_{T}, \rho_{T})$ be a rooted real tree
  such that $(T, d_{T})$ is assumed to be compact.
  Fix an integer $k \geq 1$.
  Let $(U_{i}^{n})_{i=1}^{k}$ be vertices of $\tau_{n}$
  and $(U_{i})_{i=1}^{k}$ be elements of $T$.
  If there exist scaling factors $a_{n}>0$ such that
  $a_{n} \to 0$ and
  \begin{equation}
    (\tau_{n}, a_{n}d_{n}, \rho_{n}, (U_{i}^{n})_{i=1}^{k})
    \to
    (T, d_{T}, \rho_{T}, (U_{i})_{i=1}^{k})
  \end{equation}
  in the extended Gromov-Hausdorff-type topology
  with additional points,
  then it holds that
  \begin{equation}
    \left(\tau_{n}, a_{n} d_{n}, \rho_{n}, (U_{i}^{n})_{i=1}^{k},
    ( p_{\tau_{n}}^{U_{i}^{n}}(\cdot \hght_{\tau_{n}}(U_{i}^{n})) )_{i=1}^{k}
    \right)
    \to
    \left(T, d_{T}, \rho_{T}, (U_{i})_{i=1}^{k},
    ( p_{T}^{U_{i}}(\cdot \hght_{T}(U_{i})) )_{i=1}^{k}
    \right)
  \end{equation}
  in the extended Gromov-Hausdorff-type topology 
  with additional points and cadlag functions defined on $[0,1]$
  (we use the supremum metric for the functions, 
  see Remark \ref{8. rem: GH-type metric for cadlag curves with sup norm}).
\end{prop}

\begin{proof}
  We may assume that
  $(\tau_{n}, a_{n} d_{n})$ and $(T,d_{T})$ are embedded into a common compact metric space $(M, d^{M})$
  in such a way that
  $\tau_{n} \to T$ in the Hausdorff topology in $M$,
  $\rho_{n} \to \rho_{T}$ in $M$ and 
  $U_{i}^{n} \to U_{i}$ in $M$ for each $i$.
  Set
  \begin{equation}
    \varepsilon_{n}
    \coloneqq
    d_{H}^{M} (\tau_{n}, T) \vee
    d^{M}(\rho_{n}, \rho_{T}) \vee
    d_{P}^{M}(\nu_{n}, \nu_{T})
    \vee
    \left(
    \max_{1 \leq i \leq k} d^{M}(U_{i}^{n}, U_{i})
    \right),
  \end{equation}
  which converges to $0$.
  By the triangle inequality,
  we have that
  \begin{equation} \label{eq: height evaluation for two trees, CR models}
    | a_{n} \hght_{\tau_{n}} (U_{i}^{n}) - \hght_{T}(U) |
    =
    | d^{M}(\rho_{n}, U_{i}^{n}) - d^{M} (\rho_{T}, U_{i}) |
    \leq
    d^{M}(\rho_{n}, \rho_{T}) + d^{M}(U_{i}^{n}, U_{i})
    \leq 2 \varepsilon_{n}.
  \end{equation}
  Fix $t \in [0,1]$ and set
  $y_{n} \coloneqq p_{\tau_{n}}^{U_{i}^{n}} ( t \hght_{\tau_{n}}(U_{i}^{n}) ),\,
    y \coloneqq p_{T}^{U_{i}} ( t \hght_{T}(U_{i}) )$.
  Choose $z \in T$ such that $d^{M}(y_{n}, z) \leq \varepsilon_{n}$.
  We then have that
  \begin{equation}
    d^{M}(y_{n}, y)
    \leq
    d^{M}(y_{n}, z) + d^{M}(z, y)
    \leq
    \varepsilon_{n} + d_{T}(z,y).
  \end{equation}
  From the four-point condition (cf.\ \cite{Evans_08_Probability}),
  it follows that
  \begin{equation} \label{eq: four-point inequality, CR models}
    d_{T}(z,y) + d_{T} (\rho_{T}, U_{i})
    \leq
    \left( d_{T}(z, \rho_{T}) + d_{T}(y, U_{i})\right)
    \vee
    \left( d_{T}(z, U_{i}) + d_{T}(y, \rho_{T})\right).
  \end{equation}
  By \eqref{eq: quasi-distance-preserving of path map of plane trees, CR models}
  and \eqref{eq: height evaluation for two trees, CR models},
  we deduce that
  \begin{align}
    d_{T}(z, \rho_{T}) + d_{T}(y, U_{i})
     & \leq
    d^{M}(z, y_{n}) + a_{n} d_{\tau_{n}}(y_{n}, \rho_{n}) + d^{M}(\rho_{n}, \rho) + d_{T}(y, U_{i}) \notag       \\
     & \leq
    2 \varepsilon_{n} + t\cdot a_{n} \hght_{\tau_{n}}(U_{i}^{n}) + (1-t) \hght_{T}(U_{i}) \notag                 \\
     & \leq
    4 \varepsilon_{n} + d_{T}(\rho_{T}, U_{i}),                                                                  \\
    d_{T}(z,U_{i}) + d_{T}(z, \rho_{T})
     & \leq
    d^{M}(z,y_{n}) + a_{n} d_{\tau_{n}} (y_{n}, U_{i}^{n}) + d^{M}(U_{i}^{n}, U_{i}) + t \hght_{T}(U_{i}) \notag \\
     & \leq
    2 \varepsilon_{n} + (1-t) a_{n} \hght_{\tau_{n}}(U_{i}^{n}) + a_{n} + t \hght_{T}(U_{i}) \notag              \\
     & \leq
    4 \varepsilon_{n} + a_{n} + d_{T}(\rho_{T}, U_{i}).
  \end{align}
  Combining this with \eqref{eq: four-point inequality, CR models} yields that
  $d^{M}(y_{n}, y) \leq 5 \varepsilon_{n} + a_{n}$.
  Now the desired result is immediate.
\end{proof}

\begin{rem} \label{8. rem: GH-type metric for cadlag curves with sup norm}
  To define the topology used in Proposition \ref{prop: convergence of path measures},
  we use the theory of \cite{Noda_pre_Metrization}.
  In particular,
  for the notion of functors discussed below,
  see \cite[Section 3.2]{Noda_pre_Metrization}.
  For a compact metric space $(K, d^{K})$,
  set $\tau(K)$ to be the set of cadlag functions from $[0,1]$ to $K$ 
  equipped with the supremum norm.
  For compact metric spaces $K_{1}, K_{2}$
  and a distance-preserving map $f:K_{1} \to K_{2}$,
  set $\tau_{f}(\xi) \coloneqq \xi \circ f$.
  Then, 
  it is easy to check that the functor $\tau$ is a continuous functor.
  Hence,
  by considering the product of $k$-copies of $\tau$ 
  and the functor for points (see \cite[Section 4.2]{Noda_pre_Metrization})
  and applying \cite[Theorem 3.23]{Noda_pre_Metrization},
  we obtain a metric inducing the desired topology as follows:
  for $\mathcal{X}_{i} = (K_{i}, d^{K_{i}}, \rho_{i}, (a_{i,j})_{j=1}^{k}, (\xi_{i,j})_{j = 1}^{k}),\, i=1,2$
  where $(K_{i}, d^{K_{i}}, \rho_{i})$ is a rooted compact metric space,
  $(a_{i, j})_{j=1}^{k}$ are elements of $K_{i}$,
  and $(\xi_{i,j})_{j=1}^{k}$ are cadlag functions from $[0,1]$ to $K_{i}$,
  \begin{equation}  \label{8. rem eq: dfn of metric for the modified topology}
  \begin{split}
    d(\mathcal{X}_{1}, \mathcal{X}_{2})
    =
    &\inf_{M,f_{1}, g_{1}}
    \Bigl\{
      d_{H}^{M}(K_{1}, K_{2}) \vee d^{M}(f_{1}(\rho_{1}), f_{2}(\rho_{2})) 
      \vee 
      \max_{j}
        d^{M}(f_{1}(a_{1, j}), f_{2}(a_{2,j})) 
      \Bigr.\\
    &\qquad
    \Bigl.
      \vee
      \max_{j}
      \sup_{0 \leq t \leq 1} 
      d^{M}\left( f_{1} \circ \xi_{1,j}(t), f_{2} \circ \xi_{2,j}(t) \right)
    \Bigr\},
  \end{split}
  \end{equation}
  where the infimum is taken over all compact metric spaces $(M, d^{M})$
  and distance-preserving maps $f_{i}:K_{i} \to M,\, i=1,2$.
  We mention that although in the theory of \cite{Noda_pre_Metrization} 
  the spaces are assumed to be boundedly compact,
  it is easy to obtain corresponding results for compact spaces.
\end{rem}

Let $\bt$ be a finite plane tree with the root $\rho$.
For a leaf $x \in \scrL(\bt)$,
define a probability measure $\tildmu_{\bt}^{x}$ on $[\rho, x)$ by setting
\begin{equation}
  \tildmu_{\bt}^{x}(\{y\})
  \coloneqq
  \frac{1}{| \bA (\bt, x)|} | \{z \in \bA(\bt, x) : \gparent(z) = y\} |.
\end{equation}
If $|\bA(\bt, x)|=0$,
then we define $\tildmu_{\bt}^{x}$ to be the Dirac measure at $\rho$.
It is possible to describe $\tilde{\mu}_{\bt}^{x}$ as a pushforward measure.
We define a probability measure $\tilde{L}_{\bt}^{x}$ on $[0,1]$
with atoms $j/ \hght_{\bt}(x),\, j=0,1,\ldots, \hght_{\bt}(x)-1$ by setting
\begin{equation}
  \tilde{L}_{\bt}^{x} \left( \left\{ \frac{j}{\hght_{\bt}(x)} \right\} \right)
  \coloneqq
  \frac{1}{| \bA(\bt, x) |}
  | \{ z \in \bA(\bt, x) : \hght_{\bt}(\gparent(z)) = j\} |.
\end{equation}
When $|\bA(\bt, x)|=0$,
we define $\tilde{L}_{\bt}^{x}$ to be the Dirac measure at $\rho$.
Then, it holds that
\begin{equation}
  \tilde{\mu}_{\bt}^{x}
  =
  \tilde{L}_{\bt}^{x} \circ \left( p_{\bt}^{x} ( \cdot \hght_{\bt}(x) ) \right)^{-1}.
\end{equation}
The probability measure $\tildmu_{\bt}^{x}$ is tilted
in a sense that
it gives more mass to vertices that have more grandchildren
such that they form an admissible pair with vertex $x$.
An important observation is that,
in our setting,
the tilted probability measure $\tildmu_{\scrTs}^{x}$
converges to a path measure on the continuous random tree $T_{2e}$.

\begin{lem} {(\cite[Lemma 7.8]{Bhamidi_Sen_20_Geometry})}
  \label{lem: path function convergence for the configuration model}
  Under Assumption \ref{assum: CFD sequence},
  it holds that
  \begin{equation}
    \frac{1}{\sqrt{\tildm}}
    \max_{0 \leq k \leq \hght_{\scrTs}(U^{\tildm})}
    \left|
    \left| \bA(\scrTs, U^{\tildm}) \right|
    \tilde{L}_{\scrTs}^{U^{\tildm}} \left( [0, k/ \hght_{\scrTs}(U^{\tildm}) ] \right)
    -
    \frac{p_{0} \sigma^{2} k}{2}
    \right|
    \xrightarrow{\mathrm{p}}
    0,
  \end{equation}
  where $U^{\tildm}$ is a random leaf of $\scrTs$
  sampled according to $\mu_{\scrTs}^{\scrL}$.
\end{lem}

\begin{lem} \label{lem: convergence of Ts with path measures CR models}
  Suppose that Assumption \ref{assum: CFD sequence} is satisfied.
  Sample random leaves $(U_{i}^{\tildm})_{i=1}^{k}$
  in an i.i.d.\ fashion according to $\mu_{\scrTs}^{\scrL}$
  and $(U_{i})_{i=1}^{k}$ according to $\mu_{T_{2e}}$.
  It then holds that
  \begin{equation}
    \begin{split}
      &\left(\scrTs, \frac{\sigma}{\sqrt{\tildm}} d_{\scrTs}, \rho_{\scrTs}, \mu_{\scrTs},
      (U_{i}^{\tildm})_{i=1}^{k}, ( \tildmu_{\scrTs}^{U_{i}^{\tildm}} )_{i=1}^{k},
      \left(\frac{1}{\sqrt{\tildm}} | \bA(\scrTs, U_{i}^{\tildm}) | \right)_{i=1}^{k}
      \right)\\
      \xrightarrow{\mathrm{d}}
      &\left(T_{2e}, d_{T_{2e}}, \rho_{T_{2e}}, \mu_{T_{2e}},
      (U_{i})_{i=1}^{k}, ( \mu_{T_{2e}}^{U_{i}} )_{i=1}^{k},
      \left( \frac{p_{0}\sigma}{2} \hght_{T_{2e}} (U_{i}) \right)_{i=1}^{k}
      \right),
    \end{split}
  \end{equation}
  where the convergence of the first six coordinates
  takes place in the Gromov-Hausdorff-Prohorov topology with additional points and measures,
  and that of the last coordinate takes place in $\mbR^{k}$.
\end{lem}

\begin{proof}
  Theorem \ref{8. thm: convergence of uniform trees with given degrees} implies that 
  \begin{equation} \label{eq: convergence of path measures of Ts, CR models}
    \left(
      \scrTs, \frac{\sigma}{\sqrt{\tildm}} d_{\scrTs}, \rho_{\scrTs}, \mu_{\scrTs}, (U_{i}^{\tildm})_{i=1}^{k}
    \right)
    \xrightarrow{\mathrm{d}}
    \left(
      T_{2e}, d_{T_{2e}}, \rho_{T_{2e}}, \mu_{T_{2e}}, (U_{i})_{i=1}^{k}
    \right)
  \end{equation}
  in the extended Gromov-Hausdorff-type topology with measures and additional points.
  We may assume that, on some probability space $(\Omega, \mathcal{F}, P)$,
  the above convergence takes place almost-surely.
  Then, by Proposition \ref{prop: convergence of path measures},
  we have that, with probability $1$,
  \begin{equation} \label{eq: convergence of path measures of Ts, CR models}
    \begin{split}
      &\left(\scrTs, \frac{\sigma}{\sqrt{\tildm}} d_{\scrTs}, \rho_{\scrTs}, \mu_{\scrTs},
      (U_{i}^{\tildm})_{i=1}^{k}, ( p_{\scrTs}^{U_{i}^{\tildm}}(\cdot \hght_{\scrTs}(U_{i}^{n})) )_{i=1}^{k}
      \right)\\
      \to
      &\left(T_{2e}, d_{T_{2e}}, \rho_{T_{2e}}, \mu_{T_{2e}},
      (U_{i})_{i=1}^{k}, ( p_{T_{2e}}^{U_{i}}(\cdot \hght_{T_{2e}}(U_{i})) )_{i=1}^{k}
      \right)
    \end{split}
  \end{equation}
  in the extended Gromov-Hausdorff-type topology with measures, additional points
  and cadlag curves defined on $[0,1]$ 
  (we use the supremum metric for the functions).
  This immediately yields that
  \begin{equation}  \label{eq: convergence of height of trees Ts, CR models}
    \frac{1}{\sqrt{\tildm}}
    \frac{p_{0} \sigma^{2} \hght_{\scrTs}(U_{i}^{\tildm})}{2}
    \xrightarrow{\mathrm{a.s.}}
    \frac{p_{0} \sigma}{2} \hght_{T_{2e}}(U_{i}),
    \quad
    \forall i = 1,2,\ldots,k.
  \end{equation}
  Fix a subsequence $(\tildm_{l})_{l}$.
  By Lemma \ref{lem: path function convergence for the configuration model},
  we can find a further subsequence $(\tildm_{l_{k}})_{k}$ such that,
  for each $i =1, 2, \ldots, k$,
  \begin{equation}  \label{eq: almost-surely convergence of measures for path measures, CR models}
    \frac{1}{\sqrt{\tildm}}
    \max_{0 \leq k \leq \hght_{\scrTs}(U_{i}^{\tildm})}
    \left|
    \left| \bA(\scrTs, U_{i}^{\tildm}) \right|
    \tilde{L}_{\scrTs}^{U_{i}^{\tildm}} \left( [0, k/ \hght_{\scrTs}(U_{i}^{\tildm}) ] \right)
    -
    \frac{p_{0} \sigma^{2} k}{2}
    \right|
    \xrightarrow{\mathrm{a.s.}}
    0.
  \end{equation}
  Combining this with \eqref{eq: convergence of height of trees Ts, CR models} immediately yields that
  \begin{equation} \label{eq: convergence of cardinality of admissible pairs of Ts, CR models}
    \frac{1}{\sqrt{\tildm_{l_{k}}}}
    | \bA (\scrTs, U_{i}^{\tildm_{l_{k}}}) |
    \xrightarrow{\mathrm{a.s.}}
    \frac{p_{0} \sigma^{2}}{2} \hght_{T_{2e}}(U_{i}^{\tildm_{l_k}}),
    \quad
    \forall i=1,2,\ldots, k.
  \end{equation}
  By \eqref{eq: convergence of height of trees Ts, CR models},
  \eqref{eq: almost-surely convergence of measures for path measures, CR models}
  and \eqref{eq: convergence of cardinality of admissible pairs of Ts, CR models},
  we deduce that
  \begin{equation}
    \left|
    \tilde{L}_{\scrTs}^{U_{i}^{\tildm_{l_{k}}}}
    \left( [ 0, j / \hght_{\scrTs}(U_{i}^{\tildm_{l_{k}}}) ]  \right)
    -
    \frac{j}{\hght_{\scrTs}(U_{i}^{\tildm_{l_{k}}})}
    \right|
    \xrightarrow{\mathrm{a.s.}} 0,
    \quad
    \forall j=0,1,\ldots, \hght_{\scrTs}(U_{i}^{\tildm_{l_{k}}}),
  \end{equation}
  which implies that
  $\tilde{L}_{\scrTs}^{U_{i}^{\tildm_{l_{k}}}}$
  converges to the Lebesgue measure on $[0,1]$,
  almost-surely.
  Combining this with
  $p_{\scrTs}^{U_{i}^{\tildm}} \left(\cdot \hght_{\scrTs}(U_{i}^{\tildm}) \right)
    \to
    p_{T_{2e}}^{U_{i}} \left(\cdot \hght_{T_{2e}}(U_{i}) \right)$,
  which follows from \eqref{eq: convergence of path measures of Ts, CR models},
  we deduce that
  $\tilde{\mu}_{\scrTs}^{U_{i}^{\tildm_{l_{k}}}} \to \tilde{\mu}_{T_{2e}}^{U_{i}}$.
  Since the subsequence $(\tildm_{l})_{l}$ was chosen arbitrarily,
  we obtain the desired result.
\end{proof}

\begin{lem} {(\cite[Lemma 7.3]{Bhamidi_Sen_20_Geometry})}
  \label{lem: technical lemma for configuration models}
  Under Assumption \ref{assum: CFD sequence},
  the following assertions hold.
  \begin{enumerate}
    \item \label{lem item: sup of A is bounded CR models}
          Let $U_{\tildm}$ be uniformly distributed over $\scrL(\scrTs)$ conditional on $\scrTs$.
          Then, for every $k \geq 1$, it holds that
          \begin{equation}
            \sup_{\tildm}
            \bfE
            \left(
            \left(
            \frac{|\bA(\scrTs)|} {s_{0} \sqrt{\tildm}}
            \right)^{k}
            \right)
            \leq
            \sup_{\tildm}
            \bfE
            \left(
            \left(
            \frac{|\bA(\scrTs, U_{\tildm})|} {\sqrt{\tildm}}
            \right)^{k}
            \right)
            < \infty.
          \end{equation}
    \item \label{lem item: closeness of versions of admissible pairs sets}
          For every $k \geq 1$, it holds that
          \begin{equation}
            \frac{1}{\tildm^{3k/2}}
            (
            |\bA(\scrTs)|^{k} - |\bA_{k}^{\mathrm{ord}}(\scrTs)|
            )
            \xrightarrow{\mathrm{p}}
            0
          \end{equation}
  \end{enumerate}
\end{lem}

\begin{rem}
  As mentioned in Remark \ref{rem: difference of the definitions of admissible pairs set},
  our definition of $\bA_{k}(\bt)$ is slightly different from that of \cite{Bhamidi_Sen_20_Geometry},
  say $\tilde{\bA}_{k}(\bt)$.
  However, their asymptotic behaviors coincide.
  Indeed, in a similar way to \cite[Proof of Lemma 7.3 (iii)]{Bhamidi_Sen_20_Geometry},
  one can check that, for every $k \geq 1$,
  \begin{equation}
    \tildm^{-3k/2}
    \left(
    |\bA_{k}(\scrTs)| - |\tilde{\bA}_{k}(\scrTs)|
    \right)
    \xrightarrow{L^{1}}
    0,
  \end{equation}
  which is enough to obtain
  Lemma \ref{lem: technical lemma for configuration models}
  \eqref{lem item: closeness of versions of admissible pairs sets}
  from the corresponding result of $\tilde{\bA}_{k}(\scrTs)$
  (i.e. \cite[Lemma 7.3 (iii) ]{Bhamidi_Sen_20_Geometry}).
\end{rem}

Recall that $e=(e(t), 0 \leq t \leq 1)$ denotes the normalized Brownian excursion.
Write $\upsilon$ for the law of the normalized Brownian excursion on $\mathscr{E}$,
the space of excursions (recall it from Section \ref{subsubsec: introduction of trees}).
For $k \geq 0$,
let $\tilde{e}_{(k)}$ be a random excursion with distribution $\tilde{\upsilon}_{k}$
given via the following Radon-Nikodym density with respect to $\upsilon$:
\begin{equation}
  \frac{d \tilde{\upsilon}_{k}} {d \upsilon} (h)
  \coloneqq
  \frac{ (\int_{0}^{1} h(u) du)^{k}} {\bfE \left( (\int_{0}^{1} e(u) du) ^{k} \right)},
  \quad
  h \in \mathscr{E}.
\end{equation}
Fix $k \geq 0$.
We define a random rooted-and-measured resistance metric space $(\Mk, R_{\Mk}, \rho_{\Mk}, \mu_{\Mk})$
as follows.
\begin{const} [{cf.\ \cite[Construction 5.5]{Bhamidi_Sen_20_Geometry}}]
\label{const: the limiting space of CR models}
  Fix $k \geq 1$.
  \begin{enumerate} [label=\textbf{\arabic*.}, itemsep=0.03cm, parsep=0.06cm]
    \item
          Let $(\Ttildek, d_{\Ttildek}, \rho_{\Ttildek})$ be the rooted real tree coded by $2\tilde{e}_{(k)}$.
          Write $\mu_{\Ttildek}$ for the canonical measure on $\Ttildek$.
    \item
          Given $\Ttildek$, sample $k$ leaves $(\tilde{U}_{i})_{i=1}^{k}$
          in an i.i.d.\ fashion from $\Ttildek$
          with density proportional to $\hght_{\Ttildek}(x) \mu_{\Ttildek}(dx)$.
    \item
          Conditional on the two steps above,
          for each of the sampled leaves $\tilde{U}_{i}$,
          sample a vertex $\tilde{V}_{i}$
          according to the path measure $\mu_{\Ttildek}^{\tilde{U}_{i}}$.
    \item
          Define $(\Mk, R_{\Mk}, \rho_{\Mk}, \mu_{\Mk})$
          by fusing $(\Ttildek, d_{\Ttildek}, \rho_{\Ttildek}, \mu_{\Ttildek})$
          over $(\{\tilde{U}_{i}, \tilde{V}_{i}\})_{i=1}^{k}$.
  \end{enumerate}
  When $k=0$, we define
  $(M^{(0)}, R_{M^{(0)}}, \rho_{M^{(0)}}, \mu_{M^{(0)}}) \coloneqq (T_{2e}, d_{T_{2e}}, \rho_{T_{2e}}, \mu_{T_{2e}})$.

\end{const}

Given a finite connected graph $G$ with labeled vertices,
we write $V_{G}$ for the vertex set of $G$,
$R_{G}$ for the resistance metric on $V_{G}$,
$\rho_{G}$ for the smallest-labeled vertex among the vertices with the smallest degree,
and $\mu_{G}$ for the uniform probability measure on $V_{G}$.
If degree sequences $\bd^{(\tildm)}$ satisfies Assumption \ref{assum: degree sequence tildd},
then the associated CFD sequence $\bs=\bs(\bd^{(\tildm)})$ satisfies Assumption \ref{assum: CFD sequence}
with a p.m.f.\ $p_{i} \coloneqq \tildp_{i+1},\, i \geq 0$.

\begin{figure}[htbp]
  \centering 
  \begin{minipage}[htbp]{0.3\linewidth}
    \centering
    \begin{tikzpicture}
      \coordinate(v1)at(0,0);
      \coordinate(v2)at(1,1);
      \coordinate[label=below right:$\parent(\tilde{V}_{j}^{\tilde{m}})$](v3)at(2,2);
      \coordinate[label=above: $\tilde{V}_{j}^{\tilde{m}}$](v4)at(3,3);
      \coordinate[label=below left:$\parent(\tilde{U}_{j}^{\tilde{m}})$](v5)at(0,2);
      \coordinate[label=above: $\tilde{U}_{j}^{\tilde{m}}$](v6)at(0,3);
      \coordinate(v7)at(1,3);

      \draw(v1)--(v2)--(v3)--(v4);
      \draw(v2)--(v5)--(v6);
      \draw(v3)--(v7);

      \foreach\P in{v2,v3,v4,v5,v6,v7} \fill[black](\P)circle(0.1);
    \end{tikzpicture}
  \end{minipage}
  \begin{minipage}[htbp]{0.28\linewidth}
    \centering
    \begin{tikzpicture}
      \coordinate(v1)at(0,0);
      \coordinate(v2)at(1,1);
      \coordinate(v3)at(2,2);
      \coordinate(v4)at(3,3);
      \coordinate(v4')at(2.93,2.93);
      \coordinate(v5)at(0,2);
      \coordinate(v6)at(0,3);
      \coordinate(v6')at(0,2.93);
      \coordinate(v7)at(1,3);

      \draw(v1)--(v2)--(v3);
      \draw[dotted](v3)--(v4');
      \draw(v2)--(v5);
      \draw[dotted](v5)--(v6');
      \draw(v3)--(v7);
      \draw(v5)--(v3);

      \foreach\P in{v2,v3,v5,v7} \fill[black](\P)circle(0.1);
      \draw[dotted](3,3)circle[radius=0.1];
      \draw[dotted](0,3)circle[radius=0.1];
    \end{tikzpicture}
  \end{minipage}
  \begin{minipage}[htbp]{0.23\linewidth}
    \centering
    \begin{tikzpicture}
      \coordinate(v1)at(0,0);
      \coordinate(v2)at(1,1);
      \coordinate(v3)at(1,2);
      \coordinate(v4)at(0,3);
      \coordinate(v4')at(0,2.9);
      \coordinate(v5)at(1,3);
      \coordinate(v6)at(2,3);
      \coordinate(v6')at(2,2.9);

      \draw(v1)--(v2);
      \draw(v3)--(v5);
      \draw(v2)to[out=120, in=240](v3);
      \draw(v2)to[out=60, in=300](v3);
      \draw[dotted](v3)--(v4');
      \draw[dotted](v3)--(v6');

      \foreach\P in{v2,v3,v5} \fill[black](\P)circle(0.1);
      \draw[dotted](0,3)circle[radius=0.1];
      \draw[dotted](2,3)circle[radius=0.1];
    \end{tikzpicture}
  \end{minipage}
  \caption{From left to right, a sampled tree $\tildscrTs$, $\scrGdcon$ and 
  the space obtained by fusing over $( \{ \parent(\tilde{U}_{j}^{\tildm}), \parent(\tilde{V}_{j}^{\tildm}) \} )_{j=1}^{k}$.}
  \label{fig: graphs for configuration model}
\end{figure}

\begin{prop}  \label{prop: convergence of Gdcon}
  Under Assumption \ref{assum: degree sequence tildd}, it holds that
  \begin{equation}
    \left(
    V(\scrGdcon), \frac{\sigma}{\sqrt{\tildm}} R_{\scrGdcon}, \rho_{\scrGdcon}, \mu_{\scrGdcon}
    \right)
    \xrightarrow{\mathrm{d}}
    (\Mk, R_{\Mk}, \rho_{\Mk}, \mu_{\Mk})
  \end{equation}
  in the Gromov-Hausdorff-Prohorov topology,
  where $k$ is the non-negative integer
  in Assumption \ref{assum: degree sequence tildd}\eqref{assum item: surplus of the given degree sequence, tilted}.
\end{prop}

\begin{proof}
  When $k = 0$,
  the result is obvious by Theorem \ref{8. thm: convergence of uniform trees with given degrees}.
  Assume that $k \geq 1$.
  Sample $(\tildscrTs, \tildbX)$ from $\mbTsk$ uniformly,
  where we write
  $\tildbX = ( (\tilde{U}_{1}^{\tildm}, \tilde{V}_{1}^{\tildm}),
    \ldots, (\tilde{U}_{k}^{\tildm}, \tilde{V}_{k}^{\tildm}) )$.
  We then sample a bijection $\tau$ from $\{1,2,\ldots,k\}$ to itself
  from the symmetric group $\mathfrak{S}_{k}$ uniformly.
  Set $\hat{U}_{j}^{\tildm} \coloneqq \tilde{U}_{\tau(j)}^{\tildm},
    \hat{V}_{j}^{\tildm} \coloneqq \tilde{V}_{\tau(j)}^{\tildm}$ and
  $\hat{\bm{X}} \coloneqq
    (\hat{U}_{1}^{\tildm}, \hat{V}_{1}^{\tildm}, \ldots,
    \hat{U}_{k}^{\tildm}, \hat{V}_{k}^{\tildm})$.
  Let $T_{\tildm}$ be the plane tree
  obtained by deleting $\hat{U}_{j}^{\tildm}, \hat{V}_{j}^{\tildm}$ and two edges incident to them
  from $\tildscrTs$,
  attaching a vertex $\rho_{T_{\tildm}}$ to the root of $\tildscrTs$
  and regarding $\rho_{T_{\tildm}}$ as a new root.
  Write $d_{T_{\tildm}}$ for the graph metric on $T_{\tildm}$
  and $\mu_{T_{\tildm}}$ for the uniform probability measure on $T_{\tildm}$.
  Note that,
  by Theorem \ref{thm: algorithm generating the uniform random tree with a prescribed degree},
  we obtain $\scrGdcon$ by adding an edge between $\parent(\hat{U}_{j}^{\tildm})$ and $\parent(\hat{V}_{j}^{\tildm})$
  for $j=1,2, \ldots, k$
  (see Figure \ref{fig: graphs for configuration model}).
  Also, observe that the Gromov-Hausdorff-Prohorov distance, which is modified to deal with additional points,
  between
  \begin{equation}
    \left(
    \tildscrTs, \frac{\sigma}{\sqrt{\tildm}} d_{\tildscrTs}, \rho_{\tildTs}, \mu_{\tildTs},
    \hat{U}_{1}^{\tildm}, \gparent(\hat{V}_{1}^{\tildm}), \ldots,
    \hat{U}_{k}^{\tildm}, \gparent(\hat{V}_{k}^{\tildm})
    \right)
  \end{equation}
  and
  \begin{equation}
    \left(
    T_{\tildm}, \frac{\sigma}{\sqrt{\tildm}} d_{T_{\tildm}}, \rho_{T_{\tildm}}, \mu_{T_{\tildm}},
    \parent(\hat{U}_{1}^{\tildm}), \parent(\hat{V}_{1}^{\tildm}), \ldots,
    \parent(\hat{U}_{k}^{\tildm}), \parent(\hat{V}_{k}^{\tildm})
    \right)
  \end{equation}
  converges to $0$.
  Let $f$ be a bounded continuous function
  on the space of rooted-and-measured compact metric spaces with additional points.
  For $\bt\in \mbTs$ and $\bx = (u_{1},v_{1}, \ldots, u_{k}, v_{k}) \in \bt^{2k}$,
  we set
  \begin{equation}
    G(\bt, \bx)
    =
    G\left(\bt, (u_{i}, v_{i})_{i=1}^{k}\right)
    \coloneqq
    \left(
    \bt, \frac{\sigma}{\sqrt{\tildm}} d_{\bt}, \rho_{\bt}, \mu_{\bt}, \bx
    \right).
  \end{equation}
  Write $\hat{\bm{Y}} \coloneqq (\hat{U}_{1}^{\tildm}, \gparent(\hat{V}_{1}^{\tildm}), \ldots,
    \hat{U}_{k}^{\tildm}, \gparent(\hat{V}_{k}^{\tildm}))$.
  We then deduce that
  \begin{equation} \label{eq: expectation expression by using the Ts CR models}
    \bfE \left( f (G(\scrTs, \tilde{\bm{Y}})) \right)
    =
    \frac{\sum_{\bt \in \mbTs} \sum_{\bx \in \bA_{k}^{\mathrm{ord}}(\bt)} f(G(\bt, \bx))}
    {\sum_{\bt \in \mbTs} | \bA_{k}^{\mathrm{ord}} (\bt) |}
    =
    \frac{\bfE \left( \sum_{\bx \in \bA_{k}^{\mathrm{ord}}(\scrTs)} f(G(\scrTs, \bx))\right)}
    {\bfE ( |\bA_{k}^{\mathrm{ord}} (\scrTs) |)}.
  \end{equation}
  By \cite[Equation (8.8)]{Bhamidi_Sen_20_Geometry},
  we have that
  \begin{equation}  \label{eq: convergence of bAord from Bhamidi and Sen}
    \frac{ \bfE ( | \bA_{k}^{\mathrm{ord}} (\scrTs) | ) }
    {s_{0}^{k} \tildm^{k/2}}
    \to
    \left( \frac{p_{0} \sigma}{2} \right)^{k}
    \bfE
    \left(
    \left( \int_{T_{2e}} \hght_{T_{2e}}(x) \mu_{T_{2e}} (dx) \right)^{k}
    \right).
  \end{equation}
  Lemma \ref{lem: technical lemma for configuration models}
  \eqref{lem item: closeness of versions of admissible pairs sets}
  yields that
  \begin{equation}  \label{eq: passing ord to product sets CR models}
    \lim_{ \tildm \to \infty}
    \frac{1}{s_{0}^{k} \tildm^{k/2}}
    \bfE
    \left(
    \sum_{\bx \in \bA_{k}^{\mathrm{ord}}(\scrTs)} f \left(G(\scrTs, \bx)\right)
    \right)
    =
    \lim_{ \tildm \to \infty}
    \frac{1}{s_{0}^{k} \tildm^{k/2}}
    \bfE
    \left(
    \sum_{\bx \in \bA(\scrTs)^{k}} f \left(G(\scrTs, \bx)\right)
    \right).
  \end{equation}
  By the same calculation as \cite[Equation (8.4)]{Bhamidi_Sen_20_Geometry},
  we deduce that
  \begin{align}
     & \frac{1}{s_{0}^{k} \tildm^{k/2}}
    \bfE
    \left(
    \sum_{\bx \in \bA(\scrTs)^{k}} f \left(G(\scrTs, \bx)\right)
    \right) \notag                      \\
    =
     & \bfE
    \left(
    \prod_{j=1}^{k} \frac{| \bA(\scrTs, U_{j}^{\tildm}) |}{\sqrt{\tildm}}
    \int_{\scrTs}\! \cdots \int_{\scrTs}
    f \left( G \left( \scrTs, (U_{i}^{\tildm}, y_{i})_{i=1}^{k} \right)\right)
    \tilde{\mu}_{\scrTs}^{U_{1}^{\tildm}}(dy_{1}) \cdots \tilde{\mu}_{\scrTs}^{U_{k}^{\tildm}}(dy_{k})
    \right),
    \label{eq: expression via int wrt path measures}
  \end{align}
  where $(U_{i}^{\tildm})_{i=1}^{k}$ is defined as in Lemma \ref{lem: convergence of Ts with path measures CR models}.
  By Lemma \ref{lem: technical lemma for configuration models}\eqref{lem item: sup of A is bounded CR models},
  we have that
  \begin{align}
    \sup_{\tildm}
    \bfE
    \left(
    \prod_{j=1}^{k} \frac{|\bA(\scrTs, U_{j}^{\tildm})|}{\sqrt{\tildm}}
    \right)
     & =
    \sup_{\tildm}
    \bfE
    \left(
    \bfE
    \left(
      \left.
      \frac{|\bA(\scrTs, U_{1}^{\tildm})|}{\sqrt{\tildm}}
      \right| \scrTs
      \right)^{k}
    \right) \notag \\
     & \leq
    \sup_{\tildm}
    \bfE
    \left( \frac{|\bA(\scrTs, U_{1}^{\tildm})|}{\sqrt{\tildm}} \right)^{k}
    <\infty.
    \label{eq: sup of prod of ATs is bounded}
  \end{align}
  From \eqref{eq: expression via int wrt path measures},
  \eqref{eq: sup of prod of ATs is bounded} and Lemma \ref{lem: convergence of Ts with path measures CR models},
  it follows that
  \begin{align} \label{eq: convergence of the numerator CR models}
      &
    \lim_{\tildm \to \infty}
    \frac{1}{s_{0}^{k} \tildm^{k/2}}
    \bfE
    \left(
    \sum_{\bx \in \bA(\scrTs)^{k}} f \left(G(\scrTs, \bx)\right)
    \right) \\
    = &
    \left( \frac{p_{0} \sigma}{2} \right)^{k}
    \bfE
    \left(
    \prod_{j=1}^{k} \hght_{T_{2e}} (U_{j})
    \int_{T_{2e}}\! \cdots \int_{T_{2e}}
    f \left( G \left( T_{2e}, (U_{i}, y_{i})_{i=1}^{k} \right)\right)
    \mu_{T_{2e}}^{U_{1}}(dy_{1}) \cdots \mu_{T_{2e}}^{U_{k}}(dy_{k})
    \right),
  \end{align}
  where $(U_{i})_{i=1}^{k}$ is defined as in Lemma \ref{lem: convergence of Ts with path measures CR models}.
  By \eqref{eq: expectation expression by using the Ts CR models},
  \eqref{eq: convergence of bAord from Bhamidi and Sen},
  \eqref{eq: passing ord to product sets CR models} and
  \eqref{eq: convergence of the numerator CR models},
  we deduce that
  \begin{equation}
    \lim_{ \tildm \to \infty}
    \bfE \left( f (G(\scrTs, \tilde{\bm{Y}})) \right)
    =
    \bfE
    \left(
    f
    \left(
      \left( \tilde{T}^{(k)}, d_{ \tilde{T}^{(k)}}, \rho_{ \tilde{T}^{(k)}}, \mu_{ \tilde{T}^{(k)}},
        \tilde{U}_{1}, \tilde{V}_{1}, \ldots, \tilde{U}_{k}, \tilde{V}_{k} \right)
      \right)
    \right),
  \end{equation}
  where we recall the space
  $\left( \tilde{T}^{(k)}, d_{ \tilde{T}^{(k)}}, \rho_{ \tilde{T}^{(k)}}, \mu_{ \tilde{T}^{(k)}},
    \tilde{U}_{1}, \tilde{V}_{1}, \ldots, \tilde{U}_{k}, \tilde{V}_{k} \right)$
  from Construction \ref{const: the limiting space of CR models}.
  This yields that
  \begin{equation}
    \begin{split}
      &\left(
      \tildscrTs, \frac{\sigma}{\sqrt{\tildm}} d_{\tildscrTs}, \rho_{\tildTs}, \mu_{\tildTs},
      \hat{U}_{1}^{\tildm}, \gparent(\hat{V}_{1}^{\tildm}), \ldots,
      \hat{U}_{k}^{\tildm}, \gparent(\hat{V}_{k}^{\tildm})
      \right)\\
      \xrightarrow{\mathrm{d}}
      &\left( \tilde{T}^{(k)}, d_{ \tilde{T}^{(k)}}, \rho_{ \tilde{T}^{(k)}}, \mu_{ \tilde{T}^{(k)}},
      \tilde{U}_{1}, \tilde{V}_{1}, \ldots, \tilde{U}_{k}, \tilde{V}_{k} \right).
    \end{split}
  \end{equation}
  It then follows that
  \begin{equation}
    \begin{split}
      &\left(
      T_{\tildm}, \frac{\sigma}{\sqrt{\tildm}} d_{T_{\tildm}}, \rho_{T_{\tildm}}, \mu_{T_{\tildm}},
      \parent(\hat{U}_{1}^{\tildm}), \parent(\hat{V}_{1}^{\tildm}), \ldots,
      \parent(\hat{U}_{k}^{\tildm}), \parent(\hat{V}_{k}^{\tildm})
      \right)\\
      \xrightarrow{\mathrm{d}}
      &\left( \tilde{T}^{(k)}, d_{ \tilde{T}^{(k)}}, \rho_{ \tilde{T}^{(k)}}, \mu_{ \tilde{T}^{(k)}},
      \tilde{U}_{1}, \tilde{V}_{1}, \ldots, \tilde{U}_{k}, \tilde{V}_{k} \right).
    \end{split}
  \end{equation}
  Therefore, by \cite[Proposition 8.4]{Croydon_18_Scaling},
  we deduce that the space obtained by fusing
  $(T_{\tildm}, \sigma \tildm^{-1/2} d_{T_{\tildm}}, \rho_{T_{\tildm}}, \mu_{T_{\tildm}})$
  over $( \{ \parent(\hat{U}_{j}^{\tildm}), \parent(\hat{V}_{j}^{\tildm}) \} )_{j=1}^{k}$
  converges to $(M^{(k)}, d_{M^{(k)}}, \rho_{M^{(k)}}, \mu_{M^{(k)}})$.
  Since the Gromov-Hausdorff-Prohorov distance between the fused space and
  the space obtained by adding an edge between $\hat{U}_{j}^{\tildm}$ and $\hat{V}_{j}^{\tildm}$
  for each $j$ converges to $0$ in probability,
  we obtain the desired convergence.
\end{proof}

As in the case of the critical \erdosrenyi\ random graphs (see Section \ref{subsec: critical ER random grpah}),
the volume estimate of $\scrGdcon$ is deduced from that of tilted trees $\tildscrTs$.
In \cite[Proposition 3]{Marzouk_22_On_scaling},
the tightness of the sequence of the depth-first walk of $\scrTs$,
the uniform random tree of $\mbTs$,
in the H\"{o}lder norms was established,
and in \cite[Proposition 5]{Broutin_Marckert_14_Asymptotics},
it was proven
that the rescaled depth-first walk and the rescaled height function of $\scrTs$ are close.
Combining these results with Lemma  \ref{lem: technical lemma for configuration models},
we obtain the tightness of the sequence of the height functions of $\tildscrTs$,
which is enough for the volume estimate as we saw in Proposition \ref{prop: volume estimates for trees}.
The following assertions are versions of \cite[Proposition 3]{Marzouk_22_On_scaling}
and \cite[Proposition 5]{Broutin_Marckert_14_Asymptotics} modified for our setting.

\begin{lem} [{cf.\ \cite[Proposition 3]{Marzouk_22_On_scaling}}]
  \label{lem: Holder tightness of DFW of trees with prescribed degrees}
  Let $\bd=\bd^{(\tildm)}$ be degree sequences satisfying Assumption \ref{assum: degree sequence tildd}.
  Set $V_{\tildm} \coloneqq \tildm -1 +2k$,
  which is the number of vertices of a plane tree with the CFD $\bs=\bs(\bd^{(\tildm)})$.
  Then, for every $\varepsilon \in (0,1)$ and $\delta \in (0, 1/2)$,
  there exists a constant $C > 0$ such that,
  for all $\tildm$,
  \begin{equation}
    | W_{\tildm} ( \lfloor V_{\tildm} s \rfloor) - W_{\tildm} ( \lfloor V_{\tildm} t \rfloor) |
    \leq
    C \sqrt{V_{\tildm}} |t - s|^{\delta}
  \end{equation}
  holds uniformly for $0 \leq s \leq t \leq 1$
  with probability at least $1-\varepsilon$,
  where $(W_{\tildm}(i))_{i=0}^{V_{\tildm}}$ is the depth-first walk of $\scrTs$.
\end{lem}

\begin{lem} [{cf. \cite[Proposition 5]{Broutin_Marckert_14_Asymptotics}}]
  \label{lem: DFW and HF are close, CR models}
  Assume that we are
  in the same setting as Lemma \ref{lem: Holder tightness of DFW of trees with prescribed degrees}.
  Write $(H_{\tildm}(i))_{i=0}^{V_{\tildm -1}}$ for the height function of $\scrTs$.
  We extend the domain of $W_{\tildm}$ and $H_{\tildm}$
  to $[0,V_{\tildm}]$ and $[0,V_{\tildm}-1]$ by linear interpolation, respectively.
  Then, there exists $c_{\tildm} = o (\sqrt{\tildm})$ such that
  \begin{equation}
    \bfP
    \left(
    \sup_{0 \leq t \leq 1}
    \left|
    W_{\tildm}(t V_{\tildm}) - \frac{\sigma_{\tildm}^{2}}{2} H_{\tildm} (t (V_{\tildm} -1))
    \right|
    \geq c_{\tildm}
    \right)
    \to 0,
  \end{equation}
  where we set $\sigma_{\tildm}^{2} \coloneqq (V_{\tildm}-1)^{-1} \sum_{i} i^{2} s_{i} -1$.
  (Note that, in our setting, we have $\sigma_{\tildm}^{2} \to \sigma^{2}$.)
\end{lem}

\begin{prop}  \label{prop: volume estimate of Gdcon}
  Let $\bd=\bd^{(\tildm)}$ be degree sequences satisfying Assumption \ref{assum: degree sequence tildd}.
  Then, for every $\varepsilon>0$ and $\delta \in (0,1/2)$,
  there exists a constant $C_{\varepsilon,\delta}>0$ such that
  \begin{equation}
    \liminf_{\tildm \to \infty}
    \bfP
    \left(
    \inf_{x \in V(\scrGdcon)}
    \mu_{\scrGdcon} \left(D_{R_{\scrGdcon}} (x, \sqrt{\tildm} r) \right)
    \geq
    (C_{\varepsilon, \delta} r^{1/\delta} ) \wedge 1, \quad
    \forall r \geq 0
    \right)
    \geq 1-\varepsilon.
  \end{equation}
\end{prop}

\begin{proof}
  Let $\bs=\bs ( \bd^{\tildm} )$ be the CFD assciated with $\bd^{\tildm}$.
  Set $V_{\tildm} \coloneqq \tildm - 1 + 2k$,
  where $k$ is the non-negative number
  in Assumption \ref{assum: degree sequence tildd}\eqref{assum item: surplus of the given degree sequence, tilted}.
  Let $c_{\tildm}=o (\sqrt{\tildm})$ be the constant in Lemma \ref{lem: DFW and HF are close, CR models}.
  From Lemma \ref{lem: Holder tightness of DFW of trees with prescribed degrees}
  and Lemma \ref{lem: DFW and HF are close, CR models},
  we deduce the tightness of the sequence of the height functions of $\scrTs$
  in the H\"{o}lder norms.
  Following \cite[Proof of Lemma 4.15]{Andriopoulos_23_Convergence},
  we obtain the tightness of the sequence of the height functions of $\tildscrTs$
  in the H\"{o}lder norms
  (recall the tilted random tree $\tildscrTs$ from Construction \ref{const: uniformly random graph on Gdcon}).
  Namely,
  for every $\varepsilon \in (0,1)$ and $\delta \in (0,1/2)$,
  one can find a constant $C >0$ such that
  \begin{equation}  \label{eq: holder continuity of height function of tilted trees, CR models}
    \liminf_{\tildm \to \infty}
    \bfP
    \left(
    \left| \tilde{H}_{\tildm}( (V_{\tildm}-1) t ) - \tilde{H}_{\tildm}( (V_{\tildm} -1) s ) \right|
    \leq
    C \sqrt{\tildm} |t - s|^{\delta} + c_{\tildm},
    \quad
    \forall t, s \in [0,1]
    \right)
    \geq 1-\varepsilon,
  \end{equation}
  where $\tilde{H}_{\tildm}$ is the height function of $\tildscrTs$
  regarded as a function of $C([0,V_{\tildm}-1], \mbR)$ by linear interpolation.
  Let $\tilde{U}_{j}^{\tildm}, \tilde{V}_{j}^{\tildm},\, j=1,2,\ldots, \tildm$
  be the vertices of $\scrTs$ as in Construction \ref{const: uniformly random graph on Gdcon}.
  Define $\hat{\mathscr{T}}_{\bs}$ to be the plane tree
  obtained by deleting $\tilde{U}_{j}^{\tildm}, \tilde{V}_{j}^{\tildm},\, j=1,2, \ldots, \tildm$ from $\tildscrTs$
  and adding a vertex $\rho_{\hat{\mathscr{T}}_{\bs}}$ to the root of $\tildscrTs$
  as a new root.
  Let $\hat{H}_{\tildm} = (\hat{H}_{\tildm}(t), 0 \leq t \leq \tildm)$
  be the height function of $\tilde{\mathscr{T}}_{\bs}$.
  Then, it is observed
  that a similar result to \eqref{eq: holder continuity of height function of tilted trees, CR models}
  holds.
  Since adding an edge between $\parent(\tilde{U}_{j}^{\tildm})$ and $\parent(\tilde{V}_{j}^{\tildm})$
  for each $j$ yields $\scrGdcon$,
  the desired result is deduced
  by a similar argument to the proof of Proposition \ref{prop: volume estimates for trees}.
\end{proof}

\begin{rem}
  In \cite[Lemma 4.15]{Andriopoulos_23_Convergence},
  the tightness of the contour functions of $\tildscrTs$ in a H\"{o}lder norm
  is shown.
  However,
  it is not easy to remove the conditioning given at the end of the proof,
  and the proof appears to need revision.
\end{rem}

Finally, we give a proof of Theorem \ref{thm: main result for CR models}.

\begin{const} \label{const: limiting space of CR models}
  Write
  \begin{equation}
    \mathcal{M}^{(k)}
    \coloneqq
    (\Mk, R_{\Mk}, \rho_{\Mk}, \mu_{\Mk})
  \end{equation}
  (recall this space from Construction \ref{const: the limiting space of CR models}).
  Sample $(\bm{Z}, \bm{S})$,
  the sequence of excursion lengths and the number of points that fell in the excursions,
  and assume that $(\bm{Z}, \bm{S})$ and  $(\mathcal{M}^{(k)}, k \geq 0)$ are independent.
  We then define
  \begin{equation}  \label{eq: definition of the limit of configuration models}
    \mathcal{M}
    \coloneqq
    \left(
    M^{(S_{1})}, \frac{\alpha \sqrt{Z_{1}}}{\sqrt{\eta}} R_{M^{(S_{1})}}, \rho_{M^{(S_{1})}}, \mu_{M^{(S_{1})}}
    \right),
  \end{equation}
  where we recall the parameters from \eqref{eq: parameters for CR models}.
\end{const}

\begin{lem}  [{\cite[Proposition 9.2]{Bhamidi_Sen_20_Geometry}}]
  \label{lem: behavior of degree sequences of CR models}
  Assume that degree sequences $\bd=\bd^{(n)}$ satisfy Assumption \ref{assum: degree sequence of configuration models}.
  Let $D^{\circ}$ be the size-biased random variable corresponding to $D$, i.e.,
  \begin{equation}
    p_{i}^{\circ}
    \coloneqq
    P(D^{\circ}=i)
    =
    \frac{i P(D =i)}{E(D)}, \quad
    i=1,2, \ldots.
  \end{equation}
  Then, for each $i \geq 1$,
  the following assertions hold:
  \begin{gather}
    \frac{1}{| V(\mathscr{C}_{i}^{n}) |}
    \sum_{j \in V(\mathscr{C}_{i}^{n})} d_{j}^{2}
    \xrightarrow{\mathrm{p}}
    \sum_{j \geq 1} j^{2} p_{j}^{\circ} < \infty,\\
    \bfP ( \mathscr{C}_{i}^{n}\ \text{is simple}) \to 1,\\
    \frac{1}{| V(\mathscr{C}_{i}^{n}) |} | \{j \in V(\mathscr{C}_{i}^{n}) : d_{j} = l\} |
    \xrightarrow{\mathrm{p}}
    p_{l}^{\circ}\
    \text{for}\  l \geq 1.
  \end{gather}
\end{lem}

\begin{proof} [Proof of Theorem \ref{thm: main result for CR models}]
  Fix a subset $V \subseteq [n]$.
  We rearrange $(d_{j}^{(n)}, j \in V)$
  and write $\btildd(V)=\btildd^{(n)}(V) = (\tildd_{j}^{(n)} , 1\leq j \leq |V|)$
  so that $\tildd_{1}^{(n)} \leq \tildd_{2}^{(n)} \leq \cdots \leq \tildd_{|V|}^{(n)}$.
  Let $\Pi$ be the set of degree sequences $\bd'$.
  Sample $\mathscr{G}_{\bd'}^{\mathrm{con}},\, \bd' \in \Pi$ independently,
  and assume that $(\mathscr{G}_{\bd'}, \bd' \in \Pi)$ is independent of $\mathscr{C}_{1}^{n}$.
  Then, conditional on being simple,
  $\mathscr{C}_{1}^{n}$ has the same distribution as $\mathscr{G}_{\btildd(V(\mathscr{C}_{1}^{n}))}^{\mathrm{con}}$
  up to unimportant relabelling
  (cf.\ \cite[Proposition 7.7]{Hofstad_17_Random}).
  By Lemma \ref{lem: behavior of degree sequences of CR models},
  it is observed
  that the degree sequences $\bd(V(\mathscr{C}_{1}^{n}))$ satisfy Assumption \ref{assum: degree sequence tildd}.
  Thus the desired convergence is deduced
  from the corresponding convergence of  $\mathscr{G}_{\btildd(V(\mathscr{C}_{1}^{n}))}^{\mathrm{con}}$,
  which follows
  from Proposition \ref{prop: convergence of Gdcon} and Proposition \ref{prop: volume estimate of Gdcon}.
\end{proof}

\section*{Appendix}
\addcontentsline{toc}{section}{Appendix}  

\setcounter{section}{0}
\setcounter{subsection}{0}

\renewcommand \thesubsection {\Alph{subsection}}

\setcounter{equation}{0}
\renewcommand{\theequation}{\Alph{subsection}.\arabic{equation}}

\setcounter{exm}{0}
\renewcommand \theexm {\Alph{subsection}.\arabic{exm}}

\subsection{Convergence of Gaussian processes}
\label{application to Gaussian processes}

\newcommand{\Hcov}{\mathbb{H}_{\mathrm{cov}}}
\newcommand{\Hpr}{\mathbb{H}_{\mathrm{pr}}}
\newcommand{\dcov}{d_{\mathrm{cov}}}
\newcommand{\taucov}{\tau^{\mathrm{cov}}}
\newcommand{\taupr}{\tau^{\mathrm{pr}}}

In this appendix, we deduce convergence of Gaussian processes from the metric-entropy condition.
First, we introduce spaces $\Hcov$ and $\Hpr$ to state the main result
(see Theorem \ref{A. thm: convergence of Gaussian processes} below).

\begin{dfn} [{The sets $\hatC_{c}(M, \mbR)$}]
  Let $(M, d)$ be a compact metric space.
  We define 
  \begin{equation}
    \hatC_{c} (M, \mbR) 
    \coloneqq 
    \bigcup_{X \in \cC(F)} 
    C(X, \mbR).
  \end{equation} 
  Note that $\hatC(M, \mbR)$ contains 
  the empty map $\emptyset_{\mbR}: \emptyset \to \mbR$.
  For each $f \in \hatC(M, \mbR)$,
  if $f \in C(X, \mbR)$,
  then we write $\dom(f) \coloneqq X$.
\end{dfn}

\begin{dfn} [{The metric $d_{\hatC_{c}, M}$}]
  For $f_{1}, f_{2} \in \hatC_{c}(M, \mbR)$ and $\varepsilon>0$, 
  consider the following condition.
  \begin{enumerate} [label = (H), leftmargin = *]
    \item \label{2. dfn item: epsilon condition for metric on hatC_c for gaussian}
      For any $x \in \dom(f_{1})$, 
      there exists an element $y \in \dom(f_{2})$ such that 
      \begin{equation}
        d(x, y)
        \vee 
        |f_{1}(x) - f_{2}(y)| 
        \leq 
        \varepsilon.
      \end{equation}
      Similarly,
      for any $y \in \dom(f_{2})$, 
      there exists an element $x \in \dom(f_{1})$ such that 
      the above inequality holds.
  \end{enumerate}
  We then define
  \begin{equation}
    d_{\hatC_{c}, M}(f_{1}, f_{2}) 
    \coloneqq
    \left(
      \inf 
      \{
        \varepsilon > 0 \mid \varepsilon\ \text{satisfies \ref{2. dfn item: epsilon condition for metric on hatC_c for gaussian}}
      \}
    \right)
    \wedge 
    1,
  \end{equation}
  where the infimum over the empty set is defined to be $\infty$.
\end{dfn}

The function $d_{\hatC_{c}, M}$ is a separable metric on $\hatC_{c}(M, \mbR)$
and the induced topology is Polish.
(NB. the metric $d_{\hatC_{c}, M}$ is not necessarily complete.)
For details, see \cite[Section 2.2.3]{Noda_pre_Metrization}.

Let $\Hcov^{\circ}$ be the collection of $(F, d, \Sigma)$
such that $(F,d)$ is a compact metric space and $\Sigma$ is an element of $\hat{C}(F \times F, \mbR)$.
Given a distance-preserving map $f:F_{1} \to F_{2}$ where 
$F_{1}$ and $F_{2}$ are compact metric spaces,
we define $\taucov_{f}: \hatC(F_{1} \times F_{1}, \mbR) \to \hatC(F_{2} \times F_{2}, \mbR)$
by setting 
\begin{equation}
  \taucov_{f}(\Sigma) 
  \coloneqq 
  \Sigma \circ (f \times f)^{-1},
  \quad 
  \Sigma \in \hatC(F_{1} \times F_{1}, \mbR).
\end{equation}
For $S_{i}=(F_{i}, d_{i}, \Sigma_{i}) \in \Hcov^{\circ},\, i=1,2$, 
we say that $S_{1}$ is $\taucov$-equivalent to $S_{2}$
if and only if 
there exists an isometry $f: F_{1} \to F_{2}$ 
satisfying $\taucov_{f}(\Sigma_{1}) = \Sigma_{2}$.

\begin{dfn} [{The set $\Hcov$}]
  We define $\Hcov$ 
  to be the collection of $\taucov$-equivalence classes of elements $\Hcov^{\circ}$.
\end{dfn}

\begin{dfn} [{The metric $d_{\Hcov}$}]
  For $S_{i}=(F_{i}, d_{i}, \Sigma_{i}) \in \Hcov,\, i=1,2$,
  we define 
  \begin{equation}
      d_{\Hcov}(S_{1}, S_{2})
      \coloneqq 
      \inf_{f_{1}, f_{2}, M}
      \left(
        d_{H}(f_{1}(F_{1}), f_{2}(F_{2}))
        \vee 
        d_{\hatC_{c}, M \times M}( \taucov_{f_{1}}(\Sigma_{1}), \taucov_{f_{2}}(\Sigma_{1}) )
      \right)
  \end{equation}
  where the infimum is taken over all compact metric spaces $(M,d)$
  and all distance-preserving maps $f_{i}:F_{i} \to M,\, i=1,2$.
\end{dfn}

We define the set $\Hpr^{\circ}$ to be the collection of $(F, d, \Sigma, P)$ 
such that $(F,d, \Sigma)$ is an element of $\Hcov$ and $P$ is an element of $\mathcal{P}(\hatC(F, \mbR))$,
i.e., a probability measure on $\hatC(F, \mbR)$. 
Given a distance-preserving map $f:F_{1} \to F_{2}$ where 
$F_{1}$ and $F_{2}$ are compact metric spaces,
we define $\taupr_{f}: \hatC(F_{1}, \mbR) \to \hatC(F_{2}, \mbR)$
by setting 
\begin{equation}
  \taupr_{f}(G) 
  \coloneqq 
  G \circ f^{-1},
  \quad 
  G \in \hatC(F_{1}, \mbR).
\end{equation}
For $\cX_{i}=(F_{i}, d_{i}, \Sigma_{i}, P_{i}) \in \Hpr,\, i=1,2$, 
we say that $\cX_{1}$ is $\taucov \times (\taupr)^{-1}$-equivalent to $\cX_{2}$
if and only if 
there exists an isometry $f: F_{1} \to F_{2}$
such that $\taucov_{f}(\Sigma_{1}) = \Sigma_{2}$ and $P_{2} = P_{1} \circ (\taupr_{f})^{-1}$.

\begin{dfn} [{The set $\Hpr$}]
  We define $\Hpr$ 
  to be the collection of $\taucov \times (\taupr)^{-1}$-equivalence classes of elements $\Hpr^{\circ}$.
\end{dfn}

\begin{dfn}
  For $\cX_{i}=(F_{i}, d_{i}, \Sigma_{i}, P_{i}) \in \Hpr,\, i=1,2$,
  we define 
  \begin{equation}
    \begin{split}
      d_{\Hpr}(\cX_{1}, \cX_{2})
      \coloneqq 
      \inf_{f_{1}, f_{2}, M}
      \bigl(
        &d_{H}(f_{1}(F_{1}), f_{2}(F_{2}))
        \vee 
        d_{\hatC_{c}, M \times M}( \taucov_{f_{1}}(\Sigma_{1}), \taucov_{f_{2}}(\Sigma_{1}) )
      \bigr.\\
      &\quad 
      \bigl.
        \vee
        \tilde{d}_{P}( P_{1} \circ (\taupr_{f_{1}})^{-1}, P_{2} \circ (\taupr_{f_{2}})^{-1} )
      \bigr),
    \end{split}
  \end{equation}
  where the infimum is taken over all compact metric spaces $(M,d)$
  and all distance-preserving maps $f_{i}:F_{i} \to M,\, i=1,2$,
  and $\tilde{d}_{P}$ denotes the Prohorov metric on $\mathcal{P}(\hatC(M, \mbR))$.
\end{dfn}
 
The set $\Hcov$ will be the space for covariance functions and $\Hpr$ will be for Gaussian processes. 
Like the metric space $(\mbM_{L}, d_{\mbM_{L}})$, 
we obtain the following results
by the corresponding results of \cite{Noda_pre_Metrization}
(the space $(\mbM_{L}, d_{\mbM_{L}})$ is defined in Section \ref{sec: the space M_L}).
We mention here that 
although in \cite{Noda_pre_Metrization} spaces are assumed to be boundedly compact,
it is easy to deduce the corresponding results for compact spaces.

\begin{thm} [{cf.\ \cite[Theorem 3.23]{Noda_pre_Metrization}}]
  The functions $d_{\Hcov}$ and $d_{\Hpr}$ are metrics on $\Hcov$ and $\Hpr$ respectively. 
  Moreover, both spaces are Polish
  (but not necessarily with $d_{\Hcov}$ or $d_{\Hpr}$).
\end{thm}

\begin{thm} [{cf.\ \cite[Theorem 3.24]{Noda_pre_Metrization}}]\label{A. thm: convergence in H_pr}
  For each $n \in \mathbb{N} \cup \{\infty\}$,
  let $\cX_{n} = (F_{n}, d_{n}, \Sigma_{n}, P_{n})$ be an element of $\Hpr$.
  The elements $\cX_{n}$ converges to $\cX_{\infty}$ in $\Hpr$
  if and only if 
  there exist a compact metric space $(M, d)$ 
  and distance-preserving maps $f_{n}: F_{n} \to M$
  and $f_{\infty}: F_{\infty} \to M$
  such that 
  $f_{n}(F_{n}) \to f_{\infty}(F_{\infty})$ in the Hausdorff topology in $M$,
  $\taucov_{f_{n}}(\Sigma_{n}) \to \taucov_{f_{\infty}}(\Sigma_{\infty})$ in $\hatC(M, \mbR)$
  and $P_{n} \circ (\taupr_{f_{n}})^{-1} \to P_{\infty} \circ (\taupr_{f_{\infty}})^{-1}$
  weakly as probability measures on $\hatC(M, \mbR)$.
  (A similar result holds for the space $\Hcov$.)
\end{thm}

\begin{thm} [{cf.\ Theorem \ref{2. thm: precompactness in M_L}}] \label{A. thm: precompactness in H_pr}
  For each $n \in \mathbb{N}$,
  let $\cX_{n} = (F_{n}, d_{n}, \Sigma_{n}, P_{n})$ be an element of $\Hpr$.
  Choose a random element $G_{n}$ of $\hatC(F, \mbR)$ 
  whose law coincides with $P_{n}$.
  We denote the underlying probability measure by $P$.
  The sequence $(\cX_{n})_{n \geq 1}$ is precompact in $\Hpr$
  if and only if the following conditions are satisfied.
  \begin{enumerate} [label = (\roman*)]
    \item 
      The sequence $(F_{n}, d_{n}, \Sigma_{n})$ is precompact in $\Hcov$.
    \item 
      It holds that 
      $\displaystyle \lim_{M \to \infty} 
      \limsup_{n \to \infty}
      P
      \left(
        \sup_{x \in \dom(G_{n})} |G_{n}(x)| > M
      \right)
      =0$.
    \item 
      For every $\varepsilon >0$,
      it holds that 
      $\displaystyle 
      \lim_{\delta \to 0} 
      \limsup_{n \to \infty}
      P
      \left(
        \sup_{\substack{x, y \in \dom(G_{n}),\\ d_{n}(x,y) < \delta}} 
        |G_{n}(x) - G_{n}(y)| > \varepsilon
      \right)
      =0$.
  \end{enumerate}
\end{thm}

\begin{rem}
  One can obtain a precompactness criterion for the space $\Hcov$ 
  similarly to \cite[Theorem 4.30]{Noda_pre_Metrization}.
\end{rem}

For a positive definite function $\Sigma \in C(F \times F, \mbR)$,
we write $G_{\Sigma} = (G_{\Sigma}(x))_{x \in F}$ for a mean-zero Gaussian process
with covariance function $\Sigma$ 
built on a probability space with probability measure $P_{\Sigma}$
and we define 
\begin{equation}
  d_{\Sigma}(x, y)
  \coloneqq 
  \sqrt{E_{\Sigma}( (G_{\Sigma}(x)-G_{\Sigma}(y))^{2} )}
  =
  \sqrt{\Sigma(x,x) + \Sigma(y,y) - 2 \Sigma(x, y)},
  \quad 
  x,y \in F.
\end{equation}
Note that $d_{\Sigma}$ is a pseudometric on $F$,
that is,
$d_{\Sigma}(x,y) = 0$ does not necessarily imply $x = y$.
Let $\Hcov^{\dagger}$ be a subset of $\Hcov$ 
consisting of $(F, d, \Sigma)$ such that 
$\Sigma$ is a positive definite function defined on $F \times F$
and $d= d_{\Sigma}$.
We then define $\check{\mathbb{H}}_{\mathrm{cov}}^{\dagger}$
to be the set of $(F, d_{\Sigma}, \Sigma) \in \Hcov^{\dagger}$
such that  
\begin{equation} \label{A. eq:int-type condition for continuity}
  \int_{0}^{\infty}
  \sqrt{\log N_{d_{\Sigma}}(F, u)}\,
  du
  < \infty.
\end{equation}
By \cite[Theorem 6.1.2]{Marcus_Rosen_06_Markov},
we may assume that the Gaussian process $G_{\Sigma}$ is continuous on $(F,d_{\Sigma})$ almost-surely. 
For $S = (F, d_{\Sigma}, \Sigma)$,
we write $\cX_{S} = (F, d_{\Sigma}, \Sigma, P_{\Sigma}(G_{\Sigma} \in \cdot ))$,
which is an element of $\Hpr$.

\begin{lem} [{\cite[Theorem 6.1.2]{Marcus_Rosen_06_Markov}}]
  \label{A. lem: for precompactness of Gaussian processes}
  Let $(F, d_{\Sigma}, \Sigma)$ be an element of $\Hcov^{\dagger}$.
  There exists a universal constant $c>0$ such that 
  \begin{align}
    E_{\Sigma}
    \left(
      \sup_{x \in F} |G(x)| 
    \right)
    \leq 
    c 
    \int_{0}^{\infty}
    \sqrt{\log N_{d_{\Sigma}}(F, u)}\,
    du,\\
    E_{\Sigma}\left(
      \sup_{\substack{x, y \in F,\\ d_{\Sigma}(x, y) < \delta}} 
      |G_{\Sigma}(x) - G_{\Sigma}(y)|
    \right)
    \leq 
    c 
    \int_{0}^{\delta}
    \sqrt{\log N_{d_{\Sigma}}(F, u)}\,
    du.
  \end{align}
\end{lem}

\begin{thm} \label{A. thm: convergence of Gaussian processes}
  For each $n \in \mathbb{N}$,
  let $S_{n}=(F_{n}, d_{\Sigma_{n}}, \Sigma_{n})$ be an element of $\check{\mathbb{H}}_{\mathrm{cov}}^{\dagger}$.
  We assume the following conditions.
  \begin{enumerate} [label = (G\arabic*), leftmargin = *]
    \item \label{A. thm item: convergence of GP, covariance convergence}
      There exists $S = (F, d_{\Sigma}, \Sigma) \in \Hcov^{\dagger}$
      such that $d_{\Hcov}(S_{n}, S) \to 0$.
    \item \label{A. thm item: convergence of GP, metric-entropy condition}
      It holds that 
      $\displaystyle
        \lim_{\delta \downarrow 0}
        \limsup_{n \to \infty}
        \int_{0}^{\delta}
        \sqrt{\log N_{d_{\Sigma_{n}}} (F_{n}, u)}\,
        du
        =0.
      $
  \end{enumerate}
  Then, $S$ belongs to $\check{\mathbb{H}}_{\mathrm{cov}}^{\dagger}$
  and $\cX_{S_{n}}$ converges to $\cX_{S}$ in $\Hpr$.
 \end{thm}

\begin{proof}
  By \cite[Theorem 3.12]{Noda_pre_Metrization} and Fatou's lemma,
  we deduce that 
  \begin{equation}
    \int_{0}^{\infty}
    \sqrt{N_{d_{\Sigma}}(F, u)}\,
    du
    \leq 
    \liminf_{n \to \infty}
    \int_{0}^{\infty}
    \sqrt{N_{d_{\Sigma_{n}}}(F_{n}, u)}\,
    du.
  \end{equation}
  This, combined with \ref{A. thm item: convergence of GP, metric-entropy condition},
  implies that $S$ belongs to $\check{\mathbb{H}}_{\mathrm{cov}}^{\dagger}$.

  It remains to show the convergence result. 
  By Lemma \ref{A. lem: for precompactness of Gaussian processes},
  Markov's inequality and Theorem \ref{A. thm: precompactness in H_pr},
  we deduce that the sequence $(\cX_{S_{n}})_{n \geq 1}$ is precompact.
  Assume that a subsequence $(\cX_{S_{n_{k}}})_{k \geq 1}$ converges to some $\cX \in \Hpr$.
  Note that,
  by \ref{A. thm item: convergence of GP, covariance convergence},
  we can write $\cX = (F, d_{\Sigma}, \Sigma, Q)$.
  Using Theorem \ref{A. thm: convergence in H_pr},
  we may assume that 
  $(F_{n_{k}}, d_{\Sigma_{n_{k}}})$ and $(F, d_{\Sigma})$ are embedded into a common compact metric space $(M, d)$
  in such a way that $F_{n} \to F$ in the Hausdorff topology in $M$,
  $\Sigma_{n} \to \Sigma$ in $\hatC(M \times M, \mbR)$
  and $P_{\Sigma_{n_{k}}}(G_{\Sigma_{n_{k}}} \in \cdot) \to Q$
  as probability measures on $\hatC(M, \mbR)$.
  Let $G$ be a random element of $\hatC(M, \mbR)$ whose law coincides with $Q$.
  By the Skorohod representation,
  we may assume that $G_{\Sigma_{n_{k}}} \to G$ almost-surely
  on some probability space with probability measure $P$.
  Since $\dom(G_{\Sigma_{n_{k}}}) = F_{n_{k}} \to F$,
  we have that $\dom(G) = F$ with probability $1$.
  Hence, we can regard $G$ as a random element of $C(F, \mbRp)$.
  Fix a finite subset $\{x_{i}\}_{i=1}^{N}$ of $F$.
  For each $n$,
  choose a finite subset $\{x_{i}^{(n_{k})}\}_{i = 1}^{N}$ of $F_{n}$
  such that $x_{i}^{(n_{k})} \to x_{i}$ in $M$ for each $i$.
  The convergence $\Sigma_{n_{k}} \to \Sigma$ implies that 
  $\Sigma_{n}(x_{i}^{(n_{k})}, x_{j}^{(n_{k})}) \to \Sigma(x_{i}, x_{j})$
  for each $i,j$.
  This yields that 
  $(G_{\Sigma_{n_{k}}}(x_{i}^{(n_{k})}))_{i=1}^{N} \xrightarrow{\mathrm{d}} (G_{\Sigma}(x_{i}))_{i=1}^{N}$
  as random elements of $\mbR^{N}$.
  On the other hand,
  the almost-sure convergence $G_{\Sigma_{n_{k}}} \to G$ implies that 
  $(G_{\Sigma_{n_{k}}}(x_{i}^{(n_{k})}))_{i=1}^{N} \to (G(x_{i}))_{i=1}^{N}$
  in $\mbR^{N}$ almost-surely. 
  Hence,
  we deduce that $G_{\Sigma} \stackrel{\mathrm{d}}{=} G$
  as random elements of $C(F, \mbR)$,
  which completes the proof.
\end{proof}

\section*{Acknowledgement}
I would like to thank my supervisor Dr David Croydon for leading me to the problem, 
his support and fruitful discussions. 
This work was supported by 
JSPS KAKENHI Grant Number JP 24KJ1447
and 
the Research Institute for Mathematical Sciences, 
an International Joint Usage/Research Center located in Kyoto University.
\appendix

\bibliographystyle{amsplain}
\bibliography{convergence_of_local_times}
\end{document}